\documentclass{amsbook}

\usepackage{amssymb,amsmath, amsthm,mathrsfs,enumitem,ulem,cancel}
\usepackage[at]{easylist}
\usepackage{mathscinet,amsxtra}
\usepackage[all,cmtip]{xy}
\usepackage{tikz}
\usetikzlibrary{patterns}
\usepackage{hyperref}
\newtheorem{theorem}{Theorem}[chapter]
\newtheorem{lemma}[theorem]{Lemma}

\newtheorem{proposition}[theorem]{Proposition}%ADDED

\newtheorem{corollary}[theorem]{Corollary}%ADDED

\newtheorem{conjecture}[theorem]{Conjecture}%ADDED

\theoremstyle{definition}
\newtheorem{definition}[theorem]{Definition}

\theoremstyle{remark}
\newtheorem{remark}[theorem]{Remark}

\numberwithin{section}{chapter}
\numberwithin{equation}{chapter}

\newcommand\Z{\mathbb Z}
\newcommand\R{\mathbb R}
\newcommand\C{\mathbb C}
\newcommand\e{\epsilon}
\newcommand{\Irr}{\operatorname{Irr}}
\newcommand{\GL}{\operatorname{GL}}
\newcommand{\Cusp}{\mathcal{C}}
\newcommand{\Disc}{\mathcal{D}}
\newcommand{\Cuspsd}{\Cusp^{sd}}
\newcommand{\Cuspcl}{\Cusp^{cl}}
\newcommand{\Irrcl}{\Irr^{cl}}
\newcommand{\Discl}{\mathcal{D}^{cl}}
\newcommand{\Tempcl}{\mathcal{T}^{cl}}
\newcommand{\spsi}{\operatorname{s.p.}}
\newcommand{\Alg}{\operatorname{Alg}}
\newcommand{\ms}{\mathfrak{s}}
\newcommand{\fl}{\operatorname{f.l.}}
\newcommand{\soc}{\operatorname{soc}}
\newcommand{\ssm}{\operatorname{s.s.}}
\newcommand{\reg}{\operatorname{reg}}

\newcommand{\shrt}{\operatorname{sla}}
\newcommand{\lng}{\operatorname{lev}}
\newcommand{\abs}[1]{\left|{#1}\right|}
\newcommand{\regshrt}{\operatorname{reg-sla}}
\newcommand{\vregshrt}{\operatorname{v-reg-sla}}
\newcommand{\reglng}{\operatorname{reg-lev}}
\newcommand{\vreglng}{\operatorname{v-reg-lev}}

\newcommand\jrp{\Jord_\rho(\pi)}
\newcommand{\Jord}{\operatorname{Jord}}
\newcommand\h{\hookrightarrow}
\newcommand\ra{\rightarrow}
\newcommand\tha{\twoheadrightarrow}
\newcommand{\Ker}{\operatorname{Ker}}
\newcommand{\Img}{\operatorname{Im}}
\newcommand{\ASS}{DL{ }}
\newcommand{\Ind}{\operatorname{Ind}}
\newcommand\g{\gamma}
\newcommand\D{\Delta}
\newcommand\G{\Gamma}
\newcommand\s{\sigma}
\newcommand\SC{ \mathcal S(\Cusp)}
\newcommand{\JH}{\operatorname{JH}}
\newcommand{\Hom}{\operatorname{Hom}}
\newcommand\supp{\operatorname{supp}}
\begin{document}

\frontmatter

\title[Unitarizability in Corank Three]{Unitarizability in Corank Three for Classical $p$-adic Groups}

\author{Marko Tadi\'c}
\address{Department of Mathematics, University of Zagreb, Bijeni\v{c}ka 30, 10000 Zagreb, Croatia}
\email{\tt tadic{\char'100}math.hr}
\urladdr{http://www.hazu.hr/~tadic/}
\thanks{This work has been supported by Croatian Science Foundation under the project IP-2018-01-3628.}

%  \date is required; it is the date received by the editor.
\date{}



\dedicatory{In memory of my parents}

\begin{abstract}
Let $G$ be the $F$-points of a classical group defined over a $p$-adic field $F$ of characteristic $0$.
We classify the irreducible unitarizable representation of $G$ that are subquotients of the parabolic induction
of cuspidal representations of Levi subgroup of corank at most 3 in $G$.
\end{abstract}

\maketitle

\tableofcontents

\mainmatter

\chapter{Introduction}
\label{intro}

One of the fundamental questions in harmonic analysis is the classification of the unitary dual of a locally compact group.
An important class of locally compact groups is the $F$-points of reductive groups over a local field $F$.
We will only consider $p$-adic groups in this paper.
The history of this problem goes back at least to \cite{MR0023246} and \cite{MR0046370}.\footnote{Note that these papers were
published well before \cite{MR0093558}, which is usually considered the beginning of the representation theory of reductive $p$-adic groups.}

The classification of the unitary dual for general linear groups in the $p$-adic case was solved in \cite{MR870688}.
(The archimedean case can be solved along the same lines -- see \cite{TadicBonn1985} and \cite{MR2537046}.)
An important input is Bernstein's result \cite{MR748505} on irreducibility of unitary parabolic induction.
Subsequently, the classification of the unitary dual of inner forms of the general linear group was obtained
(\cite{MR2055385} and \cite{MR2492994}; see also \cite{MR1040995}).
A unified and simplified proof was recently discovered by E. Lapid and A. M\'inguez \cite{MR3573961} (see also \cite{MR3269346}).
Their point of departure is the reducibility points for parabolic induction of cuspidal representations of general linear groups,
which in the split case are always $\pm1$.
This also gives a uniform and simplified approach for the classification of the admissible dual
(see \cite[Appendix]{MR3573961} and \cite{MR3456591}).

Little is known in general about the classification of the unitary dual in case of (other) classical groups, except for some important subclasses
of unitary representations, such as generic representation \cite{MR2046512} or spherical representations \cite{MR2767523}.

A natural question is whether one can attack the unitarizability problem for the classical groups
using the cuspidal reducibility points as the starting point (as is the case for the general linear groups).
Such an approach was proposed in \cite{MR3969882}.
The main goal of this paper is carry out this approach in the corank (at most) 3 case, namely to classify,
irreducible unitarizable subquotients of representations Ind$_P^G(\tau)$, where $G$ is a classical group over a $p$-adic field of
characteristic zero, $P$ is a parabolic subgroup of $G$ of corank (at most) 3 and $\tau$ is an irreducible cuspidal
representation of a Levi factor $M$ of $P$.

The first order of business is to understand cuspidal reducibility for classical groups.
This was done by C. M\oe glin in terms of the local Langlands correspondence, which is a consequence of
Arthur's endoscopic classification (\cite{MR3135650}, which relies on \cite{MR3823813} and \cite{MR3823814} among other things).
The bottom line is that the points of cuspidal reducibilities are half-integers
\[
0,\pm\tfrac12,1,\pm\tfrac32,\pm2, \dots.
\]
(All half-integers can occur.) This is considerably more complicated than in the case of general linear groups
(and their inner forms) where there is a single cuspidal reducibility point (up to $\pm$).

Another important difference is that the parabolic induction of irreducible unitarizable representations is not irreducible in general.

For any reductive $p$-adic group $G$, denote by $\tilde G$ the set of all equivalence classes of its irreducible smooth representations
(the admissible dual of $G$) and by $\hat G$ the subset of the unitarizable classes (the unitary dual of $G$).
The unitarizability problem is the determination of the subset $\hat G$ of $\tilde G$.
It breaks down naturally into two parts: construction and exhaustion.
The exhaustion is usually achieved by showing that all the classes in $\tilde G$ other than the ones constructed in the first step
are non-unitarizable.
We may call such an approach to the exhaustion ``a proof by elimination''.

In the construction of the representations of $\hat G$, arguably the hardest part is the construction of representations that are isolated
in the natural topology of $\hat G$.
(Explicating the reducibility of unitary parabolic induction and the complementary series are other difficult problems.)
Our expectation is that at least for split classical $p$-adic groups, all isolated representations are of Arthur type
(and consequently, they occur as local constituents of automorphic representations in the discrete spectrum).
This is known to be the case for spherical representations by \cite{MR2767523}.
Here, M\oe glin's results on the structure of Arthur packets provide a powerful tool for the construction of isolated representations.

Regarding exhaustion, proof by elimination (which is used in this paper) is unfeasible in the higher rank case
(as this paper clearly indicates for rank 3).
Namely, it requires a very detailed knowledge of the structure of the representations in $\tilde G\backslash \hat G$.
Thus, especially in higher rank (where $\tilde G\backslash \hat G$ is much larger than $\hat G$)
we are spending most of our effort on the ``wrong'' class of representations.

Unfortunately, the prospect of finding a direct approach to the exhaustion problem for classical groups does not seem
to be on the horizon at the moment.
In fact, so far, the only successful direct approach to the exhaustion problem for reductive groups
in higher rank seems to be that of \cite{MR870688}, \cite{TadicBonn1985} for the case of general linear groups
(see also \cite{MR2537046} and \cite{MR1181278}).
The statement (if not the proof) of the classification of the unitary dual in this case is rather simple
(see Theorem \ref{ud-gl} below).\footnote{There is also D. Vogan's classification of unitary duals of
$\GL(n,\C), \GL(n,\R)$ and $\GL(n,\mathbb{H})$ (Theorem 6.18 of \cite{MR827363}).
One can find at the end of the seventh section of \cite{MR2684298} remarks
about relation between our approach and that of Vogan. We shall quote here only a part which indicates the main difference
between these two approaches: ``Vogan's classification is conceptually very different from Tadi\'c's classification.
It has its own merits, but the final result is quite difficult to state and to understand, since it uses sophisticated concepts
and techniques of the theory of real reductive groups.''}
This was arguably surprising at the time, although it should be noted that the first lists of candidates for the unitary duals
of the closely related groups $SL(n,\C)$, which go back to 1947 in \cite{MR0023246}, albeit incomplete, were very simple
(and not far from the actual unitary duals).
It took almost four decades to get a direct approach to the exhaustion in the case of general linear groups
(thereby fulfilling the vision of I. M. Gelfand and M. A. Naimark).

Although this paper is about unitarizability, most of it deals with non-uni\-tarizability because of the
exhaustion proof by elimination. A very small part of the admissible dual is unitarizable, and its unitarizability,
excluding only the representations \eqref{intro-dist}, is very natural to expect (and not too hard to prove).
In the analysis of the non-unitarizability of representations, the most delicate ones are those whose $\GL$-support
is contained in a segment of cuspidal representations which contains the reducibility point, and which are not fully induced
(non-unitarizability of the other representations is obtained by deformation to these representations or reducing to the
non-unitarizability in the case of general linear groups).
In \S\ref{sec: unit3} we settle the non-unitarizability in the most delicate cases, save for a few exceptions.
The remaining five cases are dealt with separately in \S\ref{basic 1},\ref{basic 0}.

In order to show non-unitarizability of $\pi$ in the most delicate case we consider the parabolic induction $\Pi$ of $\pi$ tensored
with a suitable irreducible unitarizable representation $\tau$ of a general linear group.
We show that the length $\ell$ of $\Pi$ is larger than the multiplicity $m$ of $\tau\otimes\pi$ in the Jacquet module of $\Pi$.
This implies that $\Pi$ cannot be semisimple, let alone unitarizable. Hence, $\pi$ cannot be unitarizable.
(In one case, we will also use the fact that if $\Pi$ is semisimple and $\ell$ is equal to $m$, then each copy of $\tau\otimes\pi$
must occur as a direct summand in the Jacquet module of $\Pi$.)

We already noted that in the construction of new irreducible unitarizable representations,
the most difficult cases are the isolated representations.
The simplest examples of representations that are often (though not always) isolated in the unitary dual are the square-integrable ones
(whose unitarizability is obvious) and their dual representations
(whose unitarizability is not obvious, except for the trivial representation which is isolated by \cite{MR0209390} if the split rank of
the simple group is not one).
In the case of $p$-adic general linear groups, the first case of an isolated representation (modulo the center) not of this type is for
$\GL(9)$.\footnote{In general, these representations are $u(\delta(\rho,m),n)$ where $m,n>2$ -- see below for notation.}
In the case of classical groups, such examples first occur in corank 3, when the reducibility point is $>1$.
These are the representations \eqref{intro-dist} below whose unitarizability
was proved by M\oe glin. (The smallest group which accommodates such a representation is the split $SO(11,F)$.)

We shall now briefly describe some parts of the strategy in \cite{MR3969882} proposed to handle the unitarizability problem
in the case of classical $p$-adic groups.
We first note that it is easy to reduce to the case of representations
supported on real twists of selfcontragredient irreducible cuspidal representations of general linear groups and irreducible cuspidal
representations of classical groups (see \S\ref{reduction-u-w} for more details).\footnote{In the case of unitary groups we need to
consider $F'/F$-contragredients, whose definition is recalled in the second chapter of the paper.
For simplicity, in the introduction we only consider symplectic and orthogonal groups.}
Because of this, we shall consider in the sequel only such representations, which we call weakly real representations.

Jantzen decomposition attaches to an irreducible (weakly real) representation $\pi$, irreducible representations
supported on single cuspidal lines
\[
\pi\ra (\pi_1,\dots,\pi_k)
\]
(see \cite{MR1481814} or \cite[\S8]{MR3969882} for more details).
In chapter \ref{mixed-case} we conjecture that this decomposition preserves unitarizability in both directions,
i.e., that $\pi$ is unitarizable if and only if all $\pi_i$ are unitarizable.
(See \cite{MR3969882} for some very limited support for this conjecture.)
If true, this would reduce the general case to unitarizability pertaining to the cases of a single cuspidal lines
(and single cuspidal reducibilities).

Consider now an irreducible representation $\pi$ of a classical group that is supported on a cuspidal line
$X_\rho$ along a selfcontragredient irreducible cuspidal representation $\rho$ of a general linear group,
and an irreducible cuspidal representation $\s$ of a classical group.
To the pair $\rho,\s$ corresponds a unique non-negative reducibility exponent $\alpha_{\rho,\s}\in\tfrac12\Z$.
The next question is whether the unitarizability of $\pi$ can be described in terms of the reducibility exponent
$\alpha_{\rho,\s}$ alone. (See \cite{MR3969882} for a precise formulation.)
If true, the unitarizability problem for classical $p$-adic groups
would amount to the determination of a certain (complicated) combinatorial data.

We shall below use standard Bernstein-Zelevinsky notation $\times$ for parabolic induction in the setting of general linear groups,
and its natural extension $\rtimes$ to the setting of classical groups (see \S\ref{notation} for details).
By $|\ \ |_F$ is denoted the normalized absolute value on $F$, and by $\nu$ the character $g\mapsto |\det(g)_F|$ of $\GL(n,F)$.

The main goal of this paper is to classify, following the above proposed strategy,
the irreducible unitarizable (weakly real) subquotients of the representations
\begin{equation} \label{k-times}
\theta_1\times\dots\times\theta_k\rtimes\s, \quad k\leq 3,
\end{equation}
where $\theta_i$, $1\leq i\leq k$, and $\s$ are irreducible cuspidal representations of general linear groups and of
a classical group respectively.
In particular, we completely classify the unitary dual of classical $p$-adic groups of the split rank (at most) three.
This gives some very limited support for the possibility of the above approach to the unitarizability to work in general.

In the last chapter of the paper we prove that the Jantzen decomposition preserves unitarizability in both directions for the cases that we consider in this paper.
More precisely, we prove the following

\begin{proposition} \label{intro-red}
Let $\pi$ be a weakly real irreducible subquotient of $\theta_1\times\dots\times\theta_k\rtimes\s$, where $\theta_i$ are irreducible
cuspidal representations of general linear groups and $k\leq 3$. Then, $\pi$ is unitarizable if and only if all $\pi_i$
in the Jantzen decomposition of $\pi$ are unitarizable.
\end{proposition}

For a general connected reductive group $G$ over $F$ there is a natural involution $\pi\mapsto D_G(\pi)$ established in \cite{MR1285969}
and \cite{MR1471867}, which carries an irreducible representations of $G$ to an irreducible representation of $G$,
up to a sign. It is modeled after Deligne--Lusztig duality (in the context of groups over finite fields)
and will be henceforth referred to as \ASS involution, or duality\footnote{This involution is also called Zelevinsky involution, or Aubert involution, or Aubert-Schneider-Stuhler involution.}
(see also \cite{MR3769724}).
Take $\e_\pi\in\{\pm1\}$ such that $\e_\pi D_G(\pi)$ is a representation.
We denote then $\e_\pi D_G(\pi)$ by $\pi^t$ and call it the \ASS involution of $\pi$.

Now we shall describe unitarizability in the case of corank up to three.
We shall express the classification of irreducible subquotients in the shortest way.

By Proposition \ref{intro-red}, it is enough to consider representations supported a on single cuspidal line.
It means that we fix an irreducible selfcontragredient cuspidal representation $\rho$ of a general linear group
and an irreducible cuspidal representation
$\s$ of classical group.
Then, there exist a unique non-negative $\alpha_{\rho,\s}\in\tfrac12\Z$ such that
$$
\nu^{\alpha_{\rho,\s}}\rho\rtimes\s
$$
is reducible.
Then, to simplify notation, we denote
$$
\alpha:=\alpha_{\rho,\s}.
$$
Suppose that $\alpha>0$. Let $k$ be a non-negative integer. Define the generalized Steinberg representation
$$
\delta([\alpha,\alpha+k]^{(\rho)};\s)=\soc(
\nu^{k+\alpha}\rho\times \nu^{k-1+\alpha}\rho\times \dots \times \nu^{\alpha}\rho\rtimes\s)
$$
where $\soc$ is the socle.
These representations are irreducible and square-integrable.
(For $k=0$ we simply write $\delta([\alpha]^{(\rho)};\s)$.)

Denote
$$
\R^k_{++}=\{(x_1,\dots,k_k)\in\R^k; 0\leq x_1\leq \dots\leq x_k\}.
$$
It is well-known that for each representation $
\pi:=\nu^{x_1}\rho\times \dots \times \nu^{x_k}\rho\rtimes\s
$, $x_i\in
\R$, there exists $(x_1',\dots,x_k')\in \R^k_{++}$ such that $\pi$ and $\nu^{x_1'}\rho\times \dots \times \nu^{x_k'}\rho\rtimes\s$ have the same composition series.
Now we describe irreducible unitarizable subquotients of representationsn\eqref{k-times}. For $k=1$ the answer is:

\begin{remark}
Unitarizability in the corank 1 is very simply to describe:
an irreducible subquotient $\pi$ of $\nu^x\rho\rtimes\s$, $x\in\R_{\geq0}$, is unitarizable $\iff$ $x\leq \alpha$.
\end{remark}

The following answer to the unitarizability problem in the corank 2
is more or less already well known (although we were unable to find a complete reference in the literature).

\begin{proposition}
The irreducible unitarizable subquotients of $\nu^{x_1}\rho\times\nu^{x_1}\rho\rtimes\sigma$, $(x_1,x_2)\in\R^2_{++}$, are the following.
\begin{enumerate}
\item $(\alpha>1)$ All irreducible subquotients when $x_1+1\leq x_2\leq \alpha$.
\item $(\alpha\ne\frac12)$ All irreducible subquotients when $x_1+x_2\leq 1$.
\item $(\alpha=\tfrac12)$ All irreducible subquotients when $x_2\le\tfrac12$.
\item $(\alpha>0)$ The representations $\delta([\alpha,\alpha+1];\s)$ and $\delta([\alpha,\alpha+1];\s)^t$.
\end{enumerate}
\end{proposition}

For the answer of unitarizability for corank 3 we need additional notation.

Suppose $\alpha>1$. Denote by
\begin{equation} \label{intro-dist}
L([\alpha]^{(\rho)}, [\alpha-1]^{(\rho)};\delta([\alpha]^{(\rho)};\s))
\end{equation}
the Langlands quotient of $\nu^{\alpha}\rho\times \nu^{\alpha-1}\rho\rtimes\delta([\alpha]^{(\rho)};\s)$.
(This representation is invariant under the \ASS involution.)

For any positive integer $n$, the representation
$$
\delta(\rho,n)=\soc(\nu^{\frac{n-1}2}\rho \times \nu^{\frac{n-3}2}\rho \times\dots\times \nu^{-\frac{n-1}2}\rho)
$$
is irreducible and square-integrable modulo center.

For $\alpha=\tfrac12$, the representation $\delta(\rho,2)\rtimes\s$ contains a unique irreducible subquotient
which is not a subquotient of $[\tfrac12]^{(\rho)}\rtimes\delta ([\tfrac12]^{(\rho)};\s).$
We denote it by
$$
\delta([-\tfrac12,\tfrac12]^{(\rho)}_-;\s).
$$
Let $\alpha=0$. For any positive integer $k$, the representation
$
\nu^{\frac k2}\delta(\rho,k+1)\rtimes\s
$
has precisely two irreducible subrepresentations, and they are both square-integrable. We denote them by
$$
\delta([0,k]^{(\rho)}_+;\s)\quad \text{and}\quad \delta([0,k]^{(\rho)}_-;\s).
$$

The unitarizability statement for corank 3, which is the main result of the paper, is the following

\begin{theorem} \label{thm: mainintro}
The irreducible unitarizable subquotients of $\nu^{x_1}\rho\times \nu^{x_2}\rho\times \nu^{x_3}\rho\rtimes \sigma$
where $(x_1,x_2,x_3)\in\R^3_{++}$ are the following.
\begin{enumerate}
\item $(\alpha\ge1)$ All irreducible subquotients when $\mathbf{x}=(x_1,x_2,x_3)$ lies in the closure of one of the domains
\begin{gather*}
x_2+x_3<1, \\
x_1+x_2<1,\, x_3-x_2>1,\, x_3<\alpha,\ \ \ (\alpha>1),\\
x_1+x_2<1,\, x_1+x_3>1,\, x_3-x_1<1,\, x_3<\alpha,\\
x_2-x_1>1,\, x_3-x_2>1,\, x_3<\alpha,\ \ \ (\alpha>2).
\end{gather*}
\item $(\alpha=\tfrac12)$ All irreducible subquotients when $x_3\le\tfrac12$.
\item $(\alpha=0$) All irreducible subquotients when $x_1=0$, $x_2+x_3\le1$.
\item $(\alpha>0)$ The representations $\delta([\alpha,\alpha+2];\s)$ and $\delta([\alpha,\alpha+2];\s)^t$.
\item $(\alpha=0)$ The representations $\delta([0,2]_\pm;\s)$ and $\delta([0,2]_\pm;\s)^t$.
\item $(\alpha>0)$ The complementary series $[x]\rtimes \delta([\alpha,\alpha+1];\s)$ and
$[x]\rtimes \delta([\alpha,\alpha+1];\s)^t$ for $0\le x<\abs{\alpha-1}$ (if $\alpha\ne1$)
and its irreducible subquotients for $x=\alpha-1$.
\item $(\alpha=0)$ The complementary series (including subquotients at the ends)
\[
[x]\rtimes\delta([0,1]_\pm;\s),\ [x]\rtimes \delta([0,1]_\pm;\s)^t,\ \ 0\le x\le1
\]
\item $(\alpha>1)$ The representation $L([\alpha-1],[\alpha];\delta([\alpha];\s))$.
\item $(\alpha=\tfrac12)$ The complementary series (including subquotients at the ends)
\begin{gather*}
\nu^x\delta(\rho,2)\rtimes\delta([\tfrac12];\s),\
\nu^x\delta(\rho,2)^t\rtimes\delta([\tfrac12];\s)^t,
\ \ 0\le x\le1
\\
\nu^x\rho\rtimes\delta([-\tfrac12,\tfrac12]_-;\s),\
\nu^x\rho\rtimes\delta([-\tfrac12,\tfrac12]_-;\s)^t,\ \ 0\le x\le\tfrac32
\\
\nu^x\delta(\rho,3)\rtimes\s,\ \nu^x\delta(\rho,3)^t\rtimes\s,\ \ 0\le x\le\tfrac12.
\end{gather*}
\end{enumerate}
\end{theorem}

Note that Theorem \ref{thm: mainintro} directly implies that the
\ASS involution preserves unitarizability in the cases at hand,
a fact that is expected to hold in general (see \cite{MR794086} and \cite{MR1010153} for some steps in that direction).
Theorem \ref{thm: mainintro} also implies that all isolated representations are automorphic. (In \S\ref{sec: conjectures}
we conjecture this to hold in general.)

Furthermore, note that in the above classifications only the reducibility point $\alpha$ plays a role in determining the exponents of the representations that are unitarizable
(not $\rho$ and $\s$ themselves).

The motivation for writing this paper came from a discussion with C. M\oe glin at the
Simons Symposium on Geometric Aspects of the Trace Formula in Schloss Elmau in Germany (2016).
The paper \cite{MR3573961} of E. Lapid and A. M\'inguez was also a strong motivation for us to try to understand unitarizability based
only on the reducibility points between irreducible cuspidal representations, at least at corank 3.
Some of the results of this paper were presented in a minicourse at the Special Trimester on
Representation Theory of Reductive Groups Over
Local Fields and Applications to Automorphic forms, which was held at the Weizmann Institute in spring 2017.
We are thankful to the Simons Foundation and the Weizmann Institute. Discussions with
M. Hanzer, I. Mati\'c and A. Moy were helpful during the writing of this paper.

Especially we would like to thank C. M\oe glin and E. Lapid.
C. M\oe glin wrote the appendix with
the proof that the representation \eqref{intro-dist}
is in an Arthur packet, which proves the unitarizability of that representation.
Thanks to E. Lapid and his huge help, this paper is much easier to understand.
In particular, he helped us to change the organization of the paper, and considerably simplified the exhaustion part
by adding a geometric argument.
The final work on this paper took place at Weizmann Institute in winter 2020, and we are very thankful to the institute
for its hospitality and the very pleasant and relaxed working atmosphere.

The contents of the paper are the following.
Chapter \ref{notation} introduces notation and recalls known results that we use throughout the paper.
In chapters \ref{CC-corank12} -- \ref{CC-0} the unitarizability is solved for representations supported on
a segment of cuspidal representations that contains the reducibility point.
The heart of the paper is chapter \ref{sec: unit3} where we solve the cases $\alpha>1$.
In chapters \ref{basic 1} and \ref{CC-0} we address the remaining cases $\alpha=1,\tfrac12,0$ which require additional work.
All these cases are completely new.
In chapters \ref{rem-on-uni}, \ref{corank-3} and \ref{mixed-case} we complete the solution of the unitarizability in corank $\le 3$.
Appendix \ref{appendix-M} by C. M\oe glen shows the unitarizability of the isolated representations $L([\alpha-1],[\alpha];\delta([\alpha];\s))$
(for $\alpha>1$). Appendix \ref{App-JM} provides formulas for the Jacquet modules of these representations.

\chapter{Notation and Preliminary Results} \label{notation}

We fix a local non-archimedean field $F$ of characteristic zero.\footnote{We expect all the results to hold also
in positive characteristic. However, we would need to verify some facts such as the unitarizability of representations \eqref{intro-dist}.}
Let $G$ be the group of $F$-points of a reductive group defined over $F$.
In this paper, by a representation of $G$ we shall always mean a complex, smooth representation.
The Grothendieck group of the category $\Alg_{\fl}(G)$ of all finite length representations of
$G$ is denoted by $ \mathfrak R(G)$. It carries a natural ordering $\leq$.
We denote by $\ssm(\tau)$ the semi simplification of $\tau\in\Alg_{\fl}(G)$.
For brevity, if $\pi_1,\pi_2\in\Alg_{\fl}(G)$, the condition
$\ssm(\pi_1)\leq \ssm(\pi_2)$ will be written simply as $\pi_1\leq \pi_2$.

The contragredient representation of $\pi$ is denoted by $\tilde\pi$, while the complex conjugate representation is denoted by $\bar\pi$.
We call $\tilde{\bar \pi}$ the hermitian contragredient of $\pi$ and denote it by $\pi^+$.
Then, $\pi\mapsto \pi^+$ is an (exact) contravariant functor. It is well known that if $\pi$ is unitarizable, then $\pi^+\cong\pi$
(i.e., $\pi$ is hermitian).

\section{General linear groups}
Let $F'$ be either $F$ itself or a (separable) quadratic extension thereof.
(The second case is pertaining to unitary groups and the first case to all other classical groups considered below.)
If $F'\ne F$, then
$
\Theta
$
denotes the non-trivial $F$-automorphism of $F'$. Otherwise, $\Theta$ denotes the identity mapping on $F$. The representation
\begin{equation} \label{cont-var}
\check\pi=\tilde\pi\circ\Theta
\end{equation}
will be called the $F'/F$-contragredient of $\pi$.
The representation $\pi\circ\Theta$ will be denoted by $\pi^\Theta$.

We shall now recall notation for the general linear groups (mainly following \cite{MR584084}).
The modulus character of $F'$ is denoted by $\abs{\cdot}_{F'}$.
The character $\abs{\det}_{F'}$ of $\GL(n,F')$ will be denoted by $\nu$.

We fix the Borel subgroup of upper triangular matrices.
For $0\leq k\leq n$, let $P_{(k,n-k)}=M_{(k,n-k)}\ltimes N_{(k,n-k)}$ be the standard parabolic subgroup
of $\GL(n,F' )$ of type $(k,n-k)$ with Levi factor $M_{(k,n-k)}\cong\GL(k,F' )\times\GL(n-k,F' )$.
For $\pi_i\in\Alg_{\fl}(\GL(n_i,F' ))$, $i=1,2$ denote by $\pi_1\times
\pi_2\in\Alg_{\fl}(\GL(n_1+n_2,F' ))$ the representation parabolically inducted from
$\pi_1\otimes\pi_2$ on $P_{(n_1,n_2)}$ (normalized induction). Let
$R=\oplus_{n\geq 0} \mathfrak R(\GL(n,F' ))$ endowed with the structure of a commutative graded ring by $\times$.
The biadditive map $\times:R\times R\ra R$ gives rise to a map $m:R\otimes R\ra R$.

The normalized Jacquet module of $\pi \in \Alg_{\fl}(\GL(n,F' ))$ with respect
to $P_{(k,n-k)}$ is denoted by $r_{(k,n-k)}(\pi)$.
The comultiplication $m^*(\pi)$ of $\pi$ is defined by
$$
m^*(\pi)=\sum_{k=0}^n \ssm(r_{(k,n-k)}(\pi)) \ \in \ R\otimes R.
$$
One extends $m^*$ additively to a ring homomorphism $m^*:R\rightarrow R\otimes R$ in a natural way.
With $m$ and $m^*$, $R$ is a graded Hopf algebra.
The fact that $m^*$ is a ring homomorphism follows from the geometric lemma of Bernstein--Zelevinsky (see \cite{MR0579172}).

Denote by $\Cusp$ (resp., $\Irr$) the set of equivalence classes of all irreducible cuspidal
(resp., irreducible) representations of all $\GL(n,F' )$, $n\geq 1$ (resp., $n\ge0$).

By a \emph{$\Z$-segment} in $\R$ we shall mean a set of the form
$
\{x,x+1,\dots,x+n\},
$
where $x\in\R$ and $n\in\Z_{\geq 0}$. We shall denote the above set by
$
[x,x+n]_\Z.
$
For any $\Z$-segment $\D=[x,y]_\Z$ in $\R$ and $\rho\in\Cusp$, denote
$$
\D^{(\rho)}=[x,y]^{(\rho)}=[\nu^x\rho,\nu^y\rho]:=\{\nu^z\rho;z\in\D\}.
$$
The set $\D^{(\rho)}$ is called a segment in $\Cusp$.
The set of all segments in $\Cusp$ is denoted by $\SC$.
It is also convenient to set $\emptyset^{(\rho)}=\emptyset$.

We say that two segments $\D_1,\D_2\in\SC$ are \emph{linked} if
$\D_1\cup \D_2\in\SC$ and $\D_1\cup\D_2\not\in\{\D_1,\D_2\}$.
If the segments $\D_i=[x_i,y_i]^{(\rho)}$ are linked and $x_1<x_2$, then
we say that $\D_1$ precedes $\D_2$, and write
$$
\D_1\prec \D_2.
$$

For any set $X$, we denote by $M(X)$ the set of all finite multisets in $X$
(which we may view as functions $X\rightarrow \Z_{\geq0}$ with finite support;
note that finite subsets correspond to all functions $X\rightarrow \{0,1\}$ with finite support).
A typical elements of $M(X)$ will be denoted by $(x_1,\dots,x_n)$ (repetitions of elements can occur,
but the order of the $x_i$'s does not matter).
The set $M(X)$ has a natural structure of a commutative monoid whose zero element is the empty multiset.
The operation will be denoted additively: $(x_1,\dots,x_n)+(y_1,\dots,y_m)=(x_1,\dots,x_n,y_1,\dots,y_m)$.

For $\D\in\SC$ we define $\supp(\D)$ to be $\D$, but considered as an element of $M(\Cusp$).
For $a=(\D_1,\dots,\D_n)\in M(\SC)$ we define
$$
\supp(a)=\sum_{i=0}^n \supp(\D_i)\in M(\Cusp).
$$

\section{Classifications of admissible duals of general linear groups}

Fix a segment $\D=\{\rho,\nu\rho, \dots,\nu^n\rho\}\in\SC.$
Then, the representation
$$
\rho\times\nu\rho\times \dots\times\nu^n\rho
$$
has an irreducible socle, denoted by $\ms(\D)$, and an irreducible cosocle, denoted by $\delta(\D)$. We have
\begin{subequations}
\begin{gather} \label{m*d}
m^*(\delta([\rho,\nu^n\rho]))=\sum_{i=-1}^n \delta([\nu^{i+1}\rho,\nu^n\rho]) \otimes
\delta([\rho,\nu^{i}\rho]),\\
\label{m*s}
m^*(\ms([\rho,\nu^n\rho]))=\sum_{i=-1}^n \ms([\rho,\nu^{i}\rho])
\otimes\ms([\nu^{i+1}\rho,\nu^n\rho]).
\end{gather}
\end{subequations}

Let $a=(\D_1,\dots,\D_n)\in M(\SC)$. We can choose an enumeration satisfying
\begin{center}
if $\D_i\prec \D_j $ for some $1\leq i,j\leq n$, then $i>j$.
\end{center}
Then, up to isomorphism, the representations
\begin{align*}
&\zeta(a):= \ms(\D_1)\times \ms(\D_2)\times\dots\times \ms(\D_n),
\\
(\text{resp., }& \lambda(a):= \delta(\D_1)\times \delta(\D_2)\times\dots\times \delta(\D_n))
\end{align*}
are determined by $a$ and admit an irreducible socle (resp., cosocle) denoted by $Z(a)$ (resp., $L(a)$).
Now $Z$ (resp. $L$) is called the Zelevinsky (resp. Langlands) classification of irreducible representations of general linear groups
over $F'$. (We follow the presentation of these classifications by F. Rodier in \cite{MR689531}.)

Denote by $\Disc$ the set of equivalence classes of all irreducible essentially square-integrable representations
of $\GL(n,F')$, $n\geq1$, and by $\Disc_u$ the subset of all unitarizable classes in $\Disc$ (i.e. those having a unitary central character).
The mapping
\begin{equation} \label{esi}
(\rho,n)\mapsto \delta(\rho,n):=\delta([-\tfrac{n-1}2,\tfrac{n-1}2]^{(\rho)}), \quad \Cusp\times \Z_{>0}\ra \Disc
\end{equation}
is a bijection.

For $\delta\in \Disc$ define $\delta^u\in \Disc_u$ and $e(\delta)\in\R$ by the following requirement:
$$
\delta=\nu^{e(\delta)}\delta^u.
$$
Let $d=(\delta_1,\dots,\delta_n)\in M(\Disc)$, enumerated so that
$$
e(\delta_1)\geq e(\delta_2) \geq \dots \geq e(\delta_n).
$$
Let
$$
\lambda(d)=\delta_1\times\delta_2\times\dots\times\delta_n.
$$
Then, the representation $\lambda(d)$ has an irreducible cosocle, denoted by $L(d)$.
Again $d\mapsto L(d)$ is a version of Langlands classification for general linear groups (irreducible representations are
parameterized by elements of $M(\Disc)$).

For $d=(\delta_1,\dots,\delta_n)\in M(\Disc)$ denote $\tilde d=(\tilde\delta_1,\dots,\tilde\delta_n)\in M(\Disc)$,
$\bar d=(\bar\delta_1,\dots,\bar\delta_n)$, $d^+=(\delta_1^+,\dots,\delta_n^+)$ and
$d^\Theta=(\delta_1^\Theta,\dots,\delta_n^\Theta)$.
Then, $L(d)\tilde{\ }=L(\tilde d)$, $L(d)\bar{\ }=L(\bar d)$, $L(d)^+=L( d^+)$ and $L(d)^\Theta=L(d^\Theta)$.

Define a mapping $^t$ on $\Irr$ by $Z(a)^t=L(a), a \in M(\SC)$.
Extend $^t$ additively to $R$. Clearly, $^t$ is a positive mapping, i.e.
it satisfies: $r_1\leq r_2\implies r_1^t \leq r_2^t$. A non-trivial fact is that $^t$ is in fact a ring homomorphism
(see \cite{MR1285969} and \cite{MR1471867}). Furthermore, $^t$ is an involution, called Zelevinsky involution.

For $a\in M(\SC)$, define $a^t\in M(\SC)$ by the requirement
$$
L(a)^t=L(a^t).
$$

\section{Classical groups -- basic definitions}
We will mostly follow the notation of \cite{MR1896238} for classical $p$-adic groups.
The main difference is that the indexing of the classical groups will be slightly different.

Fix a Witt tower $\mathcal V=\{V_n\}_{n\ge0}$ of symplectic, quadratic or hermitian vector spaces over $F'$.
In the first two cases $F'=F$ and in the latter case $F'/F$ is a (separable) quadratic extension $F$
with Galois automorphism $\Theta$.
In all cases, a maximal isotropic subspace of $V_n$ has dimension $n$
(see sections III.1 and III.2 of \cite{KudlaNotes} for more details).\footnote{For some purposes a different indexing
of the groups $S_n$ may be more convenient -- see \cite{MR1896238}.}
In particular, $V_0$ is anisotropic.\footnote{In the symplectic case, $V_0=\{0\}$.}
Denote by $S_n$ the group of isometries of $V_n$.
For $0\leq k\leq n$, let $P_{(k)}$ be the stabilizer of a fixed $k$-dimensional isotropic subspace of $V_n$ --
see \cite[\S III.2]{KudlaNotes}.\footnote{One can find in \cite{MR1356358}
matrix realizations of the symplectic and split odd-orthogonal groups.
In a similar way one can make matrix realizations also for other classical groups.}
The Levi factor $M_{(k)}$ of $P_{(k)}$ is naturally isomorphic to $\GL(k,F')\times S_{n-k}$.
More generally, for any partition $\beta$ of $\ell\leq n$ we can in a natural way define a parabolic subgroup $P_\beta$
and its Levi subgroup $M_\beta$.
(For $M_\beta$ first consider $M_{(\ell)}$, and then apply the construction from the case of general linear groups.)

We remark that in the odd orthogonal case we may replace $S_n$ with the group of isometries of $V_n$ of determinant one.

We exclude in the paper the case of split even orthogonal groups, although we expect that all the results
hold also in this case, with the same proofs. (Split even orthogonal groups are not connected,
which requires some additional checks that we have not yet carried out.)

A minimal parabolic subgroup in $S_n$, which is the intersection of all $P_{(k)}$'s, will be fixed.
(Only standard parabolic subgroups with respect to the fixed minimal parabolic subgroup will be considered in this paper.)

For the rest of the paper we fix once and for all the series $\{S_n\}_{n\ge0}$ as above.
We denote by $\Cuspcl$, (resp., $\Discl$, $\Tempcl$, $\Irrcl$) the set of cuspidal (resp., square-integrable, tempered, all)
irreducible representations of $S_n$, $n\ge0$ (up to equivalence) and
by $\Cuspsd$ the subset of $\Cusp$ consisting of $F'/F$-selfcontragredient representations.
(We shall often apply Casselman's criteria from \cite{CassNotes} for representations to be square-integrable or tempered.)

\section{Twisted Hopf algebra structure}
For $\pi\in\Alg_{\fl}(\GL(k,F'))$ and $\sigma\in\Alg_{\fl}(S_{n-k})$, the representation parabolically
induced from $\pi\otimes\sigma$ is denoted by
$$
\pi\rtimes\sigma.
$$
We shall often use that
\begin{equation} \label{asso}
\pi_1\rtimes (\pi_2\rtimes\sigma )\cong (\pi_1\times\pi_2)\rtimes\sigma.
\end{equation}

For $\pi$ as above holds
\begin{equation} \label{check}
\ssm(\pi\rtimes\sigma)=\ssm(\check\pi\rtimes\sigma).
\end{equation}
Therefore, if $\pi \rtimes \sigma$ is irreducible, then
$\pi\rtimes\sigma\cong\check\pi\rtimes\sigma$.
We say that a representation $\pi$ of a general linear group
over $F'$ is $F'/F$-selfcontragredient if
$
\pi\cong\check\pi.
$

The normalized Jacquet module of $\tau\in\Alg_{\fl}(S_n)$ with respect to $P_{(k)}$ is denoted by $s_{(k)}(\tau)$.
Let $\tau$ and $\omega$ be irreducible representations of $\GL(p,F)$ and $S_q$, respectively, and let $\pi$ be an admissible
representation of $S_{p+q}$. Then, by Frobenius reciprocity
$$
\Hom_{_{S_{p+q}}}(\pi,\tau \rtimes \omega)\cong
\Hom_{_{\GL(p,F)\times S_q}}(s_{(p)}(\pi),\tau \otimes \omega),
$$
while the second adjointness implies
$$
\Hom_{_{S_{p+q}}}(\tau \rtimes \omega, \pi)\cong
\Hom_{_{\GL(p,F)\times S_q}}(\check \tau \otimes \omega,s_{(p)}(\pi)).
$$

Denote
$$
R(S)=\underset {n\geq 0} \oplus \mathfrak R(S_n).
$$

Now $\rtimes$ induces in a natural way a mapping $R\times R(S)\ra R(S)$, which is denoted again by $\rtimes$.
For $\tau\in \Alg_{\fl}(S_n)$, denote
$$
\mu^*(\tau)=\sum_{k=0}^n \ssm\left(s_{(k)}(\tau)\right).
$$
We extend $\mu^*$ additively to $\mu^*:R(S)\ra R\otimes R(S)$. Denote
\begin{equation} \label{M*}
 M^*= (m \otimes 1) \circ (\, \check{\ }\, \otimes m^\ast) \circ
\kappa \circ m^\ast : R\ra R\otimes R,
\end{equation}
where $\check{\ }\,:R\ra R$ is a
group homomorphism determined by the requirement that
$\pi\mapsto\check\pi$ for all $\pi\in\Irr$, and $\kappa:R\times R\ra R\times R$ maps $\sum x_i\otimes y_i$ to $\sum y_i\otimes x_i.$
The action $\rtimes$ of $R\otimes R$ on $R\otimes R(S)$ is defined in a natural way. Then,
\begin{equation}
\label{mu*}
\mu^*(\pi \rtimes \sigma)= M^*(\pi) \rtimes \mu^*(\sigma)
\end{equation}
holds for $\pi\in R$ and $\sigma \in R(S)$.

For any finite length representation $\pi$ of $\GL(k,F')$, the component of $M^*(\pi)$ which is in
$
\mathfrak R(\GL(k,F'))\otimes \mathfrak R(\GL(0,F'))
$, will be denoted by
$$
M^*_{\GL}(\pi)\otimes 1.
$$

Let $\pi$ be a representation of $\GL(k,F')$ of finite length, and let $\sigma\in\Cuspcl$. Suppose
that $\tau$ is a subquotient of $\pi\rtimes \sigma$. Then, we shall denote $s_{(k)}(\tau)$ also by
\begin{equation*}
s_{\GL}(\tau).
\end{equation*}
If in addition, $\tau$ is irreducible, then we shall say that
\begin{equation} \label{pcs}
\s
\end{equation}
is the partial cuspidal support of $\tau$.
We say that $\theta\in\Cusp$ is a factor of $\tau$ if there exists an irreducible subquotient
$\beta\otimes\s$ of $s_{\GL}(\tau)$ such that $\theta$ is in the support of $\beta$.

Let $\pi$ be a finite length representation of a general linear group, and let $\tau$ be a representation of
$S_n$ as above. Then, \eqref{mu*} implies
\begin{equation} \label{s-GL}
\ssm(s_{\GL}(\pi\rtimes\tau))=M^*_{\GL}(\pi)\times
\ssm(s_{\GL}(\tau))
\end{equation}
($\times$ in the above formula denotes multiplication in $R$ of
$M^*(\pi)$ with the factors on the left-hand side of $\otimes$ in $\ssm(s_{\GL}(\tau))$).

Let $\tau$ be a representation of some $\GL(m,F)$ and let
$
m^{\ast}(\tau)=\sum x\otimes
y.
$
Then, the formula \eqref{M*} implies directly
\begin{equation} \label{M-GL}
M^*_{\GL}(\tau)=\sum x\times \check{y}.
\end{equation}
Furthermore, the sum of the irreducible subquotients of the form $1\otimes \ast$ in $M^{\ast}(\tau)$ is
\begin{equation} \label{1o}
1\otimes \tau.
\end{equation}

Now assume that $\pi$ is a representation of $\GL(d,F)$ and $\s$ is a representation of a classical group.
Let $\pi_1\otimes\pi_2\otimes\pi_3$ be an irreducible subquotient of some $r_{(n_1,n_2,n_3)}(\pi)$ ($n_1+n_2+n_3=n$)
and let $\pi_4\otimes\s_0$ be an irreducible subquotient of some $s_{(m_1)}(\s)$ ($m_1\leq m$). Then,
$$
\pi_1 \times\pi_4\times \tilde\pi_3 \otimes \pi_2\rtimes\s_0
$$
is a subquotient of the corresponding Jacquet module (see \cite[Lemma 5.1]{MR1356358} and the discussion preceding it).

\section{Some formulas for \texorpdfstring{$M^*$}{M*}} \label{some form.}
Let $\rho\in\Cuspsd$. Suppose that $x,y\in \R$ satisfy $y-x\in \Z_{\geq0}$.
Then, one directly gets from \eqref{m*d} and \eqref{M*}
\begin{equation} \label{M-seg}
M^*\big(\delta([x,y]^{(\rho)})\big) =
\sum_{i= x-1}^{ y}\sum_{j=i}^{ y}
\delta([-i,-x]^{(\rho)})\times\delta([j+1,y]^{(\rho)}) \otimes\delta([i+1,j]^{(\rho)}),
\end{equation}
where $y-i,y-j\in \Z_{\geq 0}$ in the above sums. In particular
\begin{equation}\label{M-seg-GL}
M_{\GL}^*\big(\delta([x,y]^{(\rho)})\big) =
\sum_{i= x-1}^{ y}\delta([-i,-x]^{(\rho)})\times\delta([i+1,y]^{(\rho)}).
\end{equation}

In a similar way, one gets for Zelevinsky segment representations
$$
M^*(\ms([x,y]^{(\rho)}))=\sum_{x -1\leq i \leq y}\sum_{ x -1\leq j \leq i}
\ms ([-y,-i -1]^{(\rho)})\times\ms ([x ,j]^{(\rho)})\otimes \ms ([j +1,i]^{(\rho)}).
$$

More generally, let $\pi=L(\D_1,\dots,\D_k)$ be a ladder representations, i.e., we can write $\D_i=[a_i,b_i]^{(\rho)}$ where
$a_k<\dots<a_1$ and $b_k<\dots<b_1$ (we continue to assume below
$
\rho\cong\check\rho).
$
Then, using \cite{MR2996769} we get

\begin{equation}\label{lad-GL}
M^*_{\GL}(\pi)=\sum_{\substack{a_i-1\leq x_i\leq b_i,
\\x_k<\dots<x_1}}L(\ ([-x_i,-a_i]^{(\rho)})_{1\leq i\leq k} \ )
\times L(\ ([x_i+1,b_i]^{(\rho)})_{1\leq i\leq k}\ ).
\end{equation}

\section{Langlands classification for classical groups
\texorpdfstring{(\cite{MR0507262}, \cite{MR1721403}, \cite{MR2050093}, \cite{MR2567785}, \cite{MR584084})}{}} \label{LCCG}
Denote
$$
\Disc_+=\{\delta\in \Disc: e(\delta)>0\}.
$$
For
$
t=((\delta_1,\delta_2,\dots,\delta_k),\tau)\in
M(\Disc_+)\times \Tempcl
$
take a permutation $p$ of $\{1,\dots,k\}$ such that
\begin{equation}\label{dec}
e(\delta_{p(1)})\geq e(\delta_{p(2)})\geq\dots\geq e(\delta_{p(k)}).
\end{equation}
Then, the representation
$$
\lambda(t):=\delta_{p(1)}\times\delta_{p(2)}\times\dots\times\delta_{p(k)}\rtimes \tau
$$
has an irreducible cosocle, denoted by
$$
L(t).
$$
The mapping
$$
t\mapsto L(t)
$$
defines a bijection between $M(\Disc_+)\times \Tempcl$ and $\Irrcl$. This is the Langlands classification for classical groups.
The multiplicity of $L(t)$ in $\lambda(t) $ is one.

Write $t=(d;\tau)$. Then, $L(d;\tau)\bar{\ }\cong L(\bar d;\bar\tau)$ and $L(d;\tau)\tilde{\ }\cong L(d^\Theta;\tilde\tau)$.

Let $t=((\delta_1,\delta_2,\dots,\delta_k),\tau) \in M(\Disc_+)\times \Tempcl$ and suppose that a permutation $p$ satisfies \eqref{dec}.
Suppose that $\delta_{p(i)}$ is a representation of $\GL(n_i,F)$ and $L(t) $ a representation of $S_n$. Define
$$
e_*(t)=(\underbrace{e(\delta_{p(1)}),\dots,e(\delta_{p(1)})}_{n_1\text{ times}},\dots,
\underbrace{e(\delta_{p(k)}),\dots,e(\delta_{p(k)})}_{n_k\text{ times}},\underbrace{0,\dots,0}_{n'\text{ times}}),
$$
where $n'=n-(n_1+\dots+n_k)$. Consider the partial ordering on $\R^n$ given by
$$
(x_1,\dots,x_n)\leq (y_1,\dots,y_n)\iff\sum_{i=1}^j x_i\leq \sum_{i=1}^j y_i, \quad 1\leq j\leq n.
$$
Suppose $t,t'\in M(\Disc_+)\times \Tempcl$ and $L(t')$ is a subquotient of $\lambda(t)$. Then,
\begin{equation} \label{BPLC}
e_*(t')\leq e_*(t), \text{ with an equality }\iff t'=t.
\end{equation}
(See \cite[\S6]{MR1266251} for the symplectic groups -- the same proof works for all classical groups
other than the split even orthogonal groups.)

For $\Delta\in \mathcal S$ define $\mathfrak c(\Delta)$ to be $e(\delta(\Delta))$. Let
$$
\SC_+=\{\Delta\in\SC; \mathfrak c (\Delta)>0\}.
$$
In this way we can define in a natural way the Langlands classification $(a,\tau)\mapsto L(a;\tau)$ using $M(\SC_+)\times \Tempcl$
as the parameters.

\section{Irreducible subquotients of induced representations of classical groups} \label{irr-sq}
We will recall a very useful fact from \cite{MR2504024}.

For $d=(\delta_1,\dots,\delta_k)\in M(\Disc)$ denote by
$$
d^\uparrow
$$
the element of $M(\Disc_+)$ obtained from $d$ by removing every $\delta_i$ such that $e(\delta_i)=0$
and replacing every $\delta_i$ for which $e(\delta_i)<0$ by $\check\delta_i$. Also, denote by
$$
d_u
$$ the multiset in $M(\Disc)$ obtained from $d$ by retaining only the $\delta_i$'s such that $e(\delta_i)=0$.

\begin{proposition}[\cite{MR2504024}] \label{prop: addparms}
Let $d\in M(\Disc)$ and $t=(d',\tau)\in M(\Disc_+)\times\Tempcl$. The tempered representation
$
\lambda(d_u)\rtimes\tau
$
is unitarizable and multiplicity free.
For every irreducible constituent $\tau'$ of $\lambda(d_u)\rtimes\tau$,
the representation
$$
L(d^\uparrow+d';\tau')
$$
occurs with multiplicity one in the Jordan--H\"older sequence of the induced representation
$$
L(d)\rtimes L(d';\tau).
$$
\end{proposition}

\section{Involution}
The Zelevinsky involution is a special case of an involution $D_G$ which exists on the Grothendieck group of the representations
of any connected reductive $p$-adic group. This involution is constructed in \cite{MR1285969} and \cite{MR1471867}.
It takes any irreducible representation to an irreducible representation up to a sign.
For any irreducible representation $\pi$, let $\pi^t$ be the irreducible representation such that $D_G(\pi)=\pm\pi^t$.
We call $\pi^t$ the \ASS involution of $\pi$, or \ASS dual of $\pi$.

This involution is compatible with parabolic induction in the sense that
$$
(\pi\rtimes\tau)^t=\pi^t\rtimes\tau^t
$$
(on the level of Grothendieck groups).

Furthermore, for Jacquet modules, the mapping
$$
\pi_1\otimes\dots\pi_l\otimes\mu\mapsto \check\pi_1^t\otimes\dots\check\pi_l^t\otimes\mu^t,
$$
is a bijection from the semi simplification of $s_\beta(\pi)$ onto the semisimplification of $s_\beta(\pi^t)$
($\beta$ is the partition which parametrizes the corresponding parabolic subgroup).

\section{Reducibility point and generalized Steinberg representations}

Let $\rho\in\Cusp$ with a unitary central character and $\sigma\in\Cuspcl$.
If $\rho\notin\Cuspsd$ then $\nu^x\rho\rtimes\s$ is irreducible for all $x\in\R$.
Otherwise,
\begin{equation}\label{alpha}
\nu^{\alpha_{\rho,\sigma}}\rho\rtimes\s
\end{equation}
is reducible for a unique $\alpha_{\rho,\sigma}\geq0$ (\cite{MR577138}). C. M\oe glin has proved that
$
\alpha_{\rho,\sigma}\in \tfrac12\Z.
$

Given $\rho\in\Cuspsd$, either $\alpha_{\rho,\sigma}\in\Z$ for all $\sigma\in\Cuspcl$ or
$\alpha_{\rho,\sigma}\in\frac12+\Z$ for all $\sigma\in\Cuspcl$.
We say that $\rho$ is of odd type (or parity) in the former case and of even type (or parity) in the latter.
The parity of $\rho$ depends only on $\rho$ and the Witt tower and can be detected by the existence of a pole for a suitable $L$-function.
For instance, for odd orthogonal groups, $\rho$ is of odd (resp., even) type if and only if
the exterior (resp., symmetric) square $L$-function of $\rho$ has a pole at $s=0$.
These conditions are reversed for symplectic groups.

From now on we fix $\rho\in\Cuspsd$ and $\sigma\in\Cuspcl$
and denote the reducibility point $ \alpha_{\rho,\sigma}$ simply by
$$
\alpha.
$$

The representation
$
\nu^{\alpha+ n} \rho \times \nu^{\alpha+ n-1} \rho \times \cdots \times
\nu^{\alpha+ 1} \rho \times \nu^\alpha\rho \rtimes \sigma
$
admits an irreducible socle, which is denoted by
$
\delta([\nu^\alpha\rho, \nu^{\alpha+ n} \rho]; \sigma) \ \ (n\geq 0).
$
It is square-integrable and called a generalized Steinberg representation. We have
\begin{equation} \label{eq: muforgenstn}
\mu^*\left(\delta([\nu^\alpha\rho, \nu^{\alpha+ n}\rho];\sigma)\right) =
\sum^n_{k=-1} \delta([\nu^{\alpha+ k + 1}\rho, \nu^{\alpha+n} \rho]) \otimes \delta([\nu^\alpha\rho,
\nu^{\alpha+ k}\rho]; \sigma),
\end{equation}
$$
\delta([\nu^\alpha\rho, \nu^{\alpha+ n} \rho];\sigma)\tilde{\ }
\cong
\delta([\nu^\alpha\rho,\nu^{\alpha+ n} \rho]; \tilde{\sigma}).
$$
Applying the \ASS involution, we get
\begin{multline} \label{eq: muforgenstnt}
\mu^*\left(L(\nu^{\alpha+n} \rho. \dots , \nu^{\alpha+ 1}\rho, \nu^{\alpha}\rho; \sigma)\right)=
\\\sum^n_{k=-1}L(\nu^{-(\alpha+ n)}\rho,\dots , \nu^{-(\alpha+ k+2)}\rho,\nu^{-(\alpha+k+1)} \rho)
\otimes L(\nu^{\alpha+k} \rho. \dots , \nu^{\alpha+ 1}\rho, \nu^{\alpha}\rho; \sigma).
\end{multline}

\label{sec: Castrick}

The generalized Steinberg representation and its \ASS dual are the only unitarizable irreducible subquotients of $
\nu^{\alpha+ n} \rho \times \nu^{\alpha+ n-1} \rho \times \cdots \times
\nu^{\alpha+ 1} \rho \times \nu^\alpha\rho \rtimes \sigma$ (\cite{MR2652536}, \cite{MR3052686};
see also \cite[\S 13]{MR3969882}).

\section{Representations of segment type} \label{segment}
We shall recall the formulas for Jacquet modules obtained in \cite{MR3360752}.\footnote{The results of \cite{MR3360752}
are only stated for symplectic and split odd-orthogonal groups but the proof works for all classical groups.
Note that the proof does not use the classification of $\Discl$ in terms of $\Cuspcl$.}
We fix $\rho\in\Cuspsd$ and $\sigma\in\Cuspcl$ and
consider irreducible subquotients of $\delta([\nu^{-c}\rho,\nu^{d}\rho])\rtimes\s$, where
$c+d\in \Z_{\geq 0}$.
As above, $\alpha\in\tfrac12\Z_{\geq0}$ denotes the reducibility exponent \eqref{alpha}.

The representation $\delta([\nu^{-c}\rho,\nu^{d}\rho])\rtimes\s$ is multiplicity free of length at most three.
It is reducible (resp., of length three) if and only if $[-c,d]_\Z\cap\{-\alpha,\alpha\}\ne \emptyset$ (resp.,
$\{-\alpha,\alpha\}\subseteq[-c,d]_\Z$ and $c\ne d$).

Assume that $d\ge c$ and $d-\alpha\in\Z$.
We define terms $\delta([\nu^{-c}\rho,\nu^{d}\rho]_+;\s)$, $\delta([\nu^{-c}\rho,\nu^{d}\rho]_-;\s)$ and $L_\alpha([\nu^{-c}\rho,\nu^{d}\rho];\s)$.\footnote{In \cite{MR3360752}
we denoted the last term by $L_\alpha(\delta([\nu^{-c}\rho,\nu^{d}\rho]);\s)$.}
Each of these terms is either an irreducible representation or the trivial (zero-dimensional) representation.
They satisfy
\begin{equation} \label{uuvodu}
\delta([\nu^{-c}\rho,\nu^{d}\rho])\rtimes\s=\delta([\nu^{-c}\rho,\nu^{d}\rho]_+;\s)+\delta([\nu^{-c}\rho,\nu^{d}\rho]_-;\s)+
L_\alpha([\nu^{-c}\rho,\nu^{d}\rho];\s)
\end{equation}
in the corresponding Grothendieck group.

Suppose first that $\delta([\nu^{-c}\rho,\nu^{d}\rho])\rtimes\s$ is irreducible.
Then, we define that $\delta([\nu^{-c}\rho,\nu^{d}\rho]_-;\s)=0$ and require that
$\delta([\nu^{-c}\rho,\nu^{d}\rho]_+;\s)\ne0$ if and only if $[-c,d]\subseteq[-\alpha+1,\alpha-1]$.
By \eqref{uuvodu} this determines $L_\alpha([\nu^{-c}\rho,\nu^{d}\rho];\s)$.

Suppose now that $\delta([\nu^{-c}\rho,\nu^{d}\rho])\rtimes\s$ is reducible.
If $c=d$, we define that $L_\alpha([\nu^{-c}\rho,\nu^{d}\rho];\s)=0$. Otherwise, $L_\alpha([\nu^{-c}\rho,\nu^{d}\rho];\s)=L([\nu^{-c}\rho,\nu^{d}\rho];\s)$.

If $\alpha> 0$, then there is a unique irreducible subquotient $\g$ of $\delta([\nu^{-c}\rho,\nu^{d}\rho])\rtimes\s$
which has in $s_{\GL}(\g)$ an irreducible subquotient $\tau\otimes\s$ such that $\tau$ is generic,
and $e(\theta)\geq 0$ for all $\theta$ in $\supp(\tau)$.
We denote this $\g$ by $\delta([\nu^{-c}\rho,\nu^{d}\rho]_+;\s)$.

If $\alpha= 0$, we write $\rho \rtimes \sigma$ as a sum of irreducible subrepresentations
$\tau_+ \oplus \tau_{-}$.
We denote also $\tau_\pm$ by $\delta([\rho]_\pm;\s)$.
Then, there exists a unique irreducible subquotient of $\delta([\nu^{-c}\rho,\nu^{d}\rho]) \rtimes \s$ that
contains an irreducible representation
of the form $\tau \otimes \tau_\pm$ in its Jacquet module with respect to an appropriate standard parabolic subgroup,
such that $\tau$ is generic, and $e(\theta)\geq 0$ for all $\theta$ in $\supp(\tau)$.
We denote this subquotient by $\delta([\nu^{-c}\rho,\nu^{d}\rho]_\pm;\s)$.

If $c = d$ or the length of $\delta([\nu^{-c}\rho,\nu^{d}\rho])\rtimes\s$ is three, then
$\delta([\nu^{-c}\rho,\nu^{d}\rho])\rtimes\s$ contains a unique irreducible subrepresentation different from
$\delta([\nu^{-c}\rho,\nu^{d}\rho]_+;\s)$ and we denote it by $\delta([\nu^{-c}\rho,\nu^{d}\rho]_-;\s)$.
Otherwise, we take $\delta([\nu^{-c}\rho,\nu^{d}\rho]_-;\s)=0$.

The representations $\delta([\nu^{-c}\rho,\nu^{d}\rho]_\pm;\s)$ are called representations of segment type.

The representation $\delta([\nu^{-c}\rho,\nu^{d}\rho]_+;\s)$ is square-integrable if and only if $c\ne d$
and $\{-\alpha,\alpha\}\subseteq [-c,d]$ or $\alpha=-c$.
If $\delta([\nu^{-c}\rho,\nu^{d}\rho]_+;\s)$ is square-integrable, then $\delta([\nu^{-c}\rho,\nu^{d}\rho]_-;\s)$
is also square-integrable if it is non-zero.
Conversely, if $\delta([\nu^{-c}\rho,\nu^{d}\rho]_-;\s)$ is square-integrable (and non-zero),
then $\delta([\nu^{-c}\rho,\nu^{d}\rho]_+;\s)$ is square-integrable (and non-zero).

In the two formulas below, we symmetrize notation in the following way. We define
$$
\delta([\nu^{-d}\rho,\nu^{c}\rho]_+;\s),\
\delta([\nu^{-d}\rho,\nu^{c}\rho]_-;\s)\text{ and
}L_\alpha([\nu^{-d}\rho,\nu^{c}\rho];\s)
$$
to denote $\delta([\nu^{-c}\rho,\nu^{d}\rho]_+;\s)$, $ \delta([\nu^{-c}\rho,\nu^{d}\rho]_-;\s)$ and
$L_\alpha([\nu^{-c}\rho,\nu^{d}\rho];\s)$ respectively (assumptions on $c$ and $d$ are as above).

\begin{remark}
Now we recall the formulas for the Jacquet modules of segment representations and associated Langlands quotient from \cite{MR3360752}.
We take this opportunity to correct several typographical errors in \cite{MR3360752}.
The upper limit in the first sum of the second row of \eqref{jm-seg-ds} below is $d-1$
(instead of $c$).\footnote{This correction refers to the formulas on page 441 and Corollaries 4.3, 5.4 and 6.4 of \cite{MR3360752}.}
The limits of the first sum in the third row of \eqref{jm-seg-ds} is $-c-1\leq i\leq c-1$ (instead of $-c-1\leq i\leq d$)
in fact, the indices between $c$ and $d$ do not give any contribution.
Further, the limits in the first sum in the second row of \eqref{jm-seg-q} are $-c-1\leq i\leq d-1$ (instead of $-c-1\leq i\leq d$)
since the index $d$ does not contribute show up in the formula.\footnote{Each of these two not so essential modifications
also refers to the formulas on page 441 and Corollaries 4.3, 5.4 and 6.4 of \cite{MR3360752}.}
\end{remark}

If $\delta([-c,d]^{(\rho)})\rtimes\s$ is reducible (with notation as above), then
we have the following equality
\begin{equation} \label{jm-seg-ds}
\begin{aligned}
\mu^*\big(&\delta([-c,d]^{(\rho)}_\pm;\s)\big)=
\sum_{i=-c-1}^{\pm \alpha- 1}
\delta([-i,c]^{(\rho)}) \times \delta([i+1,d]^{(\rho)})\otimes\s
 \\
+&\sum_{i= -c -1}^{ d-1}\sum_{j=i+1}^{d} \delta([-i,c]^{(\rho)}) \times \delta([j+1,d]^{(\rho)}) \otimes
\delta([i+1,j]^{(\rho)}_\pm;\s)
\\
+&\underset {\underset{ i+j<-1} {-c-1\le i\le c-1 \quad i+1\le j\le c}}{\sum \quad \sum}
\delta([-i,c]^{(\rho)}) \times
\delta([j+1,d]^{(\rho)}) \otimes
L_\alpha([i+1,j]^{(\rho)};\s).
\end{aligned}
\end{equation}
If additionally $c\ne d$, and either $c < \alpha$ or $\alpha\leq c<d$, then we have
\begin{equation}
\begin{gathered} \label{jm-seg-q}
\mu^*\big(L([-c,d]^{(\rho)};\sigma)\big)=
\mu^*\big(L_\alpha([-c,d]^{(\rho)};\sigma)\big)=
\sum_{i=\alpha}^dL([-i,c]^{(\rho)}, [i+1,d]^{(\rho)}) \otimes\sigma+
\\\sum_{-c-1\le i\le d-1}\ \sum_{\substack{i+1\le j\le d\\0\le i+j}}
L([-i,c]^{(\rho)},[j+1,d]^{(\rho)}\big) \otimes L_\alpha([i+1,j]^{(\rho)};\sigma).
\end{gathered}
\end{equation}

\section{Jordan blocks}
Now we shall recall the definition of the Jordan blocks $\Jord(\pi)$ of
an irreducible square-integrable representation $\pi$ of $S_n$.

\begin{definition}
For any $\pi\in\Discl$ denote by $\Jord(\pi)$ the set of all square-integrable representations $\delta(\rho,a)\in\Disc$ where
$\rho\in\Cuspsd$ and $a\in\Z_{>0}$ is of the same parity of $\rho$,
such that $\delta(\rho,a)\rtimes\pi$ is irreducible.
For any $\rho\in\Cuspsd$ we denote
\[
\Jord_\rho(\pi)=\{a:(\rho,a)\in \Jord(\pi)\}\subset\Z_{>0}.
\]
\end{definition}

The set $\Discl$ is classified by admissible triples (see \cite{MR1896238} for details).
Any $\pi\in\Discl$ is parameterized by a triple $(\Jord(\pi),\e_\pi,\pi_{cusp})$, where $\e_\pi$ is a
function\footnote{It is called partially defined function attached to $\pi$.} defined on a subset of $\Jord(\pi)\cup \Jord(\pi)\times \Jord(\pi)$ which takes values in $\{\pm1\}$, and $\pi_{cusp}$ is the partial cuspidal support
(which was defined earlier).

The construction of irreducible square-integrable representations in \cite{MR1896238} starts with strongly positive representations ($\epsilon_\pi$ is alternating function on  $\Jord_\rho(\pi)$ in this case).
The simplest example of such representations is the generalized Steinberg representations.
We shall give one more example of strongly positive representations.

Assume that the reducibility point $\alpha=\alpha_{\rho,\s}$ is strictly positive.
Take $k\in\Z_{\geq0}$ such that $k<\alpha$. Then, the representation
$
\nu^{\alpha-k}\rho\times \nu^{\alpha-k+1}\rho\times\dots\times \nu^{\alpha}\rho\rtimes\s
$
admits an irreducible socle, which we denote by
$$
\delta([\nu^{\alpha-k}\rho],[ \nu^{\alpha-k+1}\rho],\dots,[ \nu^{\alpha}\rho];\s).
$$
This is an example of strongly positive (square-integrable) representation.

Sometimes when we deal with strongly positive representations, to stress this we shall add subscript $\spsi$
(we shall not do this for the generalized Steinberg representations).
Therefore, the above representations we shall also denote by
$\delta_{\spsi}([\nu^{\alpha-k}\rho],[ \nu^{\alpha-k+1}\rho],\dots,[ \nu^{\alpha}\rho];\s).$

Now we recall Proposition 6.1 from \cite{MR2504024}, which we use several
times in the paper. (Note that in (vii) of Proposition 6.1 in
\cite{MR2504024}, the condition on parity was forgotten by mistake.)

\begin{proposition} \label{Pr-red-si}
Let $\pi\in\Discl$, $\rho\in\Cuspsd$ and $a>0$. Then,
\begin{enumerate}
\item $\nu^a\rho\rtimes\pi$ is reducible if and only if $\nu^{-a}\rho\rtimes\pi$ is reducible.
\item If $a\notin\tfrac12\Z$, then $\nu^a\rho\rtimes\pi$ is irreducible.
\item $\rho\rtimes \pi$ is reducible if and only if $\rho$ has odd parity and $1\not\in\Jord_\rho(\pi)$.
\item 
\label{Pr-red-si-item: a-a+1-irr}
If $a\not\in\Jord_\rho(\pi)$, then $\nu^{(a+1)/2}\rho\rtimes\pi$ is irreducible.
\item 
\label{Pr-red-si-item: a-not-a+2-red}
If $a\in\jrp$ and $a+2\not\in\jrp$, then $\nu^{(a+1)/2}\rho\rtimes\pi$ is reducible.
\item 
\label{Pr-red-si-item: a-a+1-epsilon}
Suppose that $a$ and $a+2$ are in $\jrp$. Then,
$\nu^{(a+1)/2}\rho\rtimes\pi$ is reducible if and only if
$\e((\rho,a))=\e((\rho,a+2))$.
\item $\nu^{1/2}\rho\rtimes\pi$ is reducible if and only if $\rho$ is of even parity and either $2\not\in \jrp$,
or $2\in\jrp$ and $\e((\rho,2))=1$.

\noindent
In other words, $\nu^{1/2}\rho\rtimes\pi$ is irreducible if and only if
either $\rho$ is of odd parity or $2\in \jrp$ and $\e((\rho,2))=-1$.
\end{enumerate}
\end{proposition}

\subsection{On computing of Jordan blocks}

Next, we recall how to compute the Jordan blocks of any $\pi'\in\Discl$.

Suppose that $\pi'$ is a subquotient of
$$
\rho_1\times\dots\times\rho_r\rtimes\sigma,
$$
where $\rho_i\in\Cusp$ and $\sigma\in\Cuspcl$.
Fix $\rho\in\Cuspsd$ and let $\alpha=\alpha_{\rho,\sigma}$.
Let $\eta=0$ if $\alpha\in\Z$ and $\eta=1$ otherwise.

If no $\nu^x\rho$, $x\in \tfrac12\Z$ is a factor of $\pi'$, then
$$
\Jord_\rho(\pi')=\{\eta+1,\eta+3,\dots,2\alpha-3,2\alpha-1\}.
$$
(Note that for $\alpha=0$, $\Jord_\rho(\pi')=\emptyset.)$

In the general case, we recall Proposition 2.1 of \cite{MR1896238},
from which we can compute Jordan blocks in general.

\begin{proposition} \label{JBcomputation}
Let $x,y \in\frac12\Z$ be such that $x-y\in \Z_{\geq 0}$ and $x-\alpha\in\Z$.
Suppose that $\pi\in\Discl$ embeds in the induced representation
$$
\pi \hookrightarrow \nu^x \rho\times \cdots \times
\nu^{x-i+1}\rho
\times \cdots \times \nu^y \rho\rtimes \pi'.
$$
Then,
\begin{enumerate}
\item If $y>0$, then $2y-1\in\Jord_\rho(\pi')$ and
$$
\Jord_\rho(\pi)=\left(\Jord_\rho(\pi')\setminus\{2y-1\}\right)\cup\{2x+1\}.
$$
\item If $y\leq 0$ , then $2x+1$ and $-2y+1$ are not in $\Jord_\rho(\pi')$ and
$$
\Jord_\rho(\pi)=\Jord_\rho(\pi')\cup \{2x+1,-2y+1\}.
$$
\end{enumerate}
\end{proposition}

\begin{remark} \label{general-basic}
The Local Langlands correspondence for the general linear group (\cite{MR1876802}, \cite{MR1738446}) gives a bijection $\Phi$
between the set $\Disc$ (resp., $\Cusp$) and the irreducible representations of the Weil--Deligne (resp., Weil) group.\footnote{An alternative
classification of $\Cusp$ is given in \cite{MR1204652}.}
(Recall that $(\rho,n)\mapsto \delta(\rho,n)$ gives a parametrization of $\Disc$ by $\Cusp\times \Z_{>0}$.)

J. Arthur has obtained in \cite[Theorem 1.5.1]{MR3135650} a classification of irreducible tempered representations of classical groups
attaching to them pairs of admissible homomorphisms and characters of the component groups of the admissible homomorphisms.
C. M\oe glin has proved in \cite[Theorem 1.3.1]{MR2767522} that the admissible homomorphism attached to $\pi\in\Discl$ is
\begin{equation*} \label{eq-sum-L}
\underset{\s\in \Jord(\pi)}\oplus\Phi(\s).
\end{equation*}
In this way one gets a classification of $\Discl$ in terms of (certain) finite sets of $\Cuspsd$ and functions on these sets with values in $\{\pm1\}$.
In addition, \cite[Theorem 1.5.1]{MR2767522} $\pi\in\Cuspcl$ if and only if
for any $\rho\in\Cuspsd$, $\e_\pi(\delta(k,2))=-1$ whenever $2\in\Jord_\rho(\pi)$ and if $k\in\Jord_\rho(\pi)$ with $k>2$ then
$k-2\in\Jord_\rho(\pi)$ and $\e_\pi((\delta,k))\ne\e_\pi((\delta,k-2))$.
Finally, if $\sigma\in\Cuspcl$ and $\Jord_\rho(\sigma)\ne\emptyset$ then
$$
\alpha_{\rho,\sigma}=\tfrac{1+\max\Jord_\rho(\pi)}2
$$
while if $\Jord_\rho(\sigma)=\emptyset$ then $\alpha_{\rho,\sigma}$ is $0$ or $\frac12$ according to the parity of $\rho$
(see \cite{MR1913095} or \cite{MR1896238}).
One can find in \cite{MR3156864} more details and precise statements related to the above discussion.
\end{remark}

\section{Induction of \texorpdfstring{$\GL$}{GL}-type} \label{GL-LT}
Next we shall recall the results of \cite{1703.09475}, except that we shall formulate them in terms of the Langlands classification.
As above, $\alpha=\alpha_{\rho,\s}$ denotes the reducibility point (then $\rho\cong\rho\check{\ }$).
Let $\pi\in\Irr$.

If $\supp(\pi)$ contains $\nu^\alpha\rho$ or $\nu^{-\alpha}\rho$, then $\pi\rtimes\s$ is reducible (\cite{1703.09475}).

Suppose now that $\supp(\pi)$ does not contain $\nu^\alpha\rho$ or $\nu^{-\alpha}\rho$.
Assume that all members of $\supp(\pi)$ are contained in $\{\nu^{k+x}\rho:k\in\Z\}$, for some fixed $x\in\tfrac12\Z$.
Write $\pi=L(d)$, for some $d\in M(\Disc)$. Denote by $d_{>0}$ (resp. $d_{<0}$) the multiset consisting of all $\delta$ in $d$
such that $e(\delta)>0$ (resp $e(\delta)<0$), counted with multiplicities.
Then, if $\pi$ is a ladder representation or if $\alpha\leq 1$ and all members of $\supp(\pi)$ are contained in $\{\nu^{k+\alpha}\rho;k\in\Z\}$, then holds
\begin{equation} \label{eq: ladcrit}
L(d)\rtimes\s \text{ is reducible } \iff L(d_{>0})\times L(d_{<0})\check{\ } \text{ is reducible }.
\end{equation}

A very special case is the following very useful result proved already in \cite{MR1658535}. For $\D\in \SC$ holds:
$$
\delta(\D)\rtimes\s \text{ is reducible }\iff \theta\rtimes\s \text{ is reducible for some }\theta \in \D.
$$

\begin{remark} \label{rem-irr}
\begin{enumerate}
\item
We shall often use the following simple consequence of Proposition 3.2 of \cite{MR2504024}.
Let $\rho\in\Cuspsd$ and assume that $\pi\in\Irr$ is supported on
$$
\{\nu^{x+z}\rho:z\in \Z\} \quad \text{ for some fixed } \quad x\in \R\, \setminus \,\tfrac12 \Z.
$$
Then, $\pi\rtimes\s$ is irreducible.

\item
One can combine the above fact with the Jantzen decomposition (see section 8 of \cite{MR3969882}) to get further irreducibilities.
We shall do it later in the paper. One can get these irreducibilities also directly from Proposition 3.2 of \cite{MR2504024}.

\end{enumerate}
\end{remark}

\section{Technical lemma on irreducibility}
\begin{lemma} \label{lemma-irr}
Let $d_1,d_2,d_3\in M(\Disc_+)$ and $\tau\in \Tempcl$. Write $d_i=(\delta_1^{(i)},\dots, \delta_{k_i}^{(i)})$, $i=1,2,3$. Suppose
\begin{enumerate}
\item
\label{lemma-irr-item: irr1}
 $L(d_1)\times L(d_2)$ is irreducible;

\item
\label{lemma-irr-item: irr2}
 $L(d_1)\times L(d_2)\check{\ }$ is irreducible;

\item
\label{lemma-irr-item: irr3}
 $L(d_1)\times L(d_3;\tau)$ is irreducible;

\item 
\label{lemma-irr-item: bigger}
$e(\delta_j^{(i)})\geq e(\delta_l^{(3)})$ for all \ $i=1,2$, \ $1\leq j\leq k_i$, \ $1\leq l \leq k_3$;

\item 
\label{lemma-irr-item: theta}
$\bar d_i\cong d_i^\Theta$, $i=1,2,3.$
\end{enumerate}
Then,
$$
L(d_1)\rtimes L(d_2+d_3;\tau)
$$
is irreducible.
\end{lemma}

\begin{proof}
First, $L(d_2+d_3;\tau)$ is the unique irreducible quotient of $\lambda(d_2+d_3;\tau)$.
Condition (\ref{lemma-irr-item: bigger}) implies that $L(d_2)\rtimes L(d_3;\tau) $ is also quotient of $\lambda(d_2+d_3;\tau) $.
Therefore, $L(d_2+d_3;\tau)$ is (the unique irreducible) quotient of $L(d_2)\rtimes L(d_3;\tau) $.
This implies that $L(d_1)\rtimes L(d_2+d_3;\tau)$ is a quotient of $L(d_1)\times L(d_2)\rtimes L(d_3;\tau) $.
Since $L(d_1)\times L(d_2)\rtimes L(d_3;\tau) = L(d_1+d_2)\rtimes L(d_3;\tau) $, by condition (\ref{lemma-irr-item: irr1})
the last representation is a quotient of $\lambda(d_1+d_2+d_3;\tau)$.
This implies that $L(d_1)\rtimes L(d_2+d_3;\tau)$ has a unique irreducible quotient, which is $L(d_1+d_2+d_3;\tau)$,
and that this quotient occurs with multiplicity one.
Observe that (\ref{lemma-irr-item: irr1}) -- (\ref{lemma-irr-item: irr3}) imply $L(d_1)\times L(d_2)\rtimes L(d_3;\tau)\cong L(d_1)\check{\ }\times L(d_2)\rtimes L(d_3;\tau)$.
Therefore, $L(d_1)\rtimes L(d_2+d_3;\tau)$ is a quotient of $L(d_1)\check{\ }\times L(d_2)\rtimes L(d_3;\tau)$.

Obviously, $L(d_1)\check{\ }\rtimes L(d_2+d_3;\tau)$ is a quotient of $L(d_1)\check{\ }\times L(d_2)\rtimes L(d_3;\tau)$,
which implies that $L(d_1+d_2+d_3;\tau)$ is a quotient of $L(d_1)\check{\ }\rtimes L(d_2+d_3;\tau)$.
Now observe that (\ref{lemma-irr-item: theta}) implies $L(d_1+d_2+d_3;\tau)^+\cong L(d_1+d_2+d_3;\tau)$ and $L(d_2+d_3;\tau)^+\cong L(d_2+d_3;\tau)$.
Therefore,
$$
L(d_1+d_2+d_3;\tau)\h (L(d_1)\check{\ })^+\rtimes L(d_2+d_3;\tau)\cong L(d_1)\rtimes L(d_2+d_3;\tau).
$$
This implies the irreducibility of $L(d_1)\rtimes L(d_2+d_3;\tau)$ since $L(d_1+d_2+d_3;\tau)$ is a unique irreducible quotient
of $L(d_1)\rtimes L(d_2+d_3;\tau)$, and it has multiplicity one in $L(d_1)\rtimes L(d_2+d_3;\tau)$.
\end{proof}

In the paper, we shall use the following special case of the above lemma:

\begin{corollary} \label{irr-simple}
Let $\rho\in\Cuspsd$ and $\tau\in \Tempcl$.
Suppose that $d_1,d_2 \in M(\Disc_+)$ are such that all elements in their supports are contained
in either $\{\nu^{k+\frac12}\rho:k\in\Z\}$ or $\{\nu^{k}\rho:k\in\Z\}$.
If the following three representations
$$
L(d_1)\times L(d_2), \quad L(d_1)\times L(d_2\check{\ }), \quad L(d_1)\rtimes\tau
$$
are irreducible, then
$$
L(d_1)\rtimes L(d_2;\tau)
$$
is irreducible.\qed

\end{corollary}

\section{Distinguished irreducible subquotient in induced representation} \label{dist}
Fix $\rho\in\Cuspsd$ and $\sigma\in\Cuspcl$.
Let $c$ be a multiset of elements of representations in
$$
\{\nu^{k+\frac12}\rho:k\in\Z\}\ (\subseteq M(\Cusp)\subseteq M(\Disc)).
$$
Then, $\lambda(c^\uparrow)$ has a unique generic irreducible subquotient (which has multiplicity one in $\lambda(c^\uparrow)$).
Denote it by $\lambda(c^\uparrow)_{gen}$.
Now the formula \eqref{mu*} directly implies that the multiplicity of $\lambda(c^\uparrow)_{gen}\otimes\s$ in
$s_{\GL}(\lambda(c)\rtimes\s)$ is one.
This implies that $\lambda(c)\rtimes\s$ has a unique irreducible subquotient $\pi$ which contains
$\lambda(c^\uparrow)_{gen}\otimes\s$ in $s_{\GL}(\pi)$ as a subquotient. We denote this $\pi$ by
$$
\lambda(c;\rho)_+.
$$
Clearly, it has multiplicity one in $\lambda(c)\rtimes\s$.

Let now $c$ be a (finite) multiset of elements of
$$
\{\nu^{k}\rho:k\in\Z\}.
$$
Then, $\lambda(c^\uparrow+c_u)$ has a unique generic irreducible subquotient (which has multiplicity one in $\lambda(c^\uparrow+c_u)$).
Denote it by $\lambda(c^\uparrow+c_u)_{gen}$.
Again the formula \eqref{mu*} directly implies that the multiplicity of
$\lambda(c^\uparrow+c_u)_{gen}\otimes\s$ in $s_{\GL}(\lambda(c)\rtimes\s)$ is $2^{m(\rho,c)}$,
where $m(\rho,c)$ is the multiplicity of $\rho$ in $c$.
In this case we have the following

\begin{lemma} \label{lem: crho+}
Suppose that $c$ does not contain the reducibility point $\nu^\alpha\rho$, or that $\alpha>0$.
Then, $\lambda(c)\rtimes\s$ has a unique irreducible subquotient $\pi$ which contains
$\lambda(c^\uparrow+c_u)_{gen}\otimes\s$ in $s_{\GL}(\pi)$ as a subquotient. We denote this $\pi$ by
$$
\lambda(c;\rho)_+.
$$
It has multiplicity one in $\lambda(c)\rtimes\s$ and its Jacquet module contains $\lambda(c^\uparrow+c_u)_{gen}\otimes\s$ with
multiplicity $2^{m(\rho,c)}$.
\end{lemma}

\begin{proof}
It suffices to show that the Jacquet module of $
\lambda(c;\rho)_+.
$ contains $\lambda(c^\uparrow+c_u)_{gen}\otimes\s$ with multiplicity $2^{m(\rho,c)}$. We now prove this claim.

First consider the case when $c=\sum_{i=1}^n \D_i$, for some $\D_i\in\SC$ such that $\D_i\check{\ }=\D_i$ for all $i$.
We shall see by induction that in this case the lemma holds, and we shall show that
$\lambda(c;\rho)_+$ is a subrepresentation of $ (\prod _{i=1}^n \delta(\D_i))\rtimes\s$.
By the theory of $R$-groups it is enough to prove the claim when all $\D_i$ are different, and all $\D_i$ contain $\nu^\alpha\rho$.

Let $i=1$. Denote $\D_{1, \alpha\leq }:=\{\nu^\beta\rho\in\D_i;\beta\geq \alpha\}$.
Consider $\delta(\D_1\setminus \D_{1, \alpha\leq })\rtimes\delta(\D_{1, \alpha\leq };\s)$ and $\delta(\D_1)\rtimes\s$.
The last representation has length two.
In the Jacquet module of both representations, $\lambda(c^\uparrow+c_u)_{gen}\otimes\s$ has multiplicity 2 (as well as of $\lambda(c)\rtimes\s$).
From Jacquet module easily follows that
$\delta(\D_1)\rtimes\s\not\leq \delta(\D_1\setminus \D_{1, \alpha\leq })\rtimes\delta(\D_{1, \alpha\leq };\s)$.
This, together with the above multiplicities of $\lambda(c^\uparrow+c_u)_{gen}\otimes\s$, imply the claim.

For $i=2$, we consider $\delta(\D_1)\rtimes \lambda(\D_2;\rho)_+$ and $\delta(\D_2)\rtimes \lambda(\D_1;\rho)_+$.
We conclude in a similar way. Multiplicity of $\lambda(c^\uparrow+c_u)_{gen}\otimes\s$ is now 4 in both Jacquet modules (and in Jacquet module of $\lambda(c)\rtimes\s$).

For the general case, we consider $\delta(\D_1)\rtimes \lambda(\D_2+\dots+\D_n;\rho)_+$ and
$\delta(\D_n)\rtimes \lambda(\D_1+\dots+\D_{n-1};\rho)_+$ (see also Proposition 5.1 of \cite{MR3123571}, and its proof).

Now we go to the proof of the general case.
The first observation is that one can easily show that there exists $c'\in M(\Cusp)$ such that
\begin{enumerate}
\item $\ssm(\lambda(c)\rtimes\s)=\ssm(\lambda(c')\rtimes\s)$
\item
\label{6conditions}
 there exist $\D_1,\dots,\D_k,\G_1,\dots,\G_l\in \SC$ such that

\begin{enumerate}
\item $c'=\D_1+\dots+\D_k+\G_1+\dots+\G_l$;

\item $\mathfrak c(\D_i)\geq 0$ and $\rho\in\D_i$, $i=1,\dots,k$;

\item $\D_{i+1}\cup\D_{i+1}\check{\ }\subseteq \D_{i}\cap\D_{i}\check{\ }$, $i=1,\dots, k-1$;

\item $\mathfrak c(\G_j)> 0$ and $\rho\not\in \G_j$, $j=1,\dots,l$;

\item $\G_j$ is not linked to any other $\G_{j'}$, or any $\D_i$, $j=1,\dots,l$

\item $\G_j\check{\ }$ is not linked to any $\D_i$, $j=1,\dots,l$\footnote{To get these segments,
one consider $\nu^x\rho\in c$ with maximal $|x|$.
Then, $\nu^{|x|}\rho$ is the right end of $\D_1$ or $\G_1$ (it depends on the fact if by the process that follows one
will reach $\rho$ or not).
Then, one looks if $\nu^{|x|-1}\rho\in c$ or $\nu^{-(|x|-1)}\rho\in c$
(if there is no such a member, then the first segment consists of $\nu^{|x|}\rho$ and we repeat above
search with $c-(\nu^x\rho)$).
If yes, one has the next point of the segment of cuspidal representations and we continue the above procedure
(looking for an exponent which is smaller for one then the previous exponent) as long as we can,
in forming the first segment of cuspidal representations.
After we cannot continue the above procedure, we have got the first segment (which is $\D_1$ or $\G_1$,
depending if $\rho$ is in it, or not).
Now we repeat the above procedure with $c$ from which we have removed terms used in the above process.
We repeat these steps as long as there are remaining members of $c$. In this way one gets segments in (\ref{6conditions}).}.

\end{enumerate}

\end{enumerate}

Observe that $k=m(\rho,c)$.

Suppose that the claim does not hold for this $c$ (and $\s$).
This implies that in $\lambda (c)\rtimes\s$ there exists an irreducible subquotient $\pi$ such that
the multiplicity $m$ of $\lambda(c^\uparrow+c_u)_{gen}\otimes\s$ in $s_{\GL}(\pi)$ satisfies $0<m<2^k$.
We know
\begin{equation}\label{leq1}
\pi\leq ( \prod_{i=1}^k\delta(\D_i))\times ( \prod_{j=1}^l\delta(\G_j))\rtimes\s,
\end{equation}
since the multiplicity of $\lambda(c^\uparrow+c_u)_{gen}\otimes\s$ in the Jacquet module of the right-hand side is $2^k$,
which is the same as it is in $\lambda(c')\rtimes\s$ (we shall use this argument also below, without repeating this explanation).
The above inequality implies
$$
( \prod_{i=1}^k\delta(\D_i\check{\ }\setminus \D_i))\rtimes\pi\leq
( \prod_{i=1}^k\delta(\D_i\check{\ }\setminus \D_i))\times
( \prod_{i=1}^k\delta(\D_i))\times ( \prod_{j=1}^l\delta(\G_j))\rtimes\s.
$$
Denote $\D_i''=\D_i\cup \D_i\check{\ }$ and $c''=\D_1''+\dots+\D_k''+\G_1+\dots+\G_l$.
Considering how on the right-hand side we get $\lambda((c'')^\uparrow+c''_u)_{gen}\otimes\s$
in the Jacquet module (all of them we get from terms of $\lambda(c^\uparrow+c_u)_{gen}\otimes\s$ multiplying with
$\delta(\D_i\setminus\D_i\check{\ })$'s and taking appropriate subquotient),
we conclude that its multiplicity in the left-hand side is $m$.
Therefore, there is an irreducible subquotient $\pi''$ of the left-hand side which has
$\lambda((c'')^\uparrow+c''_u)_{gen}\otimes\s$ in its Jacquet module with multiplicity $m$.
Now in the same way as in the case of \eqref{leq1}, we conclude
\begin{equation} \label{leq2}
\pi''\leq ( \prod_{j=1}^l\delta(\G_j)) \times ( \prod_{i=1}^k\delta(\D_i''))\rtimes\s.
\end{equation}
Write $\G_j=[\nu^{g_{j,b}}\rho,\nu^{g_{j,e}}\rho]$. Denote
$\G_j'=[\nu^{-g_{j,e}}\rho, \nu^{g_{j,b}-1}\rho]$,
$\G_j''=\G_j\cup \G_j'$, $c'''=\D_1''+\dots+\D_k''+\G_1''+\dots+\G_l''$. Then,
$$
( \prod_{j=1}^l\delta(\G_j'))\rtimes\pi''\leq ( \prod_{j=1}^l\delta(\G_j'))\times ( \prod_{j=1}^l\delta(\G_j)) \times ( \prod_{i=1}^k\delta(\D_i''))\rtimes\s.
$$
Considering how on the right-hand side we get $\lambda((c''')^\uparrow+c'''_u)_{gen}\otimes\s$ in the Jacquet module (all of them we get from terms of
$\lambda((c'')^\uparrow+c_u'')_{gen}\otimes\s$ multiplying with the following two $\delta([\rho,\nu^{g_{j,e}}\rho])\times\delta([\nu\rho,\nu^{g_{j,b}-1}\rho])$, $\delta([\nu\rho,\nu^{g_{j,e}}\rho])\times\delta([\rho,\nu^{g_{j,b}-1}\rho])$ subquotients of $M^*(\delta(G_j'))$'s, and taking appropriate irreducible subquotient),
we conclude that its multiplicity is $2^lm$ in the left-hand side, which is strictly smaller then $2^{k+l}$. Directly follows that this multiplicity is positive.
This is a contradiction with the first part of the proof. The proof is now complete.
\end{proof}

\begin{remark}
\begin{enumerate}
\item If $
\s$ is generic, then $\lambda(c;\rho)_+$ is generic $($then $\alpha\in\{\frac12,1\}$; we do not consider here reducibility at $0$).

\item Since $\Cusp\subseteq \Disc$, then $M(\Cusp)\subseteq M(\Disc)$.
We say that $\pi\in\Irr$ is cogeneric if $\pi=L(d)$ for some $d\in M(\Cusp)$.

Let $c$ be a multiset of elements of $\{\nu^{k+\frac12}\rho:k\in\Z\}$.
Then, $\lambda(c^\uparrow)\check{\ }$ has a unique cogeneric irreducible subquotient
(which has multiplicity one in $\lambda(c^\uparrow)\check{\ }$).
Denote it by $\lambda(c_\downarrow)_{cogen}$.
Now the formula \eqref{mu*} directly implies that the multiplicity of $\lambda(c_\downarrow)_{cogen}\otimes\s$
in $s_{\GL}(\lambda(c)\rtimes\s)$ is one.
This implies that $\lambda(c)\rtimes\s$ has a unique irreducible subquotient $\pi$ which contains
$\lambda(c_\downarrow)_{cogen}\otimes\s$ in $s_{\GL}(\pi)$ as a subquotient. We denote this $\pi$ by
$$
\lambda(c;\rho)_-.
$$
We define representation $\lambda(c;\rho)_-$ analogously for a multiset $c$ of representations in $\{\nu^{k+\frac12}\rho:k\in\Z\}$
(in this case we consider $\lambda(c^\uparrow\check{\ }+c_u)_{cogen}\otimes\s$).
Then, the analogue of Lemma \ref{lem: crho+} holds for $\lambda(c;\rho)_-$.
Furthermore,
\begin{equation} \label{dist-formula}
\lambda(c;\rho)_+^t=\lambda(c;\rho)_-.
\end{equation}
\end{enumerate}
\end{remark}

Suppose that $\pi=\lambda(c;\rho)_+$ is square-integrable, and $\D$ is a segment contained in
$\{\nu^{k+\frac12}\rho:k\in\Z\}$ or $\{\nu^{k}\rho:k\in\Z\}$, such that $\D=\D\check{\ }$.
Then, we denote $\pi=\lambda(c+\D;\rho)_+$ by
$$
\tau(\D_+;\pi).
$$
One directly sees that $\tau(\D_+;\pi)\leq \delta(\D)\rtimes\pi$.
Furthermore, if $\delta(\D)\rtimes\pi$ is reducible, then it decomposes into a direct sum of two nonequivalent irreducible (tempered) representations.
The other one we denote by
$$
\tau(\D_-;\pi).\footnote{Note that $\tau(\D_-;\pi)$ is not related to the representations of type $\lambda(c;\rho)_-$.
Furthermore, note that if $\pi$ is a cuspidal representation $\s$, then $\tau(\D_\pm;\pi)=\delta(\D_\pm;\pi)$.}
$$

\section{Some well-known ways of obtaining unitarizability} \label{ways of getting unitary}
The standard way of obtaining new unitarizable representations from old ones is by parabolic induction,
which preserves unitarity (unitary parabolic induction).

One can also deduce unitarizability in the opposite direction.
Namely, if $\theta$ is an irreducible hermitian representation of a Levi subgroup $M$
of a parabolic subgroup $P$ of a reductive group $G$, and if Ind$_P^G(\theta)$ is irreducible and unitarizable, then $\theta$ is unitarizable.
This method of proving unitarizability will be called \emph{unitary parabolic reduction}.

A third way of proving unitarizability is by considering limits.
If $\pi_n$ is a sequence of irreducible unitarizable representations of a reductive group $G$,
$\tau_i$ irreducible representations of $G$ and $m_i\in\Z_{>0}$ such that distribution characters
$\Theta_{\pi_n}$ of $\pi_n$ converge pointwise to $\sum_im_i\Theta_{\tau_i}$, then all $\tau_i$ are unitarizable (\cite{MR0324429}).

A fourth way of proving unitarizability is by considering families.
If a continuous family of irreducible hermitian representations of a reductive groups $G$
contains at least one unitarizable representation, then all representations in the family are unitarizable.
(For a definition of continuous family of representation see \cite[\S 3 (b)]{MR1181278}.)

Furthermore, if a continuous family of irreducible hermitian representations is parameterized by unbounded set of unramified parameters,
then all the representations in the family are non-unitarizable (see \cite{MR733166} for more details).

The above methods of proving unitarizability can be easily modified for proving non-unitarizability.

\subsection{}
We will use the following elementary, but powerful, consequence of unitarizability.
\begin{lemma} \label{lem: nonunit}
Let $P$ be a parabolic subgroup of a reductive group $G$ over a non-archimedean local field.
Let $\sigma$ be an irreducible representation of the Levi part $M$ of $P$ and let $\Pi=\Ind_P^G\sigma$.
Assume that $\sigma$ or $\sigma^t$ is unitarizable.
Then, the length $\ell$ of $\Pi$ is at most the multiplicity $m$ of $\sigma$ in the Jacquet module $J_P(\Pi)$.
Moreover, if $\ell=m$, then
there exists a direct summand of $J_P(\Pi)$ isomorphic to $m\cdot\sigma=\overbrace{\sigma\oplus\dots\oplus\sigma}^m$.
\end{lemma}

\begin{proof}
Using duality, it is enough to consider the case where $\sigma$ unitarizable.
In this case $\Pi$ is unitarizable, hence semisimple and therefore
$\ell(\Pi)\le\dim\Hom_G(\Pi,\Pi)$ which by Frobenius reciprocity is equal $\dim\Hom_M(J_P(\Pi),\sigma)\le m$.

Moreover, if $\ell=m$ then necessarily $m=\dim\Hom_M(J_P(\Pi),\sigma)$.
Upon replacing $\sigma$ by its contragredient and $P$ by an opposite parabolic, we also get that
$m=\dim\Hom_M(\sigma,J_P(\Pi))$.
Therefore,
\[
\omega:=\sum_{\varphi\in\Hom_M(\sigma,J_P(\Pi))}\Img(\varphi)\cong m\cdot\sigma\text{ and }
J_P(\Pi)=\omega\oplus\bigcap_{\varphi\in\Hom_M(J_P(\Pi),\sigma)}\Ker\varphi
\]
as claimed.
\end{proof}

\section{Reduction of unitarizability to the weakly real case} \label{reduction-u-w}
$\pi\in\Irrcl$ is called weakly real if it is a subquotient of a representation of the form
$$
\nu^{r_1}\rho_1\times\dots \times\nu^{r_k}\rho_k\rtimes\s,
$$
where $\rho_i\in \Cusp$ satisfy $\rho_i\cong\rho_i\check{\ }$,
$r_i\in \R$ and $\s\in\Cuspcl$.
Now we recall \cite[Theorem 4.2 (i)]{MR2504024}:

\begin{theorem} \label{unitary-red-1}
For any unitarizable $\pi\in\Irrcl$ is unitarizable there exist a unitarizable $\theta\in\Irr$ and a weakly real unitarizable
$\pi'\in\Irrcl$ such that
$$
\pi\cong\theta\rtimes\pi'.
$$
\end{theorem}

Note that \cite[Theorem 4.2 (ii)]{MR2504024} gives a more precise reduction then the above theorem.
Since \cite[Theorm 7.5]{MR870688} (which we recall below) gives a classification of unitary duals of general linear groups,
the above theorem reduces the unitarizability problem for classical $p$-adic groups to the weakly real case.

For $\delta\in \Disc_u$ and $m\geq 1$ denote by $u(\delta,m)$ the unique irreducible quotient of
$\nu^{\frac{m-1}2}\delta\times\nu^{\frac{m-1}2-1}\delta\times\dots \times\nu^{-\frac{m-1}2}\delta,$
which is called a Speh representation. Let
$
B_{rigid}
$
be the set of all Speh representations, and
\[
B=B(F)=B_{rigid}\cup \{\nu^\alpha\s\times\nu^{-a}\s;\s\in B_{rigid}, 0<\alpha<1/2\}.
\]
Now the following simple theorem solves the unitarizability for archimedean and non-archimedean general linear groups in the uniform way:

\begin{theorem} \label{ud-gl}
The mapping $(\s_1,\dots,\s_k)\mapsto \s_1\times\dots\times\s_k$
defines a bijection between $M(B)$ and $\cup_{n\geq0}\GL(n,F)\widehat{\ }$.
\end{theorem}

\section{Computing irreducible subquotients}
The following simple lemma will be often used for computing irreducible subquotients.
\begin{lemma} \label{lem: simpac} \ \
\begin{enumerate}
\item Suppose that we have a diagram
\[
\xymatrix{
0\ar[r] & \Sigma_2 \ar[r] & \Sigma_1 \ar@{->>}[d] \ar[r] & \Pi \ar[r]&0\\
& & \Sigma
}
\]
where the top row is a short exact sequence.
Assume that $\Sigma$ is finitely generated and $\Sigma\not\le\Sigma_2$.
Then, the cosocle of $\Sigma$ and the cosocle of $\Pi$ have a common irreducible constituent.
In particular, if the cosocle of $\Pi$ is irreducible, then it occurs in the cosocle of $\Sigma$.
\item Suppose that we have a diagram
\[
\xymatrix{
0\ar[r] & \Sigma_2 \ar[r] & \Sigma_1 \ar@{->>}[d] \ar[r] & \Pi \ar[r]&0\\ & & \Sigma
}
\]
where the top row is a short exact sequence.
Assume that $\Sigma_2$ is finitely generated and $\Sigma\not\le\Pi$.
Then, $\Sigma$ and the cosocle of $\Sigma_2$
have a common irreducible subquotient.
In particular, if the cosocle of $\Sigma_2$ is irreducible, then it occurs in $\Sigma$.
\end{enumerate}
\qed
\end{lemma}

Indeed, in the first part we get a map $\Pi\twoheadrightarrow\overline{\Sigma}$ where
$\overline{\Sigma}$ is the quotient of $\Sigma$ by the image of the composition
$\Sigma_2\rightarrow\Sigma_1\rightarrow\Sigma$.
In the second part, the composition $\Sigma_2\rightarrow\Sigma_1\rightarrow\Sigma$ is non-zero.

\begin{remark} \label{rem: mdtp}
Let $\omega=\delta([-a,a]^{(\rho)})$ be a square-integrable representation and
let $\pi$ be an irreducible subquotient of $\Pi$. Then, the multiplicity of $\omega\otimes\pi$
in $\mu^*(\omega\rtimes\Pi)=M^*(\omega)\rtimes\mu^*(\Pi)$ is equal the multiplicity of $\omega\otimes\pi$ in
\[
\Big(2\sum_{i=-a}^a\delta([-i,a]^{(\rho)})\otimes\delta([i+1,a]^{(\rho)})+1\otimes\omega\Big)\rtimes\mu^*(\Pi)
\]
\end{remark}

{\chapter{Unitarizability in the Critical Case 
(Corank 1 and 2)} \label{CC-corank12}

From now until \S\ref{basic 0} we fix $\rho\in\Cuspsd$ and $\sigma\in\Cuspcl$.
Let $\alpha=\alpha_{\rho,\sigma}\in\tfrac12 \Z_{\ge0}$ be such that
$$
[\alpha]^{(\rho)}\rtimes\s \quad (=\nu^\alpha\rho\rtimes\s )
$$
is reducible. Recall that we set $\eta=0$ if $\alpha\in\Z$ and $\eta=1$ otherwise so that $\alpha-\frac12\eta\in\Z_{\ge0}$.

For simplicity, we suppress the superscript $\rho$ from the notation. Thus, we write $[x]$ instead of $[x]^{(\rho)}$.

For any $k\ge0$ and $\mathbf{x}=(x_1,\dots,x_k)\in\R^k$, consider the induced representation
$$
\Pi_{\mathbf{x}}:=[x_1]\times\dots \times [x_k]\rtimes\s.
$$
We say that $\mathbf{x}$ is \emph{critical} if $\{[x_1],\dots,[x_k]\}$ forms a segment of
cuspidal representations (possibly with multiplicities) that contains the reducibility point $[\alpha]$.
The goal of this and the following chapters is to study the irreducible subquotients of
$\Pi_{\mathbf{x}}$ and to determine which ones are unitarizable in the critical case for $k\le3$.

Since the composition series of $\Pi_{\mathbf{x}}$ depends only on the orbit of $\mathbf{x}$ under signed permutations,
we may assume without loss of generality that $\mathbf{x}\in \R^k_{++}$, where
$$
\R_{++}^k:=\{\mathbf{x}\in\R^k: 0\leq x_1\leq\dots\leq x_k\},
$$
and $\mathbf{x}$ is critical. Later on, we shall also use the notation
$$
\R_{+}^k:=\{\mathbf{x}\in\R^k: 0\leq x_1,\dots, x_k\}.
$$

\section{Extreme cases} \label{St-oposite}
Suppose that the $x_i$'s are distinct.
We have two extreme cases in this setting. The first is when $x_1=\alpha$.
Then, $\Pi_{x_1,\dots,x_n}$ is a regular (multiplicity one) representation of length $2^n$.
All but two of its irreducible subquotients are non-unitarizable,
the exceptions being the generalized Steinberg representation $\delta([\alpha,\alpha+n-1];\s)$ and its dual.

Now we shall consider the opposite extreme, when
$$
x_n=\alpha \qquad \text{and} \qquad x_1>0.
$$
Once again, one easily sees that $\Pi_{x_1,\dots,x_n}$ is a regular representation of length $2^n$.
We claim that all irreducible subquotients are unitarizable.

It is also easy to describe the Langlands parameters of all the irreducible subquotients.
The tempered parts of Langlands parameters are precisely
$$
\delta([\nu^{\alpha-k+1}\rho],\dots,[ \nu^{\alpha}\rho];\s),\quad k=0,\dots,n.
$$
Then, one gets $\GL$-parts of Langlands parameters from partitions of the remaining exponents
$$
\nu^{\alpha-n+1}\rho, \nu^{\alpha-n+1}\rho,\dots, \nu^{\alpha-k}\rho.
$$

To see that all the irreducible subquotient are unitarizable we argue by induction on $n$. The case $n=0$ is trivial.
Suppose that $n>0$ and take any irreducible subquotient $\pi$ of $\Pi_{(x_1,\dots,x_n)}$.
Clearly, there exists an irreducible subquotient $\pi_-$ of $\Pi_{x_2,\dots,x_{n}}$ such that $\pi\leq [\alpha-n+1]\rtimes\pi_-$.
Observe that $x_2\geq \frac32$, and that by Jantzen decomposition, all the representations
$$
[x]\rtimes\pi_-, \quad 0\leq x<\alpha-1
$$
are irreducible (they are hermitian, since $\rho$ is $F'/F$-selfdual).
Since $\pi_-$ is unitarizable by the inductive hypothesis, we have complementary series, and $\pi$ is at the end of these complementary series.
Therefore, it is unitarizable.

\section{Tempered representations in critical case, corank \texorpdfstring{$\leq 3$}{le3}}

We first describe the irreducible tempered representations that are subquotients of
$$
\Pi_{(x_1,\dots,x_k)},\quad 0\leq x_1\leq \dots\leq x_k, \quad k\leq 3 .
$$

(It is easy to reduce the general case to the critical case.)

First we make a simple observation.

It is a direct consequence of the classification of square-integrable representations in \cite{MR1896238} (cf. \cite{MR2908042},
in particular \S 34) that if $\alpha\in\Z$ and $x_1>0$ or if $\alpha\notin\Z$ and $x_2>\frac12$
then every irreducible tempered subquotients of $\Pi_{\mathbf{x}}$ will be strongly positive (and in particular, square-integrable).

First we shall list all the square-integrable representations.
A direct consequence of \cite{MR1896238} is that if a square-integrable $\pi$ is a subquotient of $\Pi_{\mathbf{x}}$,
then $\mathbf{x}$ is critical.

The classification of square integral representations in \cite{MR1896238} easily implies the following

\begin{proposition} \label{prop: sqrbsc}
The following are the square-integrable representations in the cases $k\le3$.
\begin{enumerate}
\item $k=0:$ $\sigma$ itself.

\item $k=1:$ The representations
$$
\delta([\alpha];\s),\quad \alpha>0.
$$

\item $k=2:$

\begin{enumerate}
\item $\delta_{\spsi}([\alpha-1],[\alpha];\s)$, \quad $\alpha>1$.

\item $\delta([\alpha,\alpha+1];\s)$, \quad $\alpha>0$.

\item $\delta([0,1]_\pm;\s)$, \quad $\alpha=0$.
\end{enumerate}

\item $k=3:$

\begin{enumerate}
\item $\delta_{\spsi}([\alpha-2],[\alpha-1],[\alpha];\s)$, \quad $\alpha>2$.

\item $\delta_{\spsi}([\alpha-1],[\alpha,\alpha+1];\s)$, \quad $\alpha>1$.\footnote{The socle of
$[\alpha-1]\times\delta([\alpha,\alpha+1])\rtimes \s$.}

\item $\delta([\alpha,\alpha+2];\s)$, \quad $\alpha>0$.

\item $\delta([-\frac12,\frac32]_\pm;\s)$, \quad $\alpha=\frac12$.

\item $\delta([0,2]_\pm;\s)$, \quad $\alpha=0$.
\end{enumerate}
\end{enumerate}
\qed
\end{proposition}

Now we shall describe the non-square-integrable tempered subquotients of representations $\Pi_{(x_1,\dots,x_k)}$, $k\leq 3$,
in the critical case.

\begin{proposition} \label{prop: tempbsc}
Suppose that we are in the critical case
(i.e. $\{x_1,\dots,x_k\}$ is a $\Z$-segment in $\Z$ or $\frac12 +\Z$, with possible multiplicities, which contains $\alpha$).
The following are the tempered, but non square-integrable, irreducible subquotients of $\Pi_{\mathbf{x}}$ for $k\le3$.

\begin{enumerate}

\item $k=1:$
$$
\tau_\pm,\quad \text{for} \quad\alpha=0,
$$
where
$$
\rho\rtimes\s=\tau_+\oplus\tau_-.\footnote{We shall denote $\tau_\pm$ also by $\delta([0]_\pm;\s).$}
$$

\item $k=2:$

\begin{enumerate}
\item $\tau([0]_\pm;\delta([1];\s))$,\quad $\alpha=1$, where
\[
[0]\rtimes \delta([1];\s)=\tau([0]_+;\delta([1];\s))\oplus\tau([0]_-;\delta([1];\s)).\footnote{+ sign corresponds to the
representation with $\delta([0,1])\otimes\s$ in its Jacquet module.}
\]

\item $\delta([-\frac12,\frac12]_\pm;\s)$, \quad $\alpha=\frac12$, where
$$
\delta([-\tfrac12,\tfrac12])\rtimes\s=\delta([-\tfrac12,\tfrac12]_+;\s)\oplus
\delta([-\tfrac12,\tfrac12]_-;\s).\footnote{$\delta([-\tfrac12,\tfrac12]_+;\s)$
has $[\tfrac12]\times[\tfrac12]\otimes\s$ in its Jacquet module}
$$

\item $[0] \rtimes\tau_\pm$, \qquad $\alpha=0$, where
$$
\rho\rtimes\s=\tau_+\oplus\tau_-.
$$
\end{enumerate}

\item $k=3:$

\begin{enumerate}
\item $[0]\times[0]\rtimes\tau_\pm$, \qquad $\alpha=0$ \quad $([0]\rtimes\s=\tau_+\oplus\tau_-)$.

\item $[0]\rtimes\tau([0]_\pm;\delta([1];\s))$,\quad $\alpha=1$.

\item $\delta([-1,1]_\pm;\s )$, \qquad $\alpha=1$, where
$$
\delta([-1,1])\rtimes\s=\delta([-1,1] _+;\s )\oplus\delta([-1,1] _-;\s ).\footnote{$ \delta([-1,1] _+;\s )$ is characterised
by the property that it has $\delta([0,1])\times[1]\otimes\s$ in its Jacquet module}
$$

\item $\delta([-1,1]_\pm;\s )$, \qquad $\alpha=0$, where
$$
\delta([-1,1])\rtimes\s =\delta([-1,1] _+;\s )\oplus\delta([-1,1] _-;\s ).\footnote{$ \delta([-1,1] _\pm;\s )$ is characterised
by the property that it has $[1]\times[1]\otimes\tau_\pm$ in its Jacquet module}
$$

\item $[0]\rtimes \tau([0]_\pm;\delta([1];\s))$,\quad $\alpha=1$.

\item $\delta([-\frac12,\frac12])\rtimes \delta([\frac12];\s)$, \qquad $\alpha=\frac12$.

\item $ \tau([0]_\pm;\delta_{\spsi}([1],[2];\s))$, \qquad $\alpha=2$, where
$$
[0]\rtimes\delta_{\spsi}([1],[2];\s)=
\tau([0]_+;\delta_{\spsi}([1],[2];\s))\oplus\tau([0]_-;\delta_{\spsi}([1],[2];\s)).
$$

\item $\tau([0]_\pm;\delta([1,2] ;\s ))$, \qquad $\alpha=1$, where
$$
[0]\rtimes\delta([1,2] ;\s )=\tau([0]_+;\delta([1,2] ;\s ))\oplus
\tau([0]_-;\delta([1,2] ;\s)).\footnote{$\tau([0]_+;\delta_{\spsi}([1],[2];\s))$ is characterised
by the property that it embeds into a representation of the form $[1]\rtimes \lambda$ (see \cite{MR3123571}).}
$$

\item
$[0]\rtimes \delta([0,1]_\pm;\s)$, \qquad $\alpha=0.$
\end{enumerate}
\end{enumerate}
\qed
\end{proposition}

\begin{remark}
For the description of composition series of representations $\Pi_{\mathbf{x}}$, $k\leq3$ in the critical case,
we will use the following non-critical irreducible tempered representation
\begin{enumerate}
\item $k=1 :$

$[0]\rtimes \sigma$, \quad $\alpha \in \{1,2\}$ (for $\Pi_{(0,1,2)}$ as well as for
$\Pi_{(0,1,1)}$ and $\Pi_{(0,0,1)}$ if $\alpha=1$).

\item $k=2 :$

\begin{enumerate}

\item
$[0]\times[0]\rtimes \sigma$, \ \ \ \ \ \quad $\alpha=1$ (for $\Pi_{(0,0,1)}$).

\item $\delta([-\frac12,\frac12])\rtimes \sigma$, \qquad $\alpha=\frac32$ (for $\Pi_{(\frac12,\frac12,\frac32)}$).

\item $[0]\rtimes \delta([2];\sigma), \ \ \ \ \ \quad \alpha=2$ (for $\Pi_{(0,1,2)}$).

\end{enumerate}

\end{enumerate}

\end{remark}

With the above description it is easy to describe $\JH(\Pi_{\mathbf{x}})$ for any given $\mathbf{x}\in\R^3$.
Henceforth, we will do so in the critical case without further explanation.

\section{Composition series in critical case, corank one}
For the rest of this chapter we deal with the case $k\le2$ (and $\mathbf{x}$ critical).

We start with the corank one case.

\subsection{\texorpdfstring{$\mathbf{x}=(\alpha)$}{xalpha} and \texorpdfstring{$\alpha\ge\tfrac12$}{alphage12}}

\begin{proposition}
Suppose that $\alpha\geq\tfrac12$. Then,
\begin{enumerate}
\item In the Grothendieck group we have
$$
\Pi_{(\alpha)}=[\alpha]\rtimes \s=L([\alpha]; \s)+\delta([\alpha]; \s).
$$
\item Both $L([\alpha]; \s)$ and $\delta([\alpha]; \s)$ are unitarizable.
\item $L([\alpha]; \s)^t=\delta([\alpha]; \s)$.
\item We have
$$
\mu^*(\delta([\alpha]; \s))=1\otimes \delta([\alpha]; \s)+[\alpha]\otimes \s,\ \
\mu^*(L([\alpha]; \s))=1\otimes L([\alpha]; \s))+[-\alpha]\otimes \s.
$$
\end{enumerate}
\qed
\end{proposition}

Moreover, by Proposition \ref{JBcomputation}
\footnote{In the sequel we shall conclude Jordan blocks from this proposition
unless otherwise indicated.}
\begin{equation}
\label{eq: Jord-a-sigma}
\Jord_\rho(\delta([\alpha];\s))=\{\eta+1,\eta+3,\dots,2\alpha-3,2\alpha+1 \}.
\end{equation}

\subsection{\texorpdfstring{$\mathbf{x}=(\alpha)$}{xalpha} and \texorpdfstring{$\alpha=0$}{alpha0}}

In this case $[0]\rtimes\s$ is unitarizable and decomposes as
a sum of two irreducible tempered representations
\[
[0]\rtimes\s=\delta ([0]_+;\s)\oplus \delta ([0]_-;\s).
\]
We have
$$
\mu^*(\delta ([0]_\pm;\s))=1\otimes \delta ([0]_\pm;\s)+[0]\otimes \s.
$$
Note that
$$
\delta ([0]_+;\s)^t=\delta ([0]_-;\s).
$$

\section{Composition series in critical case, corank two}

\subsection{\texorpdfstring{$\mathbf{x}=(\alpha,\alpha+1)$}{x01} and \texorpdfstring{$\alpha\ge\tfrac12$}{alphage12}}
\label{Steinberg-place-corank-two}

\begin{proposition} \label{prop-2-a+}
Suppose that $\alpha\geq\tfrac12$. Then,
\begin{enumerate}
\item In the Grothendieck group we have
$$
\Pi_{(\alpha,\alpha+1)}=[\alpha+1] \times[\alpha]\rtimes \s=\pi_1+\pi_2+\pi_3+\pi_4
$$
where
\begin{gather*}
\pi_1= L([\alpha+1] ,[\alpha];\s),\
\pi_2=L([\alpha+1] ;\delta([\alpha];\s))\\
\pi_3=L([\alpha,\alpha+1];\s),\ \pi_4=\delta([\alpha,\alpha+1];\s).
\end{gather*}
\item The representations $\pi_1$ and $\pi_4$ are unitarizable and $\pi_1^t=\pi_4$.
\item The representations $\pi_2$ and $\pi_3$ are not unitarizable and $\pi_2^t=\pi_3$.
\end{enumerate}
\qed
\end{proposition}

We have
\begin{subequations}
\begin{gather}
\mu^*(\pi_1) =1\otimes\pi_1+[-\alpha-1]\otimes L([\alpha];\s)+L([-\alpha-1,-\alpha])\otimes\s.\\
\begin{aligned}
\mu^*(\pi_2)&=1\otimes\pi_2+
[-\alpha-1] \otimes \delta([\alpha];\s) +[\alpha]\otimes[\alpha+1] \rtimes \s
\\&+[\alpha]\times[-\alpha-1] \otimes\s+L([\alpha],[\alpha+1] )\otimes\s,
\end{aligned}\\\label{eq: pi312}
\begin{aligned}
\mu^*(\pi_3)&=1\otimes\pi_3+
[-\alpha]\otimes[\alpha+1] \rtimes\s+[\alpha+1] \otimes L([\alpha];\s)
\\&+\delta([-\alpha-1,-\alpha])\otimes\s+[-\alpha]\times[\alpha+1] \otimes\s,
\end{aligned}\\
\mu^*(\pi_4) =1\otimes\pi_4+[\alpha+1]\otimes\delta([\alpha];\s)+\delta([\alpha,\alpha+1])\otimes\s.
\end{gather}
\end{subequations}

One easily gets that
\begin{equation} \label{eq: jord01}
\Jord_\rho(\pi_4)=\{\eta+1,\eta+3,\dots,2\alpha-3,2\alpha+3 \}.
\end{equation}

Furthermore,
\[
\begin{gathered}
{}[\alpha+1] \rtimes L([\alpha];\s)=\pi_1+\pi_3,\\
[\alpha+1] \rtimes \delta([\alpha];\s)=\pi_2+\pi_4,\\
L([\alpha],[\alpha+1])\rtimes\s=\pi_1+\pi_2,\\
\delta([\alpha,\alpha+1])\rtimes\s=\pi_3+\pi_4.
\end{gathered}
\]

\subsection{\texorpdfstring{$\mathbf{x}=(\alpha,\alpha)$}{x00} and \texorpdfstring{$\alpha\geq 1$}{alphage1}}

\begin{proposition} \label{prop: alphalpha}
Suppose that $\alpha\ge1$. Then,
\begin{enumerate}
\item In the Grothendieck group we have
\[
\Pi_{(\alpha,\alpha)}=\pi_1+\pi_1^t
\]
where
\[
\pi_1=[\alpha] \rtimes \delta([\alpha] ;\s)=L([\alpha] ;\delta([\alpha] ;\s)),\ \
\pi_1^t=[\alpha] \rtimes L([\alpha] ;\s)=L([\alpha] ,[\alpha] ;\s).
\]
\item Neither $\pi_1$ nor $\pi_1^t$ is unitarizable.
\end{enumerate}
\end{proposition}

Clearly,
\[
\Pi_{(\alpha,\alpha)}=[\alpha] \rtimes \delta([\alpha] ;\s)+[\alpha] \rtimes L([\alpha] ;\s).
\]
Since $[\alpha] \rtimes \delta([\alpha] ;\s)$ and $[\alpha] \rtimes L([\alpha] ;\s)$
are irreducible by (\ref{Pr-red-si-item: a-a+1-irr}) of Proposition \ref{Pr-red-si} and \eqref{eq: Jord-a-sigma}, we obtain the first part.

To show that $\pi_1$ is non-unitarizable consider the family of representations
\[
\gamma_s=[s] \rtimes \delta([\alpha] ;\s),\ \ s\in\R.
\]
It is irreducible for $\alpha\le s<\alpha+1$ and $\gamma_{\alpha+1}$ admits a non-unitarizable irreducible quotient by
Proposition \ref{prop-2-a+}. Therefore, $\pi_1=\gamma_\alpha$ is not unitarizable. A similar argument applies to $\pi_1^t$.

\subsection{\texorpdfstring{$\mathbf{x}=(\alpha-1,\alpha)$}{x-10} and \texorpdfstring{$\alpha\geq \tfrac32$}{alphage32}}
\label{subsec-w-prop-2-a-}

\begin{proposition} \label{prop-2-a-}
Assume $\alpha\geq \tfrac32$.
Then,
\begin{enumerate}
\item In the Grothendieck group
\[
\Pi_{(\alpha-1,\alpha)}=\pi_1+\pi_2+\pi_3+\pi_4
\]
where
\begin{gather*}
\pi_1=L([\alpha-1,\alpha];\s),\ \ \pi_2=L([\alpha-1];\delta([\alpha];\s)),\\
\pi_3= L([\alpha-1],[\alpha];\s),\ \ \pi_4=\delta_{\spsi}([\alpha-1],[\alpha];\s).
\end{gather*}
\item $\pi_1,\pi_2,\pi_3,\pi_4$ are unitarizable.
\item $\pi_4^t=\pi_1,\ \ \pi_3^t=\pi_2$.
\end{enumerate}
\end{proposition}

The representation $\Pi_{(\alpha-1,\alpha)}$ is regular and admits the following two decompositions in the Grothendieck group
\[
\Pi_{(\alpha-1,\alpha)}=[\alpha]\times[\alpha-1]\rtimes \s=\Pi_1+\Pi_2=\Pi_3+\Pi_4
\]
where
\begin{gather*}
\Pi_1=[\alpha-1]\rtimes L([\alpha];\s),\ \ \Pi_2=[\alpha-1]\rtimes \delta([\alpha];\s),\\
\Pi_3=\delta([\alpha-1,\alpha])\rtimes\s,\ \ \Pi_4=L([\alpha-1],[\alpha])\rtimes\s.
\end{gather*}
Moreover, we have the decompositions
\begin{equation} \label{eq: dec345}
\Pi_3=\pi_1+\pi_2,\ \ \Pi_4=\pi_3+\pi_4.
\end{equation}
This further implies
\[
\Pi_1=\pi_1+\pi_3,\ \ \Pi_2=\pi_2+\pi_4
\]

Furthermore,
\begin{subequations}
\begin{gather*}
\mu^*(\pi_1)=1\otimes\pi_1+[-\alpha+1]\otimes L([\alpha];\s)
+\delta([-\alpha,-\alpha+1])\otimes\s,\\
\begin{aligned}
\mu^*(\pi_2)&=1\otimes\pi_2
+[\alpha]\otimes[\alpha-1]\rtimes\s+[-\alpha+1]\otimes \delta([\alpha];\s)
\\&+[-\alpha+1]\times[\alpha]\otimes\s +\delta([\alpha-1,\alpha])\otimes\s,
\end{aligned}\\
\begin{aligned}
\mu^*(\pi_3)&=1\otimes\pi_3
+[-\alpha]\otimes [\alpha-1]\rtimes\s +[\alpha-1]\otimes L([\alpha];\s)\\
&+[\alpha-1]\times[-\alpha]\otimes\s+L([-\alpha], [-\alpha+1])\otimes\s,
\end{aligned}\\
\mu^*(\pi_4)=1\otimes\pi_4+[\alpha-1]\otimes \delta([\alpha];\s)+
L([\alpha-1],[\alpha])\otimes\s.
\end{gather*}
\end{subequations}

We also note that
$$
\Jord_\rho(\pi_4)=\{\eta+1,\eta+3,\dots,2\alpha-5,2\alpha-1,2\alpha+1 \},
$$
and the partially defined function $\e$ pertaining to $\pi_4$ takes opposite values at $2\alpha-1$ and $2\alpha+1$
(use the fact that $\pi_4$ is strongly positive; see \cite{MR1896238}).

\subsection{\texorpdfstring{$\mathbf{x}=(0,1)$}{x01} and \texorpdfstring{$\alpha=1$}{alpha1}}

\begin{proposition}
Suppose that $\alpha=1$. Then,
\begin{enumerate}
\item In the Grothendieck group we have
\[
\Pi_{(0,1)}=\pi_1+\pi_2+\pi_3^++\pi_3^-
\]
where
\[
\pi_1=L([1];[0]\rtimes\s),\ \ \pi_2=L([0,1];\s),\ \ \pi_3^\pm=\tau([0]_\pm;\delta([1];\s)).
\]
\item $\pi_1,\pi_2,\pi_3^\pm$ are unitarizable.
\item $\pi_1^t=\pi_3^+$ and $\pi_2^t=\pi_3^-$.
\end{enumerate}
\end{proposition}

Note that
\[
\Pi_{(0,1)}=\Pi_1+\Pi_2
\]
where
\[
\Pi_1=[0]\rtimes L([1];\s),\ \ \Pi_2=[0]\rtimes \delta([1];\s)
\]
and
\[
\Pi_1=\pi_1+\pi_2,\ \ \Pi_2=\pi_3^++\pi_3^-.
\]
Clearly, $\Pi_1$ and $\Pi_2$ are unitarizable.

Moreover,
\begin{equation} \label{eq: 0111}
\delta([0,1])\rtimes\s=\pi_2+\pi_3^+,\ \ L([0],[1])\rtimes\s=\pi_1+\pi_3^-.
\end{equation}

Finally,
\begin{subequations}
\begin{gather}
\begin{aligned}
\mu^*(\pi_1)&=1\otimes \pi_1
+[-1]\otimes[0]\rtimes\s+[0]\otimes L([1];\s)\\
&+2L([-1],[0])\otimes\s +\delta([-1,0])\otimes\s,
\end{aligned}\\
\mu^*(\pi_2)=1\otimes \pi_2+[0]\otimes L([1];\s)+\delta([-1,0])\otimes\s\\
\begin{aligned} \label{01-1-delta+}
\mu^*(\pi_3^+)&=1\otimes\pi_3^++
[1]\otimes[0]\rtimes\s+[0]\otimes\delta([1];\s)
\\&+2\delta([0,1])\otimes\s +L([0],[1])\otimes\s,
\end{aligned}\\\label{eq: 013-}
\mu^*(\pi_3^-)=1\otimes\pi_3^-+[0]\otimes\delta([1];\s)+L([0],[1])\otimes\s.
\end{gather}
\end{subequations}

\subsection{\texorpdfstring{$\mathbf{x}=(\tfrac12,\tfrac12)$}{x1212} and \texorpdfstring{$\alpha=\tfrac12$}{alpha12}}

\begin{proposition} \label{prop: 1212}
Suppose $\alpha=\frac12$. Then,
\begin{enumerate}
\item In the Grothendieck group we have
\[
\Pi_{(\frac12,\frac12)}=\pi_1+\pi_2+\pi_3+\pi_4
\]
where
\begin{gather*}
\pi_1=\delta([-\tfrac12,\tfrac12]_+;\s),\ \ \pi_2=\delta([-\tfrac12,\tfrac12]_-;\s),\\
\pi_3=L([\tfrac12],[\tfrac12];\s),\ \ \pi_4=L([\tfrac12];\delta([\tfrac12];\s)).
\end{gather*}
\item $\pi_1,\pi_2,\pi_3,\pi_4$ are unitarizable.
\item $\pi_1^t=\pi_3$, $\pi_2^t=\pi_4$.
\end{enumerate}
\end{proposition}

We have
\[
\delta([-\tfrac12,\tfrac12])\rtimes\s=\pi_1+\pi_2,\ \
L([-\tfrac12],[\tfrac12])\rtimes\s=\pi_3+\pi_4
\]
and also
\[
[\tfrac12]\rtimes \delta([\tfrac12];\s)=\pi_1+\pi_4,\ \
[\tfrac12]\rtimes L([\tfrac12];\s)=\pi_2+\pi_3.
\]
This implies that $\pi_1^t=\pi_3$, $\pi_2^t=\pi_4$.

Finally,
\begin{subequations}
\begin{gather*}
\begin{aligned}
\mu^*(\pi_1)&=
1\otimes\pi_1+
[\tfrac12]\otimes [\tfrac12]\rtimes \s+
[\tfrac12]\otimes \delta([\tfrac12];\s)\\
&+\delta([-\tfrac12,\tfrac12])\otimes\s+
[\tfrac12]\times[\tfrac12]\otimes\s,
\end{aligned}\\
\mu^*(\pi_2)=
1\otimes\pi_2+
[\tfrac12]\otimes L([\tfrac12];\s)+
\delta([-\tfrac12,\tfrac12]])\otimes\s,\\
\begin{aligned}
\mu^*(\pi_3)&=1\otimes\pi_3+
[-\tfrac12]\otimes [\tfrac12]\rtimes \s
+[-\tfrac12]\otimes L([\tfrac12];\s)
\\&+L([-\tfrac12],[\tfrac12]])\otimes\s+
[-\tfrac12]\times[-\tfrac12]\otimes\s,
\end{aligned}\\
\mu^*(\pi_4)=1\otimes\pi_4+[-\tfrac12]\otimes \delta([\tfrac12];\s)+
L([-\tfrac12],[\tfrac12]])\otimes\s.
\end{gather*}
\end{subequations}

\subsection{\texorpdfstring{$\mathbf{x}=(0,1)$}{x01} and \texorpdfstring{$\alpha=0$}{alpha0}}

\begin{proposition} \label{prop: 001}
Suppose that $\alpha=0$. Then,
\begin{enumerate}
\item
 \label{prop: 001-item: Gg}
 In the Grothendieck group we have
\[
\Pi_{(0,1)}=\pi_1^++\pi_1^-+2\pi_2+\pi_3^++\pi_3^-,
\]
where
$$
\pi_1^\pm=L([1];\delta ([0]_\pm;\s)),\quad
\pi_2= L([0,1];\s),\quad
\pi_3^\pm=\delta([0,1]_\pm;\s).
$$
\item All the irreducible subquotients of $\Pi_{(0,1)}$ are unitarizable.
\item 
 \label{prop: 001-item: inv}
We have
\[
(\pi_1^\pm)^t=\pi_3^\mp,\quad \pi_2^t=\pi_2.
\]
\item
\label{prop: 001-item: 01ds}
$
\Jord_\rho(\pi_3^\pm)=\{1,3\}.
$
\end{enumerate}
\end{proposition}

\begin{proof}
By Propositions \ref{prop: sqrbsc} and \ref{prop: tempbsc} it
follow that $\{\pi_1^+,\pi_1^-,\pi_2,\pi_3^+,\pi_3^-\}$ is the Jordan-H\"older series of $\Pi_{(0,1)}$. Note that
\begin{equation} \label{0-2-decomp}
\Pi_{(0,1)}=\Pi_1+\Pi_2=\Pi_3^++\Pi_3^-,
\end{equation}
where
$$
\Pi_1=\delta([0,1])\rtimes\sigma, \ \ \Pi_2=L([0],[1])\rtimes\sigma, \ \ \Pi_3^\pm=[1]\rtimes \delta([0]_\pm;\sigma).
$$
Observe that \eqref{BPLC} implies that $\pi_3^\pm$ and $\pi_2$ are the only possible subquotients of $\Pi_1$. Since $\Pi_1^t=\Pi_2$
and $\Pi_{(0,1)}$ has length $\geq 5$,
we conclude that all three representations are subquotients, which further implies that the Jordan-H\"older series of $\Pi_2$
consists of 3 irreducible representations.
Since each $\pi_3^\pm$ has in its Jacquet module $\delta([0,1])\otimes\sigma$, and the multiplicity of it in the Jacquet module of
$\Pi_{(0,1)}$ is two, we conclude that the multiplicity of each $\pi_3^\pm$ in $\Pi_{(0,1)}$ is one. Therefore,
$
\Pi_1=\pi_3^++\pi_3^-+\pi_2.
$
Both $\pi_1^\pm$ have multiplicity one in $\Pi_{(0,1)}$, which implies
$
\Pi_2=\pi_1^++\pi_1^-+\pi_2.
$
This implies (\ref{prop: 001-item: Gg}) (see \cite[\S5]{MR1212952} for the case of $Sp(4)$).

Multiplicity one of $\pi_1^\pm$ and $\pi_3^\pm$ in $\Pi_{(0,1)}$, together with \eqref{0-2-decomp}, imply $\pi_2^t=\pi_2$.
Each of $\pi_3^\pm$ is a subquotient of $\Pi_3^\pm$ (by definition of $\pi_3^\pm$). Also $\pi_1^\pm$ is a subquotient. This implies
$
\Pi_3^\pm=\pi_1^\pm+\pi_2+\pi_3^\pm.
$
Observe that $(\pi_1^\pm)^t\ne\pi_1^\pm$ and $(\pi_1^\pm)^t\ne\pi_1^\mp$ (since $\Pi_1^t=\Pi_2$), as well as $(\pi_1^\pm)^t\ne\pi_2$
(since $\pi_2^t=\pi_2$). Now $(\Pi_3^\pm)^t=\Pi_3^\mp$ implies $(\pi_1^\pm)^t=\pi_3^\mp$. This completes the proof of (\ref{prop: 001-item: inv}).

All the irreducible subquotients of $\Pi_{(0,1)}$ are unitarizable, since they are subquotients of ends of the complementary series
starting with $\delta([-\frac12,\frac12])\rtimes \sigma$ or $L([-\frac12],[\frac12])\rtimes \sigma$.

Proposition \ref{JBcomputation} implies (\ref{prop: 001-item: 01ds}). The proof is now complete.
\end{proof}

Observe that we have proved above the following equalities (which we shall use later)
\begin{equation}
\label{01-0-half}
\Pi_3^\pm=\pi_1^\pm+\pi_2+\pi_3^\pm,\quad
\Pi_1=\pi_3^++\pi_3^-+\pi_2,\quad
\Pi_2=\pi_1^++\pi_1^-+\pi_2.
\end{equation}
Further, we directly get
\begin{subequations}
\begin{gather}
\label{l1} \mu^*(\pi_3^\pm)= 1\otimes \pi_3^\pm
+[1]\otimes\delta([0]_\pm;\s)
+\delta([0,1])\otimes\s.\\
\label{l2} \mu^*(\pi_1^\pm)=1\otimes \pi_1^\pm+
[-1]\otimes \delta([0]_\pm;\s)+L([-1],[0])\otimes\s.\\
\label{l3} \mu^*(\pi_2)=1\otimes \pi_2+[0]\otimes[1]\rtimes\s+\delta([-1,0])\otimes\s+L([0],[1])\otimes\s.
\end{gather}
\end{subequations}

\subsection{\texorpdfstring{$\mathbf{x}=(0,0)$}{x00} and \texorpdfstring{$\alpha=0$}{alpha=0}}

Here $\Pi_{(0,0)}$ is a unitarizable tempered representation of length two which decomposes as
$$
\Pi_{(0,0)}[0]=[0]\rtimes\delta ([0]_+;\s)\oplus [0]\rtimes\delta ([0]_-;\s).
$$
We have
$$
([0]\rtimes\delta ([0]_+;\s))^t=[0]\rtimes\delta ([0]_-;\s).
$$

\chapter{Unitarizability in the Critical Case 
(Corank 3, \texorpdfstring{$\alpha>1$}{alpha>1})} \label{sec: unit3}

In this chapter we determine the unitarizability of the irreducible subquotients
of $\Pi_{\mathbf{x}}$ in the critical case for $k=3$ in all cases where $\alpha>1$
and in many cases where $\alpha=\frac12$ or $1$. Below we analyse various cases of $\mathbf{x}$.

\section{\texorpdfstring{$\mathbf{x}=(\alpha,\alpha+1,\alpha+2)$}{x012} and \texorpdfstring{$\alpha\geq\frac12$}{alphage12}} \label{sec: 012}

Recall (see \S\ref{sec: Castrick}) that the representation
$$
\Pi_{(\alpha,\alpha+1.\alpha+2)}
$$
is multiplicity free of length 8 and precisely two of its
irreducible subquotients, are unitarizable, namely,
the generalized Steinberg representation and its dual, i.e.
$$
\delta([\alpha,\alpha+2];\s),\qquad L([\alpha+2],[\alpha+1],[\alpha];\s).
$$
Note that
\[
\Pi_{(\alpha,\alpha+1,\alpha+2)}=\Pi_1+\Pi_2+\Pi_3+\Pi_4
\]
where
\begin{gather*}
\Pi_1=[\alpha+2]\rtimes L([\alpha+1],[\alpha];\s),\ \
\Pi_2=[\alpha+2]\rtimes L([\alpha+1];\delta([\alpha];\s))\\
\Pi_3=[\alpha+2]\rtimes L([\alpha,\alpha+1];\s),\ \
\Pi_4=[\alpha+2]\rtimes \delta([\alpha,\alpha+1];\s).
\end{gather*}
Each of $\Pi_1,\Pi_2,\Pi_3,\Pi_4$ is reducible.
The representation $L([\alpha+2],[\alpha+1],[\alpha];\s)$ is a quotient of $\Pi_1$
and $\delta([\alpha,\alpha+2];\s)$ is a subrepresentation of $\Pi_4$.

We list below all eight irreducible subquotients of $\Pi_{(\alpha,\alpha+1,\alpha+2)}$ and describe how \ASS duality acts on them
\begin{gather*}
\delta([\alpha,\alpha+2];\s)^t=L([\alpha+2],[\alpha+1],[\alpha];\s),
\\
L([\alpha+2];\delta([\alpha,\alpha+1];\s))^t=L([\alpha+1,\alpha+2],[\alpha];\s),
\\
L([\alpha+1,\alpha+2];\delta([\alpha];\s))^t=L([\alpha+2],[\alpha,\alpha+1];\s),
\\
L([\alpha+1],[\alpha+2];\delta([\alpha];\s))^t=L([\alpha,\alpha+2];\s).
\end{gather*}
Obviously both $\delta([\alpha+1,\alpha+2])\rtimes\delta([\alpha];\s)$ and $\delta([\alpha+1,\alpha+2])\rtimes L([\alpha];\s)$ reduce.

\section{\texorpdfstring{$\mathbf{x}=(\alpha,\alpha+1,\alpha+1)$}{x001} and \texorpdfstring{$\alpha\geq 1/2$}{alphage12}} \label{sec: 011}

\begin{proposition} \label{prop: 011}
Suppose $\alpha\ge\frac12$. Then,
\begin{enumerate}
\item In the Grothendieck group we have
\[
\Pi_{(\alpha,\alpha+1,\alpha+1)}=\pi_1+\pi_2+\pi_3+\pi_4
\]
where
\begin{gather*}
\pi_1=[\alpha+1]\rtimes L([\alpha+1],[\alpha];\s),\ \
\pi_2=[\alpha+1]\rtimes L([\alpha+1];\delta([\alpha];\s)),\\
\pi_3=[\alpha+1]\rtimes L([\alpha,\alpha+1];\s),\ \
\pi_4=[\alpha+1]\rtimes \delta([\alpha,\alpha+1];\s)
\end{gather*}
are irreducible.
\item \label{part: nonunit011} None of $\pi_1,\pi_2,\pi_3,\pi_4$ is unitarizable.
\item $\pi_1^t=\pi_4$, $\pi_2^t=\pi_3$.
\end{enumerate}
\end{proposition}

\begin{proof}
Recall that
$$
\Jord_\rho(\delta([\alpha,\alpha+1];\s))=
\{\eta+1,\eta+3,\dots,2\alpha-3,2\alpha+3 \}.
$$
By (\ref{Pr-red-si-item: a-a+1-irr}) of Proposition \ref{Pr-red-si}, $\pi_4$ and hence also $\pi_1=\pi_4^t$, is irreducible.

Furthermore, $\pi_3$ is irreducible by Corollary \ref{irr-simple}.
Hence, $\pi_2=\pi_3^t$ is also irreducible.

To show that $\pi_1$ is non-unitarizable consider the family
\[
\gamma_s=[s]\rtimes L([\alpha+1],[\alpha];\s),\ \ s\in\R.
\]
It is irreducible for $\alpha+1\le s<\alpha+2$ and $\gamma_{\alpha+2}$ contains a non-unitarizable quotient (see \S\ref{sec: 012}).
Hence, $\pi_1=\gamma_{\alpha+1}$ cannot be unitarizable. A similar argument applies to $\pi_2,\pi_3,\pi_4$.
\end{proof}

\section{\texorpdfstring{$\mathbf{x}=(\alpha,\alpha,\alpha+1)$}{x001} and \texorpdfstring{$\alpha\geq 1$}{alphage1}}

\begin{proposition} \label{a,a,a+1-a}
Let $\alpha\geq 1$. Then,
\begin{enumerate}
\item We have
\[
\Pi_{(\alpha,\alpha,\alpha+1)}=\pi_1+\pi_2+\pi_3+\pi_4+2\pi_5
\]
where
\begin{gather*}
\pi_1=L( [\alpha],[\alpha],[\alpha+1];\s),\ \ \pi_2=L([\alpha];\delta([\alpha,\alpha+1];\s)),\\
\pi_3= L([\alpha],[\alpha+1];\delta([\alpha];\s)),\ \
\pi_4=L( [\alpha],[\alpha,\alpha+1];\s),\\
\pi_5=L([\alpha,\alpha+1];\delta([\alpha];\s)).
\end{gather*}
\item $\pi_2=\pi_1^t$, $\pi_4=\pi_3^t$ and $\pi_5^t=\pi_5$.
\item We have
\begin{gather*}
\pi_1= [\alpha]\rtimes L([\alpha+1],[\alpha];\s),\ \
\pi_2= [\alpha]\rtimes \delta([\alpha,\alpha+1];\s),\\
\pi_3= L([\alpha],[\alpha+1])\rtimes \delta([\alpha];\s),\ \
\pi_4= \delta ([\alpha,\alpha+1])\rtimes L( [\alpha];\s).
\end{gather*}
\item None of $\pi_1,\pi_2,\pi_3,\pi_4,\pi_5$ is unitarizable.

\item $ \delta([\alpha,\alpha+1])\rtimes\delta([\alpha];\sigma) =\pi_2+\pi_5$,
\ $L([\alpha],[\alpha+1])\rtimes L([\alpha];\sigma)=\pi_1+\pi_5$.

\end{enumerate}

\end{proposition}

\begin{proof}
Write
\[
\Pi_{(\alpha,\alpha,\alpha+1)}=\Pi_1+\Pi_2+\Pi_3+\Pi_4
\]
where
\begin{gather*}
\Pi_1=\delta([\alpha,\alpha+1])\rtimes L([\alpha];\sigma),\ \
\Pi_2=L([\alpha],[\alpha+1])\rtimes L([\alpha];\sigma)\\
\Pi_3=\delta([\alpha,\alpha+1])\rtimes\delta([\alpha];\sigma),\ \
\Pi_4=L([\alpha],[\alpha+1])\rtimes\delta([\alpha];\sigma).
\end{gather*}

First observe that (\ref{Pr-red-si-item: a-a+1-irr}) of Proposition \ref{Pr-red-si} and \eqref{eq: jord01} imply that $[\alpha]\rtimes \delta([\alpha,\alpha+1];\s)$ is irreducible.
Therefore, also its dual is irreducible.
This implies that
$$
\pi_1= [\alpha]\rtimes L([\alpha+1],[\alpha];\s)\text{ and }
\pi_2= [\alpha]\rtimes \delta([\alpha,\alpha+1];\s)
$$
(using Proposition \ref{prop: addparms}). It also follows that $\pi_1^t=\pi_2$.

Furthermore, both $\pi_1$ and $\pi_2$ have multiplicity one in $\Pi_{(\alpha,\alpha,\alpha+1)}$.

Consider now $\Pi_1$.
By \eqref{BPLC}, neither $\pi_1$ nor $\pi_3$ can be a subquotient of
$\delta([\alpha,\alpha+1])\times [\alpha]\rtimes\s$, let alone of $\Pi_1$.
We show that $\pi_5\not\le\Pi_1$. Otherwise, we would get
$$
\delta([-\alpha-1 ,-\alpha])\times [\alpha]\leq\left(\delta([-\alpha-1 ,-\alpha])+[-\alpha] \times[\alpha+1]
+\delta([\alpha,\alpha+1] )\right)\times [-\alpha],
$$
which obviously cannot be sustained.

It remains to see that $\pi_2\not\le\Pi_1$.
Assume otherwise. Then,
\begin{multline*}
([\alpha]+[-\alpha])\times \delta([\alpha,\alpha+1])
\\
\leq
\left(\delta([-\alpha-1 ,-\alpha])+[-\alpha] \times[\alpha+1]
+\delta([\alpha,\alpha+1] )\right)\times [-\alpha],
\end{multline*}
which obviously does not hold.
Thus, $\Pi_1$ is irreducible and therefore equals $\pi_4$.
By applying $^t$ we get
$$
\pi_3=\Pi_4.
$$

This implies that $\pi_3^t=\pi_4$, which further implies that $\pi_5^t=\pi_5$.

The non-unitarizability of $\pi_1,\pi_2$ is proved as in \S\ref{sec: 011} using
Proposition \ref{prop: 011} part \ref{part: nonunit011} by deforming $[\alpha]$ to $[\alpha+1]$.
Similarly, the non-unitarizability of $\pi_3,\pi_4$ is obtained by deforming
$L([\alpha],[\alpha+1])$ and $\delta ([\alpha,\alpha+1])$
to $L([\alpha+1],[\alpha+2])$ and $\delta ([\alpha+1,\alpha+2])$ respectively
and using the fact that each of the representations
$L([\alpha+1],[\alpha+2])\rtimes \delta([\alpha];\s)$ and
$\delta ([\alpha+1,\alpha+2])\rtimes L( [\alpha];\s)$
admit a non-unitarizable subquotient (using Proposition \ref{prop: addparms} and the analysis of \S\ref{sec: 012}).

It remains to show that $\pi_5$ is not unitarizable.

Clearly, $\pi_1$ is a quotient of $L([\alpha],[\alpha+1])\rtimes L([\alpha];\s)$
and occurs there with multiplicity one. Hence, $\pi_2$ occurs with multiplicity one in $\Pi_3$.
Using \eqref{BPLC} we infer that
\begin{equation} \label{eq1}
\Pi_3=\pi_5+\pi_2.
\end{equation}
Hence, by duality, $\Pi_2=\pi_5+\pi_1$. All in all,
\[
\Pi_{(\alpha,\alpha,\alpha+1)}=\pi_1+\pi_2+\pi_3+\pi_4+2\pi_5.
\]

Consider
$$
\Gamma=\delta_1\rtimes\pi_5\text{ where }\delta_1=\delta([-(\alpha-1) ,\alpha-1]).
$$
Then, $\Gamma$ admits the following two irreducible subquotients (by Proposition \ref{prop: addparms})
$$
\Gamma_{\pm}:=L([\alpha,\alpha+1];\tau([-(\alpha-1),\alpha-1]_\pm;\delta([\alpha];\s))).
$$
Let
\[
\Gamma_0:=L([-(\alpha-1) ,\alpha+1];\delta([\alpha];\s)).
\]
We claim that $\Gamma_0\le\Gamma$, so that the length of $\Gamma$ is at least three.

Indeed,
$$
\Gamma_0\leq\delta([-(\alpha-1) ,\alpha+1])\rtimes\delta([\alpha];\s)\leq
\delta_1\times \Pi_3.
$$
while
\[
\delta_1\times \Pi_3=\Gamma+\delta_1\rtimes\pi_2
\]
by \eqref{eq1}. On the other hand,
$$
\Gamma_0\not\leq\delta_1\rtimes\pi_2
$$
since otherwise
$$
\Gamma_0\leq [\alpha]\times\delta_1\rtimes \delta([\alpha,\alpha+1];\s)
$$
and this is impossible by \eqref{BPLC}. Therefore, $\Gamma_0\leq\Gamma$ as claimed.

Suppose on the contrary that $\pi_5$ is unitarizable. Then, by Lemma \ref{lem: nonunit}, the multiplicity of
$\delta_1\otimes \pi_5$ in $\mu^*(\Gamma)$ would be at least three. Recall
$$
\mu^*(\Gamma)=M^*(\delta_1)\rtimes \mu^*(\pi_5).
$$
Now $\delta_1\otimes1$ has multiplicity two in $M^*(\delta_1)$.
On the other hand, since the support of $\pi_5$ is disjoint from that of $\delta_1$,
no other term from $M^*(\delta_1)$ can contribute.
Therefore, the multiplicity of $\delta_1\otimes\pi_5$ in $\mu^*(\Gamma)$ is two.
This contradiction completes the proof of the proposition.
\end{proof}

\section{\texorpdfstring{$\mathbf{x}=(\alpha,\alpha,\alpha)$}{x000} and \texorpdfstring{$\alpha\geq 1$}{alphage1}} \label{sec; 000}

Consider
\[
\Pi_{(\alpha,\alpha,\alpha)}=\pi_1+\pi_2
\]
where
\[
\pi_1=[\alpha]\times [\alpha]\rtimes L([\alpha];\s),\ \
\pi_2=[\alpha]\times [\alpha]\rtimes\delta([\alpha];\s).
\]
The representation $\pi_2$ is irreducible by (\ref{Pr-red-si-item: a-a+1-irr}) of Proposition \ref{Pr-red-si}, \eqref{eq: Jord-a-sigma}
and the factorization of the long intertwining operator in the Langlands classification.
(Here we use that $\alpha>\tfrac12$.)
Applying the \ASS involution, we get that $\pi_1=\pi_2^t$ is irreducible as well.

The non-unitarizability of $\pi_1$ (and similarly $\pi_2$) follows by deforming $[\alpha]$ to $[\alpha+1]$
and the fact that $[\alpha+1]\times [\alpha]\rtimes L([\alpha];\s)$ admits non-unitarizable
quotients (See Proposition \ref{a,a,a+1-a}.)
Alternatively, we can use the fact that
$\pi_1\cong[\alpha]\times [-\alpha]\rtimes L([\alpha];\s)$
and use unitary parabolic reduction (since $[\alpha]\times [-\alpha]$
is hermitian but not unitarizable).

\section{\texorpdfstring{$\mathbf{x}=(\alpha-1,\alpha,\alpha+1)$}{x-101} and \texorpdfstring{$\alpha>1$}{alpha>1}}

Recall that the representation $[\alpha-1]\times\delta([\alpha,\alpha+1])\rtimes\s$ has an irreducible socle,
which is a strongly positive square-integrable representation (see \cite{MR2090870} for details of
the description of the classification of strongly positive representations pertaining to this paper).
This representation was denoted by
$
\delta_{\spsi}([\alpha-1],[\alpha,\alpha+1];\s).
$

\begin{proposition} \label{a-1,a,a+1}
Suppose $\alpha>1$. Then,
\begin{enumerate}
\item In the Grothendieck group we have
\[
\Pi_{(\alpha-1,\alpha,\alpha+1)}=\pi_1+\dots+\pi_8
\]
where
\begin{gather*}
\pi_1=\delta_{\spsi}([\alpha-1],[\alpha,\alpha+1];\s),\ \
\pi_2=L([\alpha+1],[\alpha-1,\alpha];\s),\\
\pi_3=L([\alpha-1];\delta([\alpha,\alpha+1];\s)),\ \
\pi_4=L([\alpha+1],[\alpha],[\alpha-1];\s),\\
\pi_5=L([\alpha+1]; \delta_{\spsi}([\alpha-1],[\alpha];\s)),\ \
\pi_6=L([\alpha-1,\alpha+1];\s),\\
\pi_7=L([\alpha+1], [\alpha-1];\delta([\alpha];\s)),\ \
\pi_8=L([\alpha,\alpha+1],[\alpha-1];\s).
\end{gather*}
\item $\pi_2=\pi_1^t$, $\pi_4=\pi_3^t$, $\pi_6=\pi_5^t$, $\pi_8=\pi_7^t$.
\item The representations $\pi_1,\pi_2,\pi_3,\pi_4$ are unitarizable.
\item The representations $\pi_5,\pi_6,\pi_7,\pi_8$ are not unitarizable.
\end{enumerate}
\end{proposition}

\begin{proof}[Proof of first 3 parts]
Note that $\Pi_{(\alpha-1,\alpha,\alpha+1)}$ is regular (i.e., all Jacquet modules of $\Pi_{(\alpha-1,\alpha,\alpha+1)}$,
including $\Pi_{(\alpha-1,\alpha,\alpha+1)}$ itself, are multiplicity free).
The first part therefore follows from the description of $\JH(\Pi_{(\alpha-1,\alpha,\alpha+1)})$.

By Proposition \ref{prop-2-a+} we have (in the Grothendieck group)
\begin{equation} \label{dec1}
\Pi_{(\alpha-1,\alpha,\alpha+1)}=\Pi_1+\Pi_2+\Pi_3+\Pi_4
\end{equation}
where
\begin{gather*}
\Pi_1=[\alpha-1]\rtimes L([\alpha+1],[\alpha];\s),\ \
\Pi_2=[\alpha-1]\rtimes L([\alpha+1];\delta([\alpha];\s))\\
\Pi_3=\Pi_2^t=[\alpha-1]\rtimes L([\alpha,\alpha+1];\s),\ \
\Pi_4=\Pi_1^t=[\alpha-1]\rtimes \delta([\alpha,\alpha+1];\s).
\end{gather*}
Note that $\Pi_1$ and $\Pi_4$ are the extremities of complementary series
and hence, all their irreducible subquotients are unitarizable.
Furthermore, $\Pi_4$ is reducible (by (\ref{Pr-red-si-item: a-not-a+2-red}) of  Proposition \ref{Pr-red-si} and \eqref{eq: jord01}). Hence, $\Pi_1$ is reducible.

Using \eqref{BPLC} and the first part we get
\begin{equation} \label{ecs-a-1}
\Pi_4=\pi_3+\pi_1.
\end{equation}
Applying the \ASS involution, we get
\begin{equation}\label{id1}
\Pi_1=\pi_3^t+\pi_1^t.
\end{equation}
Hence, $\pi_3^t$ and $\pi_1^t$ are unitarizable.

We easily see that each of $\pi_i$, $i\ne2$ has at least one factor with positive exponent.
This implies
$$
\pi_1^t=\pi_2.
$$
We know by Proposition \ref{prop: addparms} that
$$
\pi_4\leq \Pi_1.
$$
The above two facts imply
\begin{equation}\label{id2}
\Pi_1=\pi_4+\pi_2
\end{equation}
which further implies
$$
\pi_3^t=\pi_4.
$$

Since $\Pi_2^t=\Pi_3$ we deduce from \eqref{ecs-a-1}, \eqref{id2} and the first part that each of $\Pi_2,\Pi_3$ is of length two.

Consider $\Pi_2$.
Observe that \eqref{ecs-a-1} and \eqref{id2} imply that $\pi_1,\pi_2,\pi_3,\pi_4\not\le\Pi_2$.
By \eqref{BPLC} and the multiplicity freeness of $\Pi_{(\alpha-1,\alpha,\alpha+1)}$ we get
\begin{equation} \label{necs-a-1}
\Pi_2=\pi_7+\pi_5
\end{equation}
This further implies
\begin{equation} \label{eq5}
\Pi_3=\pi_8+\pi_6.
\end{equation}
Observe that $\delta([\alpha-1,\alpha+1])^t\otimes\s\le s_{\GL}(\pi_6^t)$
and recall that $s_{\GL}(\Pi_{(\alpha-1,\alpha,\alpha+1)})$ is multiplicity free.
On the other hand,
\begin{multline*}
\pi_5\h [-\alpha-1]\rtimes \delta_{\spsi}([\alpha-1],[\alpha];\s)
\\
\h
\Pi_{(-\alpha-1,\alpha-1,\alpha)}\cong\Pi_{(\alpha-1,\alpha,-\alpha-1)}\cong\Pi_{(\alpha-1,\alpha,\alpha+1)}.
\end{multline*}
Therefore, $[\alpha-1]\otimes[\alpha]\otimes [\alpha+1]\otimes\s$ occurs in the
Jacquet module of $\pi_5$,
which implies (by the transitivity of Jacquet modules) that the same is true for $\delta([\alpha-1,\alpha+1])^t\otimes\s$.
Therefore,
\[
\pi_6^t=\pi_5
\]
which further implies
\[
\pi_8^t=\pi_7.
\]
This finishes the proof of the proposition.
\end{proof}

It remains to show the non-unitarizability of $\pi_5,\pi_6,\pi_7,\pi_8$.

\subsection{Non-unitarizability of \texorpdfstring{$\pi_5$}{pi5} and \texorpdfstring{$\pi_6$}{pi6}}

Let
\[
\Pi'=\delta([\alpha-1,\alpha+1])\rtimes\s.
\]
By \cite[Theorem 4.1 (A2)]{MR3360752}
$$
\Pi'=\pi_6+\pi_3.
$$
Let $\delta_1=\delta([-\alpha,\alpha])$ and consider
$$
\delta_1\rtimes\Pi'=\delta_1\rtimes\pi_6+\delta_1\rtimes\pi_3.
$$
Clearly, it contains
$$
\Gamma_1^{\pm}=L([\alpha-1,\alpha+1];\delta([-\alpha,\alpha]_\pm; \s))
$$
as subquotients. Moreover, since
\begin{gather*}
\delta_1\times\Pi'\ge\delta([-\alpha,\alpha+1])\times\delta([\alpha-1,\alpha])\rtimes\s\\=
\delta([-\alpha,\alpha+1])\times L([\alpha-1,\alpha];\s)+\delta([-\alpha,\alpha+1])\times L([\alpha-1];\delta([\alpha];\s)),
\end{gather*}
$\delta_1\times\Pi'$ will also contain the following irreducible subquotients
\begin{gather*}
\Gamma_2^{\pm}=L([\alpha-1,\alpha];\delta([-\alpha,\alpha+1]_\pm;\s)),\\
\Gamma_3=L([-\alpha,\alpha+1],[\alpha-1,\alpha];\s),\\
\Gamma_4=L([-\alpha,\alpha+1],[\alpha-1];\delta([\alpha];\s)).
\end{gather*}
In fact, $\Gamma_1^{\pm}, \Gamma_2^{\pm},\Gamma_3,\Gamma_4$ are already subquotients of $\Gamma:=\delta_1\rtimes\pi_6$,
since they cannot be subquotients of
$$
\delta_1\rtimes\pi_3.
$$
More precisely, by property of Langlands quotients, each of $\Gamma_1^\pm$, $\Gamma_3$ and $\Gamma_4$ has a factor with exponent $-\alpha-1$,
while $\delta_1\rtimes\pi_3$ does not have.
Furthermore, each of $s_{\GL}(\Gamma_2^{\pm})$ admits an irreducible subquotient
which has in its support the exponent $-\alpha$ twice, while $\delta_1\rtimes\pi_3$ can have it at most once (see \S\ref{segment}).

Therefore, the length of $\Gamma$ is at least 6.

On the other hand, we shall show that the multiplicity of
$$
\tau:=\delta_1\otimes\pi_6
$$
in $\mu^*(\Gamma)$ is at most 4.
By Lemma \ref{lem: nonunit}, this will imply that $\pi_6$ and $\pi_5$ are not unitarizable.

In fact, we show that the multiplicity of $\tau$ in $\mu^*(\delta_1\times \Pi')=
M^*(\delta_1)\times M^*(\delta([\alpha-1,\alpha+1]))\rtimes (1\otimes\s)$ is four.
By remark \ref{rem: mdtp} and the formula for $M^*(\delta([\alpha-1,\alpha+1]))$ we need to compute the multiplicity of $\tau$ in
\[
\begin{gathered}
\Big(2\sum_{x= -\alpha}^{ \alpha}\delta([-x,\alpha])\otimes\delta([x+1,\alpha])+1\otimes\delta_1\Big)\times\\
\Big(\sum_{i= \alpha-2}^{ \alpha+1}\sum_{j=i}^{ \alpha+1}
\delta([-i,-\alpha+1]) \times\delta([j+1,\alpha+1]) \otimes
\delta([i+1,j])\Big)\rtimes 1\otimes\s.
\end{gathered}
\]
Only the terms corresponding to $j=\alpha+1$ may contribute $\tau$ as a subquotient.
One possibility to get $\tau$ is to take $i=\alpha-2$, for which we must take $x=\alpha$.
This contributes $2\cdot\delta_1\otimes\Pi'$ which contains $\tau$ with multiplicity two.

Suppose now that $i>\alpha-2$. Then, the only possible contribution is from $i=\alpha$ and $x=\alpha-2$.
Again, this contributes multiplicity two of $\tau$ since $\Pi_{(\alpha-1,\alpha,\alpha+1)}$ is multiplicity free.

Our claim follows.

\subsection{Non-unitarizability of \texorpdfstring{$\pi_8$}{pi8} and \texorpdfstring{$\pi_7$}{pi7}}

As before, let $\delta_1=\delta([-\alpha,\alpha])$.

\begin{lemma}
Let
\[
\widehat\Gamma=\delta_1\rtimes\pi_8.
\]
Then, $\widehat\Gamma$ contains the following irreducible subquotients
\begin{gather*}
\Gamma_1^{\pm}:=L([\alpha,\alpha+1],[\alpha-1];\delta([-\alpha,\alpha]_\pm;\s)),\\
\Gamma_2^{\pm}:=L([\alpha],[\alpha-1];\delta([-\alpha,\alpha+1]_\pm;\s)),\\
\Gamma_3:=L([\alpha],[\alpha-1],[-\alpha,\alpha+1];\s),\ \
\Gamma_4:=L([-\alpha,\alpha+1];\delta_{\spsi}([\alpha-1],[\alpha];\s)).
\end{gather*}
In particular, the length of $\widehat\Gamma$ is at least six.
\end{lemma}

\begin{proof}
By Proposition \ref{prop: addparms}
$$
\Gamma_1^{\pm}\leq \widehat\Gamma.
$$

We have
$$
[\alpha-1]\times \delta([\alpha,\alpha+1])\rtimes\s=\Pi_3+\Pi_4=\pi_6+\pi_8+\Pi_4,
$$
by \eqref{eq5}, which implies
\begin{multline} \label{*}
\delta([-\alpha,\alpha+1])\times [\alpha-1]\times [\alpha]\rtimes\s\leq
\delta_1\times [\alpha-1]\times \delta([\alpha,\alpha+1])\rtimes\s=\\
\widehat\Gamma+\delta_1\rtimes\pi_6+\delta_1 \rtimes \Pi_4.
\end{multline}
The left-hand side contains, among others, $\Gamma_2^{\pm}$, $\Gamma_3$ and $\Gamma_4$ as subquotients.
On the other hand, by \eqref{BPLC}, none of $\Gamma_2^{\pm}$, $\Gamma_3$ and $\Gamma_4$ can be a subquotient of $\delta_1 \times \Pi_4$.
We claim that none of them can be a subquotient of $\delta_1\rtimes\pi_6$ as well.

Observe that
$$
\Gamma_2^{\pm}\h L([-\alpha],[-\alpha+1])\rtimes \delta([-\alpha,\alpha+1]_\pm;\s)
\h L([-\alpha],[-\alpha+1])\times \delta([-\alpha,\alpha+1])\rtimes \s
$$
and hence by Frobenius reciprocity
$$
L([-\alpha],[-\alpha+1])\times \delta([-\alpha,\alpha+1])\otimes\s\leq s_{\GL}(\Gamma_2^{\pm}).
$$
Similarly,
\[
L([-\alpha],[-\alpha+1])\times \delta([-\alpha-1,\alpha])\otimes\s\leq s_{\GL}(\Gamma_3).
\]
and
\[
L([-\alpha-1,\alpha],[\alpha-1],[\alpha])\otimes\s\le s_{\GL}(\Gamma_4).
\]
On the other hand,
\[
\begin{aligned}
s_{\GL}(\delta_1\rtimes\pi_6) \leq
&\Big(\sum_{x= -\alpha-1}^{ \alpha}\delta([-x,\alpha])\times\delta([x+1,\alpha]\Big)\times\\
&\Big(\sum_{i= \alpha-2}^{ \alpha+1}
\delta([-i,-\alpha+1])\times\delta([i+1,\alpha+1]\Big)\otimes\s.
\end{aligned}
\]
Clearly, no irreducible subquotient of the right-hand side is of the form $L([\alpha-1]+b)\otimes\s$
or $L([-\alpha]+b)\otimes\s$ for any multisegment $b$.
Therefore,
$$
\Gamma_2^\pm,\Gamma_3,\Gamma_4\not\leq \delta_1\rtimes\pi_6.
$$
By \eqref{*} we infer that
\[
\Gamma_2^\pm,\Gamma_3,\Gamma_4\le\widehat\Gamma.
\]
The lemma follows.
\end{proof}

The non-unitarizability of $\pi_8$ and $\pi_7$ would follow from Lemma \ref{lem: nonunit} once
we prove the following lemma.

\begin{lemma}
The multiplicity of $\widehat\tau:=\delta_1\otimes \pi_8$ in $\mu^*(\widehat\Gamma)$ is at most 4.
\end{lemma}

\begin{proof}
We first show that the multiplicity of $\widehat\tau$ in $\mu^*(\delta_1\rtimes\Pi_3)=M^*(\delta_1)\rtimes\mu^*(\Pi_3)$ is six.

By Remak \ref{rem: mdtp} we need to consider the multiplicity of $\widehat\tau$ in
\begin{equation} \label{eq: sumibg}
\Big(2\sum_{i=-\alpha}^\alpha\delta([-i,\alpha])\otimes\delta([i+1,\alpha])+1\otimes\delta_1\Big)\rtimes\mu^*(\Pi_3)
\end{equation}
By \eqref{eq: pi312} we have
\[
\begin{gathered}
\mu^*(\Pi_3)
=(1\otimes [\alpha-1]+[\alpha-1]\otimes1+[-\alpha+1]\otimes1)
\\\rtimes\Big(
1\otimes L([\alpha,\alpha+1];\s)+
[-\alpha]\otimes[\alpha+1]\rtimes\s+\cancel{[\alpha+1]\otimes L([\alpha];\s)}
\\+\cancel{\delta([-\alpha-1,-\alpha])\otimes\s}+\cancel{[-\alpha]\times[\alpha+1]\otimes\s}\Big),
\end{gathered}
\]
where we struck down terms which cannot possibly contribute $\widehat\tau$ in $\mu^*(\delta_1\rtimes\Pi_3)$.

By a simple analysis, the only pertinent contribution from $\mu^*(\Pi_3)$ is
\[
1\otimes\Pi_3+[-\alpha]\otimes[\alpha-1]\times [\alpha+1]\rtimes\s+
[-\alpha]\times [-\alpha+1]\otimes[\alpha+1]\rtimes\s
\]
and for each of these summands, the relevant term in \eqref{eq: sumibg} is
$i=\alpha$, $i=\alpha-1$ and $i=\alpha-2$ respectively.
Thus, we only need to consider the terms
\begin{gather*}
2\delta_1\otimes\Pi_3\\
+2[-\alpha]\times\delta([-\alpha+1,\alpha])\otimes\Pi_{(\alpha-1,\alpha,\alpha+1)}\\
+2[-\alpha]\times [-\alpha+1]\times\delta([-\alpha+2,\alpha])\otimes [\alpha+1]\times\delta([\alpha-1,\alpha])\rtimes\s
\end{gather*}
and the multiplicity of $\widehat\tau$ in each of these summands is two.

On the other hand,
\begin{gather*}
\mu^*(\delta_1\rtimes\pi_6)=M^*(\delta_1)\rtimes\mu^*(\pi_6)\\
\ge2\Big(\delta([-\alpha+2,\alpha])\otimes \delta([\alpha-1,\alpha])\Big)\rtimes
\Big(\delta([-\alpha,-\alpha+1])\otimes [\alpha+1]\rtimes\s\Big)\\
\ge2\delta_1\otimes \delta([\alpha-1,\alpha])\times[\alpha+1]\rtimes\s\ge
2\delta_1\otimes L([\alpha,\alpha+1],[\alpha-1])\rtimes\s\ge 2\widehat\tau.
\end{gather*}
Therefore, by \eqref{eq5}, the multiplicity of $\widehat\tau$ in $\mu^*(\widehat\Gamma)
=\mu^*(\delta_1\rtimes\Pi_3)-\mu^*(\delta_1\rtimes\pi_6)$
is at most 4.
\end{proof}

\section{\texorpdfstring{$\mathbf{x}=(\alpha-1,\alpha,\alpha)$}{x-100} and \texorpdfstring{$\alpha>1$}{alpha>1}}

\begin{proposition} \label{a-1,a,a-a}
Assume $\alpha>1$. Then,
\begin{enumerate}
\item 
\label{a-1,a,a-a-item: Gg}
In the Grothendieck group we have
\[
\Pi_{(\alpha-1,\alpha,\alpha)}=2\pi_0+\pi_1+\pi_2+\pi_3+\pi_4
\]
where
\begin{gather*}
\pi_0=L([\alpha],[\alpha-1]; \delta([\alpha];\s))\\
\pi_1=L([\alpha]; \delta_{\spsi}([\alpha-1],[\alpha];\s)),\ \
\pi_2=L([\alpha],[\alpha-1,\alpha];\s),\\
\pi_3=L([\alpha-1,\alpha]; \delta([\alpha];\s)),\ \
\pi_4=L([\alpha],[\alpha-1],[\alpha];\s).
\end{gather*}
\item $\pi_1^t=\pi_2$, $\pi_3^t=\pi_4$, $\pi_0^t=\pi_0$.
\item We have
\label{a-1,a,a-a-item: induction}
\begin{gather*}
\pi_1= [\alpha]\rtimes \delta_{\spsi}([\alpha-1],[\alpha];\s),\ \
\pi_2= [\alpha]\rtimes L([\alpha-1,\alpha];\s),\\
\pi_3=\delta([\alpha-1,\alpha])\rtimes \delta([\alpha]; \s),\ \
\pi_4= L([\alpha-1],[\alpha])\rtimes L([\alpha]; \s).
\end{gather*}
\item None of the representations $\pi_1,\pi_2,\pi_3,\pi_4$ is unitarizable.
\item (See Appendix \ref{appendix-M}) The representation $\pi_0$ is unitarizable.
\end{enumerate}

\end{proposition}

\begin{proof}
Recall
\begin{equation}
\label{eq: jord-st-pos-a-1-a}
\Jord_\rho(\delta_{\spsi}([\alpha-1],[\alpha];\s))=\{\eta+1,\eta+3,\dots,2\alpha-5,2\alpha-1,2\alpha+1 \},
\end{equation}
and that the partially defined function $\e$ attached to the above square-integrable representation is different on $2\alpha-1$ and $2\alpha+1$.
Now
$$
[\alpha]\rtimes \delta_{\spsi}([\alpha-1],[\alpha];\s)
$$
is irreducible by (\ref{Pr-red-si-item: a-a+1-epsilon}) of  Proposition \ref{Pr-red-si}.
Applying the \ASS involution we infer that
$$
[\alpha]\rtimes L([\alpha-1,\alpha];\s)
$$
is irreducible (since $\delta_{\spsi}([\alpha-1],[\alpha];\s)^t=L([\alpha-1,\alpha];\s)$).
This implies (using Proposition \ref{prop: addparms}) that
\[
\pi_1= [\alpha]\rtimes \delta_{\spsi}([\alpha-1],[\alpha];\s)\text{ and }
\pi_2= [\alpha]\rtimes L([\alpha-1,\alpha];\s).
\]
It also implies that $\pi_1^t=\pi_2$.

Examining $s_{\GL}(\Pi_{(\alpha-1,\alpha,\alpha)})$, we get that the multiplicity of $\pi_2$ in $\Pi_{(\alpha-1,\alpha,\alpha)}$ is one.
Therefore, also its \ASS dual $\pi_1$ has multiplicity one in $\Pi_{(\alpha-1,\alpha,\alpha)}$.

Using \eqref{BPLC} and the description of $\JH(\Pi_{(\alpha-1,\alpha,\alpha)})$ we see that
$\delta([\alpha-1,\alpha])\rtimes \delta([\alpha]; \s)$ is irreducible. Therefore,
\[
\pi_3=\delta([\alpha-1,\alpha])\rtimes \delta([\alpha]; \s)
\]
and by applying \ASS involution we get
\[
\pi_4= L([\alpha-1],[\alpha])\rtimes L([\alpha]; \s).
\]

This implies that $\pi_3^t=\pi_4$ and hence, $\pi_0^t=\pi_0$.

Since $L([\alpha+1]; \delta_{\spsi}([\alpha-1],[\alpha];\s))$ is not unitarizable
(Proposition \ref{a-1,a,a+1}), we get that
$$
\pi_1=[\alpha]\rtimes \delta_{\spsi}([\alpha-1],[\alpha];\s)
$$
is not unitarizable by the usual deformation argument.

Note that
\begin{multline*}
L([\alpha+1]; \delta_{\spsi}([\alpha-1],[\alpha];\s))\leq
[\alpha+1]\rtimes \delta_{\spsi}([\alpha-1],[\alpha];\s)
\\
=
([\alpha+1]\rtimes L([\alpha-1,\alpha];\s))^t.
\end{multline*}
Applying the \ASS involution and Proposition \ref{a-1,a,a+1} we get
$$
L([\alpha-1,\alpha+1];\s)\leq [\alpha+1]\rtimes L([\alpha-1,\alpha];\s).
$$
By deformation, it follows that $\pi_2=[\alpha]\rtimes L([\alpha-1,\alpha];\s)$ is not unitarizable
(since $L([\alpha-1,\alpha+1];\s)$ is not).

The non-unitarizability of
$$
\pi_3=\delta([\alpha-1,\alpha])\rtimes \delta([\alpha]; \s)
$$
follows from the fact that for the exponents $(\alpha,\alpha,\alpha+1)$ we do not have unitarizable irreducible subquotients
(Proposition \ref{a,a,a+1-a}) by deforming $\delta([\alpha-1,\alpha])$ to $\delta([\alpha,\alpha+1])$.
Similarly for
$$
\pi_4=L([\alpha-1],[\alpha])\rtimes L([\alpha]; \s).
$$

Consider
\[
\Pi_{(\alpha,\alpha-1,\alpha)}=[\alpha]\times[\alpha-1]\times [\alpha]\rtimes \s=
\Pi_1+\Pi_2+\Pi_3+\Pi_4
\]
where
\begin{gather*}
\Pi_1=L([\alpha],[\alpha-1])\rtimes L([\alpha]; \s),\ \
\Pi_2=\delta([\alpha-1,\alpha])\rtimes L([\alpha]; \s),\\
\Pi_3=L([\alpha],[\alpha-1])\rtimes \delta([\alpha]; \s),\ \
\Pi_4=\delta([\alpha-1,\alpha])\rtimes \delta([\alpha]; \s).
\end{gather*}

Looking at $s_{\GL}(\Pi)$, we get that the multiplicity of the representation $\pi_0$ in
$\Pi_{(\alpha,\alpha-1,\alpha)}$ is at most two. Further, $\pi_4$ has multiplicity one in $\Pi_{(\alpha-1,\alpha,\alpha)}$ (since it is the Langlands quotient of $\Pi_{(\alpha1,\alpha,\alpha-1)}$).
Therefore, $\pi_3$ has multiplicity one in $\Pi_{(\alpha-1,\alpha,\alpha)}$.

All the above discussion implies
\begin{equation}
\Pi_2=\pi_2+\pi_0,\ \Pi_3=\pi_1+\pi_0.
\end{equation}
This implies (\ref{a-1,a,a-a-item: Gg}), and the proof is now complete.
\end{proof}

From parts (\ref{a-1,a,a-a-item: Gg}) and (\ref{a-1,a,a-a-item: induction}) of the above proposition it directly follows that $[\alpha]\rtimes L([\alpha-1]; \delta([\alpha];\s))$ is reducible.
Actually, it is easy to obtain the composition series of $[\alpha]\rtimes L ([\alpha-1]; \delta([\alpha];\s))$ as follows.

Observe that
$$
\pi_4\le [\alpha]\rtimes L ([\alpha-1],[\alpha];\s)=
( [\alpha]\rtimes L ([\alpha-1]; \delta([\alpha];\s)))^t
$$
Passing to the \ASS dual, we get
\begin{equation} \label{red-to-use-1}
\pi_3\leq [\alpha]\rtimes L ([\alpha-1]; \delta([\alpha];\s)).
\end{equation}
From this we conclude
\begin{equation} \label{to-use-below-1}
\begin{gathered}
{}[\alpha]\rtimes L ([\alpha-1]; \delta([\alpha];\s))=\pi_0+\pi_3,\\
[\alpha]\rtimes L([\alpha-1],[\alpha];\s)=\pi_0+\pi_4.
\end{gathered}
\end{equation}

We remark that the representation $\pi_0$ is isolated in the unitary dual.

\section{\texorpdfstring{$\mathbf{x}=(\alpha-1,\alpha-1,\alpha)$}{xalpha-1-10} and \texorpdfstring{$\alpha>1$}{alpha>1}}

\subsection{Case \texorpdfstring{$\alpha\ge2$}{alphage2}} We first consider the case $\alpha\ge2$.

\begin{proposition} \label{prop: -1-10}
Assume $\alpha\ge2$.
\begin{enumerate}
\item In the Grothendieck group we have
\[
\Pi_{(\alpha-1,\alpha-1,\alpha)}=\pi_1+\pi_2+\pi_3+\pi_4
\]
where
\begin{gather*}
\pi_1=L([\alpha],[\alpha-1],[\alpha-1];\s),\ \
\pi_2=L([\alpha-1],[\alpha-1,\alpha];\s),\\
\pi_3=L([\alpha-1],[\alpha-1];\delta([\alpha];\s)),\ \
\pi_4=L([\alpha-1];\delta_{\spsi}([\alpha-1],[\alpha];\s).
\end{gather*}
\item We have
\begin{gather*}
\pi_1=[\alpha-1]\rtimes L([\alpha],[\alpha-1];\s),\ \
\pi_2=[\alpha-1]\rtimes L([\alpha-1,\alpha];\s),\\
\pi_3=[\alpha-1]\rtimes L([\alpha-1];\delta([\alpha];\s)),\ \
\pi_4=[\alpha-1]\rtimes \delta_{\spsi}([\alpha-1],[\alpha];\s).
\end{gather*}
\item $\pi_2^t=\pi_4$ and $\pi_3^t=\pi_1$.
\item None of $\pi_1,\pi_2,\pi_3,\pi_4$ is unitarizable.
\end{enumerate}
\end{proposition}

\begin{proof}
We may write
\[
\Pi_{(\alpha-1,\alpha-1,\alpha)}=\Pi_1+\Pi_2
\]
where
\[
\Pi_1=[\alpha-1]\times [\alpha-1]\rtimes L([\alpha];\s),\ \
\Pi_2=[\alpha-1]\times [\alpha-1]\rtimes \delta([\alpha];\s).
\]
Also,
\begin{gather*}
\Pi_1=[\alpha-1]\rtimes L([\alpha],[\alpha-1];\s)+[\alpha-1]\rtimes L([\alpha-1,\alpha];\s),\\
\Pi_2=[\alpha-1]\rtimes L([\alpha-1];\delta([\alpha];\s))+[\alpha-1]\rtimes \delta_{\spsi}([\alpha-1],[\alpha];\s).
\end{gather*}
Recall that by Proposition \ref{prop-2-a-}
\[
\delta_{\spsi}([\alpha-1],[\alpha];\s)^t=L([\alpha-1,\alpha];\s),\ \
L([\alpha],[\alpha-1];\s)^t=L([\alpha-1];\delta([\alpha];\s)),
\]
and
$$
\Jord_\rho(\delta([\alpha];\s)))=\{\eta+1,\eta+3,\dots,2\alpha-3,2\alpha+1 \},
$$
$$
\Jord_\rho(\delta_{\spsi}([\alpha-1],[\alpha];\s)))=\{\eta+1,\eta+3,\dots,2\alpha-5,2\alpha-1,2\alpha+1 \},
$$
where in the last case partially defined function $\e$ attached to the square-integrable representation is different on $2\alpha-1$ and $2\alpha+1$.

By (\ref{Pr-red-si-item: a-a+1-irr}) of  Proposition \ref{Pr-red-si} and \eqref{eq: jord-st-pos-a-1-a} we conclude that $[\alpha-1]\rtimes \delta_{\spsi}([\alpha-1],[\alpha];\s)$ is irreducible,
and hence equals to $\pi_4$. (Here we used that $\alpha\ge2$.)

Furthermore, the \ASS involution implies that $[\alpha-1]\rtimes L([\alpha-1,\alpha];\s)$ is irreducible, hence equals $\pi_2$.

Consider $[\alpha-1]\rtimes L([\alpha-1];\delta([\alpha];\s))$.
If this were not irreducible, then \eqref{BPLC} would imply that
$$
\pi_4\leq[\alpha-1]\rtimes L([\alpha-1];\delta([\alpha];\s)),
$$
which would imply further that
\begin{gather*}
([\alpha-1]+[-\alpha+1])\times L([\alpha-1],[\alpha])\otimes\s\leq\\
([\alpha-1]+[-\alpha+1])\times\Big([-\alpha+1]\times[\alpha] +\delta([\alpha-1,\alpha])\Big)\otimes\s.
\end{gather*}
This is impossible (since $[\alpha-1]\times L([\alpha-1],[\alpha])\otimes\s$ does not occur on the right-hand side).

Therefore, $[\alpha-1]\rtimes L([\alpha-1];\delta([\alpha];\s))$ is irreducible
(hence equals $\pi_3$), which implies using the \ASS involution that $[\alpha-1]\rtimes L([\alpha],[\alpha-1];\s)$
is irreducible (hence equals $\pi_1$).

Recall that each of $\pi_1,\pi_2,\pi_3,\pi_4$ is of the form $[\alpha-1]\rtimes \tau$, where $\tau $ is an irreducible subquotient of $\Pi_{(\alpha-1.\alpha)}$.
We will deform $[\alpha-1]$ to $[\alpha]$ and use Proposition \ref{a-1,a,a-a}.

The non-unitarizability of $\pi_4$ and $\pi_2$ follows from the non-unitarizability of
$$
L([\alpha]; \delta_{\spsi}([\alpha-1],[\alpha];\s))=[\alpha]\rtimes \delta_{\spsi}([\alpha-1],[\alpha];\s),
$$
and
$$
L([\alpha],[\alpha-1,\alpha];\s)=[\alpha]\rtimes L([\alpha-1,\alpha];\s)
$$
respectively.
Similarly, for $\pi_1$ we use the fact that $L([\alpha],[\alpha-1],[\alpha];\s)$ is a non-unitarizable subquotient
of $[\alpha]\rtimes L([\alpha],[\alpha-1];\s)$ (by Proposition \ref{prop: addparms}).
For $\pi_3$ we use that $L([\alpha-1,\alpha];\delta([\alpha]; \s))$ is a non-unitarizable irreducible subquotient of $[\alpha]\rtimes
L ([\alpha-1]; \delta([\alpha];\s))$ \eqref{red-to-use-1}.
This completes the proof of the proposition.
\end{proof}

\subsection{Case \texorpdfstring{$\alpha=\tfrac32$}{alpha32}} \label{sec: 121232-32}
We now consider the case $\alpha=\tfrac32$.

\begin{proposition}
\label{prop: red32-113}
 Let $\alpha=\frac32$.

\begin{enumerate}
\item
\label{prop: red32-113-item: Gg}
In the Grothendieck group we have
$$
\Pi_{(\frac12,\frac12,\frac32)}=\pi_1+\pi_2+\pi_3+\pi_4+\pi_5+\pi_6+\pi_7+\pi_8,
$$
where
\begin{gather*}
\pi_1=L([\tfrac12],[\tfrac12],[\tfrac32];\s),\ \
\pi_2=\tau([-\tfrac12,\tfrac12]_+;\delta([\tfrac32];\s)),\\
\pi_3=L([\tfrac12,\tfrac32],[\tfrac12];\s),\ \
\pi_4=\tau([-\tfrac12,\tfrac12]_-;\delta([\tfrac32];\s)),\\
\pi_5=L([\tfrac32];\delta([-\tfrac12,\tfrac12])\rtimes\s), \ \
\pi_6=L([\tfrac12],[\tfrac12];\delta([\tfrac32];\s)),\\
\pi_7=L([-\tfrac12,\tfrac32];\s), \ \
\pi_8=L([\tfrac12];\delta_{\spsi}([\tfrac12],[\tfrac32];\s)).
\end{gather*}

\item 
\label{prop: red32-113-item: all unit}
All $\pi_1,\dots,\pi_8$ are unitarizable. (They are subrepresentations of unitarily induced representations.)

 \item 
 \label{prop: red32-113-item: inv}
 $\pi_1^t=\pi_2, \ \pi_3^t=\pi_4, \ \pi_5^t=\pi_6, \ \pi_7^t=\pi_8.$

\end{enumerate}
\end{proposition}

\begin{proof} We directly check that $\pi_1,\dots,\pi_8$ is the Jordan-H\"older series of the representation $\Pi_{(\frac12,\frac12,\frac32)}$.
Further, in the Grothendieck group we have
\begin{equation} \label{sum-Pi-32}
\Pi_{(\frac12,\frac12,\frac32)}=\Pi_1+\Pi_2+\Pi_3+\Pi_4,
\end{equation}
where
\[
\begin{aligned}
&\Pi_1:=\delta([-\tfrac12,\tfrac12])\rtimes \delta([\tfrac32];\s),\ \
\Pi_2:=\delta([-\tfrac12,\tfrac12])\rtimes L([\tfrac32];\s)\\
&\Pi_3:=L([-\tfrac12],[\tfrac12])\rtimes \delta([\tfrac32];\s),\ \
\Pi_4:=L([-\tfrac12],[\tfrac12])\rtimes L([\tfrac32];\s),
\end{aligned}
\]
and each of the four summands is unitarizable, which implies that (\ref{prop: red32-113-item: all unit}) holds.

Frobenius reciprocity implies that each $\Pi_i$ is a (multiplicity one) representation of length at most two.
The above equality and the fact that we have 8 irreducible sub-quotients in the Jordan-H\"older series imply that
each of $\Pi_i$ is a length 2 representation, and that $\Pi_{(\frac12,\frac12,\frac32)}$ is a multiplicity one representation of length 8.
Thus, (\ref{prop: red32-113-item: Gg}) holds.

Further we know from \ref{dist-formula} that
$$
\pi_1^t=\pi_2.
$$
Observe that $\Pi_1^t=\Pi_4$ and $\Pi_2^t=\Pi_3$. This implies that no $\pi_i$ is fixed by DL duality.
Obviously
\[
\Pi_1=\pi_2\oplus\pi_4.
\]
Further, (\ref{prop: red32-113-item: Gg}) (which we have proved), \eqref{sum-Pi-32} and \eqref{BPLC} imply that
\[
\Pi_3=\pi_8\oplus\pi_6,
\]
and further
$$
\Pi_2=\pi_5\oplus\pi_7.
$$
Therefore,
$$
\Pi_4=\pi_1\oplus\pi_3.
$$
Now $\pi_1^t=\pi_2$ implies
$$
\pi_3^t=\pi_4.
$$
Observe that we have $[-\tfrac12]\times[- \tfrac12]\otimes \delta([\tfrac32];\sigma)$ in the Jacquet module of $\pi_6$.
A short discussion implies that we have also $[-\tfrac12]\times[- \tfrac12]\times [\tfrac32]\otimes\sigma$ in the Jacquet module of $\pi_6$.
Therefore, $[\tfrac12]\times[ \tfrac12]\times [-\tfrac32]\otimes \sigma$ must be in the Jacquet module of $\pi_6^t$.
Since this term is not in the Jacquet module of $\delta([-\tfrac12,\tfrac32])\rtimes\sigma$ and $\pi_7\leq \delta([-\tfrac12,\tfrac32])\rtimes\sigma$, we conclude $\pi_6^t\ne
\pi_7$. This implies
$$
\pi_6^t=\pi_5,
$$
which further implies
$$
\pi_7^t=\pi_8.
$$
This completes the proof of (\ref{prop: red32-113-item: inv}).
\end{proof}

\begin{remark} \label{rem-red-for-cs}
Clearly $\pi_8\leq [\tfrac12]\rtimes \delta_{\spsi}([\tfrac12],[\tfrac32];\s).$
On the other hand, $\pi_3\leq [\tfrac12]\rtimes L([\tfrac12,\tfrac32];\s).$
Now duality implies $\pi_4\leq [\tfrac12]\rtimes \delta_{\spsi}([\tfrac12],[\tfrac32];\s).$
Therefore, $ [\tfrac12]\rtimes \delta_{\spsi}([\tfrac12],[\tfrac32];\s)$ is reducible.

Further, $\pi_6\leq [\tfrac12]\rtimes L([\tfrac12];\delta([\tfrac32];\s))$.
Since $\pi_1\leq [\tfrac12]\rtimes L([\tfrac12],[\tfrac32],\s)$, we get
$\pi_2\leq [\tfrac12]\rtimes L([\tfrac12];\delta([\tfrac32];\s))$, which implies that
$[\tfrac12]\rtimes L([\tfrac12];\delta([\tfrac32];\s))$ is reducible.
\end{remark}

\section{\texorpdfstring{$\mathbf{x}=(\alpha-2,\alpha-1,\alpha)$}{xalpha-2-10} and \texorpdfstring{$\alpha\geq 2$}{alphage2}}

\subsection{Case \texorpdfstring{$\alpha>2$}{alpha>2}}
We first consider the case $\alpha>2$.

\begin{proposition}
\label{prop: a-2-a-1-a}
 Assume $\alpha>2$.

\begin{enumerate}

\item
\label{prop: a-2-a-1-a-item: Gg}
 In the Grothendieck group we have
$$
\Pi_{(\alpha-2,\alpha-1,\alpha)}=\pi_1 + \pi_2 + \pi_3 + \pi_4 + \pi_5 + \pi_6 + \pi_7 + \pi_8,
$$
where
\begin{gather*}
\pi_1=\delta_{\spsi}([\alpha-2],[\alpha-1],[\alpha];\s),\ \
\pi_2=L([\alpha-2];\delta_{\spsi}([\alpha-1],[\alpha];\s)),\\
\pi_3=L([\alpha-1],[\alpha-2];\delta([\alpha];\s)),\ \
\pi_4=L([\alpha-2,\alpha-1];\delta([\alpha];\s)),\\
\pi_5=L([\alpha],[\alpha-1],[\alpha-2];\s),\ \
\pi_6=L([\alpha],[\alpha-2,\alpha-1];\s),\\
\pi_7=L([\alpha-1,\alpha],[\alpha-2];\s),\ \
\pi_8=L([\alpha-2,\alpha];\s).
\end{gather*}

\item 
\label{prop: a-2-a-1-a-item: all unit}
All the irreducible subquotients of $\Pi_{(\alpha-2,\alpha-1,\alpha)}$
are unitarizable (they are subquotients of the ends of complementary series).

\item
\label{prop: a-2-a-1-a-item: inv}
 $\pi_1^t=\pi_8, \ \pi_2^t=\pi_7, \ \pi_3^t=\pi_6, \ \pi_4^t=\pi_5.$
\end{enumerate}
\end{proposition}

\begin{proof}
We get (\ref{prop: a-2-a-1-a-item: Gg}) and (\ref{prop: a-2-a-1-a-item: all unit}) from \S\ref{St-oposite}. It remains to prove (\ref{prop: a-2-a-1-a-item: inv}).

The fact
$
\pi_1
\h [\alpha-2]\times[\alpha-1] \times[\alpha] \rtimes\s
$
implies
$
\pi_1\h L( [\alpha-2],[\alpha-1] ,[\alpha] )\rtimes\s
$
which further implies
$$
\delta([ -(\alpha-2),-a])\otimes\s\leq s_{\GL}(\pi_1^t).
$$
Since $\delta([ -(\alpha-2),-a])\otimes\s\leq s_{\GL}(\pi_8)$, we get
$$
\pi_1^t=\pi_8.
$$

Further
$$
\pi_2
\h [-(\alpha-2)]\times[\alpha-1] \times[\alpha] \rtimes\s
$$
implies
$$
\pi_2 \h [-(\alpha-2)]\times L([\alpha-1] ,[\alpha] )\rtimes\s\cong
$$
$$
 L([\alpha-1] ,[\alpha] )\times [-(\alpha-2)]\rtimes\s
 \cong
 L([\alpha-1] ,[\alpha] )\times [\alpha-2]\rtimes\s.
$$
This implies that
$\delta([-\alpha,-(\alpha-1)])\otimes [-(\alpha-2)] \otimes\s $ is in the Jacquet module of $\pi_2^t$, which implies
$$
\pi_2^t=\pi_7.
$$

Consider now
$$
\pi_3
\h L([-(\alpha-1)],[-(\alpha-2)]) \times[\alpha] \rtimes\s
\cong
[\alpha]\times L([-(\alpha-1)],[-(\alpha-2)]) \rtimes\sigma
$$
$$
\cong
[\alpha]\times L([\alpha-2],[\alpha-1]) \rtimes\s.
$$
This implies that $[\alpha]\otimes L([\alpha-2],[(\alpha-1)]) \otimes\s$ is in the Jacquet module of $\pi_3$, and therefore $[-\alpha]\times \delta([-(\alpha-1),-(\alpha-2)]) \rtimes\s$
is in the Jacquet module of $\pi_3^t$.
Therefore,
$$
\pi_3^t= \pi_6.
$$

Finally, consider
$$
\pi_4
\h \delta([-(\alpha-1),-(\alpha-2)]) \times[\alpha] \rtimes\s
\cong
[\alpha] \times \delta([-(\alpha-1),-(\alpha-2)]) \rtimes\s
$$
$$
\cong[\alpha] \times \delta([\alpha-2,\alpha-1]) \rtimes\s
$$
This implies that $[-\alpha]\otimes L([-(\alpha-1)],[-(\alpha-2)]) \otimes\s$ is in the Jacquet module of $\pi_4^t$.
Therefore, $\pi_4^t$ embeds into
$$
[-\alpha] \times[-(\alpha-1)]\times[-(\alpha-2)]) \rtimes\s,
$$
which implies
$$
\pi_4^t=\pi_5.
$$
\end{proof}

One directly gets that $\delta([\alpha-2,\alpha-1]) \rtimes \delta([\alpha];\s)$ and
$\delta([\alpha-2,\alpha-1]) \rtimes L([\alpha];\s)$ reduce.

\subsection{Case \texorpdfstring{$\alpha=2$}{alpha2}}

Now we turn to the case $\alpha=2$.

\begin{proposition} 
\label{prop: 012alpha=2} 
Assume $\alpha=2$.

\begin{enumerate}

\item 
\label{prop: 012alpha=2-item: Gg}
In the Grothendieck group we have
$$
\Pi_{(0,1,2)}=\pi_1 + \pi_2 + \pi_3 + \pi_4 + \pi_5 + \pi_6 + \pi_7 + \pi_8,
$$
where
\begin{gather*}
\pi_1=L([1,2];[0]\rtimes\s),\quad \pi_2=\tau([0]_+;\delta_{\spsi}([1],[2];\s)),\\
\pi_3= L([0,2];\s), \quad \pi_4=\tau([0]_-;\delta_{\spsi}([1],[2];\s)),\\
\pi_5= L([1];[0]\rtimes\delta([2];\s)), \quad \pi_6=L([2],[0,1];\s),\\
\pi_7= L([0,1];\delta([2];\s)),\quad \pi_8= L([2],[1];[0]\rtimes\s).
\end{gather*}

\item
\label{prop: 012alpha=2-item: all unit}
 All $\pi_i$'s are unitarizable.

\item
\label{prop: 012alpha=2-item: inv}
 $\pi_1^t=\pi_2, \ \pi_3^t=\pi_4, \ \pi_5^t=\pi_6, \ \pi_7^t=\pi_8.$

\end{enumerate}
\end{proposition}

\begin{proof}
We conclude in a standard way that $\pi_1,\dots,\pi_8$ is the Jordan-H\"older series of $\Pi_{(0,1,2)}$.
All the irreducible subquotients of the induced representation
$$
\Pi_{(0,1,2)}=([0]\times[1])\rtimes ([2]\rtimes \s).
$$
are unitarizable (since they are subquotients of the ends of complementary series). Thus, (\ref{prop: 012alpha=2-item: all unit}) holds.

From our characterisation of $\pi_2$, one can easily compute:
$$
s_{\GL}(\pi_2)=2 L([0],[1],[2])\otimes\s+ L([0,1],[2])\otimes\s,
$$
$$
s_{\GL}(\pi_4)= L([0,1],[2])\otimes\s
$$
(see also \cite{MR3123571}, Theorem 1.2, (2), (b)).

Now we shall consider $\pi_7$ (such representations are called in \cite{MR3360752} segment representations).
From there we know that (the segment representation) $\pi_7$ has the following (semi-simplification of) Jacquet module:
\begin{equation} \label{JM-seg-02-2}
s_{\GL}(\pi_7)=\delta([0,2])\otimes\s+ [0]\times\delta([1,2])\otimes\s+ \delta([-1,0])\times[2]\otimes\s.
\end{equation}
This implies that the representation with irreducible generic term in $s_{\GL}$ whose all exponents are non-negative, is
$
\pi_7.
$
Further, one directly sees that
$
\pi_8
$
is a representation with irreducible cogeneric term in $s_{\GL}$ which has all exponents non-positive.
Now \eqref{dist-formula} implies
$$
\pi_8^t=\pi_7.
$$
This implies that the multiplicity of $\pi_7$ in $\Pi_{(0,1,2)}$ is one (since the multiplicity of $\pi_8$ is one).

In \cite{MR3360752} (see \ref{segment}) is computed
$$
s_{\GL}(\pi_3)= \delta([-2,0])\otimes\s.
$$
Write
$$
\Pi_{(0,1,2)}=\Pi_1+\Pi_2+\Pi_3+\Pi_4,
$$
where
\begin{gather*}
\Pi_1:=[0]\times L([1,2];\s),\qquad \Pi_2:= [0]\times L([1];\delta([2];\s)),\\
\Pi_3:= [0]\times L([1],[2];\s), \qquad
\Pi_4:=[0]\times \delta_{\spsi}([1],[2];\s).
\end{gather*}
Then,
$$
\Pi_1^t=\Pi_4, \qquad\Pi_2^t=\Pi_3.
$$
Obviously
$$
\Pi_4=\pi_2 \oplus \pi_4.
$$
This implies that $\Pi_1$ is a multiplicity one representation of length two.
The fact that Jordan-H\"older series of $\Pi_{(0,1,2)}$ have 8 representations, implies that also $\Pi_2$ and $\Pi_3$ are reducible.

Now \eqref{BPLC},
the above formulas for Jacquet modules of $\pi_2$ and $\pi_4$, and the formula for $s_{\GL}( L([1];\delta([2];\s)))$ imply that only possible
subquotients of $\Pi_2$ are $\pi_5$ and $\pi_7$.
This implies
$$
\Pi_2=\pi_5+\pi_7.
$$

Observe that
$$
\pi_1\leq \Pi_1,
$$
and therefore $\pi_1^t\leq \Pi_4$. Further we have in the Jacquet module of $\pi_1$ the term $2\delta([-2,-1])\otimes[0]\otimes\s$, and therefore $2L([1],[2])\otimes[0]\otimes\s$
is in the Jacquet module of $\pi_1^t$.
Since the minimal non-trivial Jacquet module of $\pi_4$ is a multiplicity one representation, we conclude
$$
\pi_1^t=\pi_2.
$$

Further \eqref{BPLC} implies $\pi_6\not\leq\Pi_1$ and $\pi_6\not\leq\Pi_2$, which implies
$$
\Pi_3=\pi_8+\pi_6,
$$
and further
$$
\Pi_1=\pi_1+\pi_3.
$$
Now we conclude that
$$
\pi_3^t=\pi_4.
$$

Observe that $\pi_8\h [-2]\times[-1]\times[0]\rtimes\s$, which implies $\pi_8\h L([-2],[-1],[0])\rtimes\s$.
Therefore, $\delta([0,2])\otimes\s$ is in the Jacquet module of $\pi_8^t$.

The fact that $\pi_5\h [-1] \times [0] \rtimes \delta([2];\s)$ implies $\pi_5\h L([-1],[0]) \rtimes \delta([2];\s)$, which implies
$$
s_{\GL}(\pi_5)\leq
( L([0],[1])+ [0]\times[-1]+ L([-1],[0])\times [2]\otimes\sigma.
$$
We obviously do not have $\delta([0,2])\otimes\s$ on the right-hand side of the above inequality. This implies $\pi_8^t\ne \pi_5$. Therefore,
$$
\pi_8^t=\pi_7,
$$
which implies
$$
\pi_6^t=\pi_5.
$$
The proof of (\ref{prop: 012alpha=2-item: Gg}) and (\ref{prop: 012alpha=2-item: inv}) is now complete.
\end{proof}

Obviously, $\pi_7\leq \delta([0,1])\rtimes\delta([2];\s)$. Using \eqref{JM-seg-02-2}, one directly gets that $\pi_7< \delta([0,1])\rtimes\delta([2];\s)$.
Therefore, $\delta([0,1])\rtimes\delta([2];\s)$ is reducible.

\chapter{Remaining Cases for \texorpdfstring{$\alpha=\frac12$}{alpha12} and \texorpdfstring{$\alpha=1$}{alpha1}} \label{basic 1}

\section{\texorpdfstring{$\mathbf{x}=(0,1,2)$}{x012} and \texorpdfstring{$\alpha=1$}{alpha1}}

\begin{proposition} \label{0,1,2-1}
Assume $\alpha=1$. Then,
\begin{enumerate}
\item In the Grothendieck group we have
\[
\Pi_{(0,1,2)}=\pi_1+\pi_2+\pi_3+\pi_4+\pi_5+\pi_6+\pi_7+\pi_8
\]
where
\begin{gather*}
\pi_1=L([2] ,[1] ;[0] \rtimes\s ),\ \ \pi_2=\tau([0]_+;\delta([1,2] ;\s )),\\
\pi_3=L([2] ,[0,1] ;\s ),\ \ \pi_4=\tau([0]_-;\delta([1,2] ;\s )),\\
\pi_5=L([0,2] ;\s),\ \ \pi_6=L([2] ;\tau([0]_-;\delta([1] ;\s ))),\\
\pi_7=L([1,2] ;[0] \rtimes\s ),\ \ \pi_8=L([2] ;\tau([0]_+;\delta([1] ;\s ))).
\end{gather*}
\item $\pi_1^t=\pi_2$, $\pi_3^t=\pi_4$, $\pi_5^t=\pi_6$, $\pi_7^t=\pi_8$.
\item $\pi_1,\pi_2,\pi_3,\pi_4$ are unitarizable.
\item $\pi_5,\pi_6,\pi_7,\pi_8$ are not unitarizable.
\end{enumerate}

\end{proposition}

\begin{proof}
Write
\[
\Pi_{(0,1,2)}=\Pi_1+\Pi_2+\Pi_3+\Pi_4
\]
where
\[
\Pi_1=[0] \rtimes \delta([1,2] ;\s ),\ \ \Pi_2=[0] \rtimes L([2] ,[1] ;\s )
\]
\[
\Pi_3=[0] \rtimes L([1,2] ;\s ),\ \ \Pi_4=[0] \rtimes L([2] ;\delta([1] ;\s )).
\]
Note that $\Pi_1$ and $\Pi_2$ are unitarizable (and $\Pi_1$ is tempered). We have
\begin{equation} \label{a=1-first}
\Pi_1= \pi_2\oplus\pi_4
\end{equation}
and hence $\pi_2$ and $\pi_4$ are unitarizable.
Moreover, $\Pi_2=\Pi_1^t$ is a sum of two irreducible unitarizable representations, and one of them is $\pi_1$.

Consider the diagram
\[
\xymatrix@C-8pt{
0\ar[r] & [2] \times L([0] ,[1] )\rtimes\s \ar[r] & \Pi_{(2,0,1)}\cong\Pi_{(0,2,1)} \ar@{->>}[d] \ar[r]
& [2] \times \delta([0,1])\rtimes\s \ar@{->>}[d] \ar[r]&0\\
& & \Pi_2 & \pi_3
}
\]
We show that
$$
\Pi_2\not\leq [2] \times L([0] ,[1] )\rtimes\s.
$$
Otherwise, taking $s_{\GL}$ on both side, we would get
$$
2[0] \times L( [-2] ,[-1] )\leq\Big(\cancel{[2]} +[-2]\Big)\times\Big(\cancel{L([0] ,[1])}+[0] \times[-1]+L([0] ,[-1] )\Big).
$$
Crossing out redundant terms on the right-hand side we get
$$
2[0] \times L( [-2] ,[-1] )\leq [-2] \times [0] \times[-1] +[-2] \times L([0] ,[-1] ).
$$
However, the multiplicity of $ L( [-2] ,[-1,0] )$ in the left-hand side is two,
while on the right-hand side it is one (since $ L([-2] , [-1,0] )$ is not a subquotient of
$$
[-2] \times L([0] ,[-1] )=L([-2] ,[0] ,[-1] )+ L([0] ,[-2,-1] ),
$$
and $[-2] \times [0] \times[-1]$ is multiplicity free).

By Lemma \ref{lem: simpac} we infer that $\pi_3\le\Pi_2$.

We conclude that
\begin{equation} \label{a=1-second}
\Pi_2=\pi_1\oplus\pi_3.
\end{equation}
In particular, $\pi_1$ and $\pi_3$ are unitarizable.

From \eqref{dist-formula} we get
$$
\pi_2^t=\pi_1
$$
which further implies (using \eqref{a=1-first} and \eqref{a=1-second})
$$
\pi_4^t=\pi_3.
$$

Consider now $\Pi_3$ and $\Pi_4$.
We know that these two induced representations must contain between themselves $\pi_5,\pi_6,\pi_7,\pi_8$
as subquotients (and possibly others a priori).
Moreover, $\Pi_4$ is reducible by Proposition \ref{prop: addparms} and $\Pi_3^t=\Pi_4$ by Proposition \ref{prop-2-a+}.
Therefore, $\Pi_3$ is also reducible.

We know that $\pi_7$ occurs in $\Pi_3$ with multiplicity one.
Observe that by \eqref{eq: pi312} we get that
\begin{gather*}
s_{\GL}(\Pi_3)=2\cdot [0]\times\delta([-2,-1])\otimes\s+
2\cdot [0]\times[-1]\times[2]\otimes\s\\
=2\cdot \delta([-2,0])\otimes\s+2\cdot L([0],[-2,-1])\otimes\s
\\+2\cdot L([0],[-1])\times[2]\otimes\s +
2\cdot \delta([-1,0])\times[2]\otimes\s.
\end{gather*}
From this it easily follows that $\Pi_3$ has no tempered subquotients.
It follows from \eqref{BPLC} that every subquotient of $\Pi_3$ other than $\pi_7$ is isomorphic to $\pi_5$.
In particular, $\JH(\Pi_3)$ has two elements, and therefore the same is true for $\JH(\Pi_4)$.
Now Proposition \ref{prop: addparms} implies
\begin{equation}\label{eq-2}
\Pi_4=\pi_8+\pi_6.
\end{equation}
Therefore,
\begin{equation} \label{eq-1}
\Pi_3=\pi_7+\pi_5.
\end{equation}
From this we see that $\Pi_{(0,1,2)}$ is a multiplicity free representation of length 8.

Since $L([-2,-1] ,[0]) \otimes\s\le s_{\GL}(\pi_7)$, we have
$L([0,1] ,[2]) \otimes\s\le s_{\GL}(\pi_7^t)$.
Suppose on the contrary that $\pi_7^t=\pi_6$.
Then, we would have by \eqref{eq: 013-}
$$
L([0,1] ,[2]) \otimes\s\le (([2]+[-2])\otimes 1)\rtimes\mu^*(\tau([0]_-;\delta([1] ;\s ))=
([2]+\cancel{[-2]})\times L([0],[1])\otimes\s,
$$
which is impossible. Therefore,
$$
\pi_7^t=\pi_8,
$$
which further implies
$$
\pi_5^t=\pi_6.
$$
In the rest of the section we deal with the non-unitarizability of $\pi_5,\pi_6,\pi_7,\pi_8$.
\end{proof}

\subsection{Non-unitarizability of \texorpdfstring{$\pi_5$}{pi5} and \texorpdfstring{$\pi_6$}{pi6}}

By \cite[Theorem 4.1 (A1)]{MR3360752} we have
\[
\Pi':=\delta([0,2])\rtimes \s= \pi_5+\pi_2.
\]
Write $\delta_1=\delta([-1,1])$ and $\Gamma=\delta_1\rtimes\pi_5$.
Consider
$$
\delta_1\times\Pi'= \delta_1\rtimes\big(\pi_5+\pi_2\big).
$$
Clearly, $L([0,2];\tau([-1,1]_\pm;\s))\le\delta_1\times\Pi'$.
Since $\delta_1$ and $\pi_2$ are tempered and
$$
\delta([-1,2])\times\delta([0,1])\rtimes \s\leq \delta_1\times\Pi'
$$
we get from \eqref{eq: 0111} that the following six non-tempered irreducible representations
\begin{gather*}
L([0,2];\tau([-1,1]_\pm;\s)), \quad L([-1,2];\tau([0]_+ ;\delta([1];\s))),\\
L([0,1];\delta([-1,2]_\pm ;\s)), \quad L([0,1],[-1,2] ;\s)
\end{gather*}
are subquotients of $\Gamma.$

If $\pi_5$ or $\pi_6$ were unitarizable, then by Lemma \ref{lem: nonunit} the multiplicity of $\tau:=\delta_1\otimes \pi_5$ in $\mu^*(\Gamma)$
would be at least six, in contradiction to the following lemma.

\begin{lemma} \label{lem: multgam4}
The multiplicity of $\tau$ in $\mu^*(\Gamma)$ is $4$.
\end{lemma}

\begin{proof}
By remark \ref{rem: mdtp} we need to consider the multiplicity of $\tau$ in
\begin{gather*}
\Big(2\cdot\dashuline{[1]\otimes \delta ([0,1])}+2\cdot\delta([0,1])\otimes[1]+2\cdot\uline{\delta_1\otimes 1}+
1\otimes\delta_1\Big)\rtimes\\
\Big(\uline{1\otimes \pi_5}+[2] \otimes L([0,1] ;\s )+[0] \otimes L([1,2] ;\s)+\\
[0] \times[2] \otimes L([1] ;\s) +\dashuline{\delta([-1,0] )\otimes[2] \rtimes \s}
+\delta([-1,0] )\times[2] \otimes\s +\delta([-2,0])\otimes\s\Big) .
\end{gather*}
We underline the terms which can actually contribute to the multiplicity of $\tau$.
Obviously, $(\delta_1\otimes1)\rtimes(1\otimes\pi_5)=\tau$.
Moreover, the multiplicity of $\tau$ in
$$
[1]\times \delta([-1,0] )\otimes \delta ([0,1]) \rtimes [2] \rtimes \s
$$
is one since the multiplicity of $\delta_1$ in $[1]\times \delta([-1,0])$ is one,
and the multiplicity of $\pi_5$ in $\delta ([0,1]) \rtimes [2]\rtimes\s$ is one
since $\pi_5\le \delta ([0,2])\rtimes\s\le\delta ([0,1]) \rtimes [2]\rtimes\s$
and $\Pi_{(0,1,2)}$ is multiplicity free.

All in all, the multiplicity of $\tau$ in $\mu^*(\Gamma)$ is four.
\end{proof}

\subsection{Non-unitarizability of \texorpdfstring{$\pi_7$}{pi7} and \texorpdfstring{$\pi_8$}{pi8}}

\begin{lemma}
Let $\delta_1=\delta([-1,1])$ as before. Then, the representation
$$
\widehat\Gamma:=\delta_1\rtimes \pi_7
$$
admits (at least) the following irreducible subquotients
\begin{gather*}
\gamma_1^\pm=L([1,2] ;[0]\rtimes \delta([-1,1]_\pm; \s )),\ \
\gamma_2^\pm=L( [1]; [0]\rtimes\delta([-1,2]_\pm;\s)),\\
\gamma_3=L( [1],[-1,2]; [0]\rtimes\s).
\end{gather*}
Hence, the length of $\widehat\Gamma$ is at least $5$.
\end{lemma}

We remark that by \cite[Proposition 2.1]{MR1896238}, $\Jord_\rho(\delta([-1,2]_\pm;\s))=\{1,3,5\}$
and therefore $[0]\rtimes\delta([-1,2]_\pm;\s)$ is irreducible.

\begin{proof}
By Proposition \ref{prop: addparms} $\widehat{\Gamma}$ admits $\gamma_1^\pm$ as irreducible subquotients.

Recall that in the Grothendieck group we have
$$
\Pi_3=\pi_7+ \pi_5.
$$
Since
$$
[0]\times \delta([1,2])\rtimes \s=\Pi_3+\Pi_1,
$$
and $\Pi_1$ is tempered, it follows that any non-tempered irreducible subquotient of
$\delta_1\times [0]\times \delta([1,2])\rtimes \s$ is necessarily
a subquotient of $\widehat\Gamma$ or of $\delta_1 \rtimes \pi_5$.
On the other hand,
$$
\delta([-1,2])\times [0]\times [1]\rtimes \s\leq \delta_1\times [0]\times \delta([1,2])\rtimes \s
$$
and the left-hand side admits $\gamma_2^\pm$ and $\gamma_3$ as (non-tempered) irreducible subquotients.
Hence, it is enough to show that none of $\gamma_2^\pm$ and $\gamma_3$ is a subquotient of $\delta_1\rtimes\pi_5$.

By Frobenius reciprocity,
\[
L( [-1], [0])\times\delta([-1,2])\otimes\s\le s_{\GL}(\gamma_2^\pm)\text{ and }
L( [-1], [0])\times\delta([-2,1])\otimes\s\le s_{\GL}(\gamma_3).
\]
On the other hand, it is easy to see that
$$
L( [-1]+b)\otimes\s\not\leq M^*(\delta_1) \times M^*(\delta([0,2]))\rtimes(1\otimes \s)
$$
for any multisegment $b$. Therefore, $\gamma_2^\pm,\gamma_3\not\le\delta_1\rtimes\pi_5$ as required.
\end{proof}

\begin{remark}
In fact, one can show that $\widehat{\Gamma}$ also contains the irreducible subquotient $L([-1,2];\tau([0]_-; \delta([1]; \s)))$.
(We will not need to use this fact.)
Indeed, as in the proof above, it suffices to show that
$$
L([-1,2];\tau([0]_-; \delta([1]; \s)))\not\leq\delta_1\rtimes \pi_5.
$$
Suppose on the contrary that this is not the case. Then, we would have
$$
\delta([-2,1])\otimes\tau([0]_-; \delta([1]; \s))\leq \mu^*(\delta_1\rtimes \pi_5)=M^*(\delta_1)\rtimes\mu^*(\pi_5).
$$
However, from the formulas for $M^*(\delta_1)$ and $\mu^*(\pi_5)$, the only possibility
to get $\delta([-2,1])\otimes\tau([0]_-; \delta([1]; \s))$ in $\mu^*(\delta_1\rtimes\pi_5)$
would be from the term $\delta([-2,0] )\otimes\s$ of $\mu^*(\pi_5)$
(in order to obtain exponent $-2$) and one of the terms $[1] \otimes \delta([0,1])$ or
$[1] \otimes \delta ([-1,0])$ from $M^*(\delta_1)$.
However, by \eqref{eq: 0111} $\delta([0,1])\rtimes\s$ (and hence also $\delta([-1,0])\rtimes\s$) does not admit
$\tau([0]_-; \delta([1]; \s))$ as a subquotient.
\end{remark}

In order to deduce the non-unitarizability of $\pi_7$ and $\pi_8$ from Lemma \ref{lem: nonunit} it remains to prove the following.

\begin{lemma}
The multiplicity of $\widehat\tau:=\delta_1\otimes\pi_7$ in $\mu^*(\widehat{\Gamma})$ is $\leq 4$.
\end{lemma}

\begin{proof}
We will show that in fact the multiplicity of $\widehat\tau$ in $\mu^*(\delta_1\rtimes\Pi_3)=M^*(\delta_1)\rtimes\mu^*(\Pi_3)$ is $4$.
By remark \ref{rem: mdtp} we need to consider the multiplicity of $\widehat\tau$ in
\begin{multline*}
\Big(2\cdot[1]\otimes \delta ([0,1])+2\cdot\delta([0,1])\otimes[1]+2\cdot\delta_1\otimes 1+
1\otimes\delta_1\Big)\rtimes\mu^*(\Pi_3)=\\
\Big(2\cdot\overline{[1]\otimes \delta ([0,1])}+2\cdot\uwave{\delta([0,1])\otimes[1]}+2\cdot\uline{\delta_1\otimes 1}+
1\otimes\delta_1\Big)
\times\Big(\uwave{\uline{1\otimes [0]}}+2\cdot\dashuline{[0]\otimes1}\Big)
\rtimes
\\
\Big(\uline{1\otimes L([1,2];\s)}+\overline{\uwave{[-1]\otimes[2]\rtimes\s}}
\\
+
[2]\otimes L([1];\s)+\delta([-2,-1])\otimes\s+[-1]\times[2]\otimes\s\Big).
\end{multline*}
We highlighted the terms which may contribute to the multiplicity of $\widehat\tau$.
The solid underlined terms give $2\cdot\delta_1\otimes\Pi_3$ which contains $\widehat\tau$ with multiplicity two.
The wavy underlined terms give
\[
\delta([0,1])\times[-1]\otimes\Pi_{(0,1,2)}
\]
which again contains $\widehat\tau$ with multiplicity two.
The overlined terms give
\[
2\cdot[1]\times[0]\times[-1]\otimes\delta([0,1])\rtimes[2]\times\sigma
\]
which does not contain $\widehat\tau$.
Indeed, $\pi_7$ does not occur in $\delta([0,1]) \times [2]\rtimes\s$
since
$$
\pi_7\leq L([1,2], [0])\rtimes \s\leq L([1], [0])\times [2]\rtimes \s
$$
and $\Pi_{(0,1,2)}$ is multiplicity free.

Our claim follows.
\end{proof}

This concludes the proof of Proposition \ref{0,1,2-1}.

\section{\texorpdfstring{$\mathbf{x}=(0,1,1)$}{x011} and \texorpdfstring{$\alpha=1$}{alpha1}}

\begin{proposition}
Assume $\alpha=1$. Then,
\begin{enumerate}
\item In the Grothendieck group we have
\[
\Pi_{(0,1,1)}=2\pi_1+\pi_2+2\pi_3+2\pi_4^++\pi_4^-+\pi_5^++\pi_5^-
\]
where
\begin{gather*}
\pi_1=L([0,1] ,[1] ;\s ),\ \ \pi_2=L([1] ,[1] ;[0] \rtimes\s ),\ \ \pi_3=L([0,1] ;\delta([1] ;\s )),\\
\pi_4^{\pm}=L([1] ;\tau([0] _\pm;\delta([1] ;\s))),\ \ \pi_5^\pm=\delta([-1,1] _\pm;\s ).
\end{gather*}
\item We have $\pi_1^t=\pi_3$, $\pi_2^t=\pi_5^+$, $(\pi_4^-)^t=\pi_5^-$, $(\pi_4^+)^t=\pi_4^+$.
\item All irreducible subquotients of $\Pi_{(0,1,1)}$ are unitarizable.
\end{enumerate}
\end{proposition}

\begin{proof}
We have
\begin{multline*}
\Pi_{(0,1,1)}=
\\
\delta([-1,1])\rtimes\s+L([1],[0],[-1])\rtimes\s+L([0,1],[-1])\rtimes\s+L([-1,0],[1])\rtimes\s\\
=\delta([-1,1])\rtimes\s+L([1],[0],[-1])\rtimes\s+2L([0,1],[-1])\rtimes\s.
\end{multline*}
Now,
\[
\delta([-1,1])\rtimes\s=\pi_5^++\pi_5^-
\]
We show that
\begin{equation} \label{eq: l111}
L([-1],[0],[1])\rtimes\s=\pi_2+\pi_4^-.
\end{equation}
Clearly $L([-1],[0],[1])\rtimes\s$ contains $\pi_2$ with multiplicity one.

Consider
\[
\xymatrix@C-12pt{
0\ar[r] & L([1] ,[0] )\rtimes L([1] ;\s ) \ar[r] & L([1] ,[0] )\times[-1] \rtimes\s \ar@{->>}[d] \ar[r]
& L([1] ,[0] )\rtimes \delta([1] ;\s ) \ar[r]&0\\
& & L([1],[0] ,[-1] )\rtimes\s
}
\]
We see easily that
$$
L([1] ,[0] ,[-1] )\rtimes\s \not\leq L([1] ,[0] )\rtimes L([1] ;\s )
$$
since $s_{\GL}(L([1] ,[0] ,[-1] )\rtimes\s )$ contains $L([1] ,[0] ,[-1])\otimes\s $
with multiplicity two while $s_{\GL}(L([1] ,[0] )\rtimes L([1] ;\s ))$ contains $L([1] ,[0] ,[-1])\otimes\s $
with multiplicity one.
Note that the cosocle of $L([1] ,[0] )\rtimes \delta([1] ;\s )$ is a
a quotient of the cosocle of $[1]\times [0]\rtimes \delta([1];\s)=
[1]\times\tau([0]_+;\delta([1];s))\oplus[1]\times\tau([0]_-;\delta([1];s)$
which is $\pi_4^+\oplus\pi_4^-$.
By Lemma \ref{lem: simpac}, at least one of $\pi_4^\pm$ occurs in $L([1],[0] ,[-1] )\rtimes\s $.

Suppose on the contrary that
$$
\pi_4^+\leq L([1] ,[0] ,[-1] )\rtimes\s .
$$
Clearly, $s_{\GL}(\pi_4^+)$ contains $[-1] \otimes \tau([0]_+; \delta([1] ;\s ))$.
This and \eqref{01-1-delta+} imply that the Jacquet module $\pi_4^+$ contains
$2\cdot [-1] \otimes\delta([0,1])\otimes\s$, and therefore also
$$
2\cdot [-1] \otimes[1] \otimes[0] \otimes\s.
$$
On the other hand
$$
s_{\GL}(L([1] ,[0] ,[-1] )\rtimes\s)= 2(L([-1] ,[0] ,[1] )\otimes\s +2L([-1] ,[0] )\times [-1] ) \otimes\s .
$$
Obviously, we cannot have $[-1] \otimes[1] \otimes[0] \otimes\s$ in the Jacquet module
of the above representation, and therefore we get a contradiction.

We conclude that
$$
\pi_4^-\leq L([1] ,[0] ,[-1] )\rtimes\s
$$
which further implies \eqref{eq: l111}.

Since $\pi_5^t=\pi_2$ by \eqref{dist-formula} it follows that $(\pi_5^-)^t=\pi_4^-$.

We also conclude that
\[
L([0,1],[-1])\rtimes\s\ge\pi_1+\pi_3+\pi_4^+.
\]
Therefore, each of $\pi_1$, $\pi_3$ and $\pi_4^+$ occurs with multiplicity bigger than one (and even) in $\Pi_{(0,1,1)}$.

In the Grothendieck group we have
\[
\Pi_{(0,1,1)}=\Pi_++\Pi_-+\Pi_+^t+\Pi_-^t
\]
where
\[
\Pi_\pm=[1]\rtimes\tau([0]_\pm;\delta([1];\s)),\ \
\Pi_+^t=[1]\rtimes L([1];[0]\rtimes\s),\ \
\Pi_-^t=[1]\rtimes L([0,1];\s).
\]
We claim that
\[
\Pi_-=\pi_4^-+\pi_3.
\]
Clearly $\pi_4^-$ occurs in $\Pi_-$ with multiplicity one.
By \eqref{BPLC}, the only other possible non-tempered irreducible subquotient of $\Pi_-$ is $\pi_3$.
On the other hand, since $s_{\GL}(\tau([0]_-;\delta([1];\s))$ $=L([1],[0])\otimes\s$, we have
$L([-1,1])\otimes\s\not\le s_{\GL}(\Pi_-)$. Therefore, $\pi_5^\pm\not\le\Pi_-$.
So the only other possible irreducible subquotient of $\Pi_-$ other than $\pi_4^-$ is $\pi_3$.
On the other hand, clearly $\pi_1$ occurs with multiplicity one in $\Pi_-^t$
and we have $\pi_1^t\ne\pi_4^-$. It necessarily follows that
$\pi_1^t=\pi_3$ and $\Pi_-=\pi_4^-+\pi_3$.

It follows that $(\pi_4^+)^t=\pi_4^+$.

We show that
\[
\Pi_+=\pi_4^++\pi_3+\pi_5^+.
\]
Clearly $\Pi_+$ contains $\pi_4^+$ with multiplicity one.
As before, the only other possible non-tempered irreducible subquotient of $\Pi_+$ is $\pi_3$.
The representation $\pi_5^+$ occurs in $\Pi_+$
since $s_{\GL}(\Pi_+)\ge\delta([0,1])\times[1]\otimes\s$
and occurs with multiplicity one in $\Pi_{(0,1,1)}$ (see \S\ref{dist}).

We know that $\pi_3$ occurs with multiplicity at least two in $\Pi_{(0,1,1)}$.
On the other hand, $\pi_3$ occurs with multiplicity one in $\Pi_-$ and it does not occur
in $\Pi_\pm^t$ since $\pi_1=\pi_3^t$ does not occur in $\Pi_\pm$.
Therefore, $\pi_3$ must occur in $\Pi_+$.

It remains to show that $\Pi_+$ contains $\pi_3$ with multiplicity one and does not contain $\pi_5^-$.

We have $s_{\GL}(\Pi_+)=$
\begin{gather*}
\Big(([1]+[-1])\times(2\delta([0,1])+L([0],[1]))\Big)\otimes\s=
\Big(2[1]\times\delta([0,1])+[1]\times L([0],[1])\\+2\delta([-1,1])+2L([-1],[0,1])
+L([-1,0],[1])+L([-1],[0],[1])\Big)\otimes\s.
\end{gather*}
In particular, $\delta([-1,1])\otimes\s$ occurs with multiplicity two and $L([-1,0],[1])$
occurs with multiplicity one.
We claim that $s_{\GL}(\pi_4^+)$ contains $\delta([-1,1])\otimes\s$.
Indeed, since $\pi_4^+\h [-1]\rtimes\tau([0]_+;\delta([1] ;\s))$, we have
$\mu^*(\pi_4^+)\ge [-1]\otimes \tau([0]_+;\delta([1] ;\s))$ and therefore
the Jacquet module of $\pi_4^+$ contains $[-1]\otimes[0]\otimes[1]\otimes\s$.
It follows that $s_{\GL}(\pi_4^+)\ge L([-1],[0],[1])\otimes\s$.
Since $(\pi_4^+)^t=\pi_4$ we infer that $s_{\GL}(\pi_4^+)\ge\delta([-1,1])\otimes\s$ as claimed.

Since $\Pi_+$ contains $\pi_5^+$ and both $s_{\GL}(\pi_5^\pm)$ contains $\delta([-1,1])\otimes\s$,
it follows now that $\pi_5^-$ cannot occur in $\Pi_+$.

$\mu^*(\pi_3)\ge\delta([-1,0])\otimes\delta([1];\s)$ and therefore the Jacquet module of $\pi_3$
contains $[0]\otimes [-1]\otimes [1]\otimes\sigma$. It follows that $s_{\GL}(\pi_3)\ge L([-1,0],[1])\otimes\sigma$.
Since $L([-1,0],[1])\otimes\s$ occurs with multiplicity one in $s_{\GL}(\Pi_+)$
we infer that $\pi_3$ cannot occur in $\Pi_+$ with multiplicity bigger than one.
\end{proof}

\begin{remark} \label{red-for-cs-2}
The fact that $\pi_5^+\leq \delta([0,1])\rtimes \delta([1];\sigma)$ implies that $\delta([0,1])\rtimes \delta([1];\sigma)$ is reducible
(since also $\pi_3\leq \delta([0,1])\rtimes \delta([1];\sigma)$).

Furthermore, $\pi_4^\pm\leq L([0],[1])\rtimes \delta([1];\sigma)$, which implies that $L([0],[1])\rtimes \delta([1];\sigma)$ is reducible,
as well as $\delta([0,1])\rtimes L([1];\sigma)$.
\end{remark}

\section{\texorpdfstring{$\mathbf{x}=(0,0,1)$}{x001} and \texorpdfstring{$\alpha=1$}{alpha1}}

\begin{proposition}
\begin{enumerate}
\item In the Grothendieck group we have
\begin{align*}
\Pi_{(0,0,1)}&=[0] \rtimes L([0,1];\s)+[0]\rtimes L([1] ;[0] \rtimes\s)\\&+
[0] \rtimes\tau([0]_+;\delta([1] ;\s ))+[0] \rtimes\tau([0]_-; \delta([1] ;\s )),
\end{align*}
where the representations above are irreducible.
\item We have
\begin{gather*}
([0] \rtimes\tau([0]_+; \delta([1] ;\s )))^t=[0]\rtimes L([1] ;[0] \rtimes\s),\\
([0] \rtimes L([0,1];\s))^t=[0] \rtimes\tau([0]_-; \delta([1] ;\s )).
\end{gather*}
\item All irreducible subquotients of $\Pi_{(0,0,1)}$ are unitarizable.
\end{enumerate}
\end{proposition}

\begin{proof} Indeed,
\begin{multline*}
\Pi_{(0,0,1)}=[0] \rtimes\Big([0] \rtimes L([1] ;\s)+ [0]\rtimes \delta([1] ;\s)\Big)\\=
[0] \rtimes L([0,1];\s)+[0]\rtimes L([1] ;[0] \rtimes\s)+
\\
 [0] \rtimes\tau([0]_+; \delta([1] ;\s ))+[0] \rtimes\tau([0]_-; \delta([1] ;\s )).
\end{multline*}
The representations $[0] \rtimes\tau([0]_\pm; \delta([1] ;\s ))$ are irreducible by the theory of $R$-groups.
The representations $[0] \rtimes L([0,1];\s)$ and $[0]\rtimes L([1] ;[0] \rtimes\s)$ are irreducible by duality.
\end{proof}

\section{\texorpdfstring{$\mathbf{x}=(\tfrac12,\tfrac12,\tfrac32)$}{x121232} and \texorpdfstring{$\alpha=\tfrac12$}{alpha12}}

\begin{proposition} \label{1/2,1/2,3/2-1/2}
Assume $\alpha=\tfrac12$.
\begin{enumerate}
\item In the Grothendieck group we have
\begin{equation} \label{eq: l12}
\Pi_{(\frac12,\frac12,\frac32)}=\pi_1+\pi_2+2\pi_3+2\pi_4+\pi_5+\pi_6+\pi_7+\pi_8+\pi_9+\pi_{10}
\end{equation}
where
\begin{gather*}
\pi_1=L([\tfrac32],[\tfrac12],[\tfrac12];\s),\ \
\pi_2=\delta([-\tfrac12,\tfrac32]_+;\s),\\
\pi_3=L([-\tfrac12,\tfrac32];\s),\ \
\pi_4=L([\tfrac12,\tfrac32];\delta([\tfrac12];\s)),\\
\pi_5=L([\tfrac32],[\tfrac12];\delta([\tfrac12];\s)),\ \
\pi_6=\delta([-\tfrac12,\tfrac32]_-;\s),\\
\pi_7=L([\tfrac12];\delta([\tfrac12,\tfrac32];\s)),\ \
\pi_8=L([\tfrac32];\delta([-\tfrac12,\tfrac12]_-;\s)),\\
\pi_9=L([\tfrac12,\tfrac32],[\tfrac12];\s),\ \
\pi_{10}=L([\tfrac32];\delta([-\tfrac12,\tfrac12]_+;\s)).
\end{gather*}
\item $\pi_1^t=\pi_2$, $\pi_3^t=\pi_4$, $\pi_5^t=\pi_6$, $\pi_7^t=\pi_8$, $\pi_9^t=\pi_{10}$.
\item The representations $\pi_1,\dots,\pi_8$ are unitarizable.
\item The representations $\pi_9,\pi_{10}$ are not unitarizable.
\end{enumerate}
\end{proposition}

\begin{remark}
In fact, each of $\pi_1,\dots,\pi_8$ is a subquotient of a representation at the end of a complementary series.
(In addition, $\pi_2$ and $\pi_6$ are square-integrable.)
\end{remark}

We proceed in several steps.

\subsection{}
First observe that the representations $\pi_2$ and $\pi_6$ are unitarizable, since they are square-integrable.
Furthermore $\pi_3$ and $\pi_1$ are unitarizable, since they are at the ends of the complementary series (starting with
$\delta ([-1,1])\rtimes \s$ and $L([-1],[0],[1])\rtimes \s$ respectively).
Analogously, $\pi_4$ (and also $\pi_1$)
is unitarizable since it is at the ends of the complementary series starting with $\delta([-\tfrac12,\tfrac12])\rtimes \delta([\tfrac12];\s)$
(and with $L([\tfrac12],[-\tfrac12])\rtimes L([\tfrac12];\s)$).
Furthermore $\pi_7$ is unitarizable (since it is at the end of the complementary series which start with
$[0]\rtimes\delta([\tfrac12,\tfrac32];\s)$).

\subsection{}
Observe that \eqref{dist-formula} implies
$$
\pi_2^t=\pi_1.
$$

We know that $\pi_2$ has multiplicity one in $\Pi_{(\frac12,\frac12,\frac32)}$.
By \eqref{BPLC} we get that
$$
[\tfrac12]\rtimes \delta([\tfrac12,\tfrac32];\s)=\pi_7+\pi_2.
$$
Indeed, $\pi_6$ cannot occurs on the left-hand side because
$\delta([-\frac12,\frac32])\otimes\sigma$ occurs with multiplicity one
in $s_{\GL}([\tfrac12]\rtimes \delta([\tfrac12,\tfrac32];\s))$, $s_{\GL}(\pi_2)$ and $s_{\GL}(\pi_6)$.

This implies
$$
[\tfrac12]\rtimes L([\tfrac12],[\tfrac32];\s)=\pi_7^t+\pi_1.
$$

\subsection{}
We show that $\pi_7^t=\pi_8$.

Consider
\[
\xymatrix@C-10pt{
0\ar[r] & [\tfrac32]\rtimes L([\tfrac12],[\tfrac12];\s) \ar[r] &
[\tfrac32]\times [-\tfrac12]\rtimes L ([\tfrac12] ;\s)
\ar@{->>}[d] \ar[r]
& [\tfrac32]\times\delta([-\tfrac12,\tfrac12]_-;\s) \ar@{->>}[d] \ar[r]&0\\
& & [-\tfrac12]\rtimes L([\tfrac32],[\tfrac12] ;\s) & \pi_8
}
\]
We show that
$$
[-\tfrac12]\rtimes L([\tfrac32],[\tfrac12] ;\s)\not\leq
[\tfrac32]\rtimes L([\tfrac12],[\tfrac12];\s).
$$
Otherwise, by passing to $s_{\GL}$ we would get
$$
([\tfrac12]+[-\tfrac12])\times L([-\tfrac12],[-\tfrac32] )
\leq([\tfrac32]+[-\tfrac32])\times (L([-\tfrac12], [\tfrac12])+[-\tfrac12]\times [-\tfrac12]).
$$
Considering the part supported on $[-\tfrac32], [-\tfrac12],[\tfrac12]$ we get
$$
[\tfrac12]\times L([-\tfrac12],[-\tfrac32] )
\leq
[-\tfrac32]\times L([-\tfrac12], [\tfrac12]),
$$
which is impossible since $ L([-\tfrac32],[-\tfrac12,\tfrac12])$ occurs in the left-hand side but not in the right-hand side.

By Lemma \ref{lem: simpac} we infer that
\begin{equation} \label{eq: pi8in}
\pi_8\leq[-\tfrac12]\rtimes L([\tfrac32],[\tfrac12] ;\s)).
\end{equation}
This implies that $\pi_8$ is unitarizable. Moreover, we get
$$
[\tfrac12]\rtimes L([\tfrac12],[\tfrac32];\s)=\pi_8+\pi_1.
$$
and
$$
\pi_7^t=\pi_8.
$$

\subsection{}
We show that
\begin{equation} \label{eq: pi5leq33}
\pi_5\leq L([\tfrac32],[\tfrac12],[-\tfrac12])\rtimes\s
\end{equation}
and hence that $\pi_5$ is unitarizable, since it a subquotient at the end of a complementary series.

Consider
\[
\xymatrix@C-14pt{
0\ar[r] & L([\tfrac32],[\tfrac12])\rtimes L([\tfrac12];\s) \ar[r] &
L([\tfrac32],[\tfrac12])\times [-\tfrac12]\rtimes\s \ar@{->>}[d] \ar[r]
& L([\tfrac32],[\tfrac12])\rtimes \delta([\tfrac12];\s) \ar@{->>}[d] \ar[r]&0\\
& & L([\tfrac32],[\tfrac12],[-\tfrac12])\rtimes\s & \pi_5
}
\]
We have
$$
L([\tfrac32],[\tfrac12],[-\tfrac12])\rtimes\s\not\leq
L([\tfrac32],[\tfrac12])\rtimes L([\tfrac12];\s),
$$
otherwise, taking $s_{\GL}$ we would get
\[
2\cdot L([-\tfrac32],[-\tfrac12],[\tfrac12])\otimes\s\le
(L([\tfrac32],[\tfrac12])+[-\tfrac32]\times[\tfrac12]+L([-\tfrac32],[-\tfrac12]))
\times[-\tfrac12]\otimes\s
\]
which is impossible.

We deduce \eqref{eq: pi5leq33} from Lemma \ref{lem: simpac}.

To conclude, $\pi_1,\dots,\pi_8$ are unitarizable.

\subsection{}
We show that $\pi_5^t=\pi_6$.

Let
\[
\Pi'=\delta([\tfrac12,\tfrac32])\rtimes L([\tfrac12];\s).
\]
Recall
$$
\delta([-\tfrac12,\tfrac32])\rtimes\s=\pi_2+\pi_6+\pi_3.
$$
From this and the fact that $\pi_1^t=\pi_2$ it follows that
$$
\pi_5^t\in\{\pi_3;\pi_6\}.
$$
Observe that $\pi_5$ is a quotient of $L([\tfrac32],[\tfrac12])\rtimes \delta([\tfrac12];\s)$.
Therefore, $\pi_5^t$ is a subquotient of $\Pi'$.
Note that
\[
s_{\GL}(\Pi')\not\ge \delta([-\tfrac32,\tfrac12])\otimes\s
\]
but
\[
s_{\GL}(\pi_3)\ge \delta([-\tfrac32,\tfrac12])\otimes\s.
\]
Therefore, necessarily $\pi_5^t=\pi_6$ as required.

\subsection{}
Consider
$$
\delta([\tfrac12,\tfrac32])\times[\tfrac12]\rtimes\s=
\Pi'
+
\delta([\tfrac12,\tfrac32])\rtimes\ \delta ([\tfrac12];\s).
$$
Observe that the left-hand side contains $\pi_2$ and $\pi_6$ as subquotients.
Furthermore, since $\pi_2\leq \delta([\tfrac12,\tfrac32])\rtimes \delta([\tfrac12]; \s)$ and the multiplicity of
$\delta([-\tfrac12,\tfrac32]])\otimes\s$ in the Jacquet module of $\delta([\tfrac12,\tfrac32])\rtimes \delta([\tfrac12]; \s)$ is one, we get
$$
\pi_6\leq \Pi'.
$$

\subsubsection{}
We show that
\begin{equation} \label{eq: Pi'dec}
\Pi'=\pi_6+\pi_9
\end{equation}
and compute $\mu^*(\pi_9)$.

Write
\begin{gather*}
\mu^*(\Pi')=
\Big(1\otimes \delta([\tfrac12,\tfrac32])+[-\tfrac12]\otimes [\tfrac32]+[\tfrac32]\otimes [\tfrac12]\\
+\delta([\tfrac12,\tfrac32])\otimes1+[-\tfrac12]\times[\tfrac32]\otimes1 + \delta([-\tfrac32,-\tfrac12])\otimes 1\Big) \rtimes
\Big(1\otimes L([\tfrac12];\s)+[-\tfrac12]\otimes\s\Big)
\end{gather*}
as
\begin{gather*}
1\otimes \Pi'\\
+[-\tfrac12]\otimes [\tfrac32]\rtimes L([\tfrac12];\s)
+\overbrace{[\tfrac32]\otimes L([\tfrac12],[\tfrac12];\s)}^{\omega_2}+
[\tfrac32]\otimes \delta([-\tfrac12,\tfrac12]_-;\s)
\\
+[-\tfrac12]\otimes \delta([\tfrac12,\tfrac32])\rtimes\s\\
+\delta([\tfrac12,\tfrac32])\otimes L([\tfrac12];\s)+
2\cdot\overbrace{[-\tfrac12]\times[\tfrac32]\otimes L([\tfrac12];\s)}^{\omega_2''}+
\delta([-\tfrac32,-\tfrac12])\otimes L([\tfrac12];\s)\\
+\overbrace{[-\tfrac12]\times [-\tfrac12]\otimes [\tfrac32]\rtimes\s}^{\omega_2'}
+[-\tfrac12]\times [\tfrac32]\otimes \delta( [\tfrac12];\s)\\
+\Big(\delta([-\tfrac12,\tfrac32])+
\overbrace{[-\tfrac12]\times [-\tfrac12]\times[\tfrac32]}^{\omega_1}+
\overbrace{ L([-\tfrac12],[\tfrac12,\tfrac32])}^{\omega_3}+
\overbrace{[-\tfrac12]\times \delta([-\tfrac32,-\tfrac12])}^{\omega_3'}\Big)\otimes\s.
\end{gather*}
We consider in $\Pi'$ the irreducible subquotient $\pi$ such that $s_{\GL}(\pi)\ge\omega_1\otimes\sigma$.
Using the transitivity of Jacquet modules one gets that $\mu^*(\pi)\ge\omega_2+\omega_2'+\omega_2''$.
Moreover, $\mu^*(\pi)\ge\omega_2$ implies that $\mu^*(\pi)\ge\omega_3$ and $\mu^*(\pi)\ge\omega_2'$ implies $\mu^*(\pi)\ge\omega_3'\otimes\s$.
Now considering $s_{\GL}(\Pi')$ we infer that $\Pi'$ is a multiplicity one representation of length two.
Therefore,
\begin{equation} \label{nfcs1}
\Pi'=\pi_9+\pi_6.
\end{equation}
Furthermore, a simple analysis using the transitivity of Jacquet modules\footnote{Recall that we know also $\mu^*(\pi_6)$.} gives
\begin{equation}\label{mu-*-1/2-nu}
\begin{gathered}
\mu^*(\pi_9)=1\otimes \pi_9+\\
+[-\tfrac12]\otimes [\tfrac32]\rtimes L([\tfrac12];\s)
+[\tfrac32]\otimes L([\tfrac12],[\tfrac12];\s)
+[-\tfrac12]\otimes \delta([\tfrac12,\tfrac32])\rtimes\s\\
+[-\tfrac12]\times[\tfrac32]\otimes L([\tfrac12];\s)+ \delta([-\tfrac32,-\tfrac12])\otimes L([\tfrac12];\s)\\
+[-\tfrac12]\times [-\tfrac12]\otimes [\tfrac32]\rtimes\s
+[-\tfrac12]\times [\tfrac32]\otimes L( [\tfrac12];\s)+[-\tfrac12]\times [\tfrac32]\otimes \delta( [\tfrac12];\s)\\
+ L([-\tfrac12],[\tfrac12,\tfrac32])\otimes\s+
[-\tfrac12]\times [-\tfrac12]\times[\tfrac32]\otimes\s + [-\tfrac12]\times \delta([-\tfrac32,-\tfrac12])\otimes\s.
\end{gathered}
\end{equation}

\subsection{}
It follows from the formula for $\mu^*(\pi_9)$ that $\pi_9^t$ is not tempered
(look at the term $[-\tfrac12]\times [-\tfrac12]\times[\tfrac32]\otimes\s$
which gives in the Jacquet module of the dual representation $[\tfrac12]\times [\tfrac12]\times[-\tfrac32]\otimes\s$).
Once again, from the formula for $\mu^*(\pi_9)$, we get that the Langlands parameter of $\pi_9^t$ must come from a Jacquet module of
$$
[-\tfrac32]\otimes L([\tfrac12],[\tfrac12];\s)^t=
[-\tfrac32]\otimes \delta([-\tfrac12,\tfrac12]_+;\s).
$$
Since $ \delta([-\tfrac12,\tfrac12]_+;\s)$ is tempered, this implies
$$
\pi_9^t=\pi_{10}.
$$

Now observe that
$$
\pi_3^t\ne \pi_3.
$$
Namely, in the Jacquet module of $\pi_3^t$ is $\delta([-\tfrac12,\tfrac32])^t\otimes\s$.
One directly sees that this term is not in the Jacquet module of $\delta([-\tfrac12,\tfrac32])\rtimes\s.$
This and the formulas for the remaining involutions imply
$$
\pi_3^t=\pi_4.
$$

We show that in the Grothendieck group we have
\[
\Theta:=L([\tfrac32],[-\tfrac12,\tfrac12])\rtimes\s=\pi_8+\pi_{10}+\pi_3.
\]
By Proposition \ref{prop: addparms}, $\pi_8$ and $\pi_{10}$ occur with multiplicity one in $\Theta$.
By \eqref{BPLC} $\pi_1,\pi_4,\pi_5,\pi_9\not\le\Theta$ and also $\pi_5,\pi_8\not\le\Theta^t=L([-\tfrac32],[\tfrac12,\tfrac32])\rtimes\s$
so that $\pi_6,\pi_7\not\le\Theta$.
We have
\[
[\tfrac32]\times\delta([-\tfrac12,\tfrac12])\rtimes\s=\Theta+\delta([-\tfrac12,\tfrac32])\rtimes\s.
\]
Therefore. $\pi_2\not\le\Theta$ since $\pi_2$ occurs in $\delta([-\tfrac12,\tfrac32])\rtimes\s$ and it has multiplicity one in $\Pi_{(\frac12,\frac12,\frac32)}$.
It remains to show that $\pi_3$ occurs in $\Theta$ with multiplicity one.
The representation $\omega:=\delta([-\tfrac32,\tfrac12])\otimes\s$ occurs with multiplicity two
in $s_{\GL}([\tfrac32]\times\delta([-\tfrac12,\tfrac12])\rtimes\s)$ and with multiplicity one in $s_{\GL}(\delta([-\tfrac12,\tfrac32])\rtimes\s)$.
Therefore, $\omega$ occurs with multiplicity one in $s_{\GL}(\Theta)$.
We show that $\omega$ does not occur in $s_{\GL}(\pi_8)$ and $s_{\GL}(\pi_{10})$.
In fact, $\omega$ does not occur in the Jacquet modules of $[\tfrac12]\rtimes L([\tfrac12],[\tfrac32];\s)$ and $L([\frac12],[\tfrac32])\rtimes\delta([\tfrac12];\s)$
(which contain $\pi_8$ and $\pi_{10}$ respectively as subquotients by \eqref{eq: pi8in} and the dual of the relation \eqref{eq: Pi'dec}).
We conclude that $\pi_3$ must occur in $\Theta$ with multiplicity one.

Passing to the dual we get
\[
\Theta^t=\pi_4+\pi_7+\pi_9.
\]
Since in the Grothendieck group we have
\[
\Pi_{(\frac12,\frac12,\frac32)}=\Pi_{(-\frac12,\frac12,\frac32)}=\delta([-\tfrac12,\tfrac32])\rtimes\s+\delta([-\tfrac12,\tfrac32])^t\rtimes\s+\Theta+\Theta^t
\]
and
\[
\delta([-\tfrac12,\tfrac32])\rtimes\s=\pi_2+\pi_6+\pi_3,
\]
we conclude \eqref{eq: l12}.

\subsection{}
Observe that $\pi_2\leq \delta([\tfrac12,\tfrac32])\rtimes \delta([\tfrac12];\s)$.
This implies that $\delta([\tfrac12,\tfrac32])\rtimes \delta([\tfrac12];\s)$ is reducible (we shall use this later).
Actually, we directly get that this is a length three representation.

\subsection{Non-unitarizability of \texorpdfstring{$\pi_9$}{pi9} and \texorpdfstring{$\pi_{10}$}{pi10}}
To finish the proof of Proposition \ref{1/2,1/2,3/2-1/2}, it remains to prove the non-unitarizability of $\pi_9$ and $\pi_{10}$.

Let $\delta_1=\delta([-\tfrac12,\tfrac12])$.

\begin{lemma} \label{lem: lngt612}
Let $\alpha=\tfrac12$. The length of the representation
$$
\Gamma=\delta_1\rtimes \pi_9.
$$
is at least 6.
\end{lemma}

\begin{proof}
First, using Proposition \ref{prop: addparms} $\Gamma$ contains
$$
L([\tfrac12,\tfrac32],[\tfrac12];\delta([-\tfrac12,\tfrac12]_\pm\s))
$$
as irreducible subquotients. Recall
\[
\Pi'=\delta([\tfrac12,\tfrac32])\rtimes L([\tfrac12];\s).
\]
Using \eqref{nfcs1} we get
$$
\delta_1\rtimes \Pi'=\Gamma + \delta_1\rtimes \pi_6.
$$
Recall that $\pi_6$ is square-integrable.
Therefore, any non-tempered irreducible subquotient of $\delta_1\rtimes \Pi'$ must be a subquotient of $\Gamma$.
Observe that
$$
\delta([-\tfrac12,\tfrac32])\times[\tfrac12]\rtimes L([\tfrac12];\s)
\leq\delta_1\rtimes \Pi'
$$
and the left-hand side admits the following non-tempered irreducible subquotients:
$$
L([-\tfrac12,\tfrac32],[\tfrac12],[\tfrac12];\s),\ \
L([-\tfrac12,\tfrac32];\delta([-\tfrac12,\tfrac12]_-;\s)).
$$
Therefore, the length of $\Gamma$ is at least four.

Consider now
\begin{equation} \label{longer1}
\delta_1\times \delta([\tfrac12,\tfrac32])\times[\tfrac12]\rtimes\s
=\delta_1\rtimes \Pi'+\delta_1\rtimes \delta([\tfrac12,\tfrac32])\rtimes \delta([\tfrac12];\s).
\end{equation}
We have
$$
\delta([-\tfrac12,\tfrac32])\times [\tfrac12]\times[\tfrac12]\rtimes\s\leq
\delta_1\times \delta([\tfrac12,\tfrac32])\times[\tfrac12]\rtimes\s.
$$
The left-hand side contains
$$
L( [\tfrac12],[\tfrac12];\delta([-\tfrac12,\tfrac32]_\pm;\s))
$$
as irreducible subquotients. We show that these two representations are subquotients of $\delta_1\rtimes \Pi'$,
and therefore of $\Gamma$.

Suppose on the contrary that
\begin{equation}
\label{cont1}
L( [\tfrac12],[\tfrac12];\delta([-\tfrac12,\tfrac32]_\pm;\s))
\leq
\delta_1\rtimes \delta([\tfrac12,\tfrac32])\rtimes \delta([\tfrac12];\s).
\end{equation}
Observe that
$$
L( [\tfrac12],[\tfrac12];\delta([-\tfrac12,\tfrac32]_\pm;\s))\h
[-\tfrac12]\times[-\tfrac12]\rtimes \delta([-\tfrac12,\tfrac32]_\pm;\s).
$$
Therefore, we have $[-\tfrac12]\times[-\tfrac12]\otimes-$ for a subquotient of the Jacquet module of the left-hand side of \eqref{cont1}.
However, one easily sees that there is no term of the form $[-\tfrac12]\times[-\tfrac12]\otimes-$
in the Jacquet module of the right-hand side of \eqref{cont1}.

We conclude that the length of $\Gamma$ is at least 6, as required.
\end{proof}

\begin{lemma} \label{lem: mult612}
The multiplicity of $\tau:=\delta_1\otimes \pi_9$ in $\mu^*(\Gamma)$ is 6.
\end{lemma}

\begin{proof}
Using Remark \ref{rem: mdtp}
and the formula \eqref{mu-*-1/2-nu} for $\mu^*(\pi_9)$, we need to find the multiplicity of $\tau$ in
$$
2\cdot \delta_1\otimes \pi_9+
2\cdot [\tfrac12]\times[-\tfrac12]\otimes [\tfrac12]\times [\tfrac32]\rtimes L([\tfrac12];\s)+
2\cdot [\tfrac12]\times[-\tfrac12]\otimes [\tfrac12]\times \delta([\tfrac12,\tfrac32])\rtimes\s.
$$
Since $\pi_9$ occurs with multiplicity one in the standard module
$[\tfrac12]\times \delta([\tfrac12,\tfrac32])\rtimes\s$, in order to
show multiplicity 6, it is enough to prove that
$$
\pi_9\not\leq L( [\tfrac12], [\tfrac32])\rtimes L([\tfrac12];\s).
$$
In turn, by passing to $s_{\GL}$, this follows from the fact that
$$
\delta([-\tfrac32,-\tfrac12])\times [-\tfrac12]\not\leq
\big(L( [\tfrac12], [\tfrac32])+[\tfrac12]\times [-\tfrac32]+L( [-\tfrac32], [-\tfrac12])\big) \times [-\tfrac12],
$$
which is easy to verify.
\end{proof}

\subsection{}
Finally, we prove that the representations $\pi_9$ and $\pi_{10}$ are not unitarizable.

Suppose on the contrary that $\pi_9$ or $\pi_{10}$ is unitarizable. Then, by Lemmas \ref{lem: nonunit},
\ref{lem: lngt612} and \ref{lem: mult612}, the Jacquet module of $\Gamma$ admits a direct summand
isomorphic to $6\cdot\tau$.
This would imply that for any subquotient $\Lambda$ of $J(\Gamma)$, the multiplicity of $\tau$ in $\Lambda$
is equal to $\dim\Hom(\Lambda,\tau)$.
On the other hand, we claim that there exists a subquotient $\Lambda$ of $J(\Gamma)$ of the form $[\tfrac12]\times[-\tfrac12]\otimes\pi_9$
(which clearly admits $\tau$ as a subrepresentation, but not as a quotient).

Indeed, it follows from the geometric lemma and the fact that the Jacquet module of $\delta_1$ is $[\tfrac12]\otimes[-\tfrac12]$, that
for any irreducible subquotient $\delta_2\otimes\tau_2$ of $J(\pi_9)$, the representation
\[
[\tfrac12]\times\delta_2\otimes[-\tfrac12]\rtimes\tau_2
\]
is a subquotient of $J(\Gamma)$. In particular we can take $\delta_2=[-\tfrac12]$ and
$\tau_2=\delta([\tfrac12,\tfrac32])\rtimes\s$ to infer that $J(\Gamma)$ admits
\[
[\tfrac12]\times[-\tfrac12]\otimes [-\tfrac12]\times \delta([\tfrac12,\tfrac32])\rtimes\s
\]
as a subquotient, and hence also
\[
[\tfrac12]\times[-\tfrac12]\otimes\pi_9
\]
since $\pi_9$ is a subquotient of $\delta_2\rtimes\tau_2$.

This completes the proof of Proposition \ref{1/2,1/2,3/2-1/2}.

\section{\texorpdfstring{$\mathbf{x}=(\tfrac12,\tfrac12,\tfrac12)$}{x121212} and \texorpdfstring{$\alpha=\tfrac12$}{alpha12}}

\begin{proposition} \label{prop: 121212}
For $\alpha=\frac12$ we have
\begin{enumerate}
\item In the Grothendieck group we have
\[
\Pi_{(\frac12,\frac12,\frac12)}=\pi_1+\pi_2+\pi_3+\pi_4+2\pi_5
\]
where
\begin{gather*}
\pi_1=\delta([-\tfrac12,\tfrac12])\rtimes\delta([\tfrac12];\s),\ \
\pi_2=[\tfrac12]\rtimes\delta([-\tfrac12,\tfrac12]_-;\s),\\
\pi_3=[\tfrac12]\rtimes L([\tfrac12];\delta([\tfrac12];\s)),\ \
\pi_4=L([-\tfrac12],[\tfrac12])\rtimes L([\tfrac12];\s)\\
\pi_5=L([\tfrac12];\delta([-\tfrac12,\tfrac12]_+;\s))
\end{gather*}
are irreducible.
\item $\pi_1,\pi_2,\pi_3,\pi_4,\pi_5$ are unitarizable.
\item $\pi_1^t=\pi_4,\quad \pi_2^t=\pi_3, \quad \pi_5^t =\pi_5$.
\end{enumerate}
\end{proposition}

\begin{proof}
First note that exactly as in the case of \S\ref{sec: 121232-32}
all the irreducible subquotients of $\Pi_{(\frac12,\frac12,\frac12)}$ are unitarizable.
Also note that $\pi_1$ is irreducible since $2\in\Jord_\rho(\delta([\tfrac12];\s))$.
Therefore, the same is true for $\pi_4=\pi_1^t$.
Moreover, by the usual analysis,
\[
\JH(\Pi_{(\frac12,\frac12,\frac12)})=\{\pi_1,
L([\tfrac12];\delta([-\tfrac12,\tfrac12]_-;\s)),
L([\tfrac12],[\tfrac12];\delta([\tfrac12];\s)),\pi_4.\pi_5\}.
\]
In the Grothendieck group we have
\begin{equation} \label{eq: Pidec2121}
\begin{gathered}
\Pi_{(\frac12,\frac12,\frac12)}=
[\tfrac12]\rtimes\delta([-\tfrac12,\tfrac12]_+; \s)
+[\tfrac12]\rtimes\delta([-\tfrac12,\tfrac12]_-; \s)\\
+[\tfrac12]\rtimes L([\tfrac12];\delta([\tfrac12];\s))
+[\tfrac12]\rtimes L([\tfrac12],[\tfrac12];\s).
\end{gathered}
\end{equation}

By Proposition \ref{prop: addparms}
$$
\delta([-\tfrac12,\tfrac12])\rtimes L([\tfrac12];\s)
$$
is reducible. Therefore, the same is true for its dual.

Furthermore $\pi_1\not\leq [\tfrac12]\rtimes\delta([-\tfrac12,\tfrac12]_-;\s)$ since
\[
[\tfrac12]\times[\tfrac12]\times[\tfrac12]\otimes\s\le s_{\GL}(\pi_1)
\]
but
\[
[\tfrac12]\times[\tfrac12]\times[\tfrac12]\otimes\s\not\le s_{\GL}([\tfrac12]\rtimes\delta([-\tfrac12,\tfrac12]_-;\s)).
\]
Therefore, \eqref{BPLC} implies that $\pi_2$ is irreducible.
By Proposition \ref{prop: 1212} $\pi_3=\pi_2^t$ and hence, $\pi_3$ is also irreducible.
It follows that
$$
\pi_5^t=\pi_5.
$$

Let $\Pi'=[\tfrac12]\rtimes\delta([-\tfrac12,\tfrac12]_+; \s)$.
We claim that
\begin{equation} \label{eq: pi'31}
\Pi'=\pi_5+\pi_1.
\end{equation}
Indeed, $[\tfrac12]\times[\tfrac12]\times[\tfrac12]\otimes\s\le s_{\GL}(\Pi')$
and this implies that $\pi_1$ occurs in $\Pi'$, necessarily with multiplicity one, since
it is the only irreducible subquotient $\pi'$ of $\Pi_{(\frac12,\frac12,\frac12)}$ such that
$[\tfrac12]\times[\tfrac12]\times[\tfrac12]\otimes\s\le s_{\GL}(\pi')$
and $[\tfrac12]\times[\tfrac12]\times[\tfrac12]\otimes\s$ occurs in $s_{\GL}(\Pi_{(\frac12,\frac12,\frac12)})$ with multiplicity one.
Also, $\pi_5$ is a the Langlands quotient of $\Pi'$, and hence occurs with multiplicity one.
We conclude \eqref{eq: pi'31} by \eqref{BPLC}.

Therefore, by Proposition \ref{prop: 1212} and duality we get
$$
[\tfrac12]\rtimes L([\tfrac12],[\tfrac12];\s)=\pi_4+\pi_5.
$$
The proof of the proposition is finished by \eqref{eq: Pidec2121}.
\end{proof}

The above proposition directly implies that $\delta([-\tfrac12,\tfrac12])\rtimes L([\tfrac12];\s)$ is reducible.

\chapter{The Case \texorpdfstring{$\alpha=0$}{0}} \label{basic 0} \label{CC-0}

\section{\texorpdfstring{$\mathbf{x}=(0,1,2)$}{x0123} and \texorpdfstring{$\alpha=0$}{alpha0}}

\begin{proposition} \label{0,1,2-0}
Assume $\alpha=0$. Then,
\begin{enumerate}
\item We have in the Grothendieck group
\begin{equation} \label{eq: decom012}
\Pi_{(0,1,2)}=\pi_1^++\pi_1^-+\pi_2^++\pi_2^-+\pi_3^++\pi_3^-+\pi_4^++\pi_4^-+2\pi_5+2\pi_6
\end{equation}
where
\begin{gather*}
\pi_1^\pm=\delta([0,2]_\pm;\s),\ \
\pi_2^\pm=L([2],[1];\delta([0]_\pm;\s))\\
\pi_3^\pm=L([2];\delta([0,1]_\pm ;\s)),\ \
\pi_4^\pm=L([1,2];\delta([0]_\pm ;\s)),\\
\pi_5=L([2],[0,1];\s),\ \
\pi_6=L([0,2];\s).
\end{gather*}
\item $(\pi_1^\pm)^t=\pi_2^\mp$, $(\pi_3^\pm)^t=\pi_4^\mp$, $\pi_5^t=\pi_6$.
\item The representations $\pi_1^\pm,\pi_2^\pm$ are unitarizable.
\item The representations $\pi_3^\pm,\pi_4^\pm,\pi_5,\pi_6$ are not unitarizable.
\end{enumerate}
\end{proposition}

We shall prove the above proposition in several steps.

\subsection{}
Using Proposition \ref{prop: 001}, we write $\Pi_{(0,1,2)}$ (in the Grothendieck group) as
\begin{equation} \label{eq: sum0123}
[2]\rtimes\Big(L([1];\delta([0]_+;\s))+L([1];\delta([0]_-;\s))+2L([0,1];\s)
+\delta([0,1]_+;\s)+\delta([0,1]_-;\s)\Big).
\end{equation}
By (\ref{Pr-red-si-item: a-not-a+2-red}) of  Proposition \ref{Pr-red-si} and (\ref{prop: 001-item: 01ds}) of Proposition \ref{prop: 001}, the representations $[2]\rtimes\delta([0,1]_\pm;\s)$ are reducible.
The same is true for $[2]\rtimes L([1];\delta([0]_\pm;\s))$ by duality.

Note that the Jacquet module of $\Pi_{(0,1,2)}$ contains each of $[2]\otimes[1]\otimes \delta([0]_\pm;\s)$
with multiplicity one. It follows that
\begin{equation} \label{eq: pipmchar}
\begin{gathered}
\pi_1^\pm\text{ is the unique irreducible subquotient of $\Pi_{(0,1,2)}$ whose Jacquet module}\\
\text{contains }[2]\otimes[1]\otimes \delta([0]_\pm;\s)\text{ as a subquotient,}\\
\text{and $\pi_1^\pm$ occurs with multiplicity one in }\Pi_{(0,1,2)}.
\end{gathered}
\end{equation}
Dually, the Jacquet module of $(\pi_1^\mp)^t$ contains $[-2]\otimes[-1]\otimes \delta([0]_{\mp};\s)$ as a subquotient and this property
characterizes it uniquely.
Since $[-2]\otimes[-1]\otimes \delta([0]_{\mp};\s)$ it also a subquotient of the Jacquet module of $\pi_2^\mp$, we get
$$
(\pi_1^\pm)^t=\pi_2^\mp.
$$
Note that $\pi_1^\pm$ are square-integrable, while $\pi_2^\pm$ are unitarizable by \cite{MR2448433}.\footnote{This reference does not cover the case of unitary groups.
These groups are covered by results of C. M\oe glin (see \cite[\S 13]{MR3969882} for more details).
She has shown that \ASS dual of a general irreducible square-integrable representation of a classical group over field of
characteristic zero is unitarizable.}

\subsection{}
We will use the following special case of \eqref{jm-seg-ds}
\begin{equation} \label{jm-seg-0}
\mu^*(\delta([0,d]_\pm;\s))=
\sum_{j=-1}^d \delta([j+1,d])\otimes\delta([0,j]_\pm;\s),
\end{equation}
where by convention $\delta(\emptyset_\pm;\s)=\s$.

We have an epimorphism
$[2]\rtimes \delta([0,1]_\pm ;\s)\tha\pi_3^\pm.$
On the other hand by \eqref{jm-seg-0} and \eqref{eq: pipmchar}
$\pi_1^\pm$ occurs with multiplicity one in $[2]\rtimes \delta([0,1]_\pm ;\s)$.
In fact,
\[
\pi_1^\pm \h [2]\rtimes \delta([0,1]_\pm ;\s).
\]
(This follows from \eqref{jm-seg-0} and Frobenius reciprocity.)

By the description of $\JH(\Pi_{(0,1,2)})$, \eqref{BPLC} and the above discussion we infer that
\begin{equation} \label{with-eps}
[2]\rtimes \delta([0,1]_\pm ;\s)=\pi_3^\pm + \pi_1^\pm.
\end{equation}
(We cannot have $\pi_1^\mp$, since it has multiplicity one in $\Pi_{(0,1,2)}$ and it occurs in $[2]\rtimes \delta([0,1]_{\mp} ;\s)$.)

Applying duality we get
\begin{equation} \label{eq: daul1234}
[2]\rtimes L([1];\delta([0]_{\mp};\s))=(\pi_3^\pm)^t+\pi_2^\mp.
\end{equation}

Consider
\[
\xymatrix@C-12pt{
0\ar[r] & \delta([1,2])\rtimes\delta([0]_\pm;\s) \ar@{->>}[d] \ar[r] & [2]\times[1]\rtimes\delta([0]_\pm;\s) \ar@{->>}[d] \ar[r]
& L([1],[2])\rtimes\delta([0]_\pm ;\s) \ar[r]&0\\
& \pi_4^\pm & [2]\rtimes L([1];\delta([0]_\pm;\s))
}
\]
We have
$$
[2]\rtimes L([1];\delta([0]_\pm ;\s))\not\leq L([1],[2])\rtimes\delta([0]_\pm ;\s)
$$
since otherwise, passing to $s_{\GL}$, we would obtain
$$
([2]+[-2])\times L([-1],[0])\otimes\s\leq ( L([1],[2])+ [1]\times[-2] + L([-2],[-1]) )\times[0]\otimes\s.
$$
However $[2]\times L([-1],[0])\otimes\s$ does not occur on the right-hand side.

By Lemma \ref{lem: simpac} we infer that $\pi_4^\pm$ is a subquotient of
$[2]\rtimes L([1];\delta([0]_\pm ;\s))$.

It follows from \eqref{eq: daul1234} that
$$
(\pi_3^\pm)^t=\pi_4^\mp.
$$

We infer that
$\{\pi_5^t, \pi_6^t\}= \{\pi_5, \pi_6\} $.
Observe that
$\delta([-2,0])\otimes\s\le s_{\GL}(\pi_6)$ but
$\delta([0,2])^t\otimes\s\not\le s_{\GL}(\pi_6)$
(and in fact $\delta([0,2])^t\otimes\s\not\le s_{\GL}(\delta([0,2])\rtimes \s)$).
This implies that $\pi_6^t\ne\pi_6$ and hence,
$$
\pi_5^t= \pi_6.
$$

It remains to show the non-unitarizability of $\pi_3^\pm,\pi_4^\pm,\pi_5,\pi_6$.

\subsection{}
Consider
\[
\Pi'=[2]\rtimes L([0,1];\s).
\]
We will show that
\begin{equation} \label{not-eps}
\Pi'= \pi_5+\pi_6.
\end{equation}

Note that $(\Pi')^t=\Pi'$ and that $\pi_5$ (and hence also $\pi_5=\pi_6^t$) occurs in $\Pi'$ with multiplicity one.

By \eqref{BPLC} the only other possible irreducible subquotients of $\Pi'$ are $\pi_1^\pm$, $\pi_3^\pm$ and $\pi_4^\pm$.
Clearly, $\pi_1^\pm$ cannot be a subquotient since its dual is not a subquotient.
Consider
\begin{gather*}
\mu^*(\Pi')=(1\otimes [2]+[2]\otimes1+[-2]\otimes1)\\
\rtimes\Big( 1\otimes L([0,1];\s)+[0]\otimes[1]\rtimes\s+
\delta([-1,0])\otimes\s+L([0],[1])\otimes\s \Big).
\end{gather*}
In particular,
$$
s_{(n_\rho)}(\Pi')=[2]\otimes L([0,1];\s) + [-2] \otimes L([0,1];\s) + [0]\otimes [2]\times [1]\rtimes\s,
$$
where $n_\rho$ is defined by the requirement that $\rho$ is a representation of
$\GL(n_\rho,F')$.
Hence,
\[
s_{(n_\rho)}(\Pi')\not\ge [-2]\otimes\delta([0,1]_\pm ;\s).
\]
It follows that $\pi_3^\pm\not\le\Pi'$ and hence, $\pi_4^\pm\not\le\Pi'$ by duality.
The relation \eqref{not-eps} follows.

The relation \eqref{eq: decom012} now follows from \eqref{eq: sum0123}, \eqref{with-eps}, \eqref{eq: daul1234} and \eqref{not-eps}.

\subsection{Non-unitarizability of \texorpdfstring{$\pi_6$}{pi6} and \texorpdfstring{$\pi_5$}{pi5}}
We show now that the representations $\pi_6$ and $\pi_5$ are not unitarizable.

Let
\[
\delta_1=\delta ([-1,1])\text{ and }\Gamma=\delta_1\rtimes \pi_6.
\]
By Proposition \ref{prop: addparms} $\Gamma$ contains the two irreducible subquotients
$$
L([0,2];\delta ([-1,1]_\pm;\s)) .
$$

The equality
$
\delta([0,2])\rtimes\s=\pi_6 +\pi_1^++\pi_1^-
$ (see \ref{segment})
implies that if we have an irreducible non-tempered subquotient of
$$
\delta_1\times\delta([0,2])\rtimes\s,
$$
then it is a subquotient of $\Gamma.$ On the other hand,
\begin{gather*}
\delta_1\times \delta([0,2])\rtimes\s\geq \delta ([-1,2])\times \delta([0,1])\rtimes\s\\
=\delta ([-1,2])\rtimes \Big(L([0,1];\s) +\delta([0,1]_+;\s)+\delta([0,1]_-;\s)\Big)\\
=\delta ([0,1])\rtimes \Big(L([-1,2];\s) +\delta([-1,2]_+;\s)+\delta([-1,2]_-;\s)\Big).
\end{gather*}
From this we conclude that $\Gamma$ contains the following non-tempered subquotients:
$$
L([-1,2],[0,1];\s) ,\quad L([-1,2],\delta([0,1]_\pm;\s)),
\quad L([0,1],\delta([-1,2]_\pm;\s)).
$$
Therefore, the length of $\Gamma $ is at least 7.

To deduce the non-unitarizability of $\pi_5$ and $\pi_6$ it suffices to show that the multiplicity
of $\tau:=\delta_1\otimes\pi_6$ in $\mu^*(\Gamma)$ is less than 7. In fact, we show that it is 4.

By remark \ref{rem: mdtp} we need to consider the multiplicity of $\tau$ in
\[
\Big(2\cdot\uline{\delta_1\otimes 1}+2\cdot\delta([0,1])\otimes[1]+2\uwave{\cdot[1]\otimes\delta([0,1])}+1\otimes\delta_1\Big)\rtimes
\mu^*(\pi_6).
\]
Recall:
\begin{equation} \label{M[0,2]psi}
\begin{gathered}
\mu^*(\pi_6)= \uline{1 \otimes\pi_6}+\\
+ [2] \otimes L([0,1];\sigma) + [0] \otimes \delta([1,2])\rtimes\sigma\\
+ [0]\times [2] \otimes [1]\rtimes\sigma +\uwave{\delta([-1,0]) \otimes [2]\rtimes\sigma}\\
+ L([0],[1,2]) \otimes\sigma +\delta([-1,0])\times [2] \otimes\sigma+\delta([-2,0]) \otimes\sigma .
\end{gathered}
\end{equation}
Thus, the only relevant terms are
\[
2\cdot\tau+2\cdot [1]\times\delta([-1,0])\otimes\delta([0,1])\times[2]\rtimes\s.
\]
Decomposing (in the Grothendieck group)
\[
[2]\times\delta([0,1])\rtimes\s=
[2]\rtimes \big(L([0,1];\s) +\delta([0,1]_+;\s)+\delta([0,1]_-;\s)\big)
\]
we note that $\pi_6$ is not a subquotient of $[2]\rtimes \delta([0,1]_\pm;\s)$ (by \eqref{BPLC})
and $\pi_6$ occurs with multiplicity one in $[2]\rtimes L([0,1];\s)=\Pi'$ (by \eqref{not-eps}).

In conclusion, the multiplicity of $\tau$ in $\mu^*(\Gamma)$ is 4, as claimed.

\subsection{Non-unitarizability of \texorpdfstring{$\pi_4^\pm$}{pi4pm}, \texorpdfstring{$\pi_3^\pm$}{pi3pm}}
The remaining, and most subtle, part of the proof of the proposition is the non-unitarizability of
$\pi_4^\pm$, $\pi_3^\pm$.

First, we compute
\begin{equation*}
\begin{gathered}
\mu^*([2]\rtimes \delta([0,1]_\pm ;\s))=
\\(1\otimes [2]+ [2]\otimes1+ [-2]\otimes1) \rtimes
\\\Big(1\otimes \delta([0,1]_\pm ;\s)
+[1]\otimes\delta([0]_\pm ;\s)
+\delta([0,1])\otimes1\Big)
\\=1\otimes [2]\rtimes \delta([0,1]_\pm ;\s)
\\+[2]\otimes\delta([0,1]_\pm ;\s) + [-2]\otimes \delta([0,1]_\pm ;\s)
+[1]\otimes[2]\rtimes\delta([0]_\pm ;\s)
\\+[2]\times[1]\otimes\delta([0]_\pm ;\s)+
[-2] \times [1]\otimes\delta([0]_\pm ;\s)+
\delta([0,1])\otimes[2]\rtimes\s
\\+\delta([0,2])\otimes1+ L([2];[0,1])\otimes1+
[-2]\times\delta([0,1])\otimes1.
\end{gathered}
\end{equation*}
Now the above formula, \eqref{with-eps} and \eqref{jm-seg-0} imply
\begin{equation} \label{jm-eps-1}
\begin{gathered}
\mu^*(\pi_3^\pm)=1\otimes \pi_3^\pm
\\+ [-2]\otimes \delta([0,1]_\pm ;\s)
+[1]\otimes[2]\rtimes\delta([0]_\pm ;\s)
\\+L([1],[2])\otimes\delta([0]_\pm ;\s)+
[-2] \times [1]\otimes\delta([0]_\pm ;\s)+
\delta([0,1])\otimes[2]\rtimes\s
\\+L([2];[0,1])\otimes\s+
[-2]\times\delta([0,1])\otimes\s.
\end{gathered}
\end{equation}
Applying duality to this and changing $\mp$ by $\pm$, we get
\begin{equation} \label{jm-eps-2}
\begin{gathered}
\mu^*(\pi_4^\pm)=\uline{1\otimes \pi_4^\pm}
\\+ [2]\otimes L([1];\delta([0]_{\pm} ;\s))
+\dashuline{[-1]\otimes[2]\rtimes\delta([0]_{\pm} ;\s)}
\\+\delta([-2,-1])\otimes\delta([0]_{\pm} ;\s)+
[2] \times [-1]\otimes\delta([0]_{\pm} ;\s)+
L([-1],[0])\otimes[2]\rtimes\s
\\+L([-2,-1],[0])\otimes\s+[2]\times L([-1],[0])\otimes\s.
\end{gathered}
\end{equation}
(Alternatively, we could have also used the relation
\[
[2]\rtimes L([1];\delta([0]_{\mp};\s))=\pi_4^\mp+\pi_2^\mp
\]
to compute the above Jacquet module formula.)

\subsection{}
We shall now describe the composition series of the representation
$$
\delta([1,2])\rtimes \delta([0]_\pm;\s).
$$
Observe
$$
\pi_1^\pm+\pi_4^\pm\leq
\delta([1,2])\rtimes \delta([0]_\pm;\s).
$$

We prove that
$$
\pi_6 \leq\delta([1,2])\rtimes \delta([0]_\pm;\s)
$$
using Lemma \ref{lem: simpac} applied to the diagram
\[
\xymatrix{
0\ar[r] & \delta([0,2])\rtimes\s \ar@{->>}[d] \ar[r] & \delta([1,2])\times [0]\rtimes\s \ar@{->>}[d] \ar[r]
& L([1,2], [0])\rtimes\s\ar[r]&0\\
& \pi_6 & \delta([1,2])\rtimes \delta([0]_\pm;\s)
}
\]
by noting that
$$
\delta([1,2])\rtimes \delta([0]_\pm;\s)\not\leq L([1,2], [0])\rtimes\s.
$$
For otherwise, passing to $s_{\GL}$ and considering the terms whose exponents are non-negative, we would get
$$
\delta([1,2])\times [0]\otimes\s\leq L([1,2], [0])\otimes\s
$$
which is impossible.

This implies that each of $\delta([1,2])\rtimes \delta([0]_\pm;\s)$ has length $\geq 3$.
The same will be therefore true for the dual representations $L([1],[2])\rtimes \delta([0]_\mp;\s)$.
The obvious decomposition (in the Grothendieck group)
\begin{align*}
\Pi_{(0,1,2)}=&
\delta([1,2])\rtimes \delta([0]_+;\s)+
\delta([1,2])\rtimes \delta([0]_-;\s)+\\
&L([1],[2])\rtimes \delta([0]_-;\s)+
L([1],[2])\rtimes \delta([0]_+;\s),
\end{align*}
and the fact that the (total) length of $\Pi_{(0,1,2)}$ is 12, imply that all representations on the right-hand side have length 3.
Therefore, in particular
\begin{equation}
\label{eps-eq}
\delta([1,2])\rtimes \delta([0]_\pm;\s)= \pi_4^\pm+\pi_6+\pi_1^\pm .
\end{equation}

\subsection{}
Consider
$$
\Gamma_\pm:=\delta_1\rtimes \pi_4^\pm
$$
where as before $\delta_1=\delta ([-1,1])$.
First we shall determine the following multiplicity:

\begin{lemma}
The multiplicity of $\tau_\pm:=\delta_1\otimes \pi_4^\pm$ in $\mu^*(\Gamma_\pm)$ is four.
\end{lemma}

\begin{proof}
By Remark \ref{rem: mdtp} and the formula \eqref{jm-eps-2} for $\mu^*(\pi_4^\pm)$
we need to compute the multiplicity of $\tau_\pm$ in
\begin{gather*}
2\cdot(\delta_1\otimes 1)\rtimes(1\otimes\pi_4^\pm)+
2\cdot(\delta([0,1])\otimes [1])\rtimes([-1]\otimes[2]\rtimes\delta([0]_\pm ;\s))\\=
2\cdot\tau_\pm+2\cdot \delta([0,1])\times[-1]\otimes [1]\times[2]\rtimes\delta([0]_\pm ;\s)).
\end{gather*}
Clearly, $\delta_1$ occurs with multiplicity one in $\delta([0,1])\times[-1]$.
Decomposing
$$
[1]\times [2]\rtimes\delta([0]_\pm ;\s)=\delta([1,2])\rtimes\delta([0]_\pm ;\s)
+L([1], [2])\rtimes\delta([0]_\pm ;\s)
$$
we note that $\pi_4^\pm$ occurs with multiplicity one in $\delta([1,2])\rtimes\delta([0]_\pm ;\s)$.
However, by \eqref{eps-eq} (passing to the dual), $\pi_4^\pm$ does not occur in
$L([1], [2])\rtimes\delta([0]_\pm ;\s)$.

In conclusion, the multiplicity of $\tau_\pm$ in $\mu^*(\Gamma_\pm)$ is 4, as required.
\end{proof}

\subsection{}
Now we shall determine several irreducible subquotients of $\Gamma_\pm$.

To that end we need to introduce some notation.
Let $\theta\in\Tempcl$ and $\D$ a segment of cuspidal representations such that
$\D\check{\ }=\D$ and $\delta(\D)\rtimes\theta$ is reducible.
Then, it decomposes into a sum of two nonequivalent irreducible tempered subrepresentations which will be denoted by
$$
\tau(\D_\pm;\theta).\footnote{One can be more specific in description of these representations, as in \cite{MR3123571},
but we do not need these details for the purpose of this paper.}
$$
First, by Proposition \ref{prop: addparms} we get that the following two irreducible representations
$$
L([1,2];\tau([-1,1]_\mu;\delta([0]_\pm;\s))),\quad \mu\in\{\pm\}
$$
are subquotients of $\Gamma_\pm$.

It will require supplementary work to determine additional irreducible subquotients of $\Gamma_\pm$.
First, we shall list some natural candidates for subquotients.

Summing the identity \eqref{eps-eq} for $\e=1$ and $-1$ we get
\begin{equation} \label{5terms}
\delta([1,2])\times [0]\rtimes\s
=\pi_1^++\pi_1^-+\pi_4^++\pi_4^- +2\cdot \pi_6.
\end{equation}

Observe that
$$
\delta_1\times \delta([1,2])\times [0]\rtimes\s\geq
\delta([-1,2])\times [1]\times [0]\rtimes\s.
$$
We know that the left-hand side of the above inequality has among others
the following non-tempered irreducible subquotients:
$$
L([0,2]; \delta([-1,1]_\pm;\s)),
$$
$$
L([-1,2],[1];\delta([0]_\pm;\s)),
$$
$$
L([-1,2];\delta([0,1]_\pm;\s)),
$$
$$
L([-1,2],[0,1];\s),
$$
$$
L([1];\tau([0]_{\e_1};\delta([-1,2]_{\e_2};\s))), \quad \e_1,\e_2\in\{\pm\}.
$$
For brevity denote
\[
\Theta_{\e_1,\e_2}=\tau([0]_{\e_1};\delta([-1,2]_{\e_2};\s))
\]
(a tempered representation).

Multiplying \eqref{eps-eq} by $\delta_1$ we get
\begin{equation} \label{first-line1}
\begin{gathered}
\delta_1\times
\delta([1,2])\rtimes \delta([0]_\pm;\s)=
 \delta_1\times (\pi_4^\pm+\pi_6+\pi_1^\pm)\\
\geq\delta([-1,2])\times [1]\rtimes \delta([0]_\pm;\s).
\end{gathered}
\end{equation}
Then, the left-hand side of \eqref{first-line1} has among others
the following irreducible subquotients:
$$
L([0,2]; \delta([-1,1]_\mu;\s)), \quad \mu\in\{\pm\},
$$
$$
L([-1,2],[1];\delta([0]_\pm;\s)).
$$

Now we prove
\begin{lemma} We have
$$
L([-1,2], [1];\delta([0]_\pm;\s))\leq\Gamma_\pm.
$$
\end{lemma}

\begin{proof}
We will use Lemma \ref{lem: simpac} for the diagram
\[
\xymatrix@C-10pt{
0\ar[r] & L([-1,1], [1,2])\rtimes\omega_\pm\ar[r] &
\delta_1\times \delta([1,2])\rtimes\omega_\pm\ar@{->>}[d] \ar[r]
&\delta([-1,2])\times [1]\rtimes\omega_\pm \ar@{->>}[d] \ar[r]&0\\
& & \Gamma_\pm & L([-1,2], [1];\omega_\pm)
}
\]
where $\omega_\pm=\delta([0]_\pm;\s)$.
It remains to show that
\begin{equation} \label{ineq-1}
\Gamma_\pm\not\leq L([-1,1], [1,2])\rtimes\delta([0]_\pm;\s).
\end{equation}
Suppose on the contrary that this is not the case. Then, we would get
\[
s_{\GL}(\Gamma_\pm)\leq s_{\GL}(L([-1,1], [1,2])\rtimes\delta([0]_\pm;\s)).
\]
Now, on the one hand, by the formula \eqref{jm-eps-2} we have
$$
s_{\GL}(\Gamma_\pm)\ge
2\cdot \delta([0,1])\times [1] \times [2]\times L([-1],[0])\otimes\s
$$
and hence,
$$
s_{\GL}(\Gamma_\pm)\ge2\cdot \delta([0,1])\times \delta([1,2])\times L([-1],[0])\otimes\s.
$$
In particular,
$$
s_{\GL}(\Gamma_\pm)\ge2\cdot L([-1],[0], [0,1],[1,2])\otimes\s.
$$
On the other hand, by \eqref{lad-GL}, the only terms in $s_{\GL}(L([-1,1], [1,2])\rtimes\delta([0]_\pm;\s))$
that can dominate the above term are
\[
[0] \times \Big(L([-1,1] ,[1,2])+ L([0,1] ,[2])\times [-1] \times[1] + L([0,1] ,[-1])\times L([1],[2])\Big).
\]
Clearly,
\[
L([-1],[0], [0,1],[1,2])\not\le[0] \times L([-1,1] ,[1,2])
\]
(since the right-hand has a segment of length three). Therefore,
\[
2\cdot L([-1],[0], [0,1],[1,2])\leq
[0] \times \big( L([0,1] ,[2])\times [-1] \times[1]
+ L([0,1] ,[-1])\times L([1],[2])\big).
\]
By passing to the Zelevinsky dual
\[
2\cdot Z([-1],[0], [0,1],[1,2])\leq
[0] \times \big( Z([0,1] ,[2])\times [-1] \times[1]
+ Z([0,1] ,[-1])\times Z([1],[2])\big).
\]
The Jacquet module of the left-hand side contains
\[
2\cdot [1]\otimes Z([-1],[0],[0,1],[2]).
\]
On the other hand, the part of the Jacquet module of the right-hand side of the form $[1]\otimes *$ is
\[
[1]\otimes [0]\times Z([0,1],[2])\times [-1]
\]
which contains $[1]\otimes Z([-1],[0],[0,1],[2])$ with multiplicity one.
We get a contradiction.
Therefore, \eqref{ineq-1} holds, and the lemma now follows from Lemma \ref{lem: simpac}.
\end{proof}

At this stage we know that the length of $\Gamma_\pm$ is at least three.

\subsection{}
We exhibit two more irreducible subquotients of $\Gamma_\pm$.
Recall
\begin{equation} \label{4leq}
\sum_{\e_1,\e_2\in \{\pm\}}L([1];\Theta_{\e_1,\e_2})
\leq
\delta_1\times \delta([1,2])\times [0]\rtimes\s.
\end{equation}

Next, we note that none of the irreducible representations
$L([1];\Theta_{\e_1,\e_2})$, $\e_1,\e_2\in \{\pm\}$ is a subquotient of $\delta_1\rtimes \pi_6$.

Indeed, the Jacquet module of $\delta_1\rtimes \pi_6$ does not admit terms of the form $[-1]\otimes\ -$,
for such terms do not show up in neither $M^*(\delta_1)$ nor $\mu^*(\pi_6)$.

It now follows from \eqref{5terms} that
\begin{equation} \label{4leq2}
\sum_{\e_1,\e_2\in \{\pm\}}L([1];\Theta_{\e_1,\e_2})\leq\Gamma_++\Gamma_-.
\end{equation}
We shall now prove that each term on the right-hand side has precisely two terms from the left-hand side as subquotients.

We analyse in the Jacquet module of $\Gamma_\pm$ terms of the form $[-1]\otimes \Theta_{\e_1,\e_2}$.

Since $[-1]\otimes -$ does not occur in $M^*(\delta_1)$, we only need to consider
the multiplicity of
$
[-1]\otimes \Theta_{\e_1,\e_2}
$
in $(1\otimes\delta_1)\rtimes\mu^*(\pi_4^\pm)$, hence (in view of \eqref{jm-eps-2}) in
$
[-1]\otimes \delta_1\times [2]\rtimes\delta([0]_\pm ;\s).
$
Equivalently, we need to consider the multiplicity of
$\Theta_{\e_1,\e_2}$ in the representation $\delta_1\times [2]\rtimes\delta([0]_\pm ;\s)$ .

Observe that the Jacquet module of each four representations $\Theta_{\e_1,\e_2}$
contains $[0]\times\delta([-1,2])\otimes\s$ as a subquotient (since
$\Theta_{\e_1,\e_2}\h[0]\times \delta([-1,2])\rtimes\s$).

We claim that the multiplicity of $[0]\times\delta([-1,2])\otimes\s$ in
$s_{\GL}(\delta_1\times [2]\rtimes\delta([0]_\pm ;\s))$ is two.
(It can only come from the part $2\cdot \delta_1\times [2]\times [0]\otimes \s$.)
Therefore, at most two of $L([1];\Theta_{\e_1,\e_2})$ can show up in each of $\Gamma_\pm$.
Since they show up four times in $\Gamma_++\Gamma_-$, this implies that they show up twice in each of $\Gamma_\pm$.

In conclusion, we have proved that the length of $\Gamma_\pm$ is at least five.

By Lemma \ref{lem: nonunit} we deduce that the representations $\pi_4^\pm$ and $\pi_3^\pm$ are not unitarizable.

\section{\texorpdfstring{$\mathbf{x}=(0,1,1)$}{x011} and \texorpdfstring{$\alpha=0$}{alpha0}}

\begin{proposition} \label{0,1,1-0}
Assume $\alpha=0$. Then,
\begin{enumerate}
\item In the Grothendieck group we have
\begin{equation} \label{eq: 12321}
\Pi_{(0,1,1)}=\pi_1^++\pi_1^-+\pi_2^++\pi_2^-+\pi_3^++\pi_3^-+2\pi_4
\end{equation}
where
\begin{gather*}
\pi_1^\pm=L([1],[1], \delta([0]_\pm;\s)),\ \
\pi_2^\pm=\delta([-1,1]_\pm;\s),\\
\pi_3^\pm=L([1];\delta([0,1]_\pm;\s)),\ \ \pi_4=L([1],[0,1];\s).
\end{gather*}
\item $\pi_4=[1]\rtimes L([0,1];\s)$.
\item $(\pi_1^\pm)^t=\pi_2^\mp$, $(\pi_3^+)^t=\pi_3^-$, $\pi_4^t=\pi_4$.
\item The representations $\pi_1^\pm,\pi_2^\pm,\pi_3^\pm$ are unitarizable.
\item The representation $\pi_4$ is not unitarizable.
\end{enumerate}
\end{proposition}

\begin{proof}
Write $\Pi_{(0,1,1)}$ in the Grothendieck group as
\begin{equation} \label{eq: 01221}
[1]\rtimes\Big(L([1];\delta([0]_+;\s))+L([1];\delta([0]_-;\s))+
2L([0,1];\s)+\delta([0,1]_+;\s)+\delta([0,1]_-;\s)\Big).
\end{equation}

The representations $[1]\rtimes\delta([0,1]_\pm;\s)$ are reducible by (\ref{Pr-red-si-item: a-a+1-epsilon}) of Proposition \ref{Pr-red-si} and (\ref{prop: 001-item: 01ds}) of Proposition \ref{prop: 001} (the partially defined function attached to $\delta([0,1]_+;\s)$ takes the same values on $\Jord_\rho(\delta([0,1]_\pm;\s))=\{1,3\}$, and the same holds for the partially defined function  attached to $\delta([0,1]_-;\s)$).
Hence, by duality, $[1]\rtimes L([1];\delta([0]_\pm;\s))$ are also reducible.
These four representations are in the ends of the complementary series. Therefore, all the subquotients there are unitarizable.
Now Proposition \ref{prop: addparms} gives the following irreducible (unitarizable) subquotients
$$
\pi_1^\pm, \qquad \pi_3^\pm.
$$
Furthermore,
$$
\delta([-1,1])\rtimes\s=\pi_2^+\oplus\pi_2^-,\ \ \delta([-1,1])^t\rtimes\s=\pi_1^+\oplus \pi_1^-.
$$

This completes the unitarizability assertions.

From \eqref{jm-seg-ds} we know that $[1]\times [1]\otimes \delta([0]_\e;\s)$
is (a direct summand) in the Jacquet module of $\delta([-1,1]_\e;\s)$ (and this term in the Jacquet module characterises $\delta([-1,1]_\e;\s)$).
This implies that $[-1]\times [-1]\otimes \delta([0]_{-\e};\s)$ is in the Jacquet module of $\delta([-1,1]_\e;\s)^t$. From this follows
$$
(\pi_2^\pm)^t=\pi_1^\mp.
$$
This implies that the multiplicity of each $\delta([-1,1]_\e;\s)$ and $\delta([-1,1]_\e;\s)^t$ in $\Pi_{(0,1,1)}$ is one.
Furthermore, \eqref{jm-seg-ds} implies that $[1]\otimes \delta([0,1]_\e;\s)$ is a subquotient of the Jacquet module of $\delta([-1,1]_\e;\s)$.

Provisionally, set
\[
\Pi'=[1]\rtimes L([0,1];\s).
\]
(We will show that in fact $\Pi'$ is irreducible, i.e., $\Pi'=\pi_4$.) Note that $(\Pi')^t=\Pi'$.
We have
\begin{gather*}
\mu^*(\Pi')=
(1\otimes [1]+ [1]\otimes1+ [-1]\otimes1) \rtimes\\
\Big(1\otimes L([0,1];\s)+
[0]\otimes[1]\rtimes\s+
\delta([-1,0])\otimes\s+L([0],[1])\otimes\s\Big).
\end{gather*}
Now
$$
s_{(n_\rho)}(\Pi')=[1]\otimes L([0,1];\s) +[-1]\otimes L([0,1];\s) + [0]\otimes [1] \times[1]\rtimes\s
$$
(recall that $n_\rho$ is defined by requirement that $\rho$ is a representation of $\GL(n_\rho,F')$).
Obviously, neither of $[1]\otimes\delta([0,1]_\pm;\s)$ is a subquotient of $\Pi'$.
This implies that neither of $\delta([-1,1]_\pm;\s)$ is a subquotient of $\Pi'$.

Furthermore, neither of $[-1]\otimes\delta([0,1]_\pm;\s)$ is a subquotient of the Jacquet module of $\Pi'$.
This implies that neither of $\pi_3^\pm$ is a subquotient of $\Pi'$.
These two observations and \eqref{BPLC} imply that $\Pi'$ is irreducible (hence, equals to $\pi_4$).
Now we can conclude that $(\pi_3^+)^t$ is either $\pi_3^+$ or $\pi_3^-$. The formula \eqref{jm-seg-0} implies that in fact
$(\pi_3^+)^t=\pi_3^-$.

Next we show that
\begin{equation} \label{eq: 123321}
[1]\rtimes\delta([0,1]_\pm;\s)=\pi_2^\pm+\pi_3^\pm.
\end{equation}
Indeed, $\pi_3^\pm$ is the Langlands quotient of $[1]\rtimes\delta([0,1]_\pm;\s)$. By
\eqref{BPLC} all other irreducible subquotient are tempered.
Now $\delta([0,1])\times[1]\otimes\sigma$ occurs with multiplicity two in $s_{\GL}(\pi_2^+)$, $s_{\GL}([1]\rtimes\delta([0,1]_\pm;\s))$ and $s_{\GL}(\Pi_{(0,1,1)})$.
Therefore, $\pi_2^+$ occurs with multiplicity one in both $[1]\rtimes\delta([0,1]_+;\s)$ and $\Pi_{(0,1,1)}$.
Since both $[1]\rtimes\delta([0,1]_\pm;\s)$ are reducible, we conclude \eqref{eq: 123321}.

By passing to the dual, we conclude \eqref{eq: 12321} from \eqref{eq: 01221}.

Finally, the non-unitarizability of $\pi_4=\Pi'$ follows by deforming it to the representation
$[2]\rtimes L( [0,1];\s)$, which by Proposition \ref{0,1,2-0} contains the non-unitarizable irreducible representation
$L([2], [0,1];\s)$ as a subquotient.
\end{proof}

\section{\texorpdfstring{$\mathbf{x}=(0,0,1)$}{x001} and \texorpdfstring{$\alpha=0$}{alpha0}}

\begin{proposition}
\label{prop: 001alpha=0}
Assume $\alpha=0$. Then,
\begin{enumerate}

\item
\label{prop: 001alpha=0-item: Gg}
 In the Grothendieck group we have
$$
\Pi_{(0,0,1)}=\pi_1^+ + \pi_1^- + 2\pi_2^+ +2 \pi_2^- + \pi_3^+ + \pi_3^- ,
$$
where
\begin{gather*}
\pi_1^\pm= L([1],[0]\rtimes\delta([0]_\pm;\s)),\quad
\pi_2^\pm= L([0,1];\delta([0]_\pm;\s)),\quad
\pi_3^\pm= [0]\rtimes \delta([0,1]_\pm;\s).
\end{gather*}

\item 
\label{prop: 001alpha=0-item: irr}
$\pi_3^\pm$ are irreducible and $\pi_1^\pm= [0]\rtimes L([1]; \delta([0]_\pm;\s))$.
\item
\label{prop: 001alpha=0-item: inv}
 $(\pi_1^\pm)^t=\pi_3^\mp, \ (\pi_2^\pm)^t=\pi_2^\mp$.
\item 
\label{prop: 001alpha=0-item: all unit}
All irreducible subquotients of $\Pi_{(0,0,1)}$ are unitarizable.
\end{enumerate}
\end{proposition}

\begin{proof}
Write $\Pi_{(0,0,1)}$ as
\[
[0]\rtimes
\Big(L([1];\delta([0]_+;\s))+L([1];\delta([0]_-;\s))+2L([0,1];\s)
+\delta([0,1]_+;\s)+\delta([0,1]_-;\s)\Big).
\]
The representations $\pi_3^\pm$ are irreducible
(since $\Jord_\rho(\delta([0,1]_\pm;\s))=\{1,3\}$).
The same is true for $[0]\rtimes L([1];\delta([0]_+;\s))$ by duality.

It remains to consider $\Pi':=[0]\rtimes L([0,1];\s)$. Observe that
$$
s_{\GL}(\Pi')=
2\cdot[0]\times\big(\delta([-1,0])+L([0],[1])\big)\otimes\s,
$$
which is a length four representation.
By Proposition \ref{prop: addparms}, both representations $\pi_2^\pm$
are irreducible subquotients of $\Pi'$. Observe that
\begin{gather*}
\pi_2^\pm
\h\delta([-1,0])\rtimes \delta([0]_\pm;\s)
\h\delta([-1,0])\times [0]\rtimes\s\\
\cong
[0]\times\delta([-1,0])\rtimes\s\h
[0]\times[0]\times [-1]\rtimes\s\cong\Pi_{(0,0,1)}.
\end{gather*}
This implies that each of $s_{\GL}(L([0,1];\delta([0]_\pm;\s)))$ has length at least two. Therefore, we have
$$
\Pi'=
\pi_2^++\pi_2^-
$$
in the Grothendieck group. In this way we have proved (\ref{prop: 001alpha=0-item: Gg}) and (\ref{prop: 001alpha=0-item: irr}).

Furthermore, (\ref{prop: 001alpha=0-item: irr}) and corank two case imply $(\pi_1^\pm)^t=\pi_3^\mp$.
Now we shall prove $(\pi_2^\pm)^t=\pi_2^\mp$. For this,
write
\begin{equation}
\label{sum001-0}
\Pi_{(0,0,1)}=\Pi_1^+ + \Pi_1^- + \Pi_2^+ + \Pi_2^-,
\end{equation}
where
$$
\Pi_1^\pm :=\delta([0,1])\rtimes \delta([0]_\pm;\sigma),\quad \Pi_2^\pm :=L([0],[1])\rtimes \delta([0]_\pm;\sigma)=(\Pi_1^\mp)^t.
$$
Obviously $\pi_1^\pm\leq\Pi_2^\pm$. This implies by duality $\pi_3^\pm\leq\Pi_1^\pm$. Obviously $\pi_2^\pm\leq\Pi_1^\pm$. Therefore,
$\pi_2^\pm+\pi_3^\pm\leq\Pi_1^\pm$. Now \eqref{sum001-0} and the fact that (total) length of $\Pi_{(0,0,1)}$ is eight, implies
\begin{equation}
\pi_2^\pm+\pi_3^\pm=\Pi_1^\pm,
\end{equation}
 and after applying duality
 $$
(\pi_2^\pm)^t+\pi_1^\pm=\Pi_2^\pm.
$$

We have
$$
[0]\times[1]\rtimes\delta([0]_\pm;\sigma)=\Pi_2^\pm+\Pi_1^\pm.
$$
Now \eqref{01-0-half} and what we proved about duality up to now, imply
$$
[0]\rtimes\Big(L([1];\delta ([0]_\pm;\s))
+ L([0,1];\s)+\delta([0,1]_\pm;\s)\Big)=(\pi_2^\pm)^t+\pi_1^\pm+\pi_2^\pm+\pi_3^\pm,
$$
i.e.
$$
\pi_1^\pm
+ [0]\rtimes L([0,1];\s)+\pi_3^\pm=(\pi_2^\pm)^t+\pi_1^\pm+\pi_2^\pm+\pi_3^\pm.
$$
Since
$$
\pi_2^++\pi_2^-\leq [0]\rtimes L([0,1];\s),
$$
the above equation implies that the above inequality is actually equality. Therefore,
$$
\cancel{\pi_1^\pm}
+ \pi_2^++\pi_2^-+\cancel{\pi_3^\pm}=(\pi_2^\pm)^t+\cancel{\pi_1^\pm}+\pi_2^\pm+\cancel{\pi_3^\pm}.
$$
This implies
$$
(\pi_2^\pm)^t=\pi_2^\mp.
$$
The proof of (\ref{prop: 001alpha=0-item: inv}) is now complete.

Complementary series in rank 2, and then unitary parabolic induction imply (\ref{prop: 001alpha=0-item: all unit}).
\end{proof}

\section{\texorpdfstring{$\mathbf{x}=(0,0,0)$}{x000} and \texorpdfstring{$\alpha=0$}{alpha0}}

The tempered unitarizable representation $\Pi_{(0,0,0)}$ is of length two and splits as
$$
[0]\times[0]\rtimes\delta ([0]_+;\s)\oplus [0]\times[0]\rtimes\delta ([0]_-;\s).
$$
We have $([0]\times[0]\rtimes\delta ([0]_+;\s))^t=[0]\times[0]\rtimes\delta ([0]_-;\s)$.

\chapter{Introductory Remarks on Unitarizability and Corank 2} \label{rem-on-uni}

In this and the following chapter we fix $\rho\in\Cuspsd$ and $\sigma\in\Cuspcl$.
As usual, let $\alpha\in\tfrac12\Z_{\ge0}$ be such that
$$
[\alpha]^{(\rho)}\rtimes\s
$$
is reducible. We will suppress $\rho$ from the notation.

We shall determine unitarizability of irreducible subquotients of
$$
\Pi_{(x_1,\dots,x_k)}=[x_1]\times\dots \times [x_k]\rtimes\s
$$
when $k\leq 3$ and
$$
0\leq x_1\leq\dots\leq x_k.
$$
\begin{remark} \label{non-uni-half}
Suppose that we are in the critical case.
\begin{enumerate}
\item A consequence of the analysis of chapters \ref{CC-corank12}--\ref{CC-0} is that the unitarizability of
irreducible subquotients of $\Pi_{\mathbf{x}}$, $k\leq3$ is preserved under \ASS duality
\item
Suppose further that for some indices $i,j$ (resp., index $i$), $[x_i]\times[x_j]$ (resp. $[x_i]\rtimes \sigma$) is reducible,
necessarily of length two. Then, we get correspondingly a decomposition in the Grothendieck group
$$
\Pi_{\mathbf{x}}=\Pi'+\Pi''
$$
with $\Pi'^t=\Pi''$.
The first part implies that if one of $\Pi'$ and $\Pi''$ contains a non-unitarizable irreducible subquotient, then so does the other.
We shall use this simple observation in the sequel.
\end{enumerate}
\end{remark}

Let $\pi$ be an irreducible subquotient of $\Pi_{(x_1,\dots,x_k)}$.

Trivially, if $k=0$, then $\pi=\s$, which is obviously unitarizable.

\section{Corank 1} \label{rank1}

Let $k=1$.
Then, $\pi$ is unitarizable if and only if
$$
x_1\leq \alpha.
$$
For $0\leq x_1 <\alpha$, $\pi=[x_1]\rtimes \s$, while for $x_1=\alpha$
we have two non-equivalent irreducible (unitarizable) subquotients. In the case $\alpha>0,$ they are $\delta([\alpha]; \s)$
and $ L([\alpha]; \s)$, while for $\alpha=0$, they are $\delta ([0]_+;\s)$ and $ \delta ([0]_-;\s)$.
These two representations are the only subquotients which are unitarizable for $\alpha=0$.

\section{Corank 2} \label{red2-rank2}

The following proposition is probably well known to experts.\footnote{Classifications of the unitary duals of split rank two
classical groups in the non-archimedean case were obtained in \cite{MR0447491} (unramified case),
\cite{MR1212952}, \cite{MR2247868}, \cite{MR2346481}, \cite{MR2557193} among others}
For convenience we provide the proof, as we will use the argument repeatedly in the sequel.

\begin{proposition} \label{PROP: RK2}
The irreducible unitarizable subquotients of $\Pi_{\mathbf{x}}$ when $\mathbf{x}=(x_1,x_2)\in\R^2_{++}$, are the following.
\begin{enumerate}
\item $(\alpha>1)$ 
All irreducible subquotients when $x_1+1\leq x_2\leq \alpha$.
\item $(\alpha\ne\frac12)$
\label{PROP: RK2-item: alpha ne 1/2}
 All irreducible subquotients when $x_1+x_2\leq 1$.
\item $(\alpha=\tfrac12)$ All irreducible subquotients when $x_2\le\tfrac12$.
\item $(\alpha>0)$ The representations $\delta([\alpha,\alpha+1];\s)$ and $L([\alpha],[\alpha+1];\s)$.
\end{enumerate}
\end{proposition}
Before proving the proposition, we make some basic remarks.

\section{General principles related to graphic interpretations (cf. \S\ref{ways of getting unitary})} \label{gen-principles-unit}

It is advantageous to use drawings that describe regions of unitarizability for $\Pi_{(x,y)}$.
(It is enough to consider the quadrant
$
\R^2_+:=\{(x,y)\in\R^2:x,y\geq 0\}.
$
In fact the picture is also symmetric with respect to interchanging $x$ and $y$.)
The lines of reducibility (singular lines) will be denoted by dashed or solid lines.
The complement of the singular lines is partitioned into connected components. Each connected component consists of irreducible
representations $\Pi_{(x,y)}$ which are either all unitarizable (complementary series) or all non-unitarizable.
The former are denoted by shaded regions and the latter by blank regions.
The boundary of the shaded regions are denoted by thick solid lines -- all the irreducible subquotients there
are unitarizable.
Intersection of singular lines are critical points.
These vertices are denoted by balls (white, gray or black), according to unitarizability of
irreducible subquotients (which is known).
The singular line segments (or rays) delimited by critical points (connected components of lines) correspond to two continuous
families of hermitian representations which are either all unitarizable or all non-unitarizable.
(For example, in Figure \ref{fig: rk2alpha3}, to the open segment connecting $(0,1)$ and $(\alpha,\alpha-1)$
corresponds the families $\delta([x+1,x])\rtimes\sigma,\ 0<x<\alpha-1$ and $L([x+1],[x])\rtimes\sigma,\ 0<x<\alpha-1$.)

Unbounded regions and unbounded segments always consist entirely of non-unitarizable representations.

In the rest of this chapter we prove Proposition \ref{PROP: RK2}.

\section{Proof of Proposition \texorpdfstring{\ref{PROP: RK2}}{} for \texorpdfstring{$\alpha\geq1$}{alphage1}}
We prove Proposition \ref{PROP: RK2} in the case $\alpha\geq1$ using Figure \ref{fig: rk2alpha3}.

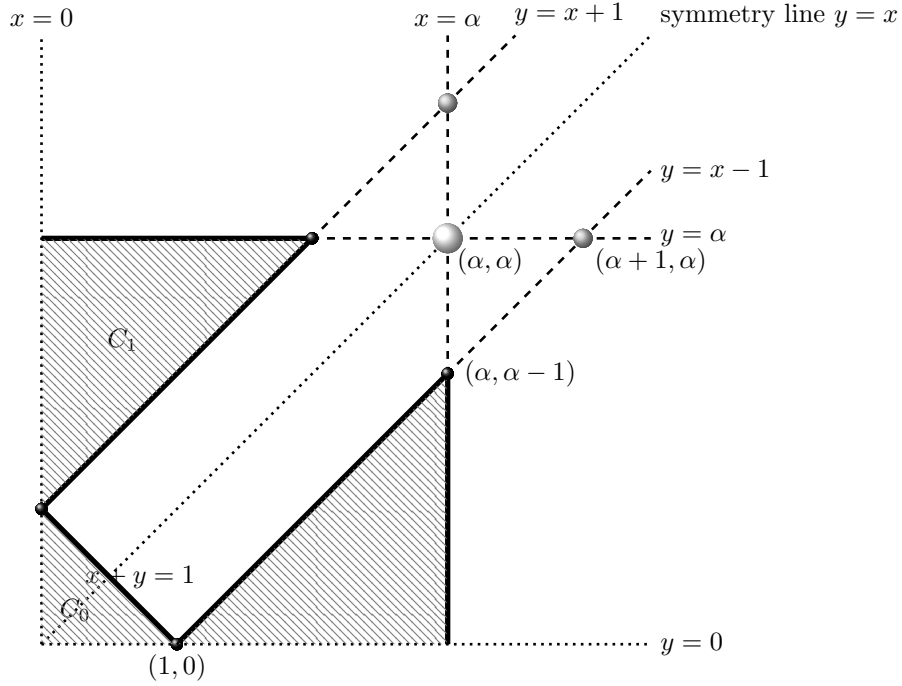
\begin{figure}
\begin{tikzpicture} [thick, scale=0.9]
\draw[style=dotted,line width=1pt] (1,1) -- (10,1);
\draw[style=dotted,line width=1pt] (1,1) -- (1,10);
\draw[style=dotted,line width=1pt] (1,1) -- (10,10);
\draw[style=dashed,line width=1pt] (1,3) -- (8,10);
\draw[style=dashed,line width=1pt] (3,1) -- (10,8);
\draw[style=dashed,line width=1pt] (7,1) -- (7,10);
\draw[style=dashed,line width=1pt] (1,7) -- (10,7);
\draw (10,1) node[right] {$y=0$};
\draw (1,10) node[above] {$x=0$};
\draw (10,10.3) node[right] {symmetry line $y=x$};
\draw (10,8) node[right] {$y=x-1$};
\draw (10,7) node[right] {$y=\alpha$};
\draw (7,10) node[above] {$x=\alpha$};
\draw (8.8,10) node[above] {$y=x+1$};
\draw (7.5,7) node[below] {$\ \ (\alpha,\alpha)$};
\draw (10,7) node[below] {$(\alpha+1,\alpha)$};
\draw (3,1) node[below] {$(1,0)$};
\draw (7.1,5) node[right] {$(\alpha,\alpha-1)$};
\draw (1.5,2) node[right] {$x+y=1$};
\draw (1.5,1.5) node {$C_0$};
\draw (2.2,5.5) node {$C_1$};
\draw[line width=2pt] (3,1) -- (1,3);
\draw[line width=2pt] (3,1) -- (7,5);
\draw[line width=2pt] (7,1) -- (7,5);
\draw[line width=2pt] (1,3) -- (5,7);
\draw[line width=2pt] (1,7) -- (5,7);
\shade[shading=ball,ball color=black] (3,1) circle (.09);
\fill [pattern=north west lines, pattern color=gray] (1,1) -- (3,1) -- (1,3);
\draw [pattern=north west lines, pattern color=gray] (3,1) -- (7,5) -- (7,1);
\draw [pattern=north west lines, pattern color=gray] (1,3) -- (5,7) -- (1,7);
\shade[shading=ball,ball color=black] (7,5) circle (.09);
\shade[shading=ball,ball color=black] (5,7) circle (.09);
\shade[shading=ball,ball color=black] (1,3) circle (.09);
\shade[shading=ball,ball color=lightgray] (7,9) circle (.14);
\shade[shading=ball,ball color=lightgray] (9,7) circle (.14);
\shade[shading=ball,ball color=white] (7,7) circle (.22);
 \end{tikzpicture}
\caption{Unitarizability for $\Pi_{(x,y)}$ (case $\alpha=3$)} \label{fig: rk2alpha3}
\end{figure}

\subsection{Legend (for Figure \ref{fig: rk2alpha3} and all subsequent drawings)} \label{legend}
~\\

\begin{tikzpicture}
\draw[style=dotted,line width=1pt] (1,1) -- (2,1);
\end{tikzpicture}
\quad coordinate axes, or symmetry axis;

\begin{tikzpicture}
\shade[shading=ball,ball color=white] (7,7) circle (.22);
\end{tikzpicture}
\quad all irreducible subquotients are non-unitarizable;

\begin{tikzpicture}
\shade[shading=ball,ball color=lightgray] (9,7) circle (.12);
\end{tikzpicture}
\quad unitarizable and non-unitarizable irreducible subquotients show up;

\begin{tikzpicture}
\shade[shading=ball,ball color=black] (1,1) circle (.09);
\end{tikzpicture}
\quad all irreducible subquotients are unitarizable;

\begin{tikzpicture}
\draw[style=dashed,line width=1pt] (1,1) -- (2,1);
\end{tikzpicture}
\quad only hermitian families of non-unitarizable representations show up;

\begin{tikzpicture}
\draw[line width=1pt] (1,1) -- (2,1);
\end{tikzpicture}
\quad both complementary series and hermitian family of non-unitarizable representations show up;

\begin{tikzpicture}
\draw[line width=2pt] (1,1) -- (2,1);
\end{tikzpicture}
\quad all irreducible subquotients are unitarizable (i.e. all belong to complementary series);

\begin{tikzpicture}
\draw [pattern=north west lines, pattern color=gray] (1,1) -- (2,1) -- (2,2) -- (1,2) -- (1,1);
\end{tikzpicture}
\quad two-dimensional complementary series;

\begin{tikzpicture}
\draw[line width=.2pt] (1,1) -- (2,1);
\draw[line width=.2pt] (2,1) -- (2,2);
\draw[line width=.2pt] (1,2) -- (2,2);
\draw[line width=.2pt] (1,1) -- (1,2);
\end{tikzpicture}
\quad two-dimensional regions of non-unitarizable representations.

\smallskip
\noindent
Two-dimensional connected regions containing the origin will be denoted always by $C_0$ (they will always consist of unitarizable representations).

\subsection{}

The region $C_0$ is unitarizable since $\Pi_{(0,0)}$ is unitarizable.
Furthermore, if $\alpha>1$, the region $C_1$ is unitarizable since
$\Pi_{(x,0)}=[0]\rtimes ([x]\rtimes\s)$ is unitarizable for any $0\le x<\alpha$.

Because of the point $(\alpha,\alpha)$ where all irreducible subquotients
are non-uni\-tarizable (by Proposition \ref{prop: alphalpha}), it remains to explain why the hermitian families corresponding to the
segment from $(\alpha,\alpha-1)$ to $(\alpha+1,\alpha)$ are not unitarizable.
At the end of these families are the representations $\delta([\alpha,\alpha+1])\rtimes\sigma$ and $L([\alpha],[\alpha+1])\rtimes\sigma$.
Since both representations at the end of two families contain a non-unitarizable irreducible subquotient by \S\ref{Steinberg-place-corank-two},
this implies non-unitarizability of both families, and completes the proof of Proposition \ref{PROP: RK2} in this case.

\section{Proof of Proposition \texorpdfstring{\ref{PROP: RK2}}{} for \texorpdfstring{$\alpha=\frac12$}{alpha12}}
Next we prove the proposition in the case $\alpha=\frac12$ -- see Figure \ref{fig: rk2alpha12}.

\begin{figure}
\begin{tikzpicture} [thick, scale=0.8]
\draw[style=dotted,line width=1pt] (1,1) -- (8,1);
\draw[style=dotted,line width=1pt] (1,1) -- (1,8);
\draw[style=dotted,line width=1pt] (1,1) -- (8,8);
\draw[style=dashed,line width=1pt] (1,5) -- (5,1);
\draw[style=dashed,line width=1pt] (1,5) -- (4,8);
\draw[style=dashed,line width=1pt] (5,1) -- (8,4);
\draw[style=dashed,line width=1pt] (3,1) -- (3,8);
\draw[style=dashed,line width=1pt] (1,3) -- (8,3);
\draw (8,1) node[right] {$y=0$};
\draw (1,8) node[above] {$x=0$};
\draw (9,8) node[above] {symmetry line $y=x$};
\draw (3.3,2) node[right] {$x+y=1$};
\draw (8,3) node[right] {$y=\frac12$};
\draw (8,4) node[right] {$y=x-1$};
\draw (3,8) node[above] {$x=\frac12$};
\draw (5,8) node[above] {$y=x+1$};
\draw (4.4,3) node[above] {$(\frac12,\frac12)$};
\draw (7.6,3) node[below] {$(\frac32,\frac12)$};
\draw (2,2) node {$C_0$};
\fill [pattern=north west lines, pattern color=gray] (1,1) -- (1,3) -- (3,3) -- (3,1);
\draw[line width=2pt] (3,1) -- (3,3);
\draw[line width=2pt] (1,3) -- (3,3);
\shade[shading=ball,ball color=black] (3,3) circle (.09);
\shade[shading=ball,ball color=lightgray] (7,3) circle (.14);
\shade[shading=ball,ball color=lightgray] (3,7) circle (.14);
\end{tikzpicture}
\caption{Unitarizability for $\Pi_{(x,y)}$ (case $\alpha=\frac12$)} \label{fig: rk2alpha12}
\end{figure}

As before, the region $C_0$ is unitarizable, since $\Pi_{(0,0)}$ is unitarizable.

For non-unitarizability, we need to consider only two segments.
The first one is the vertical line from $(\frac12,\frac12)$ to $(\frac32,\frac12)$.
At the end of the two families corresponding to this segment are the representations $[\frac32]\rtimes \delta([\frac12];\sigma)$ and
$[\frac32]\rtimes L([\frac12];\sigma)$.
Since both of these representations contain non-unitarizable irreducible subquotient(s) by \S\ref{Steinberg-place-corank-two},
we conclude that both hermitian families are non-unitarizable.
Now consider the segment from $(\frac12,\frac12)$ to $(1,0)$, and from $(1,0)$ to $(\frac32,\frac12)$.
Since at $(1,0)$ the ends of the two families of representations are irreducible, these two families
extend to $(\frac32,\frac12)$ (as families of hermitian irreducible representations).
They end with $\delta([\frac12,\frac32])\rtimes\sigma$ and $L([\frac12],[\frac32])\rtimes\sigma$.
Since both representations contain non-unitarizable irreducible subquotients (again by \ref{Steinberg-place-corank-two}),
this implies the non-unitarizability of both families, and completes the proof of (\ref{PROP: RK2-item: alpha ne 1/2}).

\section{Proof of Proposition \texorpdfstring{\ref{PROP: RK2}}{} for \texorpdfstring{$\alpha=0$}{alpha0}}

Finally, we consider the case $\alpha=0$ using Figure \ref{fig: r2alpha0}.

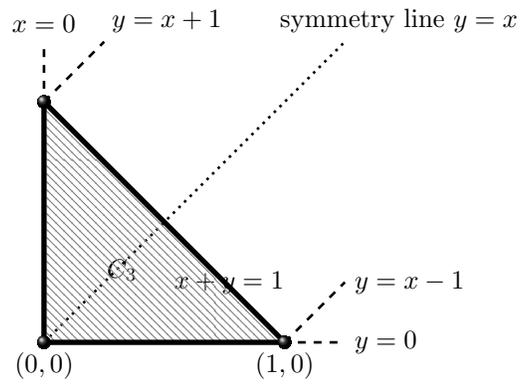
\begin{figure}
\begin{tikzpicture} [thick, scale=0.8]
\draw[style=dashed,line width=1pt] (1,1) -- (6,1);
\draw[style=dashed,line width=1pt] (1,1) -- (1,6);
\draw[style=dotted,line width=1pt] (1,1) -- (6,6);
\draw[style=dashed,line width=1pt] (1,5) -- (5,1);
\draw[style=dashed,line width=1pt] (1,5) -- (2,6);
\draw[style=dashed,line width=1pt] (5,1) -- (6,2);
\draw (6,1) node[right] {$y=0$};
\draw (1,6) node[above] {$x=0$};
\draw (6.9 ,6) node[above] {symmetry line $y=x$};
\draw (6,2) node[right] {$y=x-1$};
\draw (2.6,6) node[above] {$\ \ \ \ \ \ y=x+1$};
\draw (1,1) node[below] {$(0,0)$};
\draw (5,1) node[below] {$(1,0)$};
\draw (3,2) node[right] {$x+y=1$};
\draw (2.3,2.2) node {$C_3$};
\draw [pattern=north west lines, pattern color=gray] (1,1) -- (5,1) -- (1,5);
\draw[line width=2pt] (5,1) -- (1,5);
\draw[line width=2pt] (1,1) -- (1,5);
\draw[line width=2pt] (1,1) -- (5,1);
\shade[shading=ball,ball color=black] (5,1) circle (.11);
\shade[shading=ball,ball color=black] (1,1) circle (.11);
\shade[shading=ball,ball color=black] (1,5) circle (.11);
\end{tikzpicture}
\caption{Unitarizability for $\Pi_{(x,y)}$ (case $\alpha=0$)} \label{fig: r2alpha0}
\end{figure}

The region $C_3$
 is unitarizable, since $\Pi_{(x,x)}\cong ([x]\times[-x])\rtimes\s$ is unitarizable
for all $0\le x<\frac12$.

The non-unitarizability of the remaining part is obvious (all connected components are unbounded).

\chapter{Unitarizability in Corank 3} \label{corank-3}

\section{One-parameter complementary series}
\begin{proposition} \label{prop: unitlines}
In Table \ref{tab: redpnt} we list the one-parameter families $\Pi_x$, $x\in\R$ emanating from unitarizable irreducible representations
of a maximal Levi subgroup with critical parameters.
For each pair of a representation and its dual we record the irreducibility points for $x\ge0$.
When irreducible, $\Pi_x$, $x\ge0$ is unitarizable up to the first point of reducibility, and non-unitarizable otherwise.
When reducibility occurs at $0$, there is no complementary series.

\begin{table}
\[
\begin{array}{|@{}c@{}|l@{}|l@{}|l@{}|@{}c@{}|}
\hline
\text{\small{\rm N}}^{\circ} & \pi_x&\pi_x^t&\mathrm{reducibility\, points\,}&\mathrm{cases}
\\
\hline
1. & \nu^x\delta([-1,1])\rtimes\s & \nu^x L([-1],[0],[1])\rtimes\s &\abs{\alpha-1},\alpha,\alpha+1 & \mathrm{all}
\\
\hline
2. & \nu^x\delta([-\tfrac12,\tfrac12])\rtimes\delta([\alpha];\s)&
\nu^x L([-\tfrac12],[\tfrac12]) \rtimes L([\alpha];\s)& \abs{\alpha\pm\tfrac32}, \alpha+\tfrac12 & \alpha\ne0
\\
3. & \nu^x \delta([-\tfrac12,\tfrac12])\rtimes L([\alpha];\s)&\nu^x L([-\tfrac12],[\tfrac12])\rtimes\delta([\alpha];\s)&
\abs{\alpha\pm\tfrac32},\alpha-\tfrac12&\alpha\ne0
\\
4. & \nu^x\delta([-\tfrac12,\tfrac12])\rtimes\delta([0]_\pm;\s)&
\nu^x L([-\tfrac12],[\tfrac12])\rtimes\delta([0]_\mp;\s)
\,
&\tfrac12,\tfrac32&\alpha=0
\\
\hline
5. & {[x]}\rtimes \delta([\alpha,\alpha+1];\s) & [x]\rtimes L([\alpha],[\alpha+1];\s) & \abs{\alpha-1},\alpha+2 & \alpha\ne0
\\
6. & {[x]}\rtimes \delta([0,1]_\pm;\s) & [x]\rtimes L([1];\delta([0]_\mp;\s))&1,2&\alpha=0
\\
7. & {[x]}\rtimes L([0,1];\s) & [x]\rtimes L([0,1];\s) & 0,2&\alpha=0
\\
\hline
8. & {[x]}\rtimes\delta_{\spsi}([\alpha-1],[\alpha];\s)&[x]\rtimes L([\alpha-1,\alpha];\s)&
\abs{\alpha-2},\alpha+1&\alpha>1
\\
9. & {[x]}\rtimes L([\alpha-1];\delta([\alpha];\s))
\,
&
{[x]}\rtimes L([\alpha-1],[\alpha];\s)&\abs{\alpha-2},\alpha,\alpha+1&\alpha>1
\\
10. & {[x]}\rtimes\tau([0]_+;\delta([1];\s)) &[x]\rtimes L([1];[0]\rtimes\s)&1,2&\alpha=1
\\
11. & {[x]}\rtimes\tau([0]_-;\delta([1];\s)) &[x]\rtimes L([0,1];\s)&1,2&\alpha=1
\\
12. & {[x]}\rtimes\delta([-\tfrac12,\tfrac12]_+;\s) &
{[x]}\rtimes L([\tfrac12],[\tfrac12];\s)&\tfrac12,\tfrac32&\alpha=\tfrac12
\\
13. & {[x]}\rtimes\delta([-\tfrac12,\tfrac12]_-;\s)&
{[x]}\rtimes L([\tfrac12];\delta([\tfrac12];\s))&\tfrac32&\alpha=\tfrac12
\\
\hline
14. & {[x]}\rtimes [0]\rtimes\delta([0]_+;\sigma)&{[x]}\rtimes [0]\rtimes\delta([0]_-;\sigma)&1&\alpha=0
\\
\hline
\end{array}
\]
\caption{Irreducibility points for corank one complementary series} \label{tab: redpnt}
\end{table}
\end{proposition}

\begin{proof}
The reducibility part follows from the previous chapters.
The unitarity statement is clear if there is a unique reducibility point for $x\ge0$.
Otherwise, since there are at most 3 of them, it is enough to exhibit an irreducible non-unitarizable subquotient
at the second reducibility point (from $0$ upwards).
In Table \ref{tab: 2nd} we list for all the families where there is more than one reducibility point for $x\ge0$,
the representation $\Pi_{x_0}$ at the second reducibility point $x_0>0$ as well as
an irreducible non-unitarizable subquotient $\pi$. (By duality, it is enough to consider the second column in Table \ref{tab: redpnt}.)
\end{proof}

\begin{table}
\[
\begin{array}{|l|l|l|l|}
\hline
\Pi_{x_0}&\pi&x_0&\mathrm{cases}\\
\hline
\delta([\alpha-1,\alpha+1])\rtimes\s & L([\alpha-1,\alpha+1];\s) & \alpha & \alpha\ge1\\
\delta([\tfrac12,\tfrac52])\rtimes\s & L([\tfrac12,\tfrac52];\s) & \tfrac32 & \alpha=\tfrac12\\
\delta([0,2])\rtimes\s & L([0,2];\s) & 1 & \alpha=0\\
\hline
\delta([\alpha,\alpha+1])\rtimes\delta([\alpha];\s)&
L([\alpha,\alpha+1];\delta([\alpha];\s))& \alpha+\tfrac12 & \alpha\ge1\\
\delta([\tfrac32,\tfrac52])\rtimes\delta([\tfrac12];\s)&
L([\tfrac32,\tfrac52];\delta([\tfrac12];\s))& 2 & \alpha=\tfrac12\\
\delta([\alpha-1,\alpha])\rtimes L([\alpha];\s)&
L([\alpha-1,\alpha],[\alpha];\s)&\alpha-\tfrac12 &\alpha>1
\\\delta([2,3])\rtimes L([1];\s)&L([2,3],[1];\s)&\tfrac52 &\alpha=1
\\\delta([\tfrac12,\tfrac32])\rtimes L([\tfrac12];\s)&
L([\tfrac12,\tfrac32],[\tfrac12];\s)&1 &\alpha=\tfrac12\\
\delta([1,2])\rtimes\delta([0]_\pm;\s)&
L([1,2];\delta([0]_\pm;\s))&\tfrac32&\alpha=0\\
\hline
{[\alpha+2]}\rtimes \delta([\alpha,\alpha+1];\s) & L([\alpha+2],\delta([\alpha,\alpha+1];\s)) & \alpha+2 & \alpha\ne0\\
{[2]}\rtimes \delta([0,1]_\pm;\s) & L([2];\delta([0,1]_\pm;\s) &2&\alpha=0\\
{[2]}\rtimes L([0,1];\s) & L([2],[0,1];\s) & 2&\alpha=0\\
\hline
{[\alpha+1]}\rtimes\delta_{\spsi}([\alpha-1],[\alpha];\s)&
L([\alpha+1];\delta_{\spsi}([\alpha-1],[\alpha];\s))&\alpha+1&\alpha>1\\
{[\alpha]}\rtimes L([\alpha-1];\delta([\alpha];\s))&
L([\alpha-1,\alpha];\delta([\alpha];\s))&\alpha &\alpha>1\\
{[2]}\rtimes\tau([0]_\pm;\delta([1];\s)) & L([2];\tau([0]_\pm;\delta([1];\s))) &2&\alpha=1\\
{[\tfrac32]}\rtimes\delta([-\tfrac12,\tfrac12]_+;\s) &
L([\tfrac32];\delta([-\tfrac12,\tfrac12]_+;\s)) & \tfrac32&\alpha=\tfrac12\\
\hline
\end{array}
\]
\caption{Non-unitarizable irreducible subquotients at the second reducibility point $x_0$} \label{tab: 2nd}
\end{table}

\smallskip

Next, we consider unitarizability of families of hermitian irreducible representations induced from irreducible,
non-unitarizable representations of maximal Levi subgroups.

\begin{lemma} ~ \label{lemma-nu-m}
\begin{enumerate}
\item
 \label{lemma-nu-m-item: siegel}
 Let
$\pi_x=L([x-1,x],[x+1])\rtimes\s$ (resp. $\pi_x=L([x-1],[x,x+1])\rtimes\s$), $x\geq0$.
Then, $\pi_x$ is reducible if and only if $x\in\{|\alpha-1|,\alpha,\alpha+1\}$. If $\pi_x$ is irreducible, then it is unitarizable if and only if $0\leq x<\alpha -1$.
\item 
\label{lemma-nu-m-item: non-siegel}
For $\alpha>0$ denote
$\pi_x=[x]\rtimes L([\alpha,\alpha+1];\s)$ (resp. $\pi_x=[x]\rtimes L([\alpha+1];\delta([\alpha];\s)),$ $x\geq0.$
Then, $\pi_x$ is reducible if and only if $x\in\{|\alpha-1|,\alpha,\alpha+2\}$. If $\pi_x$ is irreducible, then it is always non-unitarizable.
\end{enumerate}
\end{lemma}

\begin{proof}
In (\ref{lemma-nu-m-item: siegel}), the reducibility of $\pi_x$ is determined by criterion \eqref{eq: ladcrit}.
The non-unitarizabi\-lity of the representations $L([0,1],[1];\sigma)$ for $\alpha=0$, $L([\tfrac12,\tfrac32],[\tfrac12];\s)$
(resp. $L([\tfrac32];\delta([-\tfrac12,\tfrac12]_+;\s))$) for $\alpha=\frac12$, $L([2],\tau([0]_-;\delta([1];\sigma)))$
(resp. $L([1,2];[0]\rtimes \sigma)$) for $\alpha=1$ and $L([\alpha-1],[\alpha,\alpha+1];\s)$
(resp. $L([\alpha-1],[\alpha+1];\delta([\alpha];\s))$) for $\alpha>1$, imply the non-unitarizability claimed in (\ref{lemma-nu-m-item: siegel}).
Unitarizability in (\ref{lemma-nu-m-item: siegel}) follows in a simple way.

The non-unitarizability of the representations $L([|\alpha-1|],[\alpha,\alpha+1];\sigma)$ and its $\ASS$ dual for $\alpha\ne1$,
$L([1,2];[0]\rtimes \sigma)$ and its $\ASS$ dual for $\alpha=0$, and any irreducible subquotient of $\Pi_{(\alpha,\alpha,\alpha+1)}$ for $\alpha\geq1$,
implies the non-unitarizability claimed in (\ref{lemma-nu-m-item: non-siegel}).
\end{proof}

\section{Regular components, unitarizability}
For convenience, we say that a point $\mathbf{x}\in\R^3$ is \emph{strongly unitary} (resp., \emph{strongly non-unitary})
if all irreducible subquotients of $\Pi_{\mathbf{x}}$ are unitarizable (resp., non-unitarizable).
This property depends only on the $W$-orbit of $\mathbf{x}$ where $W$ is the group of signed permutations.
The set of strongly unitary points is closed in $\R^3$.

Consider the singular affine hyperplanes
\[
x_i=\pm\alpha,\ \ i=1,2,3\ \ \ \ x_i\pm x_j=\e,\ \ 1\le i<j\le 3,\ \e=\pm1.
\]
We say that $\mathbf{x}\in\R^3$ is regular if it is not on any one of the singular affine hyperplanes.
We denote by $\R^3_{\reg}$ the set of regular points.
Thus $\mathbf{x}\in\R^3_{\reg}$ if and only if $\Pi_{\mathbf{x}}$ is irreducible, in which case
$\mathbf{x}$ is either strongly unitary or strongly non-unitary.
The set of strongly unitary points in $\R^3_{\reg}$ is a (possibly empty) union of connected components of $\R^3_{\reg}$
which will be called unitary.
Clearly, the set of unitary connected components is $W$-invariant.
Recall that we defined
$
\R^3_{++}=\{(x_1,x_2,x_3):0\le x_1\le x_2\le x_3\}.
$
Denote
$$
\R^3_{\reg,++}=\R^3_{\reg}\cap\R^3_{++}.
$$
The (unitary) connected components of $\R^3_{\reg,++}$ are in one-to-one correspondence with the $W$-orbits
of (unitary) connected components of $\R^3_{\reg}$.
Therefore, it is enough to consider connected components of $\R^3_{\reg,++}$.

\begin{proposition} \label{prop: compserex}
The following connected components of $\R^3_{\reg,++}$ are unitary.

For $\alpha\ge1:$
\begin{subequations}
\begin{gather}
\label{eq: xx2x31} x_2+x_3<1,\\
\label{eq: xx1x212} x_1+x_2<1,\, x_3-x_2>1,\, x_3<\alpha,\ \ \ (\alpha>1)\\
\label{eq: xx1x213} x_1+x_2<1,\, x_1+x_3>1,\, x_3-x_1<1,\, x_3<\alpha\\
\label{eq: xx1x2x14} x_2-x_1>1,\, x_3-x_2>1,\, x_3<\alpha,\ \ \ (\alpha>2).
\end{gather}
\end{subequations}
(The constraint $x_3<\alpha$ in \eqref{eq: xx1x213} is redundant unless $\alpha=1$.)

For $\alpha=\tfrac12:$
\begin{equation} \label{eq: x312}
x_3<\tfrac12.
\end{equation}

Consequently, any $\mathbf{x}\in\R^3_{++}$ in the closure of the above regions (i.e., changing strict inequalities
to non-strict ones) is strongly unitary.
\end{proposition}

\begin{proof}
First note that for $\alpha\ge1$, the non-empty regions among \eqref{eq: xx2x31}--\eqref{eq: xx1x2x14}
are distinct connected components of $\R^3_{\reg,++}$.

Suppose that $\alpha\ne0$. Then, $\Pi_{\mathbf{0}}$ is unitarizable, and thus, so is its connected component in
$\R^3_{\reg,++}$ which is given by \eqref{eq: xx2x31} if $\alpha\ge1$ and by \eqref{eq: x312} if $\alpha=\tfrac12$.

Assume now that $\alpha\ge1$.

We start with the complementary series $\pi_{x_3}=[x_3]\rtimes\s$ which is irreducible and unitarizable for $0\le x_3<\alpha$.
\begin{enumerate}
\item Consider the complementary series $[x_1]\times [-x_1]$, $0\le x_1<\tfrac12$ for the general linear group.
By parabolic induction we get a unitarizable representation $[x_1]\times [-x_1]\rtimes\pi_{x_3}$,
Taking $\tfrac12<x_3<\min(\tfrac32,\alpha)$ and $\abs{1-x_3}<x_1<\tfrac12$ we get an irreducible unitarizable
representation $\Pi_{(x_1,-x_1,x_3)}\cong\Pi_{(x_1,x_1,x_3)}$. The connected component of
\[
\{(x_1,x_1,x_3):\tfrac12<x_3<\min(\tfrac32,\alpha),\ \abs{1-x_3}<x_1<\tfrac12\}
\]
in $\R^3_{\reg,+}$ is \eqref{eq: xx1x213}.

\item Assume $\alpha>1$ and fix $1<x_3<\alpha$. Clearly $(0,0,x_3)\in\R^3_{\reg,++}$ is (strongly) unitary.
The region \eqref{eq: xx1x212} is the connected component of $(0,0,x_3)$.

\item Assume that $\alpha>2$.
Fixing $2<x_3<\alpha$ we construct a complementary series $[x_2]\rtimes\pi_{x_3}$ for $0\le x_2<x_3-1$.
Fixing $1<x_2<x_3-1$ we then construct the complementary series
$[x_1]\times [x_2]\rtimes\pi_{x_3}=\Pi_{(x_1,x_2,x_3)}$, $0\le x_1<x_2-1$.
Thus, the region \eqref{eq: xx1x2x14} is unitary.
\end{enumerate}
\end{proof}

Eventually, we show that the converse to Proposition \ref{prop: compserex} holds.

We first consider the singular affine hyperplanes.

\section{Two-parameter complementary series -- slanted hyperplanes case}
We start with the hyperplane
\[
H_{\shrt}=\{\mathbf{x}\in\R^3:x_2-x_1=1\}.
\]
All hyperplanes in the $W$-orbit of $H_{\shrt}$ will be called slanted hyperplanes.
We parameterize $H_{\shrt}$ by the coordinates $x=x_1+\frac12$, $y=x_3$.
Thus, let $\iota:\R^2\rightarrow H_{\shrt}$ be the affine isomorphism
\[
\iota(x,y)=(x-\tfrac12,x+\tfrac12,y).
\]
For any $(x,y)\in\R^2$, we decompose $\Pi_{\iota(x,y)}$ in the Grothendieck group as $\Phi_{(x,y)}^++\Phi_{(x,y)}^-$ where
\[
\Phi_{(x,y)}^+=\nu^x\delta([-\tfrac12,\tfrac12])\times [y]\rtimes\s,\ \
\Phi_{(x,y)}^-=\nu^x L([-\tfrac12],[\tfrac12])\times [y]\rtimes\s.
\]
Let $H_{\shrt}^{\circ\circ}$ be the complement in $H_{\shrt}$ of all singular affine hyperplanes other than $x_2-x_1=\pm1$,
in other words, the image under $\iota$ of the complement $\R^2_{\vregshrt}$ of the lines
\begin{gather*}
y+x=\pm\tfrac32,\ \ x+y=\pm\tfrac12,\ \ x=\pm\tfrac12,\ \ y-x=\pm\tfrac32,\ y-x=\pm\tfrac12,\\
\ y=\pm\alpha,\ \ x=\pm(\alpha-\tfrac12),\ \ x=\pm(\alpha+\tfrac12).
\end{gather*}
More importantly, let $H_{\shrt}^\circ\supset H_{\shrt}^{\circ\circ}$ be the complement in $H_{\shrt}$ of the affine hyperplanes
\[
x_3+x_2=\pm1,\ \ x_1-x_3=\pm1,\ \ x_3=\pm\alpha,\ \ x_2=\pm\alpha,\ \ x_1=\pm\alpha.
\]
Thus, $H_{\shrt}^\circ$ is the image under $\iota$ of the complement $\R^2_{\regshrt}$ of the lines
\[
y+x=\pm\tfrac32,\ \ y-x=\pm\tfrac32,\ \ y=\pm\alpha,\ \ x=\pm(\alpha-\tfrac12),\ \ x=\pm(\alpha+\tfrac12).
\]
The representations $\Phi_{(x,y)}^\pm$ are irreducible precisely when $(x,y)\in\R^2_{\regshrt}$.
Thus, for $(x,y)\in\R^2_{\regshrt}$, $\iota(x,y)$ is strongly unitary (resp., strongly non-unitary)
if and only if both $\Phi_{(x,y)}^\pm$ are unitarizable (resp., non-unitarizable).

Denote
$$
W_{\shrt}=\{w\in W:w(H_{\shrt})=H_{\shrt} \}.
$$
Then, $W_{\shrt}\cong \{\pm1\}\times\{\pm1\}$, and it is generated by the following two transformations
$$
(x_1,x_2,x_3)\mapsto (-x_2,-x_1,x_3) \qquad (x_1,x_2,x_3) \mapsto (x_1,x_2,-x_3).
$$
It acts on $H_{\shrt}\cong \R^2$ (preserving $\R^2_{\regshrt}$ and $\R^2_{\vregshrt}$) by
$(x,y)\mapsto (\e_1 x,\e_2 y)$, $\e_1,\e_2=\pm1$.

For $(x,y)\in\R^2_{\regshrt}$,
the representations $\Phi_{x,y}^\pm$ depend only on the $W_{\shrt}$-orbit of $(x,y)$.
We denote
$\R^2_+=\{(x,y)\in\R^2:x,y\ge0\}$. Set $\R^2_{\regshrt,+}=\R^2_+\cap\R^2_{\regshrt}$.

Provisionally, we say that a point $(x,y)\in\R^2_{\regshrt}$ is unitary$^+$ (resp., unitary$^-$) if $\Phi_{(x,y)}^+$
(resp., $\Phi_{(x,y)}^-$) is unitarizable. (Eventually, these two notions are equivalent.)
We also say that $(x,y)\in\R^2_{\regshrt}$ is unitary$^\pm$ if it is both unitary$^+$ and unitary$^-$,
i.e., if $\iota(x,y)$ is strongly unitary.
These properties depend only on the $W_{\shrt}$-orbit and the connected component of $(x,y)$ in $\R^2_{\regshrt}$.
We say that a connected component of $\R^2_{\regshrt}$ is unitary$^+$ (resp., unitary$^-$, unitary$^\pm$) if
the same is true for any (or each) point in it.
As before it is enough to consider the connected components of $\R^2_{\regshrt,+}$.

\label{sec: Cabcd}

Analogously, one defines unitary$^\pm$ components of $R^2_{\vregshrt}$.
For simplicity, we denote in the rest of this chapter the components \eqref{eq: xx2x31}--\eqref{eq: xx1x2x14}
of Proposition \ref{prop: compserex} by $C_a$, $C_b$, $C_c$ and $C_d$ respectively. We start the study of unitary$^\pm$ components of $R^2_{\vregshrt}$ with the following technical

\begin{lemma}(See Figure \ref{fig: s322}) \label{lem: 2dimrreg}
For $\alpha\ge1$ the following connected components of $\R^2_{\vregshrt,+}$ are unitary$^\pm$.
\begin{gather*} \label{eq: dim2a} \tag{$C_1'$}
x+y<\tfrac12,\\
\label{eq: dim2b} \tag{$C_2'$}
x+y<\tfrac32,\ x-y>\tfrac12,\ x<\alpha-\tfrac12,\ \ \ (\alpha>1)\\
\label{eq: dim2c} \tag{$C_3'$}
\tfrac12<x+y<\tfrac32,\
{y-x}<\tfrac12,\ x<\tfrac12,\\
\label{eq: dim2d} \tag{$C_4'$}
\tfrac12<x+y<\tfrac32,\ x-y<\tfrac12,\ \tfrac12<x<\alpha-\tfrac12,\ \ \ (\alpha>1)\\
\label{eq: dim2e} \tag{$C_5'$}
x+y<\tfrac32,\ y-x>\tfrac12,\ y<\alpha,\\
\label{eq: dim2f}\tag{$C_6'$}
y-x>\tfrac32,\ y<\alpha,\ x<\tfrac12,\ \ \ (\alpha\ge2)\\
\label{eq: dim2g}\tag{$C_7'$}
y-x>\tfrac32,\ y<\alpha,\ x>\tfrac12,\ \ \ (\alpha>2)\\
\label{eq: dim2h}\tag{$C_8'$}
x-y>\tfrac32,\ x<\alpha-\tfrac12\ \ \ (\alpha>2).
\end{gather*}
(The constraint $x<\alpha-\frac12$ in $C_2'$ and $C_4'$ is redundant if $\alpha\ge2$.
Similarly the constraint $y<\alpha$ in $C_5'$ is redundant for $\alpha>1$.)
Moreover, there exist $w_1,\dots,w_8\in W$ and $X_1,\dots,X_8\in \{C_a,C_b, C_c,C_d\}$ such that
\begin{enumerate}
\item $\iota(C_i')$ is contained in $\partial(w_i(X_i))$ for $i=1,\dots,8$.
\item $w_1(X_1)\cup w_4(X_4)\cup w_5(X_5)\cup w_6(X_6)$ is contained in the connected component (i.e. half space) of
$\R^3\backslash H_{\shrt}$ containing the origin.
\item $w_2(X_2)\cup w_3(X_3)\cup w_7(X_7)\cup w_8(X_8)$ is contained in the connected component of $\R^3\backslash H_{\shrt}$
that does not contain the origin.
\end{enumerate}
Finally, for $\alpha<1$ there are no unitary$^\pm$ components in $\R^2_{\vregshrt,+}$.
\end{lemma}

\begin{figure}
\begin{tikzpicture}
\draw[style=dotted,line width=1pt] (1,1) -- (11,1);
\draw[style=dotted,line width=1pt] (1,1) -- (1,7.6);
\draw[style=dashed,line width=1pt] (1,4) -- (4,1);
\draw[style=dashed,line width=1pt] (1,7) -- (11,7);
\draw[style=dashed,line width=1pt] (4,1) -- (11,8);
\draw[style=dashed,line width=1pt] (1,4) -- (5,8);
\draw[style=dashed,line width=1pt] (6,1) -- (6,8);
\draw[style=dashed,line width=1pt] (8,1) -- (8,8);
\draw (11,1) node[right] {$y=0$};
\draw (1,7.6) node[above] {$x=0$};
\draw (11,7) node[right] {$y=\alpha$};
\draw (2.5,2.6) node[right] {$y+x=\frac32$};
\draw (4,0.1) node[above] {$(\frac32,0)$};
\draw (2,0.1) node[above] {$(\frac12,0)$};
\draw (6,1) node[below] {$(\alpha-\frac12,0)$};
\draw (8,1) node[below] {$(\alpha+\frac12,0)$};
\draw (1,7) node[left] {$(0,\alpha)$};
\draw (1,2) node[left] {$(\frac12,0)$};
\draw (1,4) node[left] {$(0,\frac32)$\ \ };
\draw (6.1,3.5) node[left] {$(\alpha-\frac12,\alpha-2)$};
\draw (8.1,5) node[right] {$(\alpha+\frac12,\alpha-1)$};
\draw (3.3,7) node[above] {$(\alpha-\frac32,\alpha)$};
\draw (6.6,6.5) node[left] {$(\alpha-\frac12,\alpha)$};
\draw (9.1,6.6) node[left] {$(\alpha+\frac12,\alpha)$};
\draw (10.7,7) node[below] {$(\alpha+\frac32,\alpha)$};
\fill [pattern=north west lines, pattern color=gray] (1,1) -- (4,1) -- (1,4);
\fill [pattern=north west lines, pattern color=gray] (6,1) -- (4,1) -- (6,3);
\fill [pattern=north west lines, pattern color=gray] (1,4) -- (1,7) -- (4,7);
\draw[line width=2pt] (6,1) -- (6,3);
\draw[line width=2pt] (4,1) -- (6,3);
\draw[line width=2pt] (1,4) -- (4,1);
\draw[line width=2pt] (1,4) -- (4,7);
\draw[line width=2pt] (1,7) -- (4,7);
\shade[shading=ball,ball color=black] (4,1) circle (.09);
\shade[shading=ball,ball color=black] (6,3) circle (.09);
\shade[shading=ball,ball color=black] (1,4) circle (.09);
\shade[shading=ball,ball color=black] (4,7) circle (.09);
\draw[line width=1pt] (8,1) -- (8,5);
\draw [thick,dash pattern={on 7pt off 4pt on 1pt off 4pt}] (1,2) -- (2,1);
\draw [thick,dash pattern={on 7pt off 4pt on 1pt off 4pt}] (1,2) -- (7,8);
\draw [thick,dash pattern={on 7pt off 4pt on 1pt off 4pt}] (2,1) -- (9,8);
\draw [thick,dash pattern={on 7pt off 4pt on 1pt off 4pt}] (2,1) -- (2,8);
\draw (1.3,1.3) node {$C_1'$};
\draw (3,1.3) node {$C_2'$};
\draw (1.6,2) node {$C_3'$};
\draw (2.4,2) node {$C_4'$};
\draw (1.4,3) node {$C_5'$};
\draw (1.5,5.7) node {$C_6'$};
\draw (2.6,6.2) node {$C_7'$};
\draw (5.2,1.4) node {$C_8'$};
 \end{tikzpicture}

\caption{Unitarizability for $\Phi_{(x,y)}^\pm$ (case $\alpha=3$; added lines)} \label{fig: s322}
\end{figure}

Note however that for $\alpha=\frac12$ the boundary of the connected component in $\R^3$ that contains \eqref{eq: x312},
intersects $H_{\shrt}$ only at the point $(\frac12,\frac12,0)$.

\begin{proof} Denote
\begin{gather*}
w_1:(x_1,x_2,x_3)\mapsto (-x_2,x_3,x_1),
\\
w_2:(x_1,x_2,x_3)\mapsto (x_2,x_3,x_1),
\\
w_3:(x_1,x_2,x_3)\mapsto (-x_1,x_3,x_2),
\\
w_4:(x_1,x_2,x_3)\mapsto (x_1,x_3,x_2),
\\
w_5:(x_1,x_2,x_3)\mapsto (-x_1,x_2,x_3),
\\
w_6:(x_1,x_2,x_3)\mapsto (-x_1,x_2,x_3),
\\
w_8:(x_1,x_2,x_3)\mapsto (x_1,x_2,x_3),
\\
w_8:(x_1,x_2,x_3)\mapsto (x_2,x_3,x_1).
\end{gather*}
Take $X_1, \dots, X_8$ to be $C_a,C_b,C_c,C_c,C_c,C_b,C_d,C_d$ respectively.

The component $X_1=C_a$ is determined by the conditions $x_2+x_3<1$, $0\leq x_1\le x_2\leq x_3$, and $w_1$ takes
it to $(x_1',x_2',x_3')=(-x_2,x_3,x_1)$, where
$
x_2'-x_1'<1,
$
$0\leq x_3'\leq -x_1'\leq x_2'$. We get at the boundary for $x_2'=x_1'+1$ the region
$0\leq x_3'\leq -x_1'\leq x_2'$.
Introduce $x_1'=x-\frac12$, $y=x_3'$. From $ x_3'\leq -x_1'$ we get $x-y\leq \frac12$.
The interior of this region is $C'_1$.
The above inequality $x_2'-x_1'<1$ is equivalent to the inequality
$
x_2+x_3<1.
$
Therefore, $w_1(X_1)$ is in the connected component of $\R^3\backslash H_{\shrt}$ containing the origin.

The remaining 7 cases are proved in analogous (elementary) way. We omit the details.
\end{proof}

\begin{proposition} \label{prop: dim2unit}
For $\alpha\ge1$, the unitary$^\pm$ connected components of $\R^2_{\regshrt,+}$ are given by
\begin{subequations}
\begin{gather}
\label{eq: dim21} x+y<\tfrac32,\ x<\alpha-\tfrac12,\ y<\alpha\\
\label{eq: dim22} y-x>\tfrac32,\ y<\alpha,\ \ \ (\alpha\ge2)\\
\label{eq: dim23} x-y>\tfrac32,\ x\le\alpha-\tfrac12,\ \ \ (\alpha>2).
\end{gather}
\end{subequations}
(The constraint $x<\alpha-\tfrac12$ (resp., $y<\alpha$) in the first region is redundant
if $\alpha\ge2$ (resp., $\alpha>1$).)
For each point $(x,y)$ in the above three regions there exists $w\in W$ such that $\iota(x,y)$ is contained in the boundary of $w(C)$
where $C$ is one of the unitary connected components \eqref{eq: xx2x31}--\eqref{eq: xx1x2x14} of Proposition \ref{prop: compserex}.
The other connected components of $\R^2_{\regshrt,+}$ are neither unitary$^+$ nor unitary$^-$.

If $\alpha=0$ or $\tfrac12$ then there are no unitary$^+$ or unitary$^-$ regions in $\R^2_{\regshrt}$.
\end{proposition}

\begin{proof}
Assume that $\alpha\geq1$.
Lemma \ref{lem: 2dimrreg} implies that components \eqref{eq: dim21}--\eqref{eq: dim23} (which are the regions $C_0,C_2^u$ and
$C_2^d$ in Figure \ref{fig: s3}) are unitary.\footnote{Note that \eqref{eq: dim21} is a two-dimensional complementary series,
while \eqref{eq: dim22} and \eqref{eq: dim23} can be obtained by iterating one-dimensional complementary series twice.
This is a simpler way to conclude this unitarity.}

We turn to the converse direction.
Let us call the regions \eqref{eq: dim21}--\eqref{eq: dim23} \emph{special}.
Clearly, these regions are distinct connected components of $\R^2_{\regshrt,+}$.

We present graphical interpretation of $\R^2_{\regshrt,+}$ for $\alpha\geq 2,\alpha=\tfrac32, \alpha=1,\alpha=\tfrac12$ and $\alpha=0$ in
Figures \ref{fig: s3}, \ref{fig: s32}, \ref{fig: s1}, \ref{fig: s12} and \ref{fig: s0} respectively, where the special regions are shaded.
We need to show that the non-shaded regions are neither unitary$^+$ nor unitary$^-$.
\begin{figure}
\begin{tikzpicture} [thick, scale=0.9]
\draw[style=dotted,line width=1pt] (1,1) -- (11,1);
\draw[style=dotted,line width=1pt] (1,1) -- (1,12);
\draw[style=dashed,line width=1pt] (1,4) -- (4,1);
\draw[style=dashed,line width=1pt] (1,7) -- (12,7);
\draw[style=dashed,line width=1pt] (4,1) -- (12,9);
\draw[style=dashed,line width=1pt] (1,4) -- (9,12);
\draw[style=dashed,line width=1pt] (6,1) -- (6,12);
\draw[style=dashed,line width=1pt] (8,1) -- (8,12);
\draw (11,1) node[right] {$y=0$};
\draw (1,12) node[above] {$x=0$};
\draw (5.7,12) node[above] {$x=\alpha-\frac12$\ \ };
\draw (7.6,12) node[above] {\ \ $x=\alpha+\frac12$ };
\draw (12,7) node[right] {$y=\alpha$};
\draw (2.5,2.6) node[right] {$y+x=\frac32$};
\draw (12,9) node[above] {$\ y=x-\frac32$};
\draw (4,0.1) node[above] {$(\frac32,0)$};
\draw (1,4) node[left] {$(0,\frac32)$\ \ };
\draw (6.1,3.5) node[left] {$(\alpha-\frac12,\alpha-2)$};
\draw (8.1,5) node[right] {$(\alpha+\frac12,\alpha-1)$};
\draw (8.2,11) node[right] {$(\alpha+\frac12,\alpha+2)$};
\draw (3.1,7) node[above] {$(\alpha-\frac32,\alpha)$};
\draw (6,9.6) node[left] {$(\alpha-\frac12,\alpha+1)$};
\draw (6.1,6.5) node[left] {$(\alpha-\frac12,\alpha)$};
\draw (10.2,7.5) node[left] {$(\alpha+\frac12,\alpha)$};
\draw (10.9,7) node[below] {$(\alpha+\frac32,\alpha)$};
\draw (10,12) node[above] {$y=x+\frac32$};
\draw (2,2) node {$C_0$};
\draw (1.8,5.9) node {$C_2^u$};
\draw (5.3,1.5) node {$C_2^d$};
\fill [pattern=north west lines, pattern color=gray] (1,1) -- (4,1) -- (1,4);
\fill [pattern=north west lines, pattern color=gray] (6,1) -- (4,1) -- (6,3);
\fill [pattern=north west lines, pattern color=gray] (1,4) -- (1,7) -- (4,7);
\draw[line width=2pt] (6,1) -- (6,3);
\draw[line width=2pt] (4,1) -- (6,3);
\draw[line width=2pt] (1,4) -- (4,1);
\draw[line width=2pt] (1,4) -- (4,7);
\draw[line width=2pt] (1,7) -- (4,7);
\draw[line width=1pt] (8,1) -- (8,5);
\shade[shading=ball,ball color=black] (4,1) circle (.09);
\shade[shading=ball,ball color=black] (6,3) circle (.09);
\shade[shading=ball,ball color=black] (1,4) circle (.09);
\shade[shading=ball,ball color=black] (4,7) circle (.09);
\shade[shading=ball,ball color=lightgray] (10,7) circle (.14);
\shade[shading=ball,ball color=lightgray] (8,11) circle (.14);
\shade[shading=ball,ball color=lightgray] (6,9) circle (.14);
\shade[shading=ball,ball color=lightgray] (8,5) circle (.14);
\shade[shading=ball,ball color=lightgray] (6,7) circle (.14);
\shade[shading=ball,ball color=white] (8,7) circle (.22);
\end{tikzpicture}
\caption{Unitarizability for $\Phi_{(x,y)}^\pm$ (case $\alpha=3$)} \label{fig: s3}
\end{figure}

 \begin{figure}
\begin{tikzpicture} [thick, scale=0.8]
\draw[style=dotted,line width=1pt] (1,1) -- (9,1);
\draw[style=dotted,line width=1pt] (1,1) -- (1,9);
\draw[style=dashed,line width=1pt] (1,4) -- (4,1);
\draw[style=dashed,line width=1pt] (1,4) -- (9,4);
\draw[style=dashed,line width=1pt] (4,1) -- (9,6);
\draw[style=dashed,line width=1pt] (1,4) -- (6,9);
\draw[style=dashed,line width=1pt] (3,1) -- (3,9);
\draw[style=dashed,line width=1pt] (5,1) -- (5,9);
\draw (9,1) node[right] {$y=0$};
\draw (1,9) node[above] {$x=0$};
\draw (3,9) node[above] {$x=1$ \ \ };
\draw (5,9) node[above] {$x=2$ };
\draw (9,4) node[right] {$y=\frac32$};
\draw (3.2,2.6) node[left] {$y+x=\frac32$};
\draw (9,6) node[right] {$\ y=x-\frac32$};
\draw (4,0) node[above] {$(\frac32,0)$};
\draw (1,4) node[left] {$(0,\frac32)$\ \ };
\draw (3.1,2) node[right] {$(1,\frac12)$};
\draw (5.1,2) node[right] {$(2,\frac12)$};
\draw (5,4.5) node[right] {$(2,\frac32)$};
\draw (3,4.5) node[right] {$(1,\frac32)$};
\draw (7,3.5) node[right] {$(3,\frac32)$};
\draw (3.1,5.9) node[right] {$(1,\frac52)$};
\draw (5,7.8) node[right] {$(2,\frac72)$};
\draw (7.3,8.8) node[above] {$y=x+\frac32$};
\draw (1.8,1.7) node {$C_0$};
\draw (3.36,1.25) node {$C_3$};
\fill [pattern=north west lines, pattern color=gray] (1,1) -- (3,1) -- (3,2) -- (1,4);
\draw[line width=2pt] (1,4) -- (3,2);
\draw[line width=1pt] (5,1) -- (5,2);
\draw[line width=2pt] (3,1) -- (3,2);
\shade[shading=ball,ball color=black] (3,2) circle (.12);
\shade[shading=ball,ball color=black] (1,4) circle (.12);
\shade[shading=ball,ball color=lightgray] (3,4) circle (.16);
\shade[shading=ball,ball color=lightgray] (7,4) circle (.16);
\shade[shading=ball,ball color=lightgray] (5,2) circle (.16);
\shade[shading=ball,ball color=lightgray] (5,8) circle (.16);
\shade[shading=ball,ball color=white](5,4) circle (.25);
\shade[shading=ball,ball color=lightgray] (3,6) circle (.16);
\end{tikzpicture}
\caption{Unitarizability for $\Phi_{(x,y)}^\pm$ (case $\alpha=\frac32$)} \label{fig: s32}
\end{figure}

\begin{figure}
\begin{tikzpicture} [thick, scale=0.55]
\draw[style=dotted,line width=1pt] (1,1) -- (13,1);
\draw[style=dotted,line width=1pt] (1,1) -- (1,14);
\draw[style=dashed,line width=1pt] (1,7) -- (7,1);
\draw[style=dashed,line width=1pt] (1,5) -- (13,5);
\draw[style=dashed,line width=1pt] (7,1) -- (13,7);
\draw[style=dashed,line width=1pt] (1,7) -- (8,14);
\draw[style=dashed,line width=1pt] (3,1) -- (3,14);
\draw[style=dashed,line width=1pt] (7,1) -- (7,14);
\draw (13,1) node[right] {$y=0$};
\draw (0.5,14.2) node[above] {$x=0$};
\draw (3.3,14) node[above] {$x=\frac12$};
\draw (6.5,14) node[above] {$x=\frac32$};
\draw (9.9,14) node[above] {$y=x+\frac32$ };
\draw (13,7) node[right] {$y=x-\frac32$ };
\draw (13,5) node[right] {$y=1$};
\draw (4,5) node[above] {$(\frac12,1)$};
\draw (4,9.1) node[below] {$(\frac12,2)$};
\draw (7,1) node[below] {$(\frac32,0)$};
\draw (8,5) node[above] {$(\frac32,1)$};
\draw (7.1,13) node[right] {$(\frac32,3)$};
\draw (11,5) node[above] {$(\frac52,1)$\ \ \ \ \ };
\draw (5,3) node[above] {$x+y=\frac32$ };
\draw (2,2.5) node[above] {$C_0$ };
\draw (1.7,5) node[above] {$C_4$ };
\fill [pattern=north west lines, pattern color=gray] (1,1) -- (3,1) -- (3,5) -- (1,5);
\draw[line width=2pt] (1,5) -- (3,5);
\draw[line width=2pt] (3,1) -- (3,5);
\shade[shading=ball,ball color=black] (3,5) circle (.17);
\shade[shading=ball,ball color=lightgray] (3,9) circle (.23);
\shade[shading=ball,ball color=lightgray] (7,1) circle (.23);
\shade[shading=ball,ball color=lightgray] (11,5) circle (.23);
\shade[shading=ball,ball color=lightgray] (7,13) circle (.23);
\shade[shading=ball,ball color=white] (7,5) circle (.35);
 \end{tikzpicture}
\caption{Unitarizability for $\Phi_{(x,y)}^\pm$ (case $\alpha=1$)} \label{fig: s1}
\end{figure}

\begin{figure}
\begin{tikzpicture} [thick, scale=0.55]
\draw[style=dotted,line width=1pt] (1,1) -- (11,1);
\draw[style=dashed,line width=1pt] (1,1) -- (1,12);
\draw[style=dashed,line width=1pt] (1,3) -- (11,3);
\draw[style=dashed,line width=1pt] (5,1) -- (5,12);
\draw[style=dashed,line width=1pt] (7,1) -- (1,7);
\draw[style=dashed,line width=1pt] (1,7) -- (6,12);
\draw[style=dashed,line width=1pt] (7,1) -- (11,5);
\draw (11,1) node[right] {$y=0$};
\draw (.6,12) node[above] {$x=0$};
\draw (4.4,12) node[above] {$x=1$};
\draw (11,3) node[right] {$y=\frac12$};
\draw (2,5) node[right] {$x+y=\frac32$};
\draw (7.8,11.9) node[above] {\ \ $y=x+\frac32$ };
\draw (11,5) node[right] {$y=x-\frac32$ };
\draw (1,3) node[left] {$(0,\frac12)$};
\draw (1,7) node[left] {$(0,\frac32)$};
\draw (6,3) node[above] {$(1,\frac12)$};
\draw (6.5,10.5) node[above] {$(1,\frac52)$};
\draw (9.6,3) node[below] {$(2,\frac12)$};
\draw[line width=1pt] (1,7) -- (5,3);
\draw[line width=1pt] (1,1) -- (1,7);
\draw[line width=1pt] (5,1) -- (5,3);
\draw[line width=1pt] (1,3) -- (5,3);
\shade[shading=ball,ball color=black] (1,3) circle (.17);
\shade[shading=ball,ball color=lightgray] (1,7) circle (.23);
\shade[shading=ball,ball color=lightgray] (5,3) circle (.23);
\shade[shading=ball,ball color=lightgray] (9,3) circle (.23);
\shade[shading=ball,ball color=lightgray] (5,11) circle (.23);
\end{tikzpicture}
\caption{Unitarizability for $\Phi_{(x,y)}^\pm$ (case $\alpha=\frac12$)} \label{fig: s12}
\end{figure}

\begin{figure}
\begin{tikzpicture} [thick, scale=0.5]
\draw[style=dashed,line width=1pt] (1,1) -- (11,1);
\draw[style=dashed,line width=1pt] (1,1) -- (1,10);
\draw[style=dashed,line width=1pt] (1,7) -- (7,1);
\draw[style=dashed,line width=1pt] (3,1) -- (3,10);
\draw[style=dashed,line width=1pt] (7,1) -- (11,5);
\draw[style=dashed,line width=1pt] (1,7) -- (4,10);
\draw (11,1) node[right] {$y=0$};
\draw (.5,10.2) node[above] {$x=0$};
\draw (3,10) node[above] {$x=\frac12$};
\draw (5.9,10) node[above] {\ \ $y=x+\frac32$ };
\draw (11,5) node[right] {$y=x-\frac32$ };
\draw (1,1) node[below] {$(0,0)$};
\draw (3,1) node[below] {$(\frac12,0)$};
\draw (7,1) node[below] {$(\frac32,0)$};
\draw (1,7) node[left] {$(0,\frac32)$};
\draw (3,5) node[right] {\ $(\frac12,1)$};
\draw (3,9) node[right] {\ $(\frac12,2)$};
\draw (5,3) node[above] {$x+y=\frac32$ };
\draw[line width=1pt] (3,1) -- (3,5);
\shade[shading=ball,ball color=black] (1,1) circle (.18);
\shade[shading=ball,ball color=black] (3,1) circle (.18);
\shade[shading=ball,ball color=lightgray] (7,1) circle (.24);
\shade[shading=ball,ball color=lightgray] (3,5) circle (.24);
\shade[shading=ball,ball color=lightgray] (3,9) circle (.24);
\end{tikzpicture}
\caption{Unitarizability for $\Phi_{(x,y)}^\pm$ (case $\alpha=0$)} \label{fig: s0}
\end{figure}
Using the results of chapters \ref{sec: unit3}, \ref{basic 1} and \ref{CC-0} (and having in mind Remark \ref{non-uni-half}),
we have put in Figures \ref{fig: s3} -- \ref{fig: s0} black vertices where all subquotients are unitarizable,
and white or gray where we have at least one
irreducible subquotient non-unitarizable (see \ref{legend} for more precise description).
Suppose that we have a unitary$^+$ or unitary$^-$ component.
Then, it must be bounded, and cannot have a non-black vertex.
Now Figures \ref{fig: s3}, \ref{fig: s12} and \ref{fig: s0} imply that any unitary$^+$ or unitary$^-$ component there must be special.

In the case of $\alpha=\tfrac32$, only $C_3$ does not have a non-black vertex. Since the slanted side of this component
contains either $\nu^{x-\frac12}\delta([-1,1])\rtimes \sigma$, $1<x<\tfrac32$, or $\nu^{x-\frac12}\delta([-1,1])^t\rtimes \sigma$, $1<x<\tfrac32$, and both families consists of non-unitarizable representations by Proposition \ref{prop: unitlines} (use N$^\circ$1 in Table \ref{tab: redpnt}),
we get that also $C_3$ is neither unitary$^+$ nor unitary$^-$.

It remains to consider the case of $\alpha=1$. Here only $C_4$ does not have a non-black vertex. Here the slanted side of this component contains either
$\nu^{\frac12-x}\delta([-1,1])\rtimes \sigma$, $0<x<\tfrac12$, or $\nu^{\frac12-x}\delta([-1,1])\rtimes \sigma$, $0<x<\tfrac12$, and both families consists of non-unitarizable representations by Proposition \ref{prop: unitlines} (use again N$^\circ$1 in Table \ref{tab: redpnt}).
Therefore, $C_4$ can be neither unitary$^+$ nor unitary$^-$. This completes the proof of the proposition.
\end{proof}

\begin{corollary} \label{all-sla-v-reg-u}
The components $C_1',\dots,C_8'$ are precisely the unitary$^\pm$ components in $\R^2_{\vregshrt}$.
\end{corollary}
\begin{proof}
Observe that the components \eqref{eq: dim21}--\eqref{eq: dim23} are contained in the closure of the union of $C_1',\dots,C_8'$.
Now the claim follows from Proposition \ref{prop: dim2unit}.
\end{proof}

\section{Two-parameter complementary series -- level hyperplanes case}

We turn to the affine hyperplane
\[
H_{\lng}=\{\mathbf{x}\in\R^3:x_3=\alpha\}.
\]
All hyperplanes in the $W$-orbit of $H_{\lng}$ will be called level hyperplanes.
For any $(x,y)\in\R^2$, we decompose $\Pi_{(x,y,\alpha)}$ in the Grothendieck group as $\Psi_{(x,y)}^++\Psi_{(x,y)}^-$ where
\[
\Psi_{(x,y)}^+=[x]\times[y]\rtimes\delta([\alpha];\s),\ \
\Psi_{(x,y)}^-=[x]\times[y]\rtimes L([\alpha];\s)
\]
if $\alpha>0$ and $\Psi_{(x,y)}^\pm=[x]\times[y]\rtimes\delta([0]_\pm;\s)$ if $\alpha=0$.

Let $H_{\lng}^{\circ\circ}$ (resp., $\R^2_{\vreglng}$) be the complement in $H_{\lng}$ (resp., $\R^2$) of the
singular affine hyperplanes other than $x_3=\pm\alpha$:
\[
x_1\pm x_2=\e,\ \ x_i=\pm(\alpha+\e),\ \ x_i=\pm\alpha,\ \ \e=\pm1,\ i=1,2.
\]
Let $H_{\lng}^\circ\supset H_{\lng}^{\circ\circ}$ (resp., $\R^2_{\reglng}\supset\R^2_{\vreglng}$) be the complement in
$H_{\lng}$ (resp., $\R^2$) of the 12 (not necessarily distinct) affine hyperplanes
\[
x_1\pm x_2=\e,\ \ x_i=\pm(\alpha+\e),\ \ \e=\pm1,\ i=1,2.
\]
The representations $\Psi_{(x,y)}^\pm$ are irreducible precisely when $(x,y)\in\R^2_{\reglng}$.
Thus, for $(x,y)\in\R^2_{\reglng}$, $(x,y,\alpha)$ is strongly unitary (resp., strongly non-unitary)
if and only if both $\Psi_{(x,y)}^\pm$ are unitarizable (resp., non-unitarizable).
The group $W_{\lng}$ of signed permutations on $\{1,2\}$ (i.e., the dihedral group $D_4$)
acts on $\R^2$ and preserves $\R^2_{\reglng}$ and $\R^2_{\vreglng}$. For $(x,y)\in\R^2_{\reglng}$,
the representations $\Psi_{x,y}^\pm$ depend only on the $W_{\lng}$-orbit of $(x,y)$.
We have denoted $\R^2_{++}=\{(x,y)\in\R^2:y\ge x\ge0\}$. Set $\R^2_{\reglng,++}=\R^2_{++}\cap\R^2_{\reglng}$
(resp. $H_{\lng,++}^{\circ}$) and $\R^2_{\vreglng,++}=\R^2_{++}\cap\R^2_{\vreglng}$ (resp. $H_{\lng,++}^{\circ\circ}$).

We say that a point $(x,y)\in\R^2_{\reglng}$ is unitary$^+$ (resp., unitary$^-$) if $\Psi_{(x,y)}^+$
(resp., $\Psi_{(x,y)}^-$) is unitarizable. (As before, eventually these notions will turn out to be equivalent.)
We also say that $(x,y)\in\R^2_{\reglng}$ is unitary$^\pm$ if it is both unitary$^+$ and unitary$^-$,
i.e., if $(x,y,\alpha)$ is strongly unitary.
These properties depend only on the $W_{\lng}$-orbit and the connected component of $(x,y)$ in $\R^2_{\reglng}$.
We say that a connected component of $R^2_{\reglng}$ is unitary$^+$ (resp., unitary$^-$, unitary$^\pm$) if
the same is true for any (or each) point in it.
As before it is enough to consider the connected components of $\R^2_{\reglng,++}$.

\begin{proposition} \label{prop: 2dmlng}
The unitary$^\pm$ connected components of $\R^2_{\reglng,++}$ are as follows.
\begin{subequations}
\begin{gather}
\label{eq: >1x-x} (\alpha>2)\ \ y-x>1,\ y<\alpha-1,\\
\label{eq: >1xy1} (\alpha>1)\ \ x+y<1,\ y<\alpha-1,\\
\label{eq: 1axy1} (\alpha=1)\ \ x+y<1,\\
\label{eq: 12y12} (\alpha=\tfrac12)\ \ y<\tfrac12,\\
\label{eq: 0xy1} (\alpha=0)\ \ x+y<1.
\end{gather}
\end{subequations}
(The constraint $y<\alpha-1$ in \eqref{eq: >1xy1} is redundant for $\alpha\ge2$.)
The other connected components of $\R^2_{\reglng,++}$ are neither unitary$^+$ nor unitary$^-$.
\end{proposition}

\begin{figure}
\begin{tikzpicture} [thick, scale=0.93]
\draw[style=dotted,line width=1pt] (1,1) -- (12,1);
\draw[style=dotted,line width=1pt] (1,1) -- (1,12);
\draw[style=dotted,line width=1pt] (1,1) -- (12,12);
\draw[style=dashed,line width=1pt] (1,3) -- (3,1);
\draw[style=dashed,line width=1pt] (1,3) -- (10,12);
\draw[style=dashed,line width=1pt] (3,1) -- (12,10);
\draw[style=dashed,line width=1pt] (1,5) -- (12,5);
\draw[style=dashed,line width=1pt] (5,1) -- (5,12);
\draw[style=dashed,line width=1pt] (9,1) -- (9,12);
\draw[style=dashed,line width=1pt] (1,9) -- (12,9);
\draw (12,1) node[right] {$y=0$};
\draw (1,12) node[above] {$x=0$};
\draw (12.7,12) node[above] {sym. $y=x$};
\draw (5,12) node[above] {$x=\alpha-1$};
\draw (8.7,12) node[above] {$x=\alpha+1$};
\draw (12,5) node[right] {$y=\alpha-1$};
\draw (12,9) node[right] {$y=\alpha+1$};
\draw (10.8,12) node[above] {$y=x+1$};
\draw (12,10) node[right] {$\ y=x-1$};
\draw (2.4,2.4) node[above] {$y+x=1$};
\draw (7.6,11.5) node[below] {$(\alpha+1,\alpha+2)$};
\draw (9,9) node[above] {$(\alpha+1,\alpha+1)$};
\draw (5,5) node[above] {$(\alpha-1,\alpha-1)$};
\draw (4,7.5) node[below] {$(\alpha-1,\alpha)$};
\draw (6.4,9.8) node[below] {$(\alpha,\alpha+1)$};
\draw (3.7,9.7) node[below] {$(\alpha-1,\alpha+1)$};
\draw (1.6,1.5) node {$C_0$};
\draw (1.6,4.3) node {$C_7$};
\fill [pattern=north west lines, pattern color=gray] (1,1) -- (3,1) -- (1,3);
\fill [pattern=north west lines, pattern color=gray] (5,1) -- (3,1) -- (5,3);
\fill [pattern=north west lines, pattern color=gray] (1,3) -- (1,5) -- (3,5);
\draw[line width=2pt] (5,1) -- (5,3);
\draw[line width=2pt] (3,1) -- (5,3);
\draw[line width=2pt] (1,3) -- (3,1);
\draw[line width=2pt] (1,3) -- (3,5);
\draw[line width=2pt] (1,5) -- (3,5);
\draw[line width=1pt] (9,1) -- (9,5);
\draw[line width=1pt] (1,9) -- (5,9);
\shade[shading=ball,ball color=black] (3,1) circle (.09);
\shade[shading=ball,ball color=black] (5,3) circle (.09);
\shade[shading=ball,ball color=black] (1,3) circle (.09);
\shade[shading=ball,ball color=black] (3,5) circle (.09);
\shade[shading=ball,ball color=lightgray] (11,9) circle (.12);
\shade[shading=ball,ball color=lightgray] (9,11) circle (.12);
\shade[shading=ball,ball color=lightgray] (9,5) circle (.12);
\shade[shading=ball,ball color=lightgray] (5,9) circle (.12);
\shade[shading=ball,ball color=lightgray] (7,5) circle (.12);
\shade[shading=ball,ball color=lightgray] (5,7) circle (.12);
\shade[shading=ball,ball color=white] (9,9) circle (.19);
\shade[shading=ball,ball color=white] (5,5) circle (.19);
\shade[shading=ball,ball color=white] (9,7) circle (.19);
\shade[shading=ball,ball color=white] (7,9) circle (.19);
\end{tikzpicture}
\caption{Unitarizability for $\Psi_{(x,y)}^\pm$ (case $\alpha = 3$)} \label{fig: l3}
\end{figure}

\begin{proof}
The regions \eqref{eq: >1xy1}, \eqref{eq: 12y12}, \eqref{eq: 0xy1} are the connected components
of the origin (i.e. $C_0$) in $\R^2_{\reglng,++}$ in the cases $\alpha>1$,
$\alpha=\tfrac12$ and $\alpha=0$ respectively.
Therefore, they are unitary$^\pm$ since $\delta([\alpha];\sigma)$ and $L([\alpha];\sigma)$ (for $\alpha>0$) and $\delta([0]_\pm;\sigma)$
(for $\alpha=0$) are unitarizable; see Figures \ref{fig: l3}, \ref{fig: l2}, \ref{fig: l32}, \ref{fig: l12}, and \ref{fig: l0}).
\begin{figure}
\begin{tikzpicture} [thick, scale=0.8]
\draw[style=dotted,line width=1pt] (1,1) -- (10,1);
\draw[style=dotted,line width=1pt] (1,1) -- (1,10);
\draw[style=dotted,line width=1pt] (1,1) -- (10,10);
\draw[style=dashed,line width=1pt] (1,3) -- (3,1);
\draw[style=dashed,line width=1pt] (1,3) -- (8,10);
\draw[style=dashed,line width=1pt] (3,1) -- (10,8);
\draw[style=dashed,line width=1pt] (1,3) -- (10,3);
\draw[style=dashed,line width=1pt] (3,1) -- (3,10);
\draw[style=dashed,line width=1pt] (7,1) -- (7,10);
\draw[style=dashed,line width=1pt] (1,7) -- (10,7);
\draw (10,1) node[right] {$y=0$};
\draw (.8,10) node[above] {$x=0$};
\draw (3.3,10) node[above] {$x=1$};
\draw (6.7,10) node[above] {$x=3$};
\draw (10,3) node[right] {$y=1$};
\draw (10,7) node[right] {$y=3$};
\draw (9,10) node[above] {$y=x+1$};
\draw (10,8) node[right] {$\ y=x-1$};
\draw (2,2) node[above] {$y+x=1$};
\draw (2.4,3) node[above] {$(1,1)$};
\draw (6.4,7) node[above] {$(3,3)$};
\draw (3,5.1) node[left] {$(1,2)$};
\draw (2.3,7) node[below] {$(1,3)$};
\draw (4.5,7) node[above] {$(2,3)$};
\draw (7.2,9) node[left] {$(3,4)$ \ \ };
\draw (2,1.5) node[left] {$C_0$};
\fill [pattern=north west lines, pattern color=gray] (1,1) -- (3,1) -- (1,3);
\draw[line width=2pt] (1,3) -- (3,1);
\draw[line width=1pt] (7,1) -- (7,3);
\draw[line width=1pt] (1,7) -- (3,7);
\shade[shading=ball,ball color=black] (3,1) circle (.09);
\shade[shading=ball,ball color=black] (1,3) circle (.09);
\shade[shading=ball,ball color=lightgray] (7,9) circle (.14);
\shade[shading=ball,ball color=lightgray] (3,7) circle (.14);
\shade[shading=ball,ball color=lightgray] (7,3) circle (.14);
\shade[shading=ball,ball color=lightgray] (9,7) circle (.14);
\shade[shading=ball,ball color=lightgray] (5,3) circle (.14);
\shade[shading=ball,ball color=lightgray] (3,5) circle (.14);
\shade[shading=ball,ball color=white] (7,7) circle (.22);
\shade[shading=ball,ball color=white] (3,3) circle (.22);
\shade[shading=ball,ball color=white] (5,7) circle (.22);
\shade[shading=ball,ball color=white] (7,5) circle (.22);
 \end{tikzpicture}
 \caption{Unitarizability for $\Psi_{(x,y)}^\pm$ (case $\alpha=2$)} \label{fig: l2}
\end{figure}

\begin{figure}
\begin{tikzpicture} [thick, scale=0.7]
\draw[style=dotted,line width=1pt] (0,0) -- (11,0);
\draw[style=dotted,line width=1pt] (0,0) -- (0,11);
\draw[style=dotted,line width=1pt] (0,0) -- (11,11);
\draw[style=dashed,line width=1pt] (0,4) -- (4,0);
\draw[style=dashed,line width=1pt] (0,4) -- (7,11);
\draw[style=dashed,line width=1pt] (4,0) -- (11,7);
\draw[style=dashed,line width=1pt] (0,2) -- (11,2);
\draw[style=dashed,line width=1pt] (2,0) -- (2,11);
\draw[style=dashed,line width=1pt] (10,0) -- (10,11);
\draw[style=dashed,line width=1pt] (0,10) -- (11,10);
\draw (11,0) node[right] {$y=0$};
\draw (-0.3,11.2) node[above] {$x=0$};
\draw (2.3,11) node[above] {$x=\frac12$};
\draw (10.4,11) node[above] {$x=\frac52$};
\draw (11,2) node[right] {$y=\frac12$};
\draw (11,10) node[right] {$y=\frac52$};
\draw (6.8,11) node[above] {$y=x+1$};
\draw (11,7) node[right] {$\ y=x-1$};
\draw (3.4,0.4) node[above] {$y+x=1$};
\draw (3.4,2) node[above] {$(\frac12,\frac12)$};
\draw (2.8,5) node[above] {$(\frac12,\frac32)$};
\draw (9.2,10) node[above] {$(\frac52,\frac52)$};
\draw (6.7,9) node[above] {$(\frac32,\frac52)$};
\draw (2.8,10) node[below] {$(\frac12,\frac52)$};
\draw (1,1) node {$C_0$};
\draw (0.65,2.55) node {$C_{9}$};
\fill [pattern=north west lines, pattern color=gray] (0,0) -- (0,2) -- (2,2) -- (2,0) -- (0,0);
\draw[line width=2pt] (2,0) -- (2,2);
\draw[line width=2pt] (0,2) -- (2,2);
\draw[line width=1pt] (10,0) -- (10,2);
\draw[line width=1pt] (0,10) -- (2,10);
\shade[shading=ball,ball color=black] (2,2) circle (.15);
\shade[shading=ball,ball color=lightgray] (10,6) circle (.18);
\shade[shading=ball,ball color=lightgray] (6,10) circle (.18);
\shade[shading=ball,ball color=lightgray] (6,2) circle (.18);
\shade[shading=ball,ball color=lightgray] (2,6) circle (.18);
\shade[shading=ball,ball color=lightgray] (2,10) circle (.18);
\shade[shading=ball,ball color=lightgray] (10,2) circle (.18);
\shade[shading=ball,ball color=white] (10,10) circle (.3);
\end{tikzpicture}
\caption{Unitarizability for $\Psi_{(x,y)}^\pm$ (case $\alpha=\frac32$)} \label{fig: l32}
\end{figure}

\begin{figure}
\begin{tikzpicture} [thick, scale=0.5]
\draw[style=dashed,line width=1pt] (1,1) -- (14,1);
\draw[style=dashed,line width=1pt] (1,1) -- (1,14);
\draw[style=dotted,line width=1pt] (1,1) -- (14,14);
\draw[style=dashed,line width=1pt] (1,5) -- (5,1);
\draw[style=dashed,line width=1pt] (1,5) -- (10,14);
\draw[style=dashed,line width=1pt] (5,1) -- (14,10);
\draw[style=dashed,line width=1pt] (9,1) -- (9,14);
\draw[style=dashed,line width=1pt] (1,9) -- (14,9);
\draw (14,1) node[right] {$y=0$};
\draw (1,14) node[above] {$x=0$};
\draw (14,10) node[right] {$y=x-1$};
\draw (14,9) node[right] {$y=2$};
\draw (10.5,13.9) node[above] {$\ \ \ \ \ \ y=x+1$};
\draw (9,14) node[above] {$x=2$\ \ \ \ \ \ };
\draw (9.3,13) node[below] {$\ \ \ \ \ (2,3)$};
\draw (1,1) node[below] {$(0,0)$};
\draw (5,1) node[below] {$(1,0)$};
\draw (9,1) node[below] {$(2,0)$};
\draw (5.8,9) node[below] {$(1,2)$};
\draw (10,9) node[below] {$(2,2)$};
\draw (3,2) node[right] {$x+y=1$};
\draw (1.3,2) node[right] {$C_{10}$};
\draw [pattern=north west lines, pattern color=gray] (1,1) -- (5,1) -- (1,5);
\draw[line width=2pt] (5,1) -- (1,5);
\draw[line width=2pt] (1,1) -- (1,5);
\draw[line width=2pt] (1,1) -- (5,1);
\shade[shading=ball,ball color=black] (5,1) circle (.15);
\shade[shading=ball,ball color=black] (1,1) circle (.15);
\shade[shading=ball,ball color=black] (1,5) circle (.15);
\shade[shading=ball,ball color=lightgray] (9,1) circle (.24);
\shade[shading=ball,ball color=lightgray] (1,9) circle (.24);
\shade[shading=ball,ball color=lightgray] (13,9) circle (.24);
\shade[shading=ball,ball color=lightgray] (9,13) circle (.24);
\shade[shading=ball,ball color=white] (9,5) circle (.4);
\shade[shading=ball,ball color=white] (5,9) circle (.4);
\shade[shading=ball,ball color=white] (9,9) circle (.4);
\end{tikzpicture}
\caption{Unitarizability for $\Psi_{(x,y)}^\pm$ (case $\alpha=1$)} \label{fig: l1}
\end{figure}

\begin{figure}
\begin{tikzpicture} [thick, scale=0.6]
\draw[style=dotted,line width=1pt] (1,1) -- (12,1);
\draw[style=dotted,line width=1pt] (1,1) -- (1,12);
\draw[style=dotted,line width=1pt] (1,1) -- (12,12);
\draw[style=dashed,line width=1pt] (1,5) -- (5,1);
\draw[style=dashed,line width=1pt] (1,5) -- (8,12);
\draw[style=dashed,line width=1pt] (5,1) -- (12,8);
\draw[style=dashed,line width=1pt] (3,1) -- (3,12);
\draw[style=dashed,line width=1pt] (7,1) -- (7,12);
\draw[style=dashed,line width=1pt] (1,3) -- (12,3);
\draw[style=dashed,line width=1pt] (1,7) -- (12,7);
\draw (12,1) node[right] {$y=0$};
\draw (1,12.2) node[above] {$x=0$};
\draw (12,8) node[right] {$y=x-1$};
\draw (0,4.7) node[right] {$x+y=1$};
\draw (12,3) node[right] {$y=\frac12$};
\draw (12,7) node[right] {$y=\frac32$};
\draw (6.7,12) node[above] {$x=\frac32$};
\draw (3.4,12) node[above] {$x=\frac12$};
\draw (9.4,12) node[above] {$y=x+1$};
\draw (4,3) node[above] {$(\frac12,\frac12)$};
\draw (7.9,7) node[below] {$(\frac32,\frac32)$};
\draw (3.9,7) node[below] {$(\frac12,\frac32)$};
\draw (5.7,11.6) node[below] {$(\frac52,\frac32)$};
\draw (2,2.5) node[below] {$C_0$};
\draw (1.7,4.1) node[below] {$C_{11}$};
\fill [pattern=north west lines, pattern color=gray] (1,1) -- (1,3) -- (3,3) -- (3,1);
\draw[line width=2pt] (3,1) -- (3,3);
\draw[line width=2pt] (1,3) -- (3,3);
\draw[line width=1pt] (3,3) -- (5,1);
\draw[line width=1pt] (3,3) -- (1,5);
\draw[line width=1pt] (5,1) -- (7,3);
\draw[line width=1pt] (1,5) -- (3,7);
\draw[line width=1pt] (1,7) -- (3,7);
\draw[line width=1pt] (7,1) -- (7,3);
\draw[line width=1pt] (3,3) -- (7,3);
\draw[line width=1pt] (3,3) -- (3,7);
\shade[shading=ball,ball color=black] (3,3) circle (.15);
\shade[shading=ball,ball color=lightgray] (7,3) circle (.2);
\shade[shading=ball,ball color=lightgray] (3,7) circle (.2);
\shade[shading=ball,ball color=lightgray] (11,7) circle (.2);
\shade[shading=ball,ball color=lightgray] (7,11) circle (.2);
\shade[shading=ball,ball color=white] (7,7) circle (.33);
\end{tikzpicture}
\caption{Unitarizability for $\Psi_{(x,y)}^\pm$ (case $\alpha=\frac12$)} \label{fig: l12}
\end{figure}

\begin{figure}
\begin{tikzpicture} [thick, scale=0.5]
\draw[style=dotted,line width=1pt] (1,1) -- (10,1);
\draw[style=dotted,line width=1pt] (1,1) -- (1,10);
\draw[style=dotted,line width=1pt] (1,1) -- (10,10);
\draw[style=dashed,line width=1pt] (1,5) -- (5,1);
\draw[style=dashed,line width=1pt] (1,5) -- (6,10);
\draw[style=dashed,line width=1pt] (5,1) -- (10,6);
\draw[style=dashed,line width=1pt] (5,1) -- (5,10);
\draw[style=dashed,line width=1pt] (1,5) -- (10,5);
\draw (10,1) node[right] {$y=0$};
\draw (10,5) node[right] {$y=1$};
\draw (.6,10) node[above] {$x=0$};
\draw (4.2,10) node[above] {$x=1$};
\draw (10,6) node[right] {$y=x-1$};
\draw (7.2,10) node[above] {$y=x+1$};
\draw (6,5) node[below] {$(1,1)$};
\draw (5,1) node[below] {$(1,0)$};
\draw (6,9) node[below] {$(1,2)$};
\draw (2.3,2.8) node[below] {$C_0$};
\fill [pattern=north west lines, pattern color=gray] (1,1) -- (5,1) -- (1,5);
\draw[line width=2pt] (5,1) -- (1,5);
\draw[line width=1pt] (5,1) -- (5,5);
\draw[line width=1pt] (1,5) -- (5,5);
\shade[shading=ball,ball color=black] (5,1) circle (.15);
\shade[shading=ball,ball color=black] (1,5) circle (.15);
\shade[shading=ball,ball color=lightgray] (5,5) circle (.25);
\shade[shading=ball,ball color=lightgray] (9,5) circle (.25);
\shade[shading=ball,ball color=lightgray] (5,9) circle (.25);
\end{tikzpicture}
\caption{Unitarizability for $\Psi_{(x,y)}^\pm$ (case $\alpha=0$)} \label{fig: l0}
\end{figure}

In the case $\alpha=1$ the origin is not a regular point in $\R^2$.
However, $[\frac14]\times[\frac14]\rtimes\pi\cong ([-\frac14]\times[\frac14])\rtimes\pi$ is in \eqref{eq: 1axy1} (i.e. in $C_{10}$
of Figure \ref{fig: l1}) for $\pi=L([1];\sigma)$ or $\pi=\delta([1];\sigma)$. Therefore, this region is also unitary$^\pm$.

Suppose that $\alpha>2$.
Consider Figure \ref{fig: l3}.
The region \eqref{eq: >1x-x} (i.e. $C_7$ in Figure \ref{fig: l3}) is clearly a connected component of $\R^2_{\reglng,++}$
(not containing the origin).
For any $(x,y)\in\R^2_{\reglng,++}$ satisfying \eqref{eq: >1x-x} (where $\alpha>2$), the point
$(x,y,\alpha)$ lies in the boundary of \eqref{eq: xx1x2x14}. Hence, it is unitary$^\pm$ by Proposition \ref{prop: compserex}.

For the converse, as before, call the regions \eqref{eq: >1xy1}--\eqref{eq: 0xy1} above
(with the respective conditions on $\alpha$) \emph{special}. We proceed now in the same way as in the proof of exhaustion in Proposition \ref{prop: dim2unit}
(therefore, we shall not repeat all details from there). We need to prove that any non-shaded component is neither unitary$^+$ nor unitary$^-$.
Here, only the components $C_{9}$ in Figure \ref{fig: l32} and $C_{11}$ in Figure \ref{fig: l12} do not have a
non-black vertex (therefore all other non-shaded components in Figures \ref{fig: l3} -- \ref{fig: l0} are non-unitary).
We now show that also these two components are neither unitary$^+$ nor unitary$^-$, which will complete the proof of the proposition.

Let $\alpha=\tfrac32$. The slanted side of $C_{9}$ contains the family $\nu^{\frac12-x}\delta([-\frac12,\frac12])\rtimes \pi$, $0<x<\tfrac12$,
for either $\pi=\delta([\frac32];\sigma)$ or $\pi=L([\frac32];\sigma)$. By Proposition \ref{prop: unitlines} (use N$^\circ$2 and 3 in Table \ref{tab: redpnt})
both families consists of non-unitarizable representations. Therefore, $C_{9}$ is neither unitary$^+$ nor unitary$^-$.

Assume now $\alpha=\frac12$. Here the slanted side of this component contains either the family $\nu^{\frac12-x}\delta([-\frac12,\frac12])\rtimes L([\frac12];\sigma)$, $0<x<\tfrac12$,
or the family $\nu^{\frac12-x}L([-\frac12],[\frac12])\rtimes \delta([\frac12];\sigma)$, $0<x<\tfrac12$.
Both families consists of non-unitarizable representations by Proposition \ref{prop: unitlines} (use N$^\circ$3 in Table \ref{tab: redpnt}),
which implies that $C_{11}$ is neither unitary$^+$ nor unitary$^-$.
\end{proof}

\begin{corollary} \label{cor: 2dmlng}
The unitary$^\pm$ connected components of $\R^2_{\vreglng,++}$ are as follows.
\begin{enumerate}
\item $\alpha>1:$
\begin{gather*}
x+y<1,\ y<\alpha-1,\\
y-x>1,\ y<\alpha-1,
\end{gather*}
where the constraint $y<\alpha-1$ in the first region is redundant for $\alpha\ge2$
and the second region is empty if $\alpha\le 2$.
\item $\alpha=1:$
\[
x+y<1
\]
\item $\alpha=\tfrac12:$
\[
y<\tfrac12
\]
\item $\alpha=0:$
\[
x+y<1,\ x>0.
\]
\end{enumerate}
The other connected components of $\R^2_{\vreglng,++}$ are neither unitary$^+$ nor unitary$^-$.
\qed
\end{corollary}

\section{Three-parameter complementary series}

In the proof of Proposition \ref{prop: 3dimex} below we use the following two simple lemmas.

\begin{lemma} \label{edge-simple}
\begin{enumerate}
\item
\label{edge-simple-item: =}
Let $\alpha > 1$.
Suppose that $\mathscr C$ is a bounded connected component of $\R^3_{\reg}$ whose boundary contains a segment
$\overline{AB}$, where $A=(a_1,a_2,a_3)$ and $B=(b_1,b_2,b_3)$ are two distinct points such that
there exist an index $k$ and $\epsilon\in\{\pm1\}$ for which $a_k=b_k=\epsilon \alpha$.
Then, there exists a two-dimensional face of the closure $\overline{\mathscr C}$ of $\mathscr C$ that is contained in a level hyperplane.

\item
\label{edge-simple-item: 678}
 Suppose that a connected component $C$ of $\R^3_{\reg,++}$ contains in its boundary at least one of the sets
$\iota(C_6')$, $\iota(C_7')$ or $\iota(C_8')$ (see Lemma \ref{lem: 2dimrreg}).
Then, $\overline{C}$ admits a two-dimensional face that is contained in a level hyperplane.
\end{enumerate}
\end{lemma}

\begin{proof}
Suppose that $\mathscr C$ is a component as in (\ref{edge-simple-item: =}).
Clearly, $\overline{AB}$ is contained in either an edge or the relative interior of a two-dimensional face of $\overline{\mathscr C}$.
In the second case, there is nothing to prove. Therefore, suppose that $\overline{AB}$ is contained in an edge of $\overline{\mathscr C}$.
Clearly, it is enough to prove the lemma for $\epsilon=1$.

Suppose on the contrary that no two-dimensional face of $\overline{\mathscr C}$ is contained in a level hyperplane.
Then, the edge containing $\overline{AB}$ must be contained in the intersection of two different (non-parallel) slanted hyperplanes
given by the following equations
$$
\epsilon_1x_i+\epsilon_2x_j=1,\quad \epsilon_3x_r+\epsilon_4x_s=1, \quad \epsilon_l\in\{\pm1\},\ \ i,j,r,s\in\{1,2,3\}, \ \ i\ne j, \ r\ne s.
$$
Suppose $\{i,j\}=\{r,s\}$. Then, $(\epsilon_i,\epsilon_j)\ne \pm (\epsilon_r,\epsilon_s)$. Therefore, we are left with the hyperplanes $\epsilon_1x_i+\epsilon_1x_j=1$
and $\epsilon_1x_i-\epsilon_1x_j=1$. If $k\not\in \{i,j\}$ then $A=B$, which is a contradiction. Therefore, $k\in \{i,j\}$. If $k=i$, then $\epsilon_1\alpha=1$,
which cannot be. If $k=j$, then $\epsilon_1\alpha=0$, which again is not possible.

Therefore, $\{i,j\}\ne\{r,s\}$. Denote $\{i,j\}\cap\{r,s\}=\{t\}.$
Now hyperplanes are determined by equations
$$
\epsilon_1x_i+\epsilon_2x_t=1,\quad \epsilon_3x_t+\epsilon_4x_s=1, \quad \epsilon_l\in\{\pm1\},\ \ \{i,t,s\}=\{1,2,3\}.
$$
Suppose that $k=i$ or $k=s$. Then, $a_t=b_t$, and further $a_s=b_s$, which is a contradiction.
Similarly, if $k=t$, then again $A=B$, which is a contradiction.
This completes the proof (\ref{edge-simple-item: =}). Further, (\ref{edge-simple-item: 678}) follows directly from (\ref{edge-simple-item: =}).
\end{proof}

\begin{remark}
By the same reasoning, the closure of the connected components \eqref{eq: xx1x212} and \eqref{eq: xx1x2x14} admit two-dimensional faces
which are contained in level hyperplanes. The same is true for \eqref{eq: xx1x213} if $\alpha>1$.
\end{remark}

\begin{lemma} \label{lemma-C-23-}
Suppose that $\mathscr C$ is a connected component of $\R^3_{\reg}$ which has in its boundary $\iota(C_2')$ (with $\alpha>1$) or
$\iota(C_3')$. Assume further that $\mathscr C$ lies in the connected component of $\R^3\backslash H_{\shrt}$ that contains the origin.
Then, $\mathscr C$ is not a unitary component.
\end{lemma}

\begin{proof}
We use below the transforms $w_2:(x_1,x_2,x_3)\mapsto (x_2,x_3,x_1)$ and $w_3:(x_1,x_2,x_3)\mapsto (-x_1,x_3,x_2)$ that were introduced in
the proof of Lemma \ref{lem: 2dimrreg}.

Suppose first that $\iota(C_2')$ lies in the boundary of $\mathscr C$ (and $\alpha >1$).
Denote by $Y_2$ the non-empty convex open subset of $\R^3_{++}$ defined by the following inequalities:
$$
x_1+x_2< 1, \quad x_3-x_2< 1, \quad x_3-x_1>1,\quad x_3<\alpha.\footnote{The last condition is automatically satisfied if $\alpha\geq2$.}
$$
It is straightforward to verify that $Y_2\subset\R^3_{\reg}$.
However, $Y_2$ is not a unitary component since for instance $(0,1,\frac54)=\iota(\frac12,\frac54)\in\partial(Y_2)\cap H_{\shrt}$
and $(\frac12,\frac54)\in\R^2_{\regshrt,+}$ is not a unitary point by Proposition \ref{prop: dim2unit} (see also Figures \ref{fig: s3} and \ref{fig: s32}). Hence $w_2(Y_2)$ is not unitary as well.
However, it is immediate to see that the boundary of $w_2(Y_2)$ contains $C_2'$.
Since $w_2(Y_2)$ and $\mathscr C$ are on the same side of $H_{\shrt}$ we conclude that $w_2(Y_2)\subset\mathscr C$.
Hence $\mathscr C$ is not unitary.

\smallskip

Similarly, suppose that $\iota(C_3')$ lies in the boundary of $\mathscr C$.
Denote by $Y_3$ the nonempty open convex subset of $\R^3_{++}$ consisting of the points $(x_1,x_2,x_3)$ that satisfy
$$
x_1+x_3< 1, \quad x_3-x_2<1, \quad 1<x_2+x_3, \quad x_3<\alpha.
$$
One checks that $Y_3\subset\R^3_{\reg}$.
Moreover, $Y_3$ is not unitary since it contains for instance the point $(\frac15,\frac35,\frac35)$ and if this point were unitary
then by unitary parabolic reduction we would get that the representation $[\frac35]\times[-\frac35]$ is unitarizable, which is a contradiction (since $\frac12<\frac35$).

Once again, we check that $w_3(Y_2)$ contains $\iota(C'_3)$ in its boundary, and hence
$w_3(Y_2)\subset \mathscr C$ since $w_3(Y_2)$ and $\mathscr C$ lie on the same side of $H_{\shrt}$.
Hence $\mathscr C$ is not unitary as required.
\end{proof}

\begin{proposition} \label{prop: 3dimex}
The list of unitary connected components of $\R^3_{\reg,++}$ in Proposition \ref{prop: compserex} is exhaustive.
In particular, if $\alpha=0$ there are no unitary connected components.
\end{proposition}

\begin{proof}
Let $C$ be a unitary connected component of $\R^3_{\reg,++}$.
Denote by $\mathscr C$ the (unitary) connected component of $\R^3_{\reg}$ containing $C$.
Recall that $C$ (and also $\mathscr C$) must be bounded, and the boundary of $\mathscr C$ is contained
in the union of all reducibility hyperplanes.
Since all slanted (resp. level) hyperplanes are in the same $W$-orbit,
any two-dimensional face of the boundary of $\mathscr C$ is contained in a $W$-translate of either
$H_{\shrt}$ or $H_{\lng}$.

Consider first the case $\alpha=0$.
Clearly, $\mathscr C$ cannot be bounded only by level hyperplanes (otherwise, $\mathscr C$ would not be bounded).
Therefore, there exists a two-dimensional face of the boundary of $\mathscr C$ that lies in a slanted hyperplane.
This would contradict Proposition \ref{prop: dim2unit} (see Figure \ref{fig: s0}).
Therefore, we do not have unitary connected components in this case.

For the rest of the proof, we consider the case $\alpha>0$.
We first note that
\[
\mathscr C\subseteq\{\mathbf{x}\in \R^3_{\reg}:\abs{x_i}<\alpha,i=1,2,3\}
\]
and
$$
C\subseteq\{\mathbf{x}\in \R^3_{\reg,++}: x_3<\alpha\}.
$$
This follows immediately from Propositions \ref{prop: dim2unit} and \ref{prop: 2dmlng} by passing
to the boundary.\footnote{Observe that in $(x,y)$-planes, $x<\alpha-\frac12$ in the unitary two-dimensional slanted components.}

In particular, in the case $\alpha=\frac12$ we conclude that $C$ is the component \eqref{eq: x312}.

For the rest of the proof we consider the case $\alpha\geq 1$.

Suppose that a two-dimensional face of the boundary of $\mathscr C$ lies in a level hyperplane.
Since $\mathscr C\supset C$ we may assume that this level hyperplane is $H_{\lng}$ itself.
We therefore need to consider unitary$^\pm$ connected components of $H^{\circ\circ}_{\lng,++}$
(which correspond to $\R^2_{\vreglng,++}$).

Consider first the case $\alpha>1$. We have two possibilities for unitary$^\pm$ connected components.
The first possibility is $x+y<1, y<\alpha-1$, which clearly lies in the closure of the component \eqref{eq: xx1x212}.
Moreover, \eqref{eq: xx1x212} is the only component that contains \eqref{eq: >1x-x} in its closure
and that is contained in $\{\mathbf{x}:x_3<\alpha\}$.
We conclude that $C$ must be the component \eqref{eq: xx1x212}.
The second possibility is $y-x>1, y<\alpha-1$ (when $\alpha>2$).
Again, \eqref{eq: xx1x2x14} is the only component that contains \eqref{eq: >1xy1} in its closure
and is contained in $\{\mathbf{x}:x_3<\alpha\}$. Thus, $C$ must be the component \eqref{eq: xx1x2x14}.

Consider now the case $\alpha =1$.
Then, we have only one possibility for unitary$^\pm$ connected component: $x+y<1$.
This is in the closure of \eqref{eq: xx1x213}. In the same way as before, we conclude that $C$ must be the component \eqref{eq: xx1x213}.

It remains to consider the case where all the two-dimensional faces of the boundary of $\mathscr C$ are contained in
slanted hyperplanes. Take a slanted hyperplane $H$ such that $C$ has a two-dimensional face in it.
Since $C$ is bounded, we can find $H$ such that $C$ is in the same half-space of $\R^3\setminus H$ as the origin.
Let $w\in W$ be such that $w(H)=H_{\shrt}$.
Then, $w(C)$ has a two-dimensional face in $w(H)=H_{\shrt}$, and $w(C)$ is in the same half-space of
$\R^3\setminus H_{\shrt}$ as the origin.
Furthermore, acting by $W_{\shrt}$ (if necessary), we conclude that the following holds:

\begin{enumerate}
\item[{(Int) }]
There exists unique $w\in W$ such that $w(\partial C)\cap H_{\shrt,+}$ is the closure of a unitary
connected component $C'$ of $H_{\shrt,+}^{\circ\circ}$ and
\begin{equation}
\label{decide-side}
(w\mathbf{x})_2- (w\mathbf{x})_1<1
\end{equation}
for all $\mathbf{x}\in C$.
\end{enumerate}

Consider $w\mathbf{x}\in C'$ where $\mathbf{x}\in \partial C$. Then,
\begin{equation} \label{all8}
w(\mathbf{x})=\iota(x,y)=(x-\tfrac12,x+\tfrac12,y)
\end{equation}
for some $ (x,y)\in \R^2_{\vregshrt,+}$.
Now $C'$ is one out of the 8 unitary components $C_1',C_2',\dots,C_8'$ of $\R^2_{\vregshrt,+}$.
Since by our assumption $\mathscr C$ does not have two-dimensional faces contained in level hyperplanes,
Lemma \ref{edge-simple} implies that $C'$ must be one of the 5 unitary components $C_1',\dots,C_5'$.
We examine these 5 cases below.

Denote the components
\eqref{eq: xx2x31},
\eqref{eq: xx1x212},
\eqref{eq: xx1x2x14},
\eqref{eq: xx1x2x14},
by $C_a,C_b, C_c,C_d$ respectively in the sequel, as we did in \S\ref{sec: Cabcd}.
Further, we use below the elements $w_1,\dots,w_5\in W$ defined in the proof of Lemma \ref{lem: 2dimrreg}.

Suppose $C'=C'_i$ for some index $i\in\{1,4,5\}$. Then, $w(C)=w_i(X_i)$, where $X_i\in \{C_a,C_b,C_c,C_d\}$ by Lemma \ref{lem: 2dimrreg}. Therefore, the proposition holds in this case.

Consider the remaining case when $C'=C'_i$ for some $i\in\{2,3\}$. Then, Lemma \ref{lemma-C-23-} implies that $C$ is not unitary, which is a contradiction.
This completes the proof of the proposition.
\end{proof}

\section{Conclusion}

\begin{theorem}
\label{class-thm}
The irreducible unitarizable subquotients of $\Pi_{\mathbf{x}}$ when $\mathbf{x}=(x_1,x_2,x_3)\in\R^3_{++}$, are the following.
\begin{enumerate}

\item 
\label{class-thm-item: 3dim-geq1}
$(\alpha\ge1)$ All irreducible subquotients when $\mathbf{x}$ lies in the closure of one of the domains
\eqref{eq: xx2x31}--\eqref{eq: xx1x2x14}.

\item 
\label{class-thm-item: 3dim-1/2}
$(\alpha=\tfrac12)$ All irreducible subquotients when $x_3\le\tfrac12$.

\item
\label{class-thm-item: 2dim-0}
 $(\alpha=0$) All irreducible subquotients when $x_1=0$, $x_2+x_3\le1$.
 
\item
\label{class-thm-item: St+inv}
 $(\alpha>0)$ The representations $\delta([\alpha,\alpha+2];\s)$ and $L([\alpha],[\alpha+1],[\alpha+2];\s)$.

\item 
\label{class-thm-item: 2ds0+inv}
$(\alpha=0)$ The representations $\delta([0,2]_\pm;\s)$ and $L([1],[2];\delta([0]_\pm;\s))$.

\item $(\alpha>0)$ 
\label{class-thm-item: cs-pos}
The complementary series $[x]\rtimes \delta([\alpha,\alpha+1];\s)$ and
$[x]\rtimes L([\alpha],[\alpha+1];\s)$ for $0\le x<\abs{\alpha-1}$ (if $\alpha\ne1$) and its irreducible subquotients for
$x=\alpha-1$.

\item
\label{class-thm-item: 3dim-0}
 $(\alpha=0)$ The complementary series (including subquotients at the ends)
\[
[x]\rtimes\delta([0,1]_\pm;\s),\ [x]\rtimes L([1];\delta([0]_\pm;\s)),\ \ 0\le x\le1.
\]

\item
\label{class-thm-item: artur-moe}
 $(\alpha>1)$ The representation $L([\alpha-1],[\alpha];\delta([\alpha];\s))$.

\item 
\label{class-thm-item: cs12}
$(\alpha=\tfrac12)$ The complementary series (including subquotients at the ends)
\begin{gather*}
\delta([x-\tfrac12,x+\tfrac12])\rtimes\delta([\tfrac12];\s),\ L([x-\tfrac12],[x+\tfrac12])\rtimes L([\tfrac12];\s),\ \ 0\le x\le1\\
[x]\rtimes\delta([-\tfrac12,\tfrac12]_-;\s),\ [x]\rtimes L([\tfrac12];\delta([\tfrac12];\s)),\ \ 0\le x\le\tfrac32\\
\delta([x-1,x+1])\rtimes\s,\ L([x-1],[x],[x+1])\rtimes\s,\ \ 0\le x\le\tfrac12.
\end{gather*}
\end{enumerate}
\end{theorem}

\begin{proof}
The unitarity of the above representations follows from Proposition \ref{prop: compserex},
\eqref{eq: 0xy1} in Proposition \ref{prop: 2dmlng}, Proposition \ref{prop: unitlines} (N$^\circ$5, 6, 2, 13 and 1 in Table \ref{tab: redpnt})
and the unitarity in the critical cases dealt with in the previous chapters (where we noted that the irreducible square-integrable representations and its $\ASS$ duals are unitary,
together with the representation in (\ref{class-thm-item: artur-moe}) of the theorem). It remains to prove the exhaustion.

Suppose that $\Pi_{\mathbf{x}}$, $\mathbf{x}=(x_1,x_2,x_3)\in\R^3_{++}$ admits a unitarizable irreducible subquotient.
For $\mathbf{x}\in\R^3_{\reg}$ by Proposition \ref{prop: 3dimex}, $\alpha>0$ and $\mathbf{x}$ belongs to
one of the regions \eqref{eq: xx2x31}--\eqref{eq: xx1x2x14} if $\alpha\ge1$
and $x_3<\tfrac12$ if $\alpha=\tfrac12$.

Before we proceed further, we observe below that all the irreducible subquotients of the representations listed in Table \ref{tab: redpnt} and corresponding to the parameter $x$
between $0$ and the first reducibility point (including the end points), are listed in the above theorem.
First, such representations corresponding to N$^\circ$5 (resp. N$^\circ$6, resp. N$^\circ$13) are contained in (\ref{class-thm-item: cs-pos}) (resp. (\ref{class-thm-item: 3dim-0}), resp. (\ref{class-thm-item: cs12}))
of the theorem.
Consider now such representations in the case of N$^\circ$1. For $\alpha>1$, that representations are contained in the set of representations corresponding to the closure of $C_d$
(resp. $C_b$, resp. $C_c$) if $x> 1$ (resp. $\frac12\leq x\leq1$, resp. $0\leq x\leq \frac12$). For $\alpha=1$ such representations are contained in the set of representations
corresponding to the closure of $C_a$. Therefore, they are all contained in (\ref{class-thm-item: 3dim-geq1}) of the theorem. For $\alpha=\frac12$ (resp. $\alpha=0$),
such representations are contained in (\ref{class-thm-item: cs12}) (resp. (\ref{class-thm-item: 2dim-0})) of the theorem.
In the case of N$^\circ$2, such representations are contained in the set of representations corresponding to the closure of $C_b$ (resp. $C_c$) if $\alpha>1$ (resp. $\alpha=1$).
For $\alpha=\frac12$, they are contained in (\ref{class-thm-item: cs12}) of the theorem.
For N$^\circ$3, they are contained in the set of representations corresponding to the closure of $C_b$ (resp. (\ref{class-thm-item: 3dim-1/2}) of the theorem) if $\alpha\geq 1$ (resp. $\alpha=\frac12$).
The representations of N$^\circ$4 and N$^\circ$7 are contained in (\ref{class-thm-item: 2dim-0}) of the theorem.
In the case of N$^\circ$8 and N$^\circ$9, such representations are contained in the set of representations corresponding to the closure of $C_d$ (resp. $C_b$) if $\alpha>2$
(resp. $1<\alpha\leq2$). For N$^\circ$10 and N$^\circ$11, they are contained in the set of representations corresponding to the closure of $C_c$, for N$^\circ$12,
they are contained in (\ref{class-thm-item: 3dim-1/2}) of the theorem and for N$^\circ$14, they are contained in (\ref{class-thm-item: 2dim-0}) of the theorem.

Suppose that the $W$-orbit of $\mathbf{x}$ intersects $H_{\shrt}^{\circ}$. Then,
by Proposition \ref{prop: dim2unit} and Lemma \ref{lem: 2dimrreg}, $\alpha\ge1$ and $\mathbf{x}$ belongs to the closure of
one of the regions \eqref{eq: xx2x31}--\eqref{eq: xx1x2x14}.
Similarly, if the $W$-orbit of $\mathbf{x}$ intersects $H_{\lng}^{\circ}$ then
by Proposition \ref{prop: 2dmlng}, either:
\begin{enumerate}
\item $(\alpha>1)$ $\mathbf{x}$ belongs to the closure of \eqref{eq: xx1x212} or \eqref{eq: xx1x2x14},
\item $(\alpha=1)$ $\mathbf{x}$ belongs to the closure of \eqref{eq: xx1x213},
\item $(\alpha=\tfrac12)$ $\mathbf{x}$ belongs to the closure of \eqref{eq: x312},
\item $(\alpha=0)$ $x_1=0$ and $x_2+x_3<1$.
\end{enumerate}
Suppose now that $\mathbf{x}$ is not regular but the $W$-orbit of $\mathbf{x}$ intersects neither $H_{\shrt}^{\circ}$
nor $H_{\lng}^\circ$. Let $\pi\in\JH(\Pi_{\mathbf{x}})$.
Consider first the case when $\mathbf{x}$ is not critical. Then, there exists $x\ge0$ such that one of the following options holds.
\begin{enumerate}
\item 
\label{item: tau-siegel-x-1xx+1}
$\pi=\tau\rtimes\s$ where $\tau\in\JH([x-1]\times[x]\times[x+1])=$
\[
\{\delta([x-1,x+1]),L([x-1],[x],[x+1]),L([x-1,x],[x+1]),L([x-1],[x,x+1])\}.
\]

\item 
\label{item: midle-par}
$\pi\le\tau'\rtimes\pi'$ where
\begin{gather*}
\tau'\in\JH([x-\tfrac12]\times[x+\tfrac12])=
\{\delta([x-\tfrac12,x+\tfrac12]),L([x-\tfrac12],[x+\tfrac12])\}\\
\text{and }\pi'\in\JH([\alpha]\rtimes\s)=\begin{cases}\{\delta([\alpha];\s),L([\alpha];\s)\}&
\alpha>0\\\{\delta([0]_\pm;\s)\}&\alpha=0\end{cases}.
\end{gather*}

\item
\label{item: top-St}
 $\pi=[x]\rtimes\pi'$ where $\pi'\in\JH([\alpha]\times[\alpha+1]\rtimes\s)$.

\item
\label{item: top-st-pos-no-St}
 $(\alpha\ne0)$ $\pi=[x]\rtimes\pi'$ where $\pi'\in\JH([\alpha-1]\times[\alpha]\rtimes\s)$.
\end{enumerate}
The unitarity in the cases (\ref{item: tau-siegel-x-1xx+1}) when $\tau$ is unitarizable, (\ref{item: midle-par}) and (\ref{item: top-St}) above, as well as in the case (\ref{item: top-st-pos-no-St}) if $\pi'$ is unitarizable, was dealt with in Proposition \ref{prop: unitlines}
(Table \ref{tab: redpnt}), and there is determined when one has complementary series. We have seen above that these complementary series all appear in the irreducible unitarizable
representations listed in the above theorem. It remains to consider only the case of non-unitarizable $\tau$ and $\pi'$ as above.
The unitarity in these cases was dealt with in Lemma \ref{lemma-nu-m}. Of these two cases, unitarizable representations show up only for $\tau$ when
$\alpha>1$ and $0\leq x\leq \alpha-1$. Obviously, all these representations show up in the group (\ref{class-thm-item: 3dim-geq1}) in the theorem (corresponding to $C_d)$.

It remains to consider the critical points (the unitarity corresponding to these points was dealt with in the chapters \ref{sec: unit3}, \ref{basic 1} and \ref{basic 0}).
Unitarizability of the irreducible subquotients there,
excluding the representations listed in (\ref{class-thm-item: St+inv}), (\ref{class-thm-item: 2ds0+inv}) and (\ref{class-thm-item: 3dim-0}), was proved applying the unitary parabolic induction from a maximal parabolic subgroup, or proving that
they show up in the ends of complementary series starting from maximal parabolic subgroups (listed in Proposition \ref{prop: unitlines}).
For both methods of the proof, our above checking that
all the irreducible subquotients of the representations listed in Table \ref{tab: redpnt} and corresponding to the parameter $x$ between $0$ and the first reducibility point
(including the end points) are listed in the above theorem,
implies that all the irreducible subquotients for which we have proved in chapters \ref{sec: unit3}, \ref{basic 1} and \ref{basic 0}
that they are unitarizable, are listed in the above theorem.
For the remaining irreducible subquotients considered in chapters \ref{sec: unit3}, \ref{basic 1} and \ref{basic 0}, we proved there that they are not unitarizable.
Therefore, any irreducible unitarizable subquotient in the case of critical points is listed in the above theorem.
This finishes the proof of the theorem.
\end{proof}

\begin{remark}
Let $\tau:=\rho\otimes\rho\otimes\rho\otimes\sigma$ be a representation of a Levi subgroup $M$ of a classical group $G$.
Then, the representations listed in (\ref{class-thm-item: St+inv}), (\ref{class-thm-item: 2ds0+inv}) and (\ref{class-thm-item: artur-moe}) of the above theorem are isolated in $\hat G$, and the other listed in the theorem are not.
They are in the part $\hat G_\Omega$ of $\hat G$ corresponding to the Bernstein component $\Omega$ that contains the conjugacy class of $(M,\tau)$ (see \cite{MR771671} and \cite{MR963153} for details).
The cardinality of these representations is $4$, $2$ and $3$ if $\alpha=0$, $\alpha\in\{\tfrac12,1\}$ and $\alpha>1$ respectively,
but we have always more isolated representations in $\hat G_\Omega$.

There exists an unramified character $\chi$ of a general linear group such that $\rho':=\chi\rho$ is $F'/F$-selfdual and $\chi\rho\not\cong\rho$.
The isolated representations in $\hat G_\Omega$ are precisely the isolated representations corresponding $\rho,\sigma$ and $\rho',\sigma$ in the above theorem.
Therefore, the number of isolated representations in $\hat G_\Omega$ is between $4$ and $8$ ($\hat G_\Omega$ is an open subset of $\hat G$).
Actually, section 8 of \cite{MR3156864} implies that all numbers between $4$ and $8$ can show up as these cardinalities\footnote{Having in mind \cite{MR2742213}
(where the number of isolated unitarizable spherical representations for split classical groups is calculated), it is natural to expect that the number
of isolated representations in a component determined by $\underbrace{\rho\otimes\dots\otimes\rho}_{n-\text{times}}\otimes \sigma$ grows rapidly as $n\rightarrow\infty$.}.
As an example, observe that the component of irreducible unitarizable representations with non-trivial Iwahori-fixed vector of $Sp(6,F)$ (resp. $SO(7,F)$) has $6$ (resp. $4$) isolated representations.

If we have a component $\hat G_\Omega$ where $\Omega$ is not of the type that we just considered (in corank 3), then it does not have isolated representations.
\end{remark}

\section{Conjectures} \label{sec: conjectures}
We end this chapter by stating two conjectures, motivated by the results of \cite{MR2046512}, \cite{MR2767523} and this paper (see also \cite{MR3969882}).

\begin{conjecture} \label{conj: isoaut}
Suppose that $\pi$ is an isolated representation in the unitary dual of a split classical group $G$.
Then, $\pi$ is automorphic\footnote{See \cite{MR2331344} (and also \cite{MR2742213})}.
\end{conjecture}

By \cite{MR2742213}, assuming Arthur + $\epsilon$ conjecture from \cite{MR2331344},
there are many irreducible representations that are isolated in the automorphic dual.
All these spherical representations are subquotients in critical points.

The papers \cite{MR2046512}, \cite{MR2767523}, \cite{MR2742213} and the present paper give some evidence for the following

\begin{conjecture}
 Suppose that $\pi$ is an isolated representation in the unitary dual of a split classical group $G$.
Then, $\pi$ is a subquotient of a representation of critical type.
\end{conjecture}

We can slightly extend the notion of critical point.
\begin{definition}
Let $\rho_1,\dots,\rho_k\in\Cusp$ and $\sigma\in\Cuspcl$.
Assume that for all $i$, $\rho_i^u\cong (\rho_i^u)\check{\ }$ and the set
$$
\{e(\rho_j):\rho_j^u\cong \rho_i^u\}
$$
is a $\Z$-segment in $\frac12\Z$ (possibly with multiplicities) that contains the reducibility point $\alpha_{ \rho_i^u,\sigma}$.
Then, we say that the representation $\rho_1\times\dots\times \rho_k\rtimes\sigma$ is of critical type.
\end{definition}

We would expect that any irreducible unitarizable subquotient of a representation of critical type of a split classical group is automorphic.
Very preliminary results in this direction are the subject matter of a work in progress.

\chapter{Unitarizability in Mixed Case for Corank \texorpdfstring{$\leq 3$}{Lg}} \label{mixed-case}

In this chapter we shall use notation and terms introduced in sections 8 and 9 of \cite{MR3969882} regarding
Jantzen decomposition of an irreducible representation of a classical $p$-adic group We shall recall
some of the most basic definitions. See section 8 of \cite{MR3969882} for more details.

\section{Jantzen decomposition}

\subsection{Support of representation of classical group}
Let $X \subseteq \Cusp$ and suppose that $X$ is $F'/F$-selfcontragredient, i.e. that
$$
\check X=X,
$$
where $\check X=\{\check\rho;\rho\in\ X\}$, and let $\s\in\Cuspcl$.
We say that $\g\in\Irrcl$ is supported in $X\cup\{\s\}$ if there exist $\rho_1,\dots,\rho_k\in X$
(not necessarily distinct) such that
$$
\g\leq \rho_1\times\dots\times\rho_k\rtimes\s.
$$
We denote by $\Irrcl_{X\cup\{\s\}}$ the set of representations in $\Irrcl$ supported in $X\cup\{\s\}$.
For a not-necessarily irreducible representation $\pi$ of a classical group, one says that it is supported on $X\cup\{\s\}$
if each irreducible subquotient of it is supported on that set.

\subsection{Regular partition} Let
$$
X=X_1\cup X_2
$$
be a partition of an $F'/F$-selfcontragredient $X\subseteq \Cusp$. We shall say that this partition is regular if $X_1$ (and hence also $X_2$)
is $F'/F$-selfcontragredient,
and if among $X_1$ and $X_2$ there is no reducibility, i.e. if $\rho_1\times\rho_2$ is irreducible for all $\rho_1\in X_1$ and $\rho_2\in X_2$.

\subsection{Decomposition}
Let $\pi\in\Irrcl_{X\cup\{\s\}}$, where $X$ is $F'/F$-selfcontragredi\-ent, and let $X=X_1\cup X_2$ be a regular partition of $X$.
Fix $i \in \{1, 2 \}$.
Then, there exists $\beta\in\Irr$ supported in $X_{3-i}$ and $\g\in\Irrcl_{X_i\cup\{\s\}}$ such that
\[
\pi \hookrightarrow \beta\rtimes \g.
\]
Moreover,
$
\g
$
is uniquely determined by the above requirement. It is denoted by
$$
X_i(\pi)
$$
and called the Jantzen component of $\pi$ corresponding to the member $X_i$ in the regular partition $X=X_1\cup X_2$.
Furthermore, let $X\subset \Cusp$ be such that $\Cusp=X\cup (\Cusp\setminus X)$ is a regular partition of $\Cusp$ and let
$\pi\in\Irrcl$. Then,
$
X(\pi)
$
is defined as above for the regular partition $\Cusp=X\cup (\Cusp\setminus X)$
(and it is called again the Jantzen component of $\pi$ corresponding to $X$).

For $\rho\in\Cuspsd$, denote $X_\rho=\{\nu^x\rho:x\in \R\}$.
Let $\pi\in\Irrcl$ be weakly real.
Then, there exist finitely many distinct $\rho_1,\dots,\rho_k\in\Cuspsd$ and $\s\in\Cuspcl$ such that the support of $\pi$ is in
$X_{\rho_1}\cup\dots\cup X_{\rho_k}\cup\{\s\}$. The representations
$$
(X_{\rho_1}(\pi), \dots,X_{\rho_k}(\pi))
$$
determine $\pi$, and this defines a bijection
\[
\Irrcl_{X_{\rho_1}\cup\dots\cup X_{\rho_k}\cup\{\s\}}\rightarrow\prod_{i=1}^k\Irrcl_{X_{\rho_i}\cup\{\s\}}.
\]
The inverse map is denoted by
$$
\Psi_{X_{\rho_1},\dots,X_{\rho_k}}.
$$
The correspondence $\pi \mapsto (X_{\rho_1}(\pi), \dots,X_{\rho_k}(\pi))$ has a number of very nice properties
(see \cite{MR1481814} or section 8 of \cite{MR3969882}).
We shall now prove one additional very simple property which we shall use often.

\begin{lemma} \label{J-ind}
Let $X$ be an $F'/F$-selfcontragredient subset of $\Cusp$, and let $X=X_1\cup X_2$ be a regular partition of $X$.
Let $\theta_i\in\Irr$ supported in $X_i$ and $\pi_i\in\Irrcl_{X_i\cup\{\s\}}$, $i=1,2$.
Suppose that $\theta_i\rtimes\pi_i$, $i=1,2$ are both irreducible.
Then,
\begin{equation} \label{JDI}
\Psi_{X_1,X_2}(\theta_1\rtimes\pi_1,\theta_2\rtimes\pi_2)\cong \theta_1\times \theta_2\rtimes \Psi_{X_1,X_2}(\pi_1,\pi_2).
\end{equation}
\end{lemma}

\begin{proof}
Note that $\theta_1\times \theta_2\rtimes \Psi_{X_1,X_2}(\pi_1,\pi_2)$ is irreducible by (1) of Remark 8.9 of \cite{MR3969882}.

By the definition of $\Psi_{X_1,X_2}(\pi_1,\pi_2)$, we know that $\Psi_{X_1,X_2}(\pi_1,\pi_2)\h \tau \rtimes\pi_1$,
where $\tau $ is irreducible and supported on $X_2$. This implies $\tau\times\theta_1\cong\theta_1\times\tau$. Therefore,
$$
\theta_1\times \theta_2\rtimes \Psi_{X_1,X_2}(\pi_1,\pi_2)\h \theta_1\times \theta_2\rtimes\tau \rtimes\pi_1\cong
\theta_2\times\tau \times\theta_1 \rtimes\pi_1.
$$
If $\theta_2\times\tau$ is not irreducible, we can easily show that it admits an irreducible subquotient $\varphi$ such that
$\theta_1\times \theta_2\rtimes \Psi_{X_1,X_2}(\pi_1,\pi_2)\h \varphi \times\theta_1 \rtimes\pi_1$.
Therefore, $X_1(\theta_1\times \theta_2\rtimes \Psi_{X_1,X_2}(\pi_1,\pi_2))\cong \theta_1 \rtimes\pi_1$.
Analogously $X_2(\theta_1\times \theta_2\rtimes \Psi_{X_1,X_2}(\pi_1,\pi_2))\cong \theta_2 \rtimes\pi_2$. This proves \eqref{JDI}.
\end{proof}

\section{Preservation of unitarizability by decomposition in corank \texorpdfstring{$\leq 3$}{leq3}}

\begin{lemma}
Let $\pi$ be a weakly real irreducible subquotient of $\theta_1\times\dots\times\theta_k\rtimes\s$, where $\theta_i\in\Cusp $ and
$$
k\leq 3.
$$
Suppose that all $X_{\rho_i}(\pi)$ in the Jantzen decomposition of $\pi$ are unitarizable. Then, $\pi$ is unitarizable.
\end{lemma}

\begin{proof}
For $k=1$, there is nothing to prove.
We shall now prove the case $k=3$. (The case $k=2$ is easier, and will be omitted.)

Let $\pi \mapsto (X_{\rho_1}(\pi), \dots,X_{\rho_\ell}(\pi))$ be the Jantzen decomposition of $\pi$.
For the proof, we consider only those $\rho_i$ for which
$X_{\rho_i}(\pi)\ne \s$. We assume this for the rest of the proof.
Denote
$$
\alpha_i=\alpha_{\rho_i,\s}.
$$
If $\ell=1$, the claim obviously holds (since then $\pi=X_{\rho_1}(\pi)$).

Consider first the case $\ell=3$. Then, $\pi$ is a subquotient of $[x_1]^{(\rho_1)}\times[x_2]^{(\rho_2)} \times [x_3]^{(\rho_3)}\rtimes\s$,
where $x_i\geq0$.
Then, $X_{\rho_i}(\pi)$ are irreducible subquotients of $[x_i]^{(\rho_i)}\rtimes\s$.
Since $X_{\rho_i}(\pi)$ are unitarizable, then we know $x_i\leq\alpha_i$, $1\leq i\leq 3$.
But then each irreducible subquotient of $[x_1]^{(\rho_1)}\times[x_2]^{(\rho_2)} \times [x_3]^{(\rho_3)}\rtimes\s$ is unitarizable
(since we are in complementary series or its ends). Therefore, $\pi$ is unitarizable.

Suppose $\ell=2$. We may assume that $\pi$ is a subquotient of $[x_1]^{(\rho_1)}\times[x_2]^{(\rho_1)} \times [x_3]^{(\rho_2)} \rtimes\s$.
Then, $X_{\rho_2}(\pi)$ is an irreducible unitarizable subquotient of $[x_3]^{(\rho_2)} \rtimes\s$.
This implies $x_3\leq \alpha_2$. Furthermore, $\pi$ is an irreducible subquotient of
$[x_3]^{(\rho_2)} \rtimes X_{\rho_1}(\pi)$. Since $X_{\rho_1}(\pi)$ is unitarizable and $x_3\leq\alpha_2$,
this implies that $\pi$ is unitarizable (again we are in complementary series or its ends).
\end{proof}

\begin{lemma} \label{k=2}
Let $\pi$ be a weakly real irreducible representation of a classical group.
Suppose that some $X_{\rho_i}(\pi)$ is a non-unitarizable subquotient of $\theta_1\times\dots\times\theta_k\rtimes\s$,
where $\theta_i\in\Cusp $ and
$$
k\leq 2.
$$
Then, $\pi$ is not unitarizable.
\end{lemma}

\begin{proof}
Suppose on the contrary that $\pi$ is unitarizable. Denote $\rho_i$ simply by $ \rho$ and let
$X_\rho^c=\Cusp\setminus X_\rho.$ We also denote $\alpha=\alpha_{\rho,\s}$. Now $X_\rho(\pi)$ is a subquotient of
$$
[x_1]^{(\rho)}\times\dots\times[x_k]^{(\rho)}\rtimes\s,
$$
where $x_i\geq0 $ and $k\leq2$. Denote $\pi_\rho=X_\rho(\pi)$ and $\pi_\rho^c=X_\rho^c(\pi)$.
Clearly
$$
\pi= \Psi_{X_{\rho},X_{\rho}^c}( \pi_\rho ,\pi_\rho^c ).
$$

If $k=1$, then non-unitarizability of $\pi_\rho $ implies $\pi_\rho \cong [x_1]^{(\rho)}\rtimes\s$ where $x_1>\alpha$.
Now Lemma \ref{J-ind} implies
$$
\pi\cong [x_1]^{(\rho)}\rtimes\pi_\rho^c .
$$
This cannot be unitarizable, since we can deform $x_1$ to the right as far as we want.
We get a contradiction (with the fact that unitarizability can show up only in bounded domains -- see \cite{MR733166} for more details).

Consider now the case $k=2$. We shall suppose as usually $0\leq x_1\leq x_2$. Recall that $\pi_\rho $ is a subquotient of
$$
[x_1]^{(\rho)}\times[x_2]^{(\rho)}\rtimes\s.
$$
We consider several cases.
The first is
$$
\alpha=0.
$$
The non-unitarizability of $\pi_\rho $ implies that $x_1+x_2>1$.
This implies $\pi_\rho \cong [x_1]^{(\rho)}\times[x_2]^{(\rho)}\rtimes\s$. Now Lemma \ref{J-ind} implies
$$
\pi\cong [x_1]^{(\rho)}\times[x_2]^{(\rho)}\rtimes\pi_\rho^c.
$$
Further, we can deform $x_1$ to $x_2$, use the unitary parabolic reduction and get a contradiction with the
unitarizability in the case of general linear groups (more precisely, with the existence of the complementary series there).

It remains to consider the case
$$
\alpha>0.
$$
First recall that Theorem 1.2 of \cite{MR3969882} implies that
\begin{equation} \label{Psi1}
\Psi_{X_{\rho},X_{\rho}^c}( \tau ,\pi_\rho^c ) \text{ is not unitarizable}
\end{equation}
for any non-unitarizable irreducible subquotient $\tau$ of $ [\alpha]^{(\rho)}\times[\alpha+1]^{(\rho)}\rtimes\s $.
Therefore, we assume
$$
(x_1,x_2)\ne (\alpha,\alpha+1)
$$
in the sequel.

Consider first the case
$$
\alpha=\tfrac12.
$$
Since $\pi_\rho $ is not unitarizable, we have $x_1>\tfrac12$ or $x_2>\tfrac12$.

Suppose $x_i\ne \tfrac12$ for $i=1,2$.

Assume first that $x_2\pm x_1\ne1$. Then, $\pi\cong [x_1]^{(\rho)}\times [x_2]^{(\rho)}\rtimes\s$. Now Lemma \ref{J-ind} implies
$$
\pi\cong [x_1]^{(\rho)}\times [x_2]^{(\rho)}\rtimes\pi_\rho^c.
$$
If $x_1>\tfrac12$, then we can deform $x_1$ to $x_2$, switch one $x_2$ to $-x_2$,
use the unitary parabolic reduction and get a contradiction (with existence of complementary series
for general linear group), which implies that $\pi$ cannot be unitarizable.

If $x_1<\tfrac12$, then we can deform $x_2$ to get $x_2+\e x_1=1$ for some $\e\in\{\pm1\}$,
and take there an irreducible subquotient denoted again by $\pi$ (which must be unitarizable).
Now $\pi_\rho \cong \tau \rtimes\s$ for some irreducible subquotient of the reducible representation
$[\e x_1]^{(\rho_1)}\times [x_2]^{(\rho_1)}$. Lemma \ref{J-ind} implies
$$
\pi\cong \tau \rtimes\pi_\rho^c.
$$
Now we can deform $\tau$ to exponents $(\tfrac12,\tfrac32)$.
The properties of the Jantzen decomposition and \eqref{Psi1}
imply that $\pi$ is not unitarizable (since in the limit we have a non-unitarizable subquotient).

Assume now that $x_2+\e x_1=1$ for some $\e\in\{\pm1\}$. Then, in the same way as above we get
$$
\pi\cong \tau \rtimes\pi_\rho^c
$$
for some irreducible subquotient of the reducible representation $[\e x_1]^{(\rho)}\times [x_2]^{(\rho)}$.
Now we finish this case as the previous one.

It remains to consider the case $x_1=\tfrac12$.
Then, $\pi_\rho \cong [x_2]^{(\rho)}\rtimes\theta$ where $\theta$ is an irreducible subquotient of
$[\tfrac12]^{(\rho_2)}\rtimes\s$ (recall $x_2\ne\tfrac32$). Lemma \ref{J-ind} implies
$$
\pi\cong [x_2]^{(\rho)}\rtimes \Psi_{X_{\rho},X_{\rho}^c}( \theta ,\pi_\rho^c ).
$$
We now deform $x_2$ to $\tfrac32$, and in a similar way as before, we get a contradiction.

It remains to consider the case
$$
\alpha\geq1.
$$
First assume
$$
x_2>\alpha.
$$
We know from Proposition 2.2 of \cite{MR3711838} that $x_1\leq \alpha$ and $x_2-x_1\leq1$, which implies $x_2\leq \alpha+1$.
Therefore, if $x_2=\alpha+1$, then $x_1=\alpha$. Then, we know that $\pi$ is not unitarizable by \eqref{Psi1}.
It remains to consider the case
$$
x_2<\alpha+1.
$$

Let
$$
x_2>\alpha.
$$
Consider first the case $x_1= \alpha$. Then, $\pi_\rho =[x_2]^{(\rho)}\rtimes \theta$ where $\theta$ is $
\delta([\alpha]^{(\rho)}; \s)$ or $L([\alpha]^{(\rho)}; \s)$.
One directly gets that
$$
\pi\cong [x_2]^{(\rho)}\rtimes \Psi_{X_{\rho},X_{\rho}^c}( \theta ,\pi_\rho^c ).
$$
We now deform $x_2$ to $\alpha+1$, and in a similar way as before, we get a contradiction.

Therefore, we need to consider the case
$$
x_1<\alpha.
$$

Assume first that $x_2-x_1=1$. Then, a short analysis using Lemma \ref{J-ind} implies that
$\pi\cong \tau\rtimes\pi_\rho^c $, where $\tau$ is an irreducible subquotient of $[x_1]^{(\rho_1)}\times [x_2]^{(\rho_2)}$.
Now deforming $\tau$ to exponents $\alpha,\alpha+1$ and using the properties of the Jantzen decomposition,
we would get a contradiction.

If $x_2-x_1<1$, then similarly we get $\pi\cong [x_1]^{(\rho_1)}\times [x_2]^{(\rho_2)}\rtimes\pi_\rho^c$.
We can now deform (increase) $x_2$ to the previous case, consider a limit, and repeat the above argument.
Therefore, we get again a contradiction.

Consider now the case
$$
x_2=\alpha.
$$
We need to consider the case $\alpha-1<x_1\leq\alpha$.
Then, we know that $\pi_\rho \cong [x_1]^{(\rho)}\rtimes \theta$ for some irreducible subquotient
$\theta$ of $[\alpha]^{(\rho)}\rtimes\s$, which implies by Lemma \ref{J-ind}
$$
\pi\cong [x_1]^{(\rho)}\rtimes \Psi_{X_{\rho},X_{\rho}^c}( \theta ,\pi_\rho^c ).
$$
Now we can deform $x_1$ to $\alpha+1$, and get a contradiction in the same way as in the previous cases
(here we use that $[\alpha]\rtimes L([\alpha];\sigma)$ and $[\alpha]\rtimes \delta([\alpha];\sigma)$ are irreducible, and properties of the Jantzen decomposition).

It remains to consider the case
$$
x_2<\alpha
$$
and the region
$$
1-x_1<x_2<x_1+1.
$$
Then, representations $\pi\cong [x_1]^{(\rho)}\times [x_2]^{(\rho)}\rtimes\pi_\rho^c $, and moreover,
for the exponents satisfying above relations, $[x_1]^{(\rho)}\times [x_2]^{(\rho)}\rtimes\pi_\rho^c$ form
a continuous family of irreducible hermitian representations. Consider the point $\tfrac12<x_1=x_2<\alpha$ of the above region.
After switching $x_1$ to $-x_1$, the unitary parabolic reduction implies that this representation is not unitarizable (since $[x_1]\times[-x_1]$ is not unitarizable).
Therefore, the whole family is non-unitarizable.
This completes the proof of the
lemma\end{proof}

We infer

\begin{corollary}
Let $\pi$ be a weakly real irreducible subquotient of $\theta_1\times\dots\times\theta_k\rtimes\s$,
where $\theta_i\in\Cusp $, $k\leq 3$. Then, $\pi$ is unitarizable if and only if all $X_{\rho_i}(\pi)$ in the
Jantzen decomposition of $\pi$ are unitarizable.
\qed
\end{corollary}

Having in mind the results of \cite{MR2046512}, \cite{MR2767523}, \cite{MR3969882} and the present paper,
we conjecture that the above corollary holds for any $k$, namely

\begin{conjecture}
Jantzen decomposition preserves unitarizability in both directions.
\end{conjecture}

\appendix

\chapter{The Arthur Packet of \texorpdfstring{$L(\nu^{\alpha}\rho, \nu^{\alpha-1}\rho;\delta(\nu^\alpha\rho;\s))$\\
 by Colette M\oe glin}{}} \label{appendix-M}

 \section{The representations}

Let $\s$ be an irreducible cuspidal representation of a classical groups $G=G(F)$ and let $\rho$ be an irreducible unitarizable cuspidal representation of a $\GL(d,F)$.
Let $\alpha\in \tfrac12\Z$ be strictly greater then $1$. Suppose that $\nu^\alpha\rho\rtimes\s$ is reducible. Denote by
$
\delta(\nu^\alpha\rho;\s)
$
the unique irreducible square-integrable subquotient of the last representation.

The induced representation
$$
\nu^\alpha\rho\times\nu^{\alpha-1}\rho\rtimes \delta(\nu^\alpha\rho;\s)
$$
has a unique irreducible quotient which will be denoted by $\pi$ (this is the Langlands quotient).

\section{The parameters}

Denote by $\phi$ the Langlands parameter of $\s$. This parameter is an admissible homomorphism of $W_F\times SL(2,\C)$ into $^LG$.
We denote also by $\rho$ the morphism of $W_F$ into $\GL(d, \C)$ which parameterizes the cuspidal representation $\rho$.

The assumption on $\alpha$ regarding the reducibility implies that there exists an admissible homomorphism $\phi_-$ for a classical group of smaller rank than $G$,
which is discrete and for which holds
$$
\phi=\phi_- \ \oplus \ \rho \otimes R[2\alpha-3] \ \oplus \ \rho \otimes R[2\alpha-1],
$$
where $R[n]$ denotes the irreducible $n$-dimensional algebraic representation of the group $SL(2,\C)$.
For this it is necessary that $2\alpha-3\geq0$, i.e. $\alpha\geq \tfrac32$.
If $\alpha= \tfrac32$, there one does not need to write the
term $\rho \otimes R[2\alpha-3]$.

Moreover $\phi_-$ is non-trivial if $G$ is quasi-split, but this is not used in the sequel.

One defines now the Arthur parameter $\psi$. It is the following morphism of $W_F\times SL(2,\C) \times SL(2,\C)$ into $^LG$:
$$
\psi:= \phi_- \otimes R[1] \ \oplus \ \rho \otimes R[2\alpha+1] \otimes R[1] \ \oplus \ \rho \otimes R[1] \otimes R[2\alpha+1] .
$$

\section{The result}

\

{\bf Claim}: The representation $\pi$ is in the Arthur packet associated to $\psi$.

Denote by $\psi_+$ the morphism where one changes a dimension $2\alpha+ 1$ into $2\alpha+ 3$. We have two possibilities and we make the choice that suits us for the proof
(the result is obviously independent of the choice) That is:
$$
\psi_+:= \phi_- \otimes R[1] \ \oplus \ \rho \otimes R[2\alpha+3] \otimes R[1] \ \oplus \ \rho \otimes R[1] \otimes R[2\alpha+1] .
$$
From the section 3.1.2 of \cite{MR2767522}, we know that the representations of the packet associated with $\psi$ are obtained from those associated with $\psi_+$ as follows:

{\sl Let $\tau_+$ be a representation in the packet parameterized by $\psi_+$. Then the term $Jac_{\nu^{\alpha+1}\rho}( \tau_+)$ is either zero or an irreducible representation.
In the second case, this irreducible representation is in the packet associated with $\psi$ and all the representations of the packet associated with $\psi$ are
obtained in this way and for a unique choice of $\tau_+$.}

We will construct in the packet associated with $\psi_+$, a representation denoted $\pi_+$ which will give $\pi$ by the procedure above.

Indeed, according to the construction in the section 3.1.1 in
\cite{MR2767522} (which is a summary of \cite{MR2209850} for the case which matter here of $\ell(\psi)=0$),
 the packet associated with $\psi_+$ contains
the representations of the packet obtained as the unique irreducible sub-representations of the induced representations
$$
\nu^{-\alpha}\rho \times \nu^{-\alpha-1}\rho\rtimes\tau_-,
$$
where $\tau_-$ runs over the representations of the packet
$$
\psi_-:= \phi_- \otimes R[1] \ \oplus \ \rho \otimes R[2\alpha+3] \otimes R[1] \ \oplus \ \rho \otimes R[1] \otimes R[2\alpha-3] .
$$
This even describes all representations of the packet associated with $\psi_+$ if $2\alpha-3>0$, and some are missing if $2\alpha-3=0$, but it does not matter to us.

According to the same reference, the packet associated with $\psi_-$ is formed of the irreducible subrepresentations
of the induced representations
$$
\nu^{\alpha+1}\rho\rtimes\tau',
$$
where $\tau'$ runs over the representations of the packet
$$
\psi':= \phi_- \otimes R[1] \ \oplus \ \rho \otimes R[2\alpha+ 1] \otimes R[1] \ \oplus \ \rho \otimes R[1] \otimes R[2\alpha-3] .
$$
This packet contains the representation $\delta(\nu^\alpha\rho;\s) $, the point is somewhat subtle: the representations of the packet $\psi'$ are the unique irreducible
subrepresentations of the induced representation $ \nu^{\alpha}\rho\rtimes\tau''$ with $\tau''$ in the packet associated to
$$
\phi_- \otimes R[1] \ \oplus \ \rho \otimes R[2\alpha-1] \otimes R[1] \ \oplus \ \rho \otimes R[1] \otimes R[2\alpha-3] .
$$
But this package contains the cuspidal representation $\s$, hence the assertion.

We have thus constructed a representation $\pi_+$ which is an irreducible sub representation
of the induced representation:
$$
\nu^{-\alpha}\rho\times \nu^{-\alpha+1}\rho \times \nu^{\alpha+1}\rho \rtimes \delta(\nu^{\alpha}\rho;\s).
$$
We can exchange $\nu^{\alpha+1}\rho$ and $\nu^{-\alpha}\rho\times \nu^{-\alpha+1}\rho$ (i.e. commute), so $Jac_{\nu^{\alpha+1}\rho} (\pi_+)$ is non-zero and actually is a submodule of the induced representation
$$
\nu^{-\alpha}\rho\times \nu^{-\alpha+1}\rho \rtimes \delta(\nu^{\alpha}\rho;\s).
$$
Since this induced representation has only one irreducible sub representation, which is $\pi$, and one knows a priori that the representation
$Jac_{\nu^{\alpha+1}\rho}(\pi_+)$ is irreducible, we identify this Jacquet module with $\pi$. We have thus shown that $\pi$ is in the packet associated with $\psi$.

\chapter{Jacquet Module of \texorpdfstring{$L(\nu^{\alpha}\rho, \nu^{\alpha-1}\rho;\delta(\nu^\alpha\rho;\s))$}{}} \label{App-JM}
As in the body of the paper, let $\rho$ be a $F'/F$-self-contragredient irreducible cuspidal representation of a general linear group
(which we denote here by $\GL(n_\rho,F')$) and $\sigma$
an irreducible cuspidal representation of a classical group. We assume that $\alpha=\alpha_{\rho,\s}\ge\tfrac32$.
In this appendix we calculate the Jacquet modules of the distinguished representation
\[
\pi_0=L(\nu^{\alpha}\rho, \nu^{\alpha-1}\rho;\delta(\nu^\alpha\rho;\s)).
\]
Although the result is not used in the body of the paper, we opted to include this computation as it might be useful in the future.
For simplicity, we suppress $\rho$ from the notation.
Thus,
$$
\pi_0=L([\alpha],[\alpha-1]; \delta([\alpha];\s)).
$$
The aim of this appendix is to prove the following
\begin{proposition} \label{mu-star-of-Art-type-rep}
We have
\begin{gather*}
\mu^*(L([\alpha],[\alpha-1]; \delta([\alpha];\s)))
=1\otimes L([\alpha],[\alpha-1]; \delta([\alpha];\s))\\
+[-\alpha]\otimes L([\alpha-1];\delta([\alpha];\s))+
[\alpha]\otimes L([\alpha-1],[\alpha];\s)+
\\
[-\alpha]\times[\alpha]\otimes[\alpha-1]\rtimes\s+
L([-\alpha],[-\alpha+1])\otimes \delta([\alpha];\s)+
\delta([\alpha-1,\alpha])\otimes L([\alpha]; \s)\\
+L([-\alpha], [-\alpha+1])\times[\alpha]\otimes\s+
[-\alpha]\times \delta([\alpha-1,\alpha])\otimes\s.
\end{gather*}
\end{proposition}

\begin{proof}
From \eqref{to-use-below-1}, we know that
\[
\pi:=[\alpha]\rtimes L ([\alpha-1]; \delta([\alpha];\s))
\]
is reducible.
Furthermore we know that $\pi_0$ is a subquotient.
We have
\begin{gather*}
\mu^*(\pi)=
(1\otimes [\alpha]+[\alpha ]\otimes1+
[-\alpha]\otimes1)\rtimes\\
\Big(1\otimes L([\alpha-1];\delta([\alpha];\s))\\
+[\alpha]\otimes[\alpha-1]\rtimes\s+
[-\alpha+1]\otimes \delta([\alpha];\s)\\
+[-\alpha+1]\times[\alpha]\otimes\s +
\delta([\alpha-1,\alpha])\otimes\s\Big)
\end{gather*}
which (after multiplication) we write as
\[
1\otimes\pi \ +
\]
\begin{gather*}
[\alpha]\otimes L([\alpha-1];\delta([\alpha];\s))+
[-\alpha]\otimes L([\alpha-1];\delta([\alpha];\s))\\
+[\alpha]\otimes[\alpha]\times[\alpha-1]\rtimes\s+
[-\alpha+1]\otimes [\alpha]\rtimes\delta([\alpha];\s)+
\end{gather*}
\begin{gather*}
[\alpha]\times [\alpha]\otimes[\alpha-1]\rtimes\s+
[\alpha]\times[-\alpha+1]\otimes \delta([\alpha];\s)\\+
[-\alpha]\times[\alpha]\otimes[\alpha-1]\rtimes\s+
[-\alpha]\times[-\alpha+1]\otimes \delta([\alpha];\s)\\+
[-\alpha+1]\times[\alpha]\otimes[\alpha]\rtimes\s+
\delta([\alpha-1,\alpha])\otimes[\alpha]\rtimes \s+
\end{gather*}
\begin{gather*}
[\alpha]\times [-\alpha+1]\times[\alpha]\otimes\s +
[\alpha]\times \delta([\alpha-1,\alpha])\otimes\s\\
[-\alpha]\times [-\alpha+1]\times[\alpha]\otimes\s +
[-\alpha]\times \delta([\alpha-1,\alpha])\otimes\s
\end{gather*}
or (after decomposing into irreducible representations) as
\[
1\otimes \pi \ +
\]
\begin{gather*}
2\cdot[\alpha]\otimes L([\alpha-1];\delta([\alpha];\s)) +
\overbrace{ [-\alpha]\otimes L([\alpha-1];\delta([\alpha];\s))}^{\widehat\omega_1}+
[\alpha]\otimes L([\alpha-1,\alpha];\s)+\\
[\alpha]\otimes L([\alpha-1],[\alpha];\s)+
[\alpha]\otimes\delta_{\spsi}([\alpha-1],[\alpha];\s)+
\overbrace{[-\alpha+1]\otimes [\alpha]\rtimes\delta([\alpha];\s)}^{\omega_2}+
\end{gather*}
\begin{gather*}
\overbrace{[\alpha]\times [\alpha]\otimes[\alpha-1]\rtimes\s}^{\omega_2'}+
[\alpha]\times[-\alpha+1]\otimes \delta([\alpha];\s)+
[-\alpha]\times[\alpha]\otimes[\alpha-1]\rtimes\s+\\
L([-\alpha],[-\alpha+1])\otimes \delta([\alpha];\s)+
\delta([-\alpha,-\alpha+1])\otimes \delta([\alpha];\s)+
\\
[-\alpha+1]\times[\alpha]\otimes\delta([\alpha];\s)+
[-\alpha+1]\times[\alpha]\otimes L([\alpha];\s)+\\
\delta([\alpha-1,\alpha])\otimes\delta([\alpha]; \s)+
\delta([\alpha-1,\alpha])\otimes L([\alpha]; \s)+
\end{gather*}
\begin{gather*}
\overbrace{[\alpha]\times [-\alpha+1]\times[\alpha]\otimes\s}^{\omega_1} +
\overbrace{[\alpha]\times \delta([\alpha-1,\alpha])\otimes\s}^{\omega_3'}+\\
\overbrace{\delta([-\alpha,-\alpha+1])\times[\alpha]\otimes\s}^{\omega_3}+
\overbrace{L([-\alpha], [-\alpha+1])\times[\alpha]\otimes\s}^{\widehat\omega_2}+
\overbrace{[-\alpha]\times \delta([\alpha-1,\alpha])\otimes\s}^{\widehat\omega_2'}.
\end{gather*}
Let $\gamma$ be the irreducible subquotient of $\pi$ such that $s_{\GL}(\gamma)\ge\omega_1$.
Transitivity of Jacquet modules implies that $s_{(n_\rho)}(\gamma)$ contains an irreducible term of the form $[-\alpha+1]\otimes -$.
However, by the above formula, the only irreducible subquotient of $s_{(n_\rho)}(\pi)$ of this form is $\omega_2$.
Therefore, $\omega_2\leq \mu^*(\gamma)$. Similarly, considering $s_{(2n_\rho)}(\pi)$ and terms of form $[\alpha]\times[\alpha]\otimes-$,
we conclude that $\omega_2'\leq \mu^*(\gamma)$.
This and transitivity of Jacquet modules implies that in the Jacquet module of $\gamma$ we have
$[-\alpha+1]\otimes [-\alpha]\otimes[\alpha]\otimes\sigma$.
Again, transitivity of Jacquet modules implies that there exists an irreducible subquotient of $s_{(3n_\rho)}(\pi)$
which has $[-\alpha+1]\otimes [-\alpha]\otimes[\alpha]\otimes\sigma$ in its Jacquet module.
The above formula implies that the only possibility for such subquotient is $\omega_3$.
Therefore, $\omega_3\leq \mu^*(\gamma)$. Similarly, $\omega_2'\leq \mu^*(\gamma)$ implies $\omega_3'\leq \mu^*(\gamma)$ ($\omega_3'$
is the only term in $s_{(3n_\rho)}(\pi)$ with all the exponents in the cuspidal support positive).
Therefore $\omega_1+\omega_3+\omega_3'\leq s_{(3n_\rho)}(\gamma)$.

Let now $\widehat\gamma$ be the irreducible subquotient of $\pi$
such that $\mu^*(\widehat\gamma)\ge\widehat\omega_1$.
The formula for $s_{(n_\rho)}(\pi)$ implies that $\omega_1$ is a direct summand (consider infinitesimal character in the sense of Bernstein center). Therefore, it is also a direct summand in the Jacquet module $s_{(n_\rho)}(\widehat\gamma)$. Frobenius reciprocity now implies that $\widehat\gamma$ embeds into $ [-\alpha]\rtimes L([\alpha-1];\delta([\alpha];\s))$, which implies $\widehat\gamma=\pi_0$.

From \ref{subsec-w-prop-2-a-} we see that
$
s_{(2n_\rho)}(L([\alpha-1];\delta([\alpha];\s)))=[-\alpha+1]\times[\alpha]\otimes\s +\delta([\alpha-1,\alpha])\otimes\s.
$
This implies that $[-\alpha]\otimes[-\alpha+1]\otimes [\alpha]\otimes\sigma$ and $[-\alpha]\otimes [\alpha]\otimes [\alpha-1]\otimes\sigma$ are in the Jacquet module of $\pi_0$. The only two irreducible pieces of $s_{(3n_\rho)}(\pi)$ having these terms in their Jacquet modules are $\widehat\omega_2$ and $\widehat\omega_2'$ respectively. Thus, $\omega_2+\omega_2'\leq s_{(3n_\rho)}(\pi_0)$.
Now the fact that $s_{(3n_\rho)}(\pi)=\omega_1+\omega_3+\omega_3'+\widehat\omega_2+\widehat\omega_2'$ and the two above inequalities that we have proved for $s_{(3n_\rho)}(\gamma)$ and $s_{(3n_\rho)}(\pi_0)$ imply that $\pi=\gamma+\pi_0$ and
$$
s_{(3n_\rho)}(\pi_0)=\widehat\omega_2+\widehat\omega_2'.
$$
This implies that the semi-simplification of the minimal non-trivial Jacquet modules of $\pi_0$ is
\begin{gather*}
[\alpha]\otimes [\alpha-1]\otimes [-\alpha]\otimes \s+
[\alpha]\otimes [-\alpha]\otimes [\alpha-1]\otimes \s+
[\alpha]\otimes [-\alpha]\otimes [-\alpha+1]\otimes \s+
\\
 [-\alpha]\otimes [\alpha]\otimes [\alpha-1]\otimes \s+
[-\alpha]\otimes [\alpha]\otimes [-\alpha+1]\otimes \s+
[-\alpha]\otimes [-\alpha+1]\otimes [\alpha]\otimes \s.
\end{gather*}
We get $s_{(2n_\rho)}(\pi_0)\geq [-\alpha]\times[\alpha]\otimes[\alpha-1]\rtimes\s+
L([-\alpha],[-\alpha+1])\otimes \delta([\alpha];\s)+
\delta([\alpha-1,\alpha])\otimes L([\alpha]; \s)$ directly from the above formula. The fact that the minimal non-trivial Jacquet module of the right hand side has length 6 implies that the above inequality is actually equality, and therefore we have computed that
$
s_{(2n_\rho)}(\pi_0)=
$
$$
[-\alpha]\times[\alpha]\otimes[\alpha-1]\rtimes\s+
L([-\alpha],[-\alpha+1])\otimes \delta([\alpha];\s)+
\delta([\alpha-1,\alpha])\otimes L([\alpha]; \s).
$$

We know $s_{(n_\rho)}(\pi_0)\geq \widehat\omega_1$. From the formula for $s_{(2n_\rho)}(L([\alpha-1],[\alpha];\s))$ in \ref{subsec-w-prop-2-a-}, we get $s_{(n_\rho)}(\pi_0)\geq [\alpha]\otimes L([\alpha-1],[\alpha];\s)$ (the representation $[\alpha-1]\times[\alpha]\rtimes\sigma$ is regular). Now the inequality
$s_{(n_\rho)}(\pi_0)\geq \widehat\omega_1+ [\alpha]\otimes L([\alpha-1],[\alpha];\s)$ that we have got, is actually the equality
$$
s_{(n_\rho)}(\pi_0)= \widehat\omega_1+ [\alpha]\otimes L([\alpha-1],[\alpha];\s)
$$
since the minimal non-trivial Jacquet module of the right hand side has length 6.
This completes a proof of the proposition.
\end{proof}

\backmatter

\bibliographystyle{plain}

\def\cprime{$'$}

\end{document}